\documentclass{article}

\usepackage[utf8]{inputenc}
\usepackage{amsmath,amsfonts,amssymb}
\usepackage{amsthm,upref}
\usepackage[a4paper]{geometry}
\usepackage{color} 
\usepackage[dvipsnames]{xcolor}
\usepackage{graphicx}
\usepackage{subfigure}
\usepackage{tabularx,float}
\usepackage[hidelinks]{hyperref}
\usepackage{comment} 
\usepackage{tabularx} 
\usepackage{enumerate,mathrsfs}

\usepackage{fullpage}

\newtheorem{thm}{Theorem}[section]

\newtheorem{proposition}[thm]{Proposition}

\newtheorem{lemma}[thm]{Lemma}
\newtheorem{remark}[thm]{Remark}

\newcommand{\eps}{\varepsilon}

\newcommand{\R}{\mathbb{R}}
\newcommand{\me}{\mathrm{e}}

\usepackage{todonotes}
\newcommand\SJ[1]{{\color{black}{#1}}} 
\newcommand\SJJ[1]{{\color{black}{#1}}} 

\title{Geometric blow-up of a dynamic Turing instability in the Swift-Hohenberg equation}

\author{S.~Jelbart, F.~Hummel$^\ast$ \& C.~Kuehn}
\author{S.~Jelbart\thanks{Department of Mathematics, Technical University of Munich, 85748 Garching b.~M\"unchen, Germany.} \footnote{Corresponding author. Email: jelbart@ma.tum.de}, F.~Hummel$^\ast$ \& C.~Kuehn$^\ast$}

\date{\today}

\begin{document}
	
	\maketitle
	
	
	\begin{abstract}
		We present a rigorous analysis of the slow passage through a Turing bifurcation in the Swift-Hohenberg equation using a novel approach based on geometric blow-up. We show that the formally derived multiple scales ansatz which is known from classical modulation theory can be adapted for use in the fast-slow setting, by reformulating it as a blow-up transformation. This leads to dynamically simpler modulation equations posed in the blown-up space, via a formal procedure which directly extends the established approach to the time-dependent setting. The modulation equations take the form of non-autonomous Ginzburg-Landau equations, which can be analysed within the blow-up. The asymptotics of solutions in weighted Sobolev spaces are given in two different cases: (i) A symmetric case featuring a delayed loss of stability, and (ii) A second case in which the symmetry is broken by a source term. In order to characterise the dynamics of the Swift-Hohenberg equation itself we derive rigorous estimates on the error of the dynamic modulation approximation. These estimates are obtained by bounding weak solutions to an evolution equation for the error which is also posed in the blown-up space. Using the error estimates obtained, we are able to infer the asymptotics of a large class of solutions to the dynamic Swift-Hohenberg equation. We provide rigorous asymptotics for solutions in both cases (i) and (ii). We also prove the existence of the delayed loss of stability in the symmetric case (i), and provide a lower bound for the delay time.
	\end{abstract}
	
	

	\section{Introduction}
	\label{sec:Introduction}

	Following the first study of spatially induced instabilities in reaction-diffusion systems \cite{Turing1952}, so-called \textit{Turing instabilities} have been identified in a wide variety of biological and bio-chemical models \cite{Murray1989}, hydrodynamical systems \cite{Chossat2012,Manneville1995,Straughan2013}, problems in optics \cite{Lega1994,Lega1995,Longhi1996}, population dynamics \cite{Genieys2006,Volpert2014} and convection \cite{Andreev2012,Lappa2009,Straughan2008,Straughan2013}. Turing instabilities are characterised by the destabilisation of a spatially homogeneous steady-state in response to spatial (but not temporal) perturbations under parameter variation; see e.g.~\cite[Ch.~9]{Schneider2017}. Many authors have shown using \textit{modulation theory} that the corresponding bifurcation is associated with the emergence of spatially inhomogeneous steady states or ``patterns", which can be approximated close to the onset of instability by multiple-scale, quasi-periodic functions of the form
	\begin{equation}
		\label{eq:modulation_ansatz}
		u(x,t) \approx A(\delta x, \delta^2 t) \me^{i k_c x} + c.c. + \ldots ,
	\end{equation}
	where $A(\delta x, \delta^2 t) \in \mathbb C$, the small parameter $0 < \delta \ll 1$ is related to the bifurcation parameter, and $k_c \in \R$ is the so-called \textit{critical mode} corresponding to the onset of linear instability \cite{diPrima1971,Schneider2017}. It is known from seminal works, e.g.~\cite{diPrima1971,Newell1969,Segel1969}, that the modulation equation which governs the leading order approximation near a Turing instability takes the \textit{universal} form
	\begin{equation}
		\label{eq:GL_eqn_intro}
		\partial_{\bar t} A = \alpha_1 \partial_{\bar x}^2 A + \alpha_2 A + \alpha_3 A |A|^2 ,
	\end{equation}
	where $\bar x = \delta x$, $\bar t = \delta^2 t$ and the coefficients $\alpha_1, \alpha_2, \alpha_2 \in \mathbb C$. Equation \eqref{eq:GL_eqn_intro} is called the \textit{Ginzburg-Landau (GL)} equation, and was originally derived as a model for superconductivity in \cite{Ginzburg2002}. It is universal in the sense that it appears as a modulation equation governing the emergence of patterned states in a wide variety of applications which differ only through the values taken by the coefficients $\alpha_1, \alpha_2, \alpha_3$. We refer to \cite{Aranson2002,Saarloos1994,Schneider2017} for detailed reviews. 
	
	
	Rigorous justifications in the form of quantitative measures for the validity the GL approximation did not appear until the early 1990s. The first rigorous approximation results concerned the validity of the GL reduction near the Turing instability in the \textit{Swift-Hohenberg (SH) equation}
	\begin{equation}
		\label{eq:sh_intro}
		\partial_t u = - (1 + \partial_x^2)^2 u + v u - u^3 ,
	\end{equation}
	which were derived in \cite{Collet1990} and shortly thereafter by a different approach in \cite{Kirrmann1992}. In the special case of ``essential solutions", i.e.~solutions which exist globally both forward and backward in time, the GL reduction has been shown to arise as an exact reduction via Lyapunov-Schmidt reduction, c.f.~\cite{Mielke1992,Mielke1998}. For larger solution spaces of interest and in applications posed on large or unbounded spatial domains, however, exact reduction methods based on center manifold theory or Lyapunov-Schmidt reductions are typically not useful (in the case of large domains), or even impossible (in the case of unbounded domains). Pattern forming systems arise naturally as problems posed on large domains because the patterns emerging near a Turing instability have a spatial frequency which is small by comparison to the size of the domain. Modulation theory has emerged as an alternative to approaches based on center manifold theory or Lyapunov-Schmidt reduction, which applies for problems posed on large or unbounded domains. In this more general setting, attractivity results were given originally given in \cite{Eckhaus1993,Schneider1995b}, paving the way for subsequent results including upper-semicontinuity for the attractor \cite{Mielke1995}, shadowing properties and global existence of solutions \cite{Schneider1994b,Schneider1999}. In addition to these efforts, existence, stability and bifurcations of special solutions to the SH equation including spatially periodic and modulating fronts have been studied extensively by many authors; see \cite{Collet1990b,Eckmann2002,Eckmann1997,Mielke1997b,Schneider1996b,Schneider2017} for early works in this direction. 
	
	\
	
	In this work we are interested in the dynamic or \textit{fast-slow} extension of the SH equation \eqref{eq:sh_intro}. Specifically, we consider the SH equation with a slow parameter drift and a spatially inhomogeneous perturbation:
	\begin{equation}
		\label{eq:sh_dynamic}
		\begin{split}
			\partial_t u &= - (1 + \partial_x^2)^2 u + v u - u^3 + \eps \mu(x) , \\
			\partial_t v &= \eps , \\ 
		\end{split}
	\end{equation}
	where $u = u(x,t) \in \mathbb R$, $v = v(t) \in \mathbb R$, $x \in \mathbb R$, $t \geq 0$, $0 < \eps \ll 1$ is a (singular) perturbation parameter, and $\mu(x) \in \mathbb R$ is a source term which we assume to be $2\pi$-periodic of the form
	\begin{equation}
		\label{eq:mu}
		\mu(x) := \sum_{m \in I_N} \nu_m \me^{imx} .
	\end{equation}
	The coefficients $\nu_m \in \mathbb C$ are assumed to satisfy a reality condition $\nu_{-m} = \overline{\nu_m}$, and the index set is $I_N := \{-N, \ldots, N\} \subset \mathbb N$. 
	For $\eps = 0$, equation \eqref{eq:sh_dynamic} reduces to the classical SH equation \eqref{eq:sh_intro}, which is invariant under the symmetries $u \mapsto -u$ and $x \mapsto x + \alpha$, $\alpha \in \R$. The source term $\mu(x)$ is included in order to better understand the dynamical implications of these symmetries. There are three distinct cases to consider when $\eps > 0$: (i) the symmetric case $\mu(x) \equiv 0$, (ii) the case $\mu(x) = \nu_0$, which breaks the reflection symmetry $u \mapsto -u$ but not the translation symmetry $x \mapsto x + \alpha$, and (iii) the case that $\mu(x)$ is given by \eqref{eq:mu} with $\nu_m \neq 0$ for some $m \in \{1, \ldots, N\}$ and $\alpha \notin 2 \pi \mathbb Z$, in which case both the reflection and translation symmetries are broken.
	
	Our primary aim is to understand the evolution of solutions with initial value $v(0) < 0$ as they evolve forward in time through a neighbourhood of $v = 0$, where the (static) Turing instability occurs. This scenario may also be referred to as the \textit{slow passage through a Turing bifurcation}. A recent numerical analysis of an advection-reaction-diffusion equation in \cite{Chen2015} indicated that the slow passage through a Turing bifurcation may be associated with a uniquely dynamical phenomenon known as \textit{delayed loss of stability}, in which solutions remain close to a homogeneous and linearly unstable steady state of the limiting system for a substantial time before rapidly transitioning to a patterned state. There exists a large body of research on delayed stability loss phenomena in finite-dimensional fast-slow systems characterised by either (i) oscillatory instabilities (e.g.~Andronov-Hopf bifurcations) in the limiting system, or (ii) the presence of so-called \textit{canards}, i.e.~solutions contained within the intersection of stable and unstable Fenichel slow manifolds; for a sample of seminal works and overviews we refer to \cite{Baesens1991,Benoit1991,Dumortier1996,Fruchard2009,Hayes2016,Krupa2001a,Krupa2001b,Kuehn2015,Neishtadt1987,Neishtadt1988,Szmolyan2001,Wechselberger2012,Wechselberger2019}. The first rigorous results on the slow passage through a Turing instability were recently derived in \cite{Avitabile2020}. In this work the authors showed that under suitable assumptions, most notably the existence of a spectral gap, the local mechanism for delayed stability loss in the slow passage through an $O(2)$ symmetric Turing bifurcation is related to the existence of canard solutions in a finite-dimensional fast-slow ODE obtained via a rigorous local center manifold reduction using the center manifold theory of \cite{Haragus2010,Vanderbauwhede1992}. Following the center manifold reduction, the existence of canards implies a delay mechanism which is described by known results for finite-dimensional fast-slow systems from \cite{Krupa2001c}. These results were applied and numerically validated for a Schnakenberg model appearing in \cite{Brena2014}. 
	
	The results in \cite{Avitabile2020} provide rigorous results on the slow passage through a Turing bifurcation. Different mechanisms for delayed stability loss in infinite-dimensional fast-slow systems have been studied by numerous authors. A delayed Hopf-type instability in a FitzHugh-Nagumo type neural model with diffusion has been studied in \cite{Su1994}. In \cite{Nefedov2003} the authors proved rigorously that delayed stability loss occurs in a class of reaction-diffusion equations. Further results for reaction-diffusion equations later appeared in \cite{Maesschalck2009}, including exact estimates for the delay time associated with the delayed stability loss for the systems considered in \cite{Nefedov2003}. 
	More recently, delayed Hopf-type instabilities similar to those observed in \cite{Su1994} have been analysed using the WKB formalism in \cite{Bilinsky2018}, and formal approximations for the so-called ``space-time buffer curve" along which the delayed transition occurs have been derived and numerically validated in \cite{Goh2022,Kaper2018,Kaper2021}. See \cite{Goh2022} in particular for detailed formal and numerical study of a the slow passage through a Hopf bifurcation in a fast-slow extension of the complex GL equation which is similar in form to the fast-slow extension of the SH equation \eqref{eq:sh_intro} in \eqref{eq:sh_dynamic} (i.e.~the GL equation in \cite{Goh2022} is obtained from \eqref{eq:GL_eqn_intro} by considering slow evolution of a system parameter and including a source term to break certain symmetries). Instabilities of delayed Hopf-type were also analysed in the reference \cite{Avitabile2020} (along with numerous other types of instabilities) using the same center manifold based approach described above, and ``spatio-temporal canards" have been studied in \cite{Avitabile2017,Avitabile2017b,Vo2020}.
	
	\
	
	In this article we develop a fast-slow extension of classical modulation theory, and use it to rigorously analyse the slow passage through the Turing bifurcation in the dynamic SH equation \eqref{eq:sh_dynamic}. Given that we consider the system on an unbounded spatial domain, as is natural for pattern forming systems, the spectral gap condition in \cite{Avitabile2020} is violated and a center manifold reduction is not possible. Indeed, the SH equation \eqref{eq:sh_dynamic}$|_{\eps = 0}$ does not possess a center manifold in rich enough solution spaces. From a methodological point of view, our main contribution is to show that modulation theory can be extended to study dynamic bifurcations in infinite-dimensional fast-slow systems on unbounded spatial domains, after introducing a dynamic generalisation of the multiple-scales ansatz via a novel extension of the so-called \textit{blow-up method}. This approach is motivated by an established notion in finite-dimensional fast-slow systems that the blow-up method can be viewed as a way of formulating dynamic generalisations of singular rescalings, as shown by the seminal works \cite{Dumortier1996,Krupa2001a,Krupa2001c,Krupa2001b} and in many subsequent applications; we refer to the recent review \cite{Jardon2019b} and the many references therein.
	
	In the context of system \eqref{eq:sh_dynamic}, the idea is to resolve degeneracy at the Turing instability at $u_{ss} \equiv 0$, $v = \eps = 0$, using a blow-up transformation which involves a generalisation of the multiple-scales ansatz \eqref{eq:modulation_ansatz} to
	\[
	u(x,t) \approx r(\bar t) A(\bar t, \bar x) \me^{i k_c x} + c.c. + \ldots,
	\]
	where the small parameter $\delta$ has been replaced by a dynamic variable $r(\bar t) \geq 0$ which measures the distance from the static bifurcation point, and $\bar x$ and $\bar t$ are defined by positive, time-dependent transformations of the form $\partial_x = r(\bar t) \partial_{\bar x}$ and $\partial_t = r(\bar t)^2 \partial_{\bar t}$ instead of the simple rescalings \SJ{$\bar x = \delta x$ and $\bar t = \delta^2 t$}. After proposing a natural ansatz in the form of a blow-up transformation, we derive a set of modulation equations in the blown-up space via a formal procedure which is based on the established procedure for deriving modulation equations for the static SH equation in \eqref{eq:sh_intro}. The formal derivation of modulation equations is complicated by comparison to the static theory, due to the non-trivial geometry of the blown-up space and the additional time-dependence of the `small parameter' $r(\bar t)$, which is now a dynamic variable. The modulation equations obtained take the form of non-autonomous GL equations. These equations are posed in the blown-up space in global coordinates, however we also provide local coordinate representations of the equations in three distinguished coordinate charts better suited to applications.
	
	The blow-up method has also been applied to study a finite reduced set of ODE systems describing patterned steady-states of the SH equation in two and three spatial dimensions \cite{McCalla2013}, where the authors prove the existence of stationary spatially localised spots; see \cite{Avitabile2010,Burke2006,Lloyd2009,Lloyd2008,McCalla2010} for more on stationary spatially localised solutions to the SH equation. In the infinite-dimensional setting, a novel extension of the blow-up method for the analysis of singularities in PDEs on bounded domains was recently developed in \cite{Engel2020}, see also \cite{Engel2021}. In this work, the authors apply a spectral Galerkin discretization in space and define a suitable blow-up transformation on a finite-dimensional ODE system obtained after truncating the discretization. The blow-up method developed in the present work is related, but distinct from this approach. Given that we work on an unbounded spatial domain, we require an alternative approach which is more `direct' in the sense that it does not rely upon a preliminary discretization of space. It is also worthy to note that our approach does not depend on the existence of a spectral gap, nor does it involve reduction to a finite-dimensional system or require the existence of a lower dimensional invariant manifold.
	
	\
	
	Rigorous (as opposed to formal) results on the dynamics of system \eqref{eq:sh_dynamic} can only be obtained via modulation theory if the error associated to the approximation can be controlled. We derive rigorous estimates for the error within the blow-up by combining a-priori bounds on weak solutions to a weighted equation for the error obtained in all three coordinate charts. The error estimates are given for a large space of initial conditions in uniformly local Sobolev spaces, which are rich enough to include most of the `usual' solutions of interest in general pattern forming systems \cite{Schneider2017} (see also Section \ref{sec:Background} below). For arbitrary source functions \eqref{eq:mu}, our results are partial in the sense that we are only able to control the error up to a distance which is $O(\eps^{1/2})$ from the static bifurcation value. In the symmetric case $\mu(x) \equiv 0$, however, we prove that the error stays exponentially small up to an $O(1)$ distance from the static bifurcation point.
	
	Using the error estimates obtained, we are able to describe the dynamics of the dynamic SH equation \eqref{eq:sh_dynamic} via the an analysis of the dynamically simpler modulation equations. For general source functions $\mu(x)$ given by \eqref{eq:mu}, we show that the formal leading order approximation at `$O(r)$' is exponentially small. This is because certain symmetries of the GL-type equation governing the `leading order' approximation lead to a delayed stability loss phenomenon. 
	If the source term $\mu(x)$ is non-zero with $\nu_1 \neq 0$, say, then the relevant symmetries are broken in the first correction at $O(r^2)$. As a consequence, the first asymptotically detectable (i.e.~not exponentially small) patterns of are amplitude $O(r^2)$, which is shown to translate to $O(\eps^{1/2})$, with an envelope determined by a higher order linearized GL equation. This situation should be contrasted with the predictions of the static theory, where the leading order dynamics are generally best-approximated by solutions to a non-linear GL equation \eqref{eq:GL_eqn_intro} of amplitude $O(\delta)$ and thus in our setting $O(\eps^{1/4})$. Although we do not prove it here, the results obtained suggest that there is \textit{no} delayed stability loss if the reflection symmetry $u \mapsto -u$ is broken. We conjecture that in this case, the presence of a source term leads an exchange of stability similar to that which has been described in a sequence of papers by Butuzov, Nefedov and Schneider \cite{Butuzov2002b,Butuzov1999,Butuzov2000,Butuzov2001,Butuzov2002}. In the symmetric case $\mu(x) \equiv 0$ we show that the delayed stability loss observed in the leading order GL approximation persists in the higher order corrections, leading to delayed stability loss in the approximating system as a whole. We derive a lower bound for the delay time which is close to the symmetric estimate that one expects based on properties of the linearized equation. Using the error estimates we are able to rigorously infer the existence of and derive a lower bound for a delayed loss of stability in the dynamic SH equation \eqref{eq:sh_dynamic} when $\mu(x) \equiv 0$. The bound obtained is also $O(1)$ with respect to $\eps \to 0$, 
	but likely to be sub-optimal due to a sub-optimal bound in the error estimates. It is conjectured that the estimates could be improved by considering an approximation which accounts for exponentially small asymptotic corrections `beyond all orders'.
	
	\
	
	The manuscript is structured as follows: In Section \ref{sec:Background} we present an overview of the modulation reduction in the static SH equation \eqref{eq:sh_dynamic}$_{\eps = 0}$, including an improved version of the well-known approximation theorem from \cite{Collet1990,Kirrmann1992}. In Section \ref{sec:Main_results} we present our main results in two parts. The blow-up transformations defining a dynamic generalisation of the GL ansatz and its improvement are presented in Section \ref{sub:results_modulation_eqns}. Here we also describe the geometry and present the formally derived modulation equations, in global coordinates in the blown-up space, and in local coordinates better suited to calculations. Rigorous results which describe the dynamics of the approximation and modulation equations, the validity of the approximation and, finally, the dynamics of system \eqref{eq:sh_dynamic} are given in Section \ref{sub:results_error_estimates_and_dynamics}. Section \ref{sec:Amplitude_reduction_via_geometric_blow-up} is devoted to the formal derivation of the modulation equations, and Section \ref{sec:Proof_of_thm_dynamics} is devoted to a detailed analysis of the approximation and the modulation equations in local coordinate charts, culminating in a proof of our main result on the approximation dynamics. 
	Our main result on the validity of the approximation is proved in Section \ref{sec:Proof_of_thm_error}. In Section \ref{sec:Summary_and_outlook} we summarise our findings and conclude the manuscript. 
	A number of technical results for the proofs in Sections \ref{sec:Proof_of_thm_dynamics} and \ref{sec:Proof_of_thm_error} are deferred to Appendix \ref{app:technical_estimates}.

	\section{Validity of the classical GL approximation}
	\label{sec:Background}
	
	In this section we provide a brief overview of the established modulation theory for the classical SH equation
	\begin{equation}
		\label{eq:sh}
		\partial_t u = - (1 + \partial_x^2)^2 u + v u - u^3 ,
	\end{equation}
	where $u = u(x,t) \in \R$, $v \in \R$ is a bifurcation parameter, $x \in \R$ and $t \geq 0$. We refer to equation \eqref{eq:sh} as the \textit{classical} SH equation in order to distinguish it from the dynamic counterpart \eqref{eq:sh_dynamic} considered in detail in later sections. Equation \eqref{eq:sh} was originally derived in \cite{Hohenberg1977} as a model for convective instability, but arises frequently in the literature as a model problem for understanding the basic mechanisms of pattern forming systems. The reader is referred to \cite{Cross1993} for an early but comprehensive review with an emphasis on applications and the formal derivation of modulation equations of GL-type. 
	
	\
	
	Equation \eqref{eq:sh} has a homogeneous steady-state $u(x,t) = u_{ss}(t) \equiv 0$ which destabilises due to a Turing instability at \textit{critical modes} $k_c = \pm 1$ when the $v=0$. This can be seen by considering the linearisation $\partial w = - (1 + \partial_x^2)^2 w + v w$, which has solutions of the form $\me^{\lambda(k,v)t + ikx}$, where 
	\begin{equation}
		\label{eq:lambda}
		\lambda(k,v) = - (1 - k^2)^2 + v , \qquad k \in \mathbb R .
	\end{equation}
	If $v < 0$, then solutions are exponentially damped for all wavenumbers $k \in \R$. If $v = \delta^2 > 0$, where $\delta \ll 1$ is a small perturbation parameter, then \eqref{eq:lambda} is positive for $k$ values within two bands of width $O(\delta)$ about the critical modes $k_c^{\pm}$; see Figure \ref{fig:spectrum} and the caption. In the latter case with parameter values $v = \delta^2 \ll 1$ close to the onset of instability, solutions to the SH equation \eqref{eq:sh} can be approximated using a modulation equation obtained after substituting the multiple-scales solution ansatz 
	\begin{equation}
		\label{eq:GL_ansatz}
		\delta \psi_{GL}(x,t) = \delta A(\delta x, \delta^2 t) \me^{ix} + \delta \overline A(\delta x, \delta^2 t) \me^{-ix} ,
	\end{equation}
	where $A(\delta x, \delta^2 t) \in \mathbb C$ is a so-called \textit{modulation function}, and formally requiring that the leading order contribution to the residual
	\[
	\textup{Res}(\delta \psi_{GL}) = 
	- \partial_t (\delta \psi_{GL}) - (1 + \partial_x^2)^2 (\delta \psi_{GL}) + \delta^3 \psi_{GL} - (\delta \psi_{GL})^3 ,
	\]
	which is identically zero for exact solutions of \eqref{eq:sh}, vanishes at the critical modes (i.e.~one imposes the requirement that the $O(\delta)$ terms factoring $\me^{\pm i x}$ are zero). The ansatz \eqref{eq:GL_ansatz} is referred to as the \textit{GL approximation}, and the resulting modulation equation for $A$ is a real GL equation
	\begin{equation}
		\label{eq:GL_eqn}
		\partial_{\bar t} A = 4 \partial_{\bar x}^2 A + A - 3 A |A|^2 ,
	\end{equation}
	where $\bar x = \delta x$ and $\bar t = \delta^2 t$; c.f.~equation \eqref{eq:GL_eqn_intro} in Section \ref{sec:Introduction}.
	
	\begin{figure}[t!]
		\centering
		\includegraphics[width=.5\textwidth]{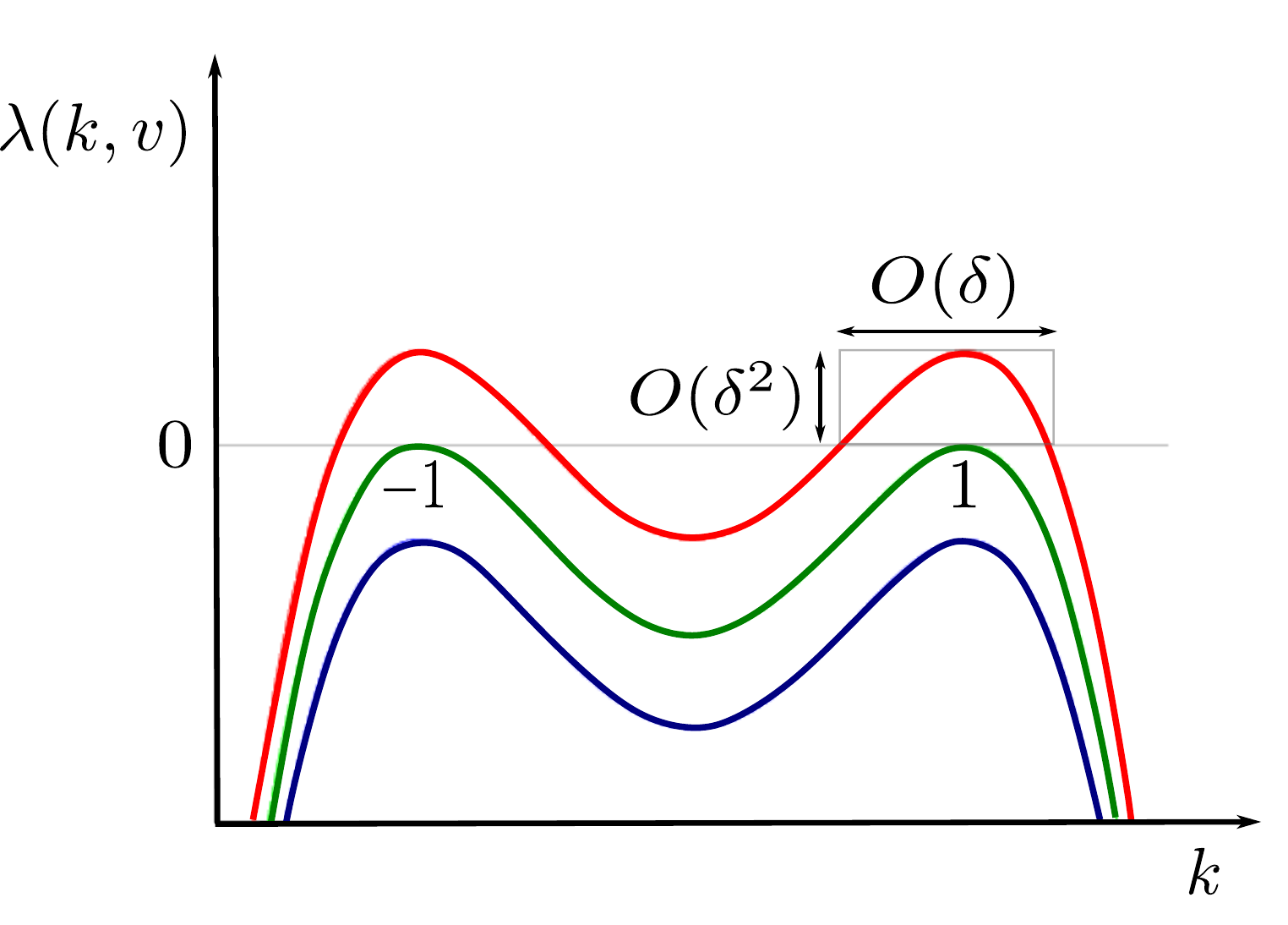}
		\caption{The stability curve \eqref{eq:lambda} for the linearized SH equation, for three different values of $v$. The blue curve has $v < 0$ and therefore $\lambda(k,v) < 0$ for all $k \in \R$, implying linear asymptotic stability for all $k \in \R$. The leading instabilities occur for $k_c^{\pm} = \pm1$ when $v = 0$, as can be seen in the green curve with $\lambda(\pm1,0) = 0$ and $\lambda(k,0) < 0$ for all $k \neq \pm 1$. The red curve illustrates the unstable case for $v = \delta^2 \ll 1$. In this case $\lambda(k,v) > 0$ within two bands of wavenumbers of height $O(\delta^2)$ and width $O(\delta)$ about 
			$k_c^{\pm}$.} 
		\label{fig:spectrum}
	\end{figure}
	
	\
	
	Rigorous justifications 
	for the (formally derived) GL approximation \eqref{eq:GL_ansatz} first appeared in \cite{Collet1990}, with simplified proofs following shortly thereafter in \cite{Kirrmann1992}. In \cite{Kirrmann1992} the authors showed that the accuracy of the approximation can be improved by adding higher order terms. 
	\SJ{Additional work by van Harten \cite{VanHarten1991} showed that for a rather general class of systems with continuous spectra similar to Figure \ref{fig:spectrum}(a), higher order corrections should chosen according to the magnitude and distribution of distinct peaks in the wave spectrum for solutions to the (spatially) Fourier transformed SH equation. The observation is that close to the onset of instability, these peaks are distributed according to the so-called \textit{clustered mode distribution} shown in Figure \ref{fig:mode_distribution}. 
	The peaks are centred around integer translations of two dominant peaks at the critical modes $k = k_c^{\pm} = \pm 1$ (where solutions of the the linearized SH equation become unstable). Higher order corrections to \eqref{eq:GL_ansatz} are attributed to the additional peaks that are centred around integer but non-critical wave numbers $k \neq \pm 1$. The magnitude of each peak scales with an integer power of $\delta$, and the magnitude of the correction term corresponding to each peak can be derived after inverse Fourier transform.} 
	
	\begin{figure}[t!]
		\centering
		\includegraphics[width=.65\textwidth]{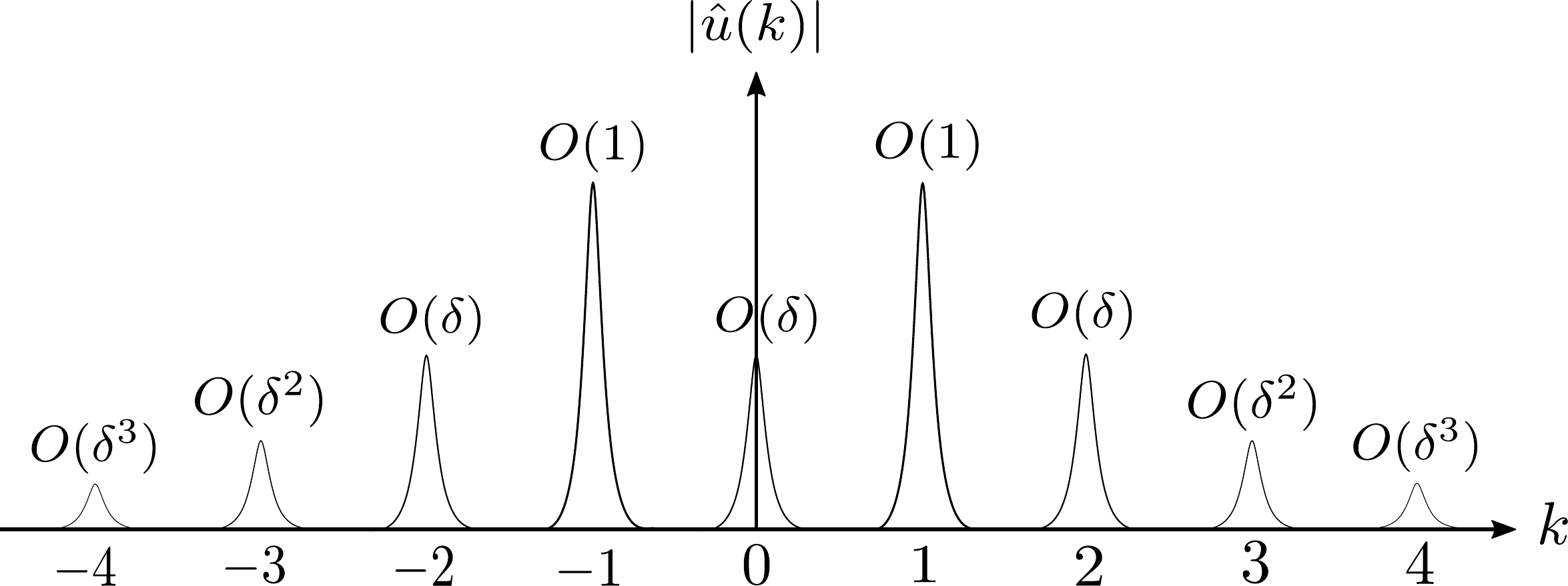}
		\caption{\SJ{The wave spectrum of solutions to the Fourier transformed SH equation \eqref{eq:sh} has distinctive peaks that are distributed according to the so-called clustered mode distribution. The form of the ansatz \eqref{eq:GL_ansatz_n} is derived based on the magnitude and distribution of the peaks, with each peak corresponding to a term in the series. Peaks are centred around integer translations of the dominant critical modes at $k_c^{\pm} = \pm 1$. Note 
		that the $L^\infty$ norm of $\widehat A(\delta^{-1}(k-1))$ is $O(1)$ in Fourier space, but $O(\delta)$ in $L^\infty$ in physical space. Thus a peak of magnitude $O(\delta^{l})$ for some $l \in \mathbb N$ corresponds to a correction of $O(\delta^{l+1})$ in \eqref{eq:GL_ansatz_n}.}}
		\label{fig:mode_distribution}
	\end{figure}
	
	\SJ{Thus,} the clustered mode distribution in Figure \ref{fig:mode_distribution} informs the choice of ansatz. In order to write down the improved ansatz for the SH equation \eqref{eq:sh}, let $N = n+1 \geq 3$ be an integer, $I_N := \{-N, \ldots , N\}$, $\alpha(m) := | |m| - 1|$ and $\tilde \alpha(m) := N - \alpha(m) - 2 \delta_{|m|,1}$. This leads to an ansatz of the form
	\begin{equation}
		\label{eq:GL_ansatz_n}
		\delta \psi_n(x,t) = \sum_{m \in I_N} \sum_{j=1}^{\tilde \alpha(m)} \delta^{\alpha(m) + j} A_{mj}(\bar x, \bar t) \me^{i m x} ,
	\end{equation}
	where the modulation functions $A_{mj}(\bar x, \bar t) \in \mathbb C$ satisfy the reality condition $A_{-mj} = \overline{A_{mj}}$. The ansatz \eqref{eq:GL_ansatz_n} has been `improved' in the sense that it can be shown to formally minimise the residual to up to order $\textup{Res}(\delta \psi_n) = O(\delta^n)$ if the critical modulation functions $A_{\pm11}$ satisfy the GL equation \eqref{eq:GL_eqn}, and the higher order modulation functions with $|m| \neq 1$ and/or $j \geq 1$ satisfy a suitable algebraic or linearized inhomogeneous GL equation; we refer again to \cite[Ch.~10]{Schneider2017} for details.
	
	\
	
	The derivation of modulation equations for the $A_{mj}$ via the method described above is purely formal. Thus, error estimates in the form of bounds in a suitable norm are necessary if one wishes to approximate solutions of the SH equation \eqref{eq:sh} in a rigorous setting. For this, one must choose a common 
	phase space for SH solutions $u$ and the approximations $\delta \psi_{GL}$ or $\delta \psi_n$. Since we would like to describe spatially inhomogeneous solutions which do not decay to zero as $|x| \to \infty$, a number of `natural choices' like the $L^2(\R,\R)$ spaces and their corresponding Sobolev spaces $H^\theta(\R,\R)$ are simply not rich enough. Such considerations have led numerous authors have chosen to work with bounded continuous spaces $C^\theta_b(\R, \R)$, however these come with significant analytical complications. Following e.g.~\cite{Mielke1995,Schneider1994b,Schneider1994,Schneider2017}, we work with the space of uniformly locally square-integral functions
	\begin{equation}
		\label{eq:L2_norm}
		L_{ul}^2 (\R,\R) := \left\{ u : \|u\|_{L_{ul}^2} = \sup_{y \in \R} \left( \int_{y-1/2}^{y+1/2} | u(x) |^2 dx \right)^{1/2} < \infty \right\} ,
	\end{equation}
	and the uniformly local Sobolev spaces $H_{ul}^\theta(\R,\R)$ consisting of all functions $u \in L_{ul}^2(\R,\R)$ such that $\partial^j_x u \in L_{ul}^2(\R,\R)$ for all $j \leq \theta$\SJ{, where $\theta \in \mathbb N$}. For notational simplicity in what follows, $L_{ul}^2$ and $H_{ul}^\theta$ are understood to denote $L_{ul}^2 (\R,\R)$ and $H_{ul}^\theta (\R,\R)$ respectively. The space of $L_{ul}^2$ functions is analytically advantageous because of the availability of Fourier methods, and sufficiently rich to contain a large class of physically interesting solutions which includes all $L^\infty$ solutions, such as spatially quasiperiodic functions and fronts. We refer to \cite{Mielke1995,Schneider1994b} and \cite[Ch.~8.3]{Schneider2017} for more details on the spaces $H_{ul}^\theta$ in the context of the GL equation and modulation theory.
	
	\
	
	The following result quantifies the error of the GL approximation \eqref{eq:GL_ansatz} and its improvement \eqref{eq:GL_ansatz_n} for solutions in $H_{ul}^\theta$, assuming sufficient regularity of the associated GL solutions. The version presented here is can be found in \cite[Thm.~10.2.9]{Schneider2017}. The original result was formulated for the GL approximation \eqref{eq:GL_ansatz} in \cite{Collet1990} in $C_b^\theta$ spaces and again via a separate method in \cite{Kirrmann1992}.
	
	\begin{thm}
		\label{thm:Error_static}
		\textup{\cite[Thm.~10.2.9]{Schneider2017}, see also \cite{Collet1990,Kirrmann1992}.} Fix $\theta \geq 1$ and let $A = A_{11} \in C([0,T_0], H_{ul}^{\theta_A})$ be a solution to the GL equation \eqref{eq:GL_eqn}, where $\theta_A = 3(n-3) + 1 + \theta$. There exists a $\delta_0 > 0$ and a constant $C > 0$ such that for all $\delta \in (0, \delta_0)$, there are SH solutions $u$ such that
		\[
		\sup_{t \in [0,T_0 / \delta^2]} \| u(\cdot,t) - \delta \psi_n(\cdot,t) \|_{H_{ul}^\theta} \leq C \delta^{n-2} .
		\]
		In particular,
		\[
		\sup_{t \in [0,T_0 / \delta^2]} \| u(\cdot,t) - \delta \psi_{GL}(\cdot,t) \|_{H_{ul}^\theta} \leq C \delta^{3/2} .
		\]
	\end{thm}
	
	Theorem \ref{thm:Error_static} applies to solutions of the SH equation \eqref{eq:sh} with initial conditions $u^\ast$ that are $O(\delta)$ in $H_{ul}^\theta$, and which are already distributed according the clustered mode distribution of the approximation \eqref{eq:GL_ansatz_n}. More precisely, it applies for initial conditions of the form
	\[
	u^\ast(x) = \delta \psi_n^\ast(x) = \sum_{m \in I_N} \sum_{j=1}^{\tilde \alpha(m)} \delta^{\alpha(m) + j} A_{mj}^\ast(\bar x) \me^{i m x} ,
	\]
	where $A^\ast_{mj}(x) = A_{mj}(x,0)$. Subsequently in \cite{Eckhaus1993,Schneider1995b} it was shown that the so-called \textit{GL manifold}, i.e.~the subset in $H_{ul}^\theta$ defined by the approximation \eqref{eq:GL_ansatz}, is an attractor (see also \cite{Mielke1997} for attractivity results on large domains). This allows Theorem \ref{thm:Error_static} to be improved so that it applies for the much larger space of solutions with initial data satisfying $\|u^\ast\|_{H_{ul}^\theta} \leq C\delta$. 
	A large number of results pertaining to e.g.~upper semi-continuity of the attractor \cite{Mielke1995}, shadowing properties and global existence have subsequently been derived for systems characterised by instabilities similar to that of the SH equation \eqref{eq:sh} (if not for the SH equation itself) \cite{Schneider1994b,Schneider1999}. We omit detailed statement of these results, since the approximation result in Theorem \ref{thm:Error_static} is the most relevant for our purposes.
	
	\
	
	A main result of this work is the formal derivation of a dynamical approximation which extends \eqref{eq:GL_ansatz_n} to the fast-slow setting, which can be applied in order to study the behaviour over an entire neighbourhood of $v = 0$. This is necessary in order to understand the corresponding dynamic bifurcation. We shall then state and prove an analogue of Theorem \ref{thm:Error_static}.

	\section{Main Results}
	\label{sec:Main_results}
	
	In this section we state and describe our main results. In Section \ref{sub:results_modulation_eqns} we present a dynamic generalisation for the multiple-scales GL ansatz \eqref{eq:GL_ansatz} and the higher order approximation \eqref{eq:GL_ansatz_n}. Both of these are given in form of blow-up transformations. We provide the corresponding modulation equations in global coordinates in the blown-up space, and in local projective coordinates well-suited to the calculations presented in later sections. Section \ref{sub:results_error_estimates_and_dynamics} is devoted to the presentation of rigorous results on the dynamics of the approximating modulation equations, as well as an approximation theorem 
	which may be considered as a dynamic generalisation of Theorem \ref{thm:Error_static}. This result is used to characterise the behaviour of a rather general space of solutions of the dynamic SH equation \eqref{eq:sh_dynamic}. 

	\subsection{Dynamic ansatz and modulation equations}
	\label{sub:results_modulation_eqns}
	
	In the following we a propose dynamic generalisation of the GL ansatz \eqref{eq:GL_ansatz} and the improved ansatz \eqref{eq:GL_ansatz_n}. Our approach is based upon an adaptation of the geometric blow-up method developed for finite-dimensional fast-slow systems in \cite{Dumortier1996} and further developed in \cite{Krupa2001a,Krupa2001c,Krupa2001b}; see \cite{Jardon2019b} for a recent survey. The idea is to use blow-up techniques to extend the static modulation theory described in Section \ref{sec:Background} to the fast-slow setting. The method is motivated by the success of blow-up based approaches to the generalisation of rescaling methods in singularly perturbed systems for applications to dynamic bifurcations in finite-dimensional fast-slow systems. 
	
	\begin{remark}
		Recently in \cite{Engel2020} the authors made first steps towards the development of a geometric blow-up method for application to infinite-dimensional fast-slow systems. In this approach, a blow-up transformation is applied to a finite-dimensional ODE system obtained after truncating a spectral Galerkin discretization at any arbitrary truncation level. The blow-up method proposed in this work provides an alternative and more direct approach, which does not rely on a particular discretization reduction to a finite-dimensional system of ODEs. This is necessary because we work with an unbounded domain in solution spaces that are too large to allow for such reductions.
	\end{remark}
	
	As is standard in blow-up approaches, we consider the perturbation parameter $\eps$ as a variable and work with the extended system
	\begin{equation}
		\label{eq:sh_extended}
		\begin{split}
			\partial_t u &= - (1 + \partial_x^2)^2 u + v u - u^3 + \eps \mu(x) , \\ 
			\partial_t v &= \eps , \\ 
			\partial_t \eps &= 0 .
		\end{split}
	\end{equation}
	Note that the subspaces defined by $\eps = const.$ are invariant, and that the classical (static) SH equation \eqref{eq:sh} is contained within the invariant subspace $\{\eps = 0\}$. 
	In parallel to the established approach for the classical SH equation \eqref{eq:sh} described in Section \ref{sec:Background}, we look for an approximation $\Psi$ which formally minimises the residual
	\begin{equation}
		\label{eq:residual_def}
		\textrm{Res}(\Psi) = - \partial_t \Psi - (1 + \partial_x^2)^2 \Psi + v \Psi - \Psi^3 + \eps \mu(x) .
	\end{equation}
	Exact solutions of the first equation in \eqref{eq:sh_extended} satisfy $\textrm{Res}(\Psi) = 0$. Note that the formal derivation of modulation equations is independent of the choice of phase space for $u$ and $\Psi$. Thus, there is no need to specify it at this point. 

	\subsubsection{Leading order approximation}
	
	Based on the fact that systems \eqref{eq:sh} and \eqref{eq:sh_extended} agree if the latter is restricted to $\{\eps = 0\}$, we propose the following blow-up transformation as a dynamic generalisation of the GL ansatz \eqref{eq:GL_ansatz}:
	\begin{equation}
		\label{eq:blowup_Psi_GL}
		r \geq 0, \ \psi_{GL} \in X, \ (\bar v, \bar \eps) \in S^1 \mapsto 
		\begin{cases}
			\Psi_{GL}(x, t) = r(\bar t) \psi_{GL}(x, \bar t) = r(\bar t) A(\bar x,\bar t) \me^{ix} + r(\bar t) \overline A(\bar x,\bar t) \me^{-ix} , \\
			v(t) = r(\bar t)^2 \bar v(\bar t) , \\
			\eps = r(\bar t)^4 \bar \eps(\bar t) .
		\end{cases}
	\end{equation}
	Here $X$ denotes a suitable Banach space for $\psi_{GL}$, \SJ{$S^1$ denotes the unit circle $\{(\bar v, \bar \eps) : \bar \eps^2 + \bar v^2 = 1 \}$,} $A_{mj}(\bar x, \bar t) \in \mathbb C$ are modulation functions, and $\bar x$, $\bar t$ denote the \textit{desingularized space and time} defined via the positive, time-dependent transformations 
	\begin{equation}
		\label{eq:desingularization}
		\partial_t = r(\bar t)^2 \partial_{\bar t} , \qquad 
		\partial_x = r(\bar t) \partial_{\bar x} ,
	\end{equation}
	referred to as \textit{desingularizations}. In the following we collect a number of observations and remarks on the blow-up map defined by equations \eqref{eq:blowup_Psi_GL}-\eqref{eq:desingularization}.
	
	\begin{figure}[t!]
		\centering
		\includegraphics[scale=0.3]{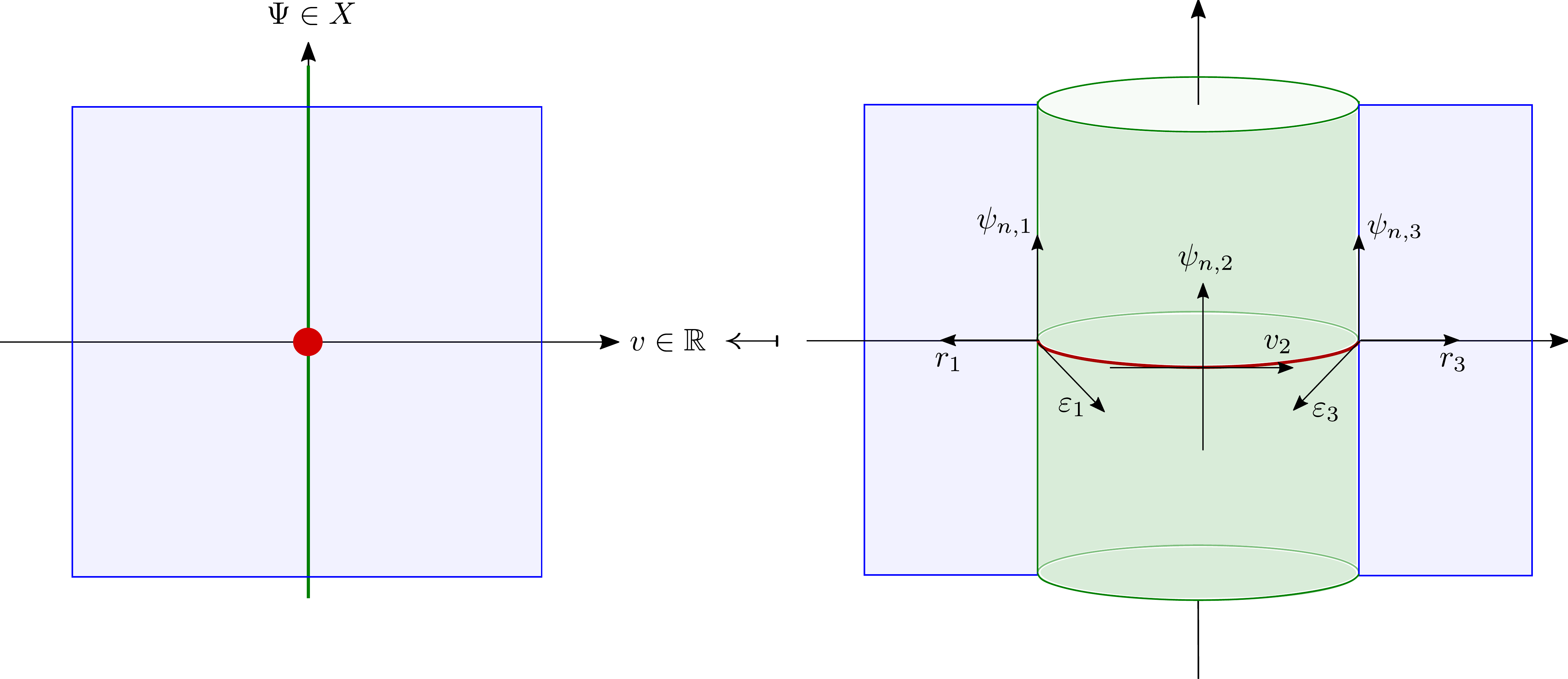}
		\caption{Action of the blow-up transformations defined in \eqref{eq:blowup_Psi_GL} and \eqref{eq:blowup_Psi}. The set $\{v = \eps = 0\}$ shown in green along the $\Psi$ axis, which contains the Turing singularity (indicated as a red disk), is blown-up to a cylinder as described in Remark \ref{rem:blow_up_v_eps}. Local coordinates in charts $\mathcal K_l$ as defined in \eqref{eq:charts} are also shown.}
		\label{fig:blowup}
	\end{figure}
	
	\begin{remark}
		\label{rem:dynamic_generalisation}
		The ansatz $\Psi_{GL} = r \psi_{GL}$ takes the same form as the GL ansatz in \eqref{eq:GL_ansatz}, except that the perturbation parameter $\delta$ has been `replaced' by the variable $r(\bar t)$, which depends on time. This is one sense in which \eqref{eq:blowup_Psi_GL} may be considered as a dynamic generalisation of the well-known GL ansatz \eqref{eq:GL_ansatz}.
	\end{remark}
	
	\begin{remark}
		\label{rem:blow_up_v_eps}
		The second two components of the map \eqref{eq:blowup_Psi_GL} decouple, i.e.~the blow-up map defined by
		\begin{equation}
			\label{eq:cyl_blow-up}
			r \geq 0, \ (\bar v, \bar \eps ) \in S^1 \mapsto
			\begin{cases}
				v = r^2 \bar v , \\
				\eps = r^4 \bar \eps ,
			\end{cases}
		\end{equation}	
		and the time desingularization in \eqref{eq:desingularization} can be considered separately. From this perspective, it is clear that the set $\{v = \eps = 0\}$ corresponding to the Turing instability is blown-up to a half-cylinder with angular coordinates $(\bar v, \bar \eps) \in S^1 \cap \{ \bar \eps \geq 0\}$; see Figure \ref{fig:blowup}. The variable $r$ is a radial coordinate which measures the Euclidean distance from the blow-up cylinder.
	\end{remark}
	
	\begin{remark}
		The maps \eqref{eq:blowup_Psi_GL} and \eqref{eq:cyl_blow-up} are defined for both $v(t) \geq 0$ and $v(t) \leq 0$, as opposed to the GL ansatz \eqref{eq:GL_ansatz} which has $v = \delta^2 > 0$. This is necessary in the dynamic setting since we are typically interested in a slow passage through the instability, i.e.~with the evolution of solutions which approach $v=0$ from an initial condition with $v(0) < 0$.
	\end{remark}
	
	\begin{remark}
		\label{rem:desingularization}
		The desingularizations in \eqref{eq:desingularization} generalise the simple rescalings $\bar x = \delta x$ and $\bar t = \delta^2 t$ in equation \eqref{eq:GL_ansatz} to the time-dependent setting. This is necessary because the `perturbation parameter' $r(\bar t) \geq 0$ now depends on time. As an alternative to the spatial desingularization defined by \eqref{eq:desingularization}, one could include space in the blow-up more directly via time-dependent transformation $x = r(\bar t)^{-1} \bar x$. In fact, $\partial_x = r(\bar t) \partial_{\bar x} \iff x = r(t)^{-1} \bar x$ since $r(\bar t)$ does not depend on $x$. Similar time-dependent transformations of the spatial domain appear frequently in the context of \textit{dynamic renormalization}; see e.g.~\cite{Bricmont1995,Chapman2021,Kevrekidis2003,Siettos2003}\SJ{, where they are typically stated in the more explicit form
		\begin{equation}
			\label{eq:renormalization}
			t = \int_0^{\bar t} \frac{1}{r(s)^2} ds , \qquad 
			x = \frac{\bar x}{r(\bar t)} .
		\end{equation}
		Since the dynamic SH equation \eqref{eq:sh_extended} only depends explicitly on $\partial_t$ (and not on $t$), it suffices to work with the formulation in \eqref{eq:desingularization} (which is implied by \eqref{eq:renormalization}). Integral transformations for $t$ are necessary for the application of blow-up to non-autonomous problems, however; see \cite{Arcidiacono2020} for an example in the context of a non-autonomous ODE system. Finally we note that for PDE} systems on bounded spatial domains, a similar effect can be achieved by applying such a transformation to the boundary in instead of the spacial variable itself; see \cite{Engel2020}.
	\end{remark}
	
	\begin{remark}
		Since \eqref{eq:sh_extended}$|_{\eps = 0}$ has a Turing singularity at $u = u_{ss} \equiv 0$, $v=0$, it is locally linearly stable with respect to time-dependent but spatially independent perturbations. More precisely, spatially homogeneous solutions $u(x,t) = u(t)$ of \eqref{eq:sh_extended}$|_{\eps = 0}$ are governed by an ODE
		\[
		\partial_t u = f_{ode}(u,v) = - u + v u - u^3 ,
		\]
		which is linearly stable with Jacobian $\partial_u f_{ode}(0,0) = -1$. This property, which is a defining property of Turing singularities, implies that there is no blow-up transformation of the standard form
		\[
		r \geq 0, \ (\bar v, \bar \eps) \in S^1, \ \bar u \in X \mapsto 
		\begin{cases}
			u = r^{a_1} \bar u , \\
			v = r^{a_2} \bar v , \\
			\eps = r^{a_3} \bar \eps ,
		\end{cases}
		\]
		where $a_1, a_2, a_3 \geq 0$, 
		such that the resulting blown-up problem admits a desingularization of the form $\partial_t = r^\alpha \partial_{\bar t}$. The non-zero linear part for $\bar u$ leads to a term of the form $\partial_{\bar t} \bar u = r^{-\alpha} \bar u$ after desingularization which cannot be avoided by a suitable choice of scaling.
	\end{remark}
	
	A formal derivation using the blow-up ansatz \eqref{eq:blowup_Psi} leads to the following non-autonomous GL equation for the modulation function $A$:
	\begin{equation}
		\label{eq:GL_eqn_global}
		\partial_{\bar t} A = 4 \partial_{\bar x}^2 A + \left( \bar v(\bar t) - r(\bar t)^{-1} \partial_{\bar t} r(\bar t) \right) A - 3 A |A|^2 .
	\end{equation}
	The detailed derivation is given in Section \ref{sec:Amplitude_reduction_via_geometric_blow-up}. \SJ{Equation \eqref{eq:GL_eqn_global}} is distinguished from the classical GL equation \eqref{eq:GL_eqn} 
	by the time-dependent coefficient $\bar v(\bar t) - r(\bar t)^{-1} \partial_{\bar t} r(\bar t)$. Note that the two equations are also posed on different spaces, since \eqref{eq:GL_eqn_global} is posed on the blown-up space with $A \in X$, $(\bar v, \bar \eps) \in S^1$ and $r \geq 0$.
	
	Although the global form of equation \eqref{eq:GL_eqn_global} is helpful for interpretive and conceptual purposes, concrete representations in local coordinates are typically preferred for calculations. Since $(\bar v, \bar \eps) \in S^1$, local coordinate charts can be defined in affine projective coordinates
	\begin{equation}
		\label{eq:charts}
		\begin{split}
			\mathcal K_1 : (\Psi_{GL}, v, \eps) &= (r_1 \psi_{GL,1}, -r_1^2, r_1^4 \eps_1) = (r_1 A_1 \me^{ix} + r_1 \overline{A_1} \me^{-ix}, -r_1^2, r_1^4 \eps_1) , \\
			\mathcal K_2 : (\Psi_{GL}, v, \eps) &= (r_2 \psi_{GL,2}, r_2^2 v_2, r_2^4 ) = (r_2 A_2 \me^{ix} + r_2 \overline{A_2} \me^{-ix}, r_2^2 v_2, r_2^4 ) , \\
			\mathcal K_3 : (\Psi_{GL}, v, \eps) &= (r_3 \psi_{GL,3}, r_3^2, r_3^4 \eps_3) = (r_3 A_3 \me^{ix} + r_3 \overline{A_3} \me^{-ix}, r_3^2, r_3^4 \eps_3) ,
		\end{split}
	\end{equation}
	obtained by setting $\bar v = -1$, $\bar \eps = 1$ and $\bar v = 1$ respectively. The local coordinate axes are sketched in Figure \ref{fig:blowup}. We also denote by $x_l$ and $t_l$ the local coordinate representation of $\bar x$ and $\bar t$ via \eqref{eq:desingularization} in charts $\mathcal K_l$, $l = 1,2,3$.
	
	\begin{remark}
		\label{rem:K2_ansatz}
		A close relationship to classical modulation theory is observed in chart $\mathcal K_2$. Due to the simple scaling $r_2 = \eps^{1/4}$, the local representation of the dynamic GL ansatz for $\Psi_{GL}$ in chart $\mathcal K_2$ coincides with the classical GL ansatz \eqref{eq:GL_ansatz} after setting $r_2 = \eps^{1/4} = \delta$. 
	\end{remark}
	
	The following result provides the modulation equations obtained in local coordinate charts $\mathcal K_l$ after imposing the formal requirement that the leading order terms of $O(r)$ vanish at modes corresponding critical wavenumbers $k_c^{\pm} = \pm1$.
	
	\begin{proposition}
		\label{prop:GL_eqns_leading_order}
		The reduced modulation equations in $\mathcal K_1$ are given by
		\begin{equation}
			\label{eq:K1_eqns_GL}
			\begin{split}
				\partial_{t_1} A_1 &= 4 \partial_{x_1}^2 A_1 + \left( - 1 + \frac{\eps_1}{2} \right) A_1 - 3 A_1 |A_1|^2 , \\
				\partial_{t_1} r_1 &= - \frac{1}{2} r_1 \eps_1 , \\
				\partial_{t_1} \eps_1 &= 2 \eps_1^2 .
			\end{split}
		\end{equation}
		The reduced modulation equations in $\mathcal K_2$ are given by
		\begin{equation}
			\label{eq:K2_eqns_GL}
			\begin{split}
				\partial_{t_2} A_2 &= 4 \partial_{x_2}^2 A_2 + v_2 A_2 - 3 A_2 | A_2 |^2 , \\
				\partial_{t_2} v_2 &= 1 .
			\end{split}
		\end{equation}
		The reduced modulation equations in $\mathcal K_3$ are given by
		\begin{equation}
			\label{eq:K3_eqns_GL}
			\begin{split}
				\partial_{t_3} A_3 &= 4 \partial_{x_3}^2 A_3 + \left( 1 + \frac{\eps_3}{2} \right) A_3 - 3 A_3 |A_3|^2 , \\
				\partial_{t_3} r_3 &= \frac{1}{2} r_3 \eps_3 , \\
				\partial_{t_3} \eps_3 &= - 2 \eps_3^2 .
			\end{split}
		\end{equation}
	\end{proposition}
	
	\begin{proof}
		The equations for $r_l$, $v_l$ and $\eps_l$ are derived in each chart directly via their local coordinate representations in \eqref{eq:charts}. The modulation equations for $A_l$ are derived by substituting the local coordinate expressions for $\Psi$ into the residual \eqref{eq:residual_def} and imposing the requirement that the leading order $O(r)$ term at the critical modes corresponding to $\me^{\pm ix}$ vanishes; we refer again to Section \ref{sec:Amplitude_reduction_via_geometric_blow-up} for details.
	\end{proof}
	
	Since $\eps_1(t_1)$, $v_2(t_2)$ and $\eps_3(t_3)$ can be solved by direct integration, the ansatz defined by the blow-up transformation \eqref{eq:blowup_Psi_GL} leads to a non-autonomous GL equation in each chart $\mathcal K_l$. If $\Psi$ can be shown to serve as a good approximation for SH solutions $u$ over a sufficiently long time interval, then Proposition \ref{prop:GL_eqns_leading_order} shows that the study of the dynamic SH equation \eqref{eq:sh_dynamic} reduces to the study of a non-autonomous GL equation
	\[
	\partial_{t_1} A_1 = 4 \partial_{x_1}^2 A_1 + \left( - 1 + \frac{\eps_1(t_1)}{2} \right) A_1 - 3 A_1 |A_1|^2
	\]
	over a time interval $t_1 \in [0,T_1]$ in chart $\mathcal K_1$, a non-autonomous GL equation
	\[
	\partial_{t_2} A_2 = 4 \partial_{x_2}^2 A_2 + v_2(t_2) A_2 - 3 A_2 |A_2|^2
	\]
	over a time interval $t_2 \in [0,T_2]$ in chart $\mathcal K_2$, and a non-autonomous GL equation
	\[
	\partial_{t_3} A_3 = 4 \partial_{x_3}^2 A_3 + \left( 1 - \frac{\eps_3(t_3)}{2} \right) A_3 - 3 A_3 |A_3|^2
	\]
	over a time interval $t_3 \in [0,T_3]$ in chart $\mathcal K_3$.
	
	\begin{remark}
		In charts $\mathcal K_1$ and $\mathcal K_3$ the subspaces $\{\eps_l = 0\}$ and $\{r_l = 0\}$ are invariant and solutions are constrained to surfaces defined by $\eps = r_l^4 \eps_l$, which defines a constant of the motion. The classical GL equation \eqref{eq:GL_eqn} arises within $\{\eps_3 = 0\}$ in chart $\mathcal K_3$, since system \eqref{eq:K3_eqns_GL} restricts to
		\[
		\begin{split}
			\partial_{t_3} A_3 &= 4 \partial_{x_3}^2 A_3 + A_3 - 3 A_3 |A_3|^2 , \\
			\partial_{t_3} r_3 &= 0.
		\end{split}
		\]
		Thus, one expects solutions to be governed by GL-type dynamics as they leave a neighbourhood of $v=0$.
	\end{remark}
	
	Of course, the validity of the formal approximation defined by the blow-up transformation \eqref{eq:blowup_Psi_GL} requires rigorous justification, i.e.~a statement similar to \SJ{Theorem \ref{thm:Error_static}} is required in order to provide rigour to the formal derivations. The statement of such a result is deferred to Section \ref{sub:results_error_estimates_and_dynamics} below.

	\subsubsection{Higher order approximation}
	
	Better estimates are obtained by including higher order corrections in the expression for $\Psi$, which serve to minimise the residual \eqref{eq:residual_def}. Given that the restricted system \eqref{eq:sh_extended}$|_{\eps = 0}$ and the classical SH equation \eqref{eq:sh} coincide, it is natural to suppose that the improved ansatz \eqref{eq:GL_ansatz_n}, whose form was based on the the clustered mode distribution in Figure \ref{fig:mode_distribution}, can also be generalised for the dynamic setting using a suitable blow-up transformation. 
	We therefore propose the following improvement of the blow-up approximation in \eqref{eq:blowup_Psi_GL}, which is obtained by replacing the expression $\Psi = r \psi_{GL}$ by $\Psi_n = r \psi_n$ where
	\begin{equation}
		\label{eq:Psi_n_global}
		\Psi_n(x, t) = r(\bar t) \psi_n(x, \bar t) = \sum_{m \in I_N} \sum_{j=1}^{\tilde \alpha(m)} r(\bar t)^{\alpha(m) + j} A_{mj}(\bar x, \bar t) \me^{imx} ,
	\end{equation}
	$n = N+1 \geq 4$ is the order of the approximation, $A_{mj}(\bar x, \bar t) \in \mathbb C$ satisfy the reality condition $A_{-mj} = \overline{A_{mj}}$, and $I_N$, $\alpha(m)$ and $\tilde \alpha(m)$ are defined as in Section \ref{sec:Background}. In full, the improved blow-up ansatz is given by the map
	\begin{equation}
		\label{eq:blowup_Psi}
		r \geq 0, \ \psi_n \in X, \ (\bar v, \bar \eps) \in S^1 \mapsto 
		\begin{cases}
			\Psi_n(x, t) = r(\bar t) \psi_n(x, \bar t) = \sum_{m \in I_N} \sum_{j=1}^{\tilde \alpha(m)} r(\bar t)^{\alpha(m) + j} A_{mj}(\bar x, \bar t) \me^{imx} , \\
			v(t) = r(\bar t)^2 \bar v(\bar t) , \\
			\eps = r(\bar t)^4 \bar \eps(\bar t) , 
		\end{cases}
	\end{equation}
	together with the desingularizations in \eqref{eq:desingularization}. Remarks \ref{rem:dynamic_generalisation}-\ref{rem:K2_ansatz} on the blow-up map \eqref{eq:blowup_Psi_GL} extend naturally to the blow-up map \eqref{eq:blowup_Psi}.
	
	\
	
	Modulation equations for the functions $A_{mj}$ can be formally derived using the blow-up transformation \eqref{eq:blowup_Psi} and a recursive procedure based on the requirement that the residual to vanishes up to $\textup{Res}(r \psi_n) = O(r^n)$. 
	Here we present the equations, referring the reader once again to Section \ref{sec:Amplitude_reduction_via_geometric_blow-up} for the derivation.
	
	\
	
	The modulation equations for $A_{\pm1j}$ at the critical modes $m=\pm1$ 
	take the form of non-autonomous (generalised) GL equations:
	\begin{equation}
		\label{eq:GL_eqns_global}
		\begin{split}
			\partial_{\bar t} A_{\pm11} &= 4 \partial_{\bar x}^2 A_{\pm11} + \left( \bar v(\bar t) - r(\bar t)^{-1} \partial_{\bar t} r(\bar t) \right) A_{\pm11} - 3 A_{\pm 11} |A_{\pm11}|^2 , \\
			\partial_{\bar t} A_{\pm12} &= 4 \partial_{\bar x}^2 A_{\pm12} + \left( \bar v(\bar t) - r(\bar t)^{-1} \partial_{\bar t} r(\bar t) \right) A_{\pm12} - 
			a_{\pm12} \mp 4i \partial_{\bar x}^3 A_{\pm11} + \bar \eps(\bar t) \nu_{\pm1} 
			, \\
			\partial_{\bar t} A_{\pm1j} &=  4 \partial_{\bar x}^2 A_{\pm1j} + \left( \bar v(\bar t) - r(\bar t)^{-1} \partial_{\bar t} r(\bar t) \right) A_{\pm1j} - a_{\pm1j} \mp 4i \partial_{\bar x}^3 A_{\pm1(j-1)} 
			- \partial_{\bar x}^4 A_{\pm1(j-2)} ,
		\end{split}
	\end{equation}
	where $j = 3, \ldots, N-2$ and each $a_{\pm1j}$ is a finite sum of cubic monomials of the form $l A_{m_1j_1} A_{m_2j_2} A_{m_3j_3}$ with $l \in \mathbb N_+$, $m = m_1 + m_2 + m_3 \in \{-1,1\}$ for $m_i \in I_N$, and $j = j_1 + j_2 + j_3 \in \{1,\ldots,N-2\}$. The non-autonomous GL equations for $A_{\pm11}$, which describe the leading order approximation at the critical modes, has the same form as equation \eqref{eq:GL_eqn_global}.
	
	\
	
	The equations for $A_{mj}$ at non-critical modes $m \neq \pm 1$ are given by
	\begin{equation}
		\label{eq:Amj_algebraic_eqns}
		\begin{split}
			A_{mj} = - (\mathcal L_m^{(0)})^{-1} \bigg( \mathcal L_m^{(1)} A_{mj-1} +  \tilde{\mathcal L_m^{(2)}} A_{mj-2} + \mathcal L_m^{(3)} A_{mj-3} + \mathcal L_m^{(4)} A_{mj-4} - a_{mj} + \delta_{\alpha(m)+j,4} \bar \eps \nu_m \bigg) ,
		\end{split}
	\end{equation}
	where $\delta_{\alpha(m)+j,4} = 1$ if $\alpha(m) + j = 4$ and $0$ otherwise,
	\[
	\begin{split}
		\mathcal L_m^{(0)} := - (1-m^2)^2 ,  \qquad &
		\mathcal L_m^{(1)} := - 4i m(1-m^2) \partial_{\bar x} ,  \qquad
		\tilde{\mathcal L_m^{(2)}} :=  - \partial_{\bar t} - r^{-1} \partial_{\bar t} r - 2 (1-3m^2) \partial_{\bar x}^2 + \bar v , \\
		& \qquad \ \mathcal L_m^{(3)} := - 4i m \partial_{\bar x}^3 ,  \qquad
		\mathcal L_m^{(4)} := - \partial_{\bar x}^4 ,
	\end{split}
	\]
	are linear operators, we set $A_{mj-k} \equiv 0$ if $j-k \leq 0$, and where the functions $a_{mj}$ are similar to $a_{\pm1j}$ above except that they satisfy the more general matching conditions
	\begin{itemize}
		\item $m = m_1 + m_2 + m_3$;
		\item $\alpha(m) + j + 2 = \alpha(m_1) + \alpha(m_2) + \alpha(m_3) + j_1 + j_2 + j_3$.
	\end{itemize}
	The recursive structure of the approximation is such that if solutions to the modulation equations for $A_{\pm11}$ exist, then the equations \eqref{eq:Amj_algebraic_eqns} defining the modulation functions $A_{mj}$ at non-critical modes $m \neq \pm 1$ are \textit{algebraic equations} depending only on $\bar v$, $\bar \eps$, and lower order (critical) modulation functions $A_{\pm1j'}$ with $j' < \alpha(m) + j$. In this case, equation \eqref{eq:Amj_algebraic_eqns} defines the so-called \textit{GL manifold}, which we shall denote by
	\[
	\mathcal M^{gl} := \left\{ (\psi_n, (\bar v, \bar \eps), r) \in X \times S^1 \times \mathbb R_+ : A_{mj} = g^{gl}_{mj}(\textbf{A}_{mj}, (\bar v, \bar \eps)) , |m| \neq 1, j \in \{ 1,\ldots,\alpha(m) \} \right\},
	\]
	where $\textbf A_{mj}$ denotes the set of lower order critical modulation functions $A_{\pm1j'}$ with $j' \in \{1,\ldots,\alpha(m) + j - 1\}$ and $g^{gl}_{mj}(\textbf{A}_{mj}, (\bar v, \bar \eps))$ is defined by the right-hand side of \eqref{eq:Amj_algebraic_eqns}.
	
	\
	
	Finally, we present the equations in local $\mathcal K_l$ coordinates, together with the local coordinate representations of the GL manifold $\mathcal M^{gl}$, denoted by $\mathcal M_l^{gl}$. In order to simplify the notation for the critical modulation equations at $m =\pm 1$, we drop the first subscript and simply write
	\[
	A_j := A_{1j} , \qquad j = 1, \ldots , N-2 ,
	\]
	globally, and
	\[
	A_{1j,l} := A_{j,l}, \qquad j = 1, \ldots , N-2 ,
	\]
	in each chart $\mathcal K_l$, $l = 1,2,3$. Moreover, let $\textbf{A}_{mj,l}$ denote the representation of the set modulation functions $\textbf A_{mj}$ in chart $\mathcal K_l$.
	
	Propositions \ref{prop:K1_eqns}-\ref{prop:K3_eqns} below follow directly from the formal derivations in Section \ref{sec:Amplitude_reduction_via_geometric_blow-up}. We present only the equations for $m \geq 0$, since those with $m < 0$ are obtained via the reality condition $A_{-mj} = \overline{A_{mj}}$.
	
	\begin{proposition}
		\label{prop:K1_eqns}
		\textup{(Modulation equations in chart $\mathcal K_1$).}
		If $A_{j,1}$, $r_1$ and $\eps_1$ satisfy
		\begin{equation}
			\label{eq:K1_eqns}
			\begin{split}
				\partial_{t_1} A_{1,1} &= 4 \partial_{x_1}^2 A_{1,1} + \left( - 1 + \frac{\eps_1}{2} \right) A_{1,1} - 3 A_{1,1} |A_{1,1}|^2 , \\
				\partial_{t_1} r_1 &= - \frac{1}{2} r_1 \eps_1 , \\
				\partial_{t_1} \eps_1 &= 2 \eps_1^2 , \\
				\partial_{t_1} A_{2,1} &= 4 \partial_{x_1}^2 A_{2,1} + \left( -1 + \frac{\eps_1}{2} \right) A_{2,1} - a_{2,1} - 4i \partial_{x_1}^3 A_{1,1} + \eps_1 \nu_{1} 
				, \\
				\partial_{t_1} A_{j,1} &=  4 \partial_{x_1}^2 A_{j,1} + \left( -1 + \frac{\eps_1}{2} \right) A_{j,1} - a_{j,1} - 4i \partial_{x_1}^3 A_{j-1,1} 
				- \partial_{x_1}^4 A_{j-2,1} ,
			\end{split}
		\end{equation}
		where $j = 3, \ldots, N-2$, and $A_{mj,1}$ with $|m| \neq 1$ satisfy $A_{mj,1} = g_{mj,1}^{gl} \left(\textbf{A}_{mj,1}, \eps_1 \right)$, where
		\begin{equation}
			\label{eq:GL_manifold_K1}
			\begin{split}
				g_{mj,1}^{gl} \left(\textbf{A}_{mj,1}, \eps_1 \right) &:= 
				- (\mathcal L_{m,1}^{(0)})^{-1} \bigg( \mathcal L_{m,1}^{(1)} A_{mj-1,1} + \tilde{\mathcal L_{m,1}^{(2)}} A_{mj-2,1} \\
				&+ \mathcal L_{m,1}^{(3)} A_{mj-3,1} + \mathcal L_{m,1}^{(4)} A_{mj-4,1} - a_{mj,1} + \delta_{\alpha(m)+j,4} \eps_1 \nu_m \bigg) ,
			\end{split}
		\end{equation}
		then the formal requirement $\textup{Res} (r_1 \psi_{n,1}) = O(r_1^n)$ is satisfied.
	\end{proposition}
	
	\begin{proposition}
		\label{prop:K2_eqns}
		\textup{(Modulation equations in chart $\mathcal K_2$).}
		If $A_{j,2}$ and $v_2$ satisfy
		\begin{equation}
			\label{eq:K2_eqns}
			\begin{split}
				\partial_{t_2} A_{1,2} &= 4 \partial_{x_2}^2 A_{1,2} + v_2 A_{1,2} - 3 A_{1,2} | A_{1,2} |^2 , \\
				\partial_{t_2} v_2 &= 1 , \\
				\partial_{t_2} A_{2,2} &= 4 \partial_{x_2}^2 A_{2,2} + v_2 A_{2,2} - a_{2, 2} - 4i \partial_{x_2}^3 A_{1,2} + \nu_{1} , \\
				\partial_{t_2} A_{j,2} &=  4 \partial_{x_2}^2 A_{j,2} + v_2 A_{j,2} - a_{j,2} - 4i \partial_{x_2}^3 A_{j-1,2} - \partial_{x_2}^4 A_{j-2,2} ,
			\end{split}
		\end{equation}
		where $j = 3, \ldots, N-2$, and $A_{mj,2}$ with $|m| \neq 1$ satisfy $A_{mj,2} = g_{mj,2}^{gl} \left(\textbf{A}_{mj,2}, v_2 \right)$, where
		\begin{equation}
			\label{eq:GL_manifold_K2}
			\begin{split}
				g_{mj,2}^{gl} \left(\textbf{A}_{mj,2}, v_2 \right) &:= 
				- (\mathcal L_{m,2}^{(0)})^{-1} \bigg( \mathcal L_{m,2}^{(1)} A_{mj-1,2} + \tilde{\mathcal L_{m,2}^{(2)}} A_{mj-2,2} \\
				&+ \mathcal L_{m,2}^{(3)} A_{mj-3,2} + \mathcal L_{m,2}^{(4)} A_{mj-4,2} - a_{mj,2} + \delta_{\alpha(m)+j,4} \nu_m \bigg) ,
			\end{split}
		\end{equation}
		then the formal requirement $\textup{Res} (r_2 \psi_{n,2}) = O(r_2^n)$ is satisfied.
	\end{proposition}
	
	
	\begin{proposition}
		\label{prop:K3_eqns}
		\textup{(Modulation equations in chart $\mathcal K_3$).}
		If $A_{j,3}$, $r_3$ and $\eps_3$ satisfy
		\begin{equation}
			\label{eq:K3_eqns}
			\begin{split}
				\partial_{t_3} A_{1,3} &= 4 \partial_{x_3}^2 A_{1,3} + \left( 1 + \frac{\eps_3}{2} \right) A_{1,3} - 3 A_{1,3} |A_{1,3}|^2 , \\
				\partial_{t_3} r_3 &= \frac{1}{2} r_3 \eps_3 , \\
				\partial_{t_3} \eps_3 &= - 2 \eps_3^2 , \\
				\partial_{t_3} A_{2,3} &= 4 \partial_{x_3}^2 A_{2,3} + \left( 1 - \frac{\eps_3}{2} \right) A_{2,3} - a_{2,3} - 4i \partial_{x_3}^3 A_{1,3} + \eps_3 \nu_{1} 
				, \\
				\partial_{t_3} A_{j,3} &=  4 \partial_{x_3}^2 A_{j,3} + \left( 1 - \frac{\eps_3}{2} \right) A_{j,3} - a_{j,3} - 4i \partial_{x_3}^3 A_{j-1,3} - \partial_{x_3}^4 A_{j-2,3} ,
			\end{split}
		\end{equation}
		where $j = 3, \ldots, N-2$, and $A_{mj,3}$ with $|m| \neq 1$ satisfy $A_{mj,3} = g_{mj,3}^{gl} \left(\textbf{A}_{mj,3}, \eps_3 \right)$, where
		\begin{equation}
			\label{eq:GL_manifold_K3}
			\begin{split}
				g_{mj,3}^{gl} \left(\textbf{A}_{mj,3}, \eps_3 \right) &:= 
				- (\mathcal L_{m,3}^{(0)})^{-1} \bigg( \mathcal L_{m,3}^{(1)} A_{mj-1,3} + \tilde{\mathcal L_{m,3}^{(2)}} A_{mj-2,3} \\
				&+ \mathcal L_{m,3}^{(3)} A_{mj-3,3} + \mathcal L_{m,3}^{(4)} A_{mj-4,3} - a_{mj,3} + \delta_{\alpha(m)+j,4} \eps_3 \nu_m \bigg) ,
			\end{split}
		\end{equation}
		then the formal requirement $\textup{Res} (r_3 \psi_{n,3}) = O(r_3^n)$ is satisfied.
	\end{proposition}
	
	
	Equations \eqref{eq:GL_manifold_K1}, \eqref{eq:GL_manifold_K2} and \eqref{eq:GL_manifold_K3} define local coordinate representations of the GL manifold $\mathcal M^{gl}$ charts $\mathcal K_1$, $\mathcal K_2$ and $\mathcal K_3$ respectively. We have
	\begin{equation}
		\begin{split}
			\label{eq:Mgl_l}
			\mathcal M_1^{gl} &= \left\{ (\psi_{n,1}, r_1, \eps_1) \in X \times \mathbb R_+ \times \mathbb R_+ : A_{mj,1} = g_{mj,1}^{gl} \left(\textbf{A}_{mj,1}, \eps_1 \right) , |m| \neq 1, j \in \{ 1,\ldots,\alpha(m) \} \right\} , \\
			\mathcal M_2^{gl} &= \left\{ (\psi_{n,2}, v_2, r_2) \in X \times \mathbb R \times \mathbb R_+ : A_{mj,2} = g_{mj,2}^{gl} \left(\textbf{A}_{mj,2}, v_2 \right) , |m| \neq 1, j \in \{ 1,\ldots,\alpha(m) \} \right\} , \\
			\mathcal M_3^{gl} &= \left\{ (\psi_{n,3}, r_3, \eps_3) \in X \times \mathbb R_+ \times \mathbb R_+ : A_{mj,3} = g_{mj,3}^{gl} \left(\textbf{A}_{mj,3}, \eps_3 \right) , |m| \neq 1, j \in \{ 1,\ldots,\alpha(m) \} \right\} .
		\end{split}
	\end{equation}
	
	\begin{remark}
		\label{rem:invariance}
		Within the subspace of spatially periodic solutions the GL manifold is a center manifold; see \cite[Ch.~2.4]{Haragus2010} and \cite{Avitabile2020} for a proofs in the static and fast-slow settings respectively. For most solution spaces of interest, however, the GL manifold derived in classical modulation theory is not a center manifold due to the continuous spectrum and absence of a spectral gap. Nevertheless, it retains the attractivity properties which make it a useful object for approximating solutions to the SH equation \cite{Eckhaus1993,Schneider1995b}. 
	\end{remark}

	\subsection{Error estimates and dynamics}
	\label{sub:results_error_estimates_and_dynamics}
	
	Having presented the dynamic ansatz $\Psi_n$ and the corresponding modulation equations in Section \ref{sub:results_modulation_eqns}, we turn to our main results on the dynamics of modulation equations, the validity of the approximation $\Psi_n$, and, finally, on the dynamics of solutions to the dynamic SH equation \eqref{eq:sh_dynamic}. Since we consider only the improved approximation $\Psi_n$, we drop the subscript $n$ for notational simplicity.
	
	\subsubsection{Approximation dynamics}
	
	Our first result describes the behaviour of the approximation $\Psi$, which is determined via \eqref{eq:Psi_n_global} by the evolution of the modulation functions $A_{mj}$.
	We are interested in initial conditions $(\Psi(x,0), v(0)) = (\Psi^\ast(x), v(0))$ with $v(0) = - \rho_{in}$ for some fixed $\rho_{in} > 0$ and
	\begin{equation}
		\label{eq:Psi_ast}
		\Psi^\ast(x) = 
		\sum_{m \in I_N} \sum_{j=1}^{\tilde \alpha(m)} \rho_{in}^{(\alpha(m) + j) / 2} A_{mj}^\ast \left( \rho_{in}^{1/2} x \right) \me^{imx} , 
	\end{equation}
	where $A^\ast_{mj} = g^{gl}(\textbf{A}_{mj}^\ast, (\bar v^\ast, \bar \eps^\ast))$ for each $m \neq 1$ and $(\bar v^\ast, \bar \eps^\ast) := (\bar v(0), \bar \eps(0))$. For the rigorous results derived in this work, we work in the uniformly local Sobolev spaces $X = H_{ul}^\theta$ defined in Section \ref{sec:Background}. 
	In particular, it will be helpful to consider the initial conditions as a subset of the section
	\begin{equation}
		\label{eq:Delta_in}
		\Delta^{in} := \left\{ (\Psi^\ast, - \rho_{in}) : \Psi^\ast \in H_{ul}^\theta , \| \Psi \|_{H_{ul}^\theta} \leq K  \right\} ,
	\end{equation}
	where $K > 0$ is small but fixed. \SJ{Here and in the following, for each fixed $\eps \in [0,\eps_0]$ we identify an initial condition $(\Psi^\ast, - \rho_{in}) \in \Delta^{in}$ with a corresponding initial condition $(\Psi^\ast, -\rho_{in}, \eps) \in \Delta_\eps^{in} := \Delta^{in} \times [0,\eps_0]$ in the extended $(\Psi, v, \eps)$-space (the identification is achieved via the projection $(\Psi^\ast, -\rho_{in}, \eps) \mapsto (\Psi^\ast, -\rho_{in})$, which is a bijection if $\eps$ is fixed)}.
	
	In order to state our results, we also define the sections
	\begin{equation}
		\label{eq:Delta_mid}
		\Delta^{mid} := \left\{ (\Psi^\ast, \rho_{mid} \eps^{1/2} ) : \Psi^\ast \in H_{ul}^\theta , \| \Psi \|_{H_{ul}^\theta} \leq K  \right\}
	\end{equation}
	and
	\begin{equation}
		\label{eq:Delta_out}
		\Delta^{out} := \left\{ (\Psi, \rho_{out}) : \Psi^\ast \in H_{ul}^\theta , \| \Psi \|_{H_{ul}^\theta} \leq K  \right\}
	\end{equation}
	for positive constants $\rho_{mid}$ and $\rho_{out} > 0$. \SJ{Similarly, for each fixed $\eps \in [0,\eps_0]$, points in $\Delta^{mid}$ and $\Delta^{out}$ have natural identifications with points in the extended sections $\Delta_\eps^{mid} := \Delta^{mid} \times [0,\eps_0]$ and $\Delta_\eps^{out} := \Delta^{out} \times [0,\eps_0]$ respectively}.
	
	The section $\Delta^{mid}$ is located at a distance of $O(\eps^{1/2})$ in $v$ away from the Turing instability in the static SH equation \eqref{eq:sh}, which occurs for $v = 0$. This is precisely the scaling regime in which classical asymptotic theory for the onset of patterns after Turing instability occurs in the static SH equation, since in chart $\mathcal K_2$ we have $\eps = r_2^4 = \delta^2$, where $\delta \ll 1$ is the small perturbation parameter used to define the classical approximations \eqref{eq:GL_ansatz} and \eqref{eq:GL_ansatz_n}. The time taken for solutions to reach $\Delta^{mid}$ is given by
	\[
	T_{mid} := \eps^{-1} (\rho_{in} + \eps^{1/2} \rho_{mid}) ,
	\]
	which is obtained by directly integration of the equation for $v$ in \eqref{eq:sh_dynamic}. 
	The section $\Delta^{out}$ is located at an $O(1)$ distance from the static Turing instability at $v=0$. The time taken for solutions to reach $\Delta^{out}$ is given by
	\[
	T := \eps^{-1} (\rho_{in} + \rho_{out}) .
	\]
	In the following we present results on the maps $\pi^{mid} : \Delta^{in} \ni (\Psi^\ast, -\rho_{in}) \mapsto (\Psi(\cdot,T_{mid}), \rho_{mid} \eps^{1/2}) \in \Delta^{mid}$ and $\pi : \Delta^{in} \ni (\Psi^\ast, -\rho_{in}) \mapsto (\Psi(\cdot,T), \rho_{out})$ induced by forward evolution of the modulation equations in Section \ref{sub:results_modulation_eqns}.
	
	\
	
	Our first result concerns the map $\pi^{mid} : \Delta^{in} \to \Delta^{mid}$. 
	The constant $\nu_1$ is the same constant appearing in the definition of the source term in \eqref{eq:mu}.
	
	\begin{lemma}
		\label{lem:Ansatz_dynamics}
		\textup{(Approximation dynamics).}
		Assume that $\mu(x)$ is given by \eqref{eq:mu} and that the initial condition $\Psi^\ast$ is such that $A_{11}^\ast \in H_{ul}^{\theta_A}$, where $\theta_A = 1 + 3 (n-4) + \theta$ and $\theta > 1/2$. Then there exists an $\eps_0 > 0$ such that for all $\eps \in (0,\eps_0)$, $\Psi(\cdot,t) \in H_{ul}^\theta$ for all $t \in [0,T_{mid}]$. In particular, the map $\pi^{mid} : \Delta^{in} \to \Delta^{mid}$ is well-defined and given by
		\begin{equation}
			\label{eq:pi_mid_asymptotics}
			\pi^{mid} : 
			\begin{pmatrix}
				\Psi^\ast(x) \\
				- \rho_{in}
			\end{pmatrix}
			\mapsto 
			\begin{pmatrix}
				\eps^{1/2} \left( A_{12,2}(\eps^{1/4} x, \eps^{1/2}T_{mid}) \me^{ix} + c.c. \right) + \eps^{3/4} R(x,\eps) \\
				\eps^{1/2} \rho_{mid}
			\end{pmatrix} ,
		\end{equation}
		where $R(\cdot,\eps) \in H_{ul}^\theta$ and the function $A_{12,2}$ satisfies
		\[
		\left\| A_{12,2}(\eps^{1/4} \cdot,\eps^{1/2} T_{mid}) \me^{i\cdot} + c.c. \right\|_{H_{ul}^\theta} \leq 
		C \left( |\nu_1| + \me^{- \kappa \rho_{in}^2 / 2 \eps} \right) ,
		\]
		for a constant $\kappa \in (0,1)$.
	\end{lemma}
	
	Lemma \ref{lem:Ansatz_dynamics} is proved in Section \ref{sec:Proof_of_thm_dynamics}, and the behaviour of the norm is sketched in Figure \ref{fig:approximation_dynamics}(a). The main idea for the proof is to derive direct estimates in charts $\mathcal K_l$ after blowing-up, using multiplier theory to control the semigroup and different techniques like integrating factors and Gr\"onwall inequalities to bound the additional terms.
	
	\begin{figure}[t!]
		\centering
		\subfigure[General $\mu(x)$.]{\includegraphics[width=.7\textwidth]{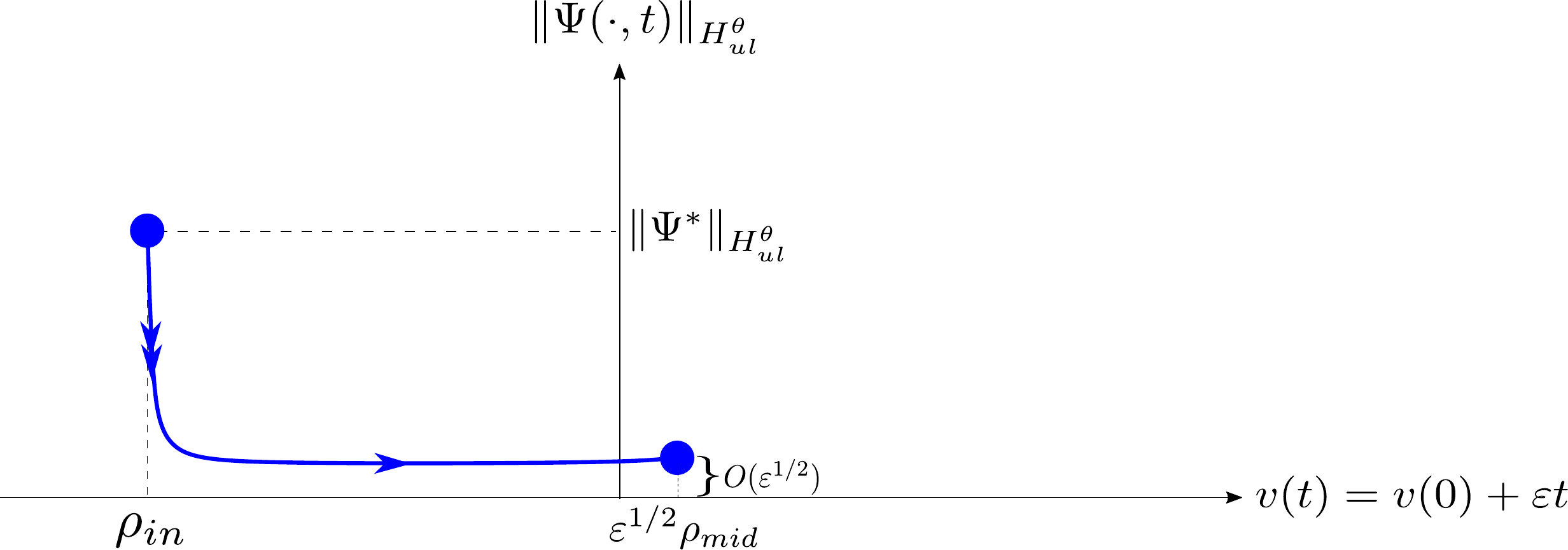}}
		\\
		\subfigure[$\mu(x) \equiv 0$]{\includegraphics[width=.7\textwidth]{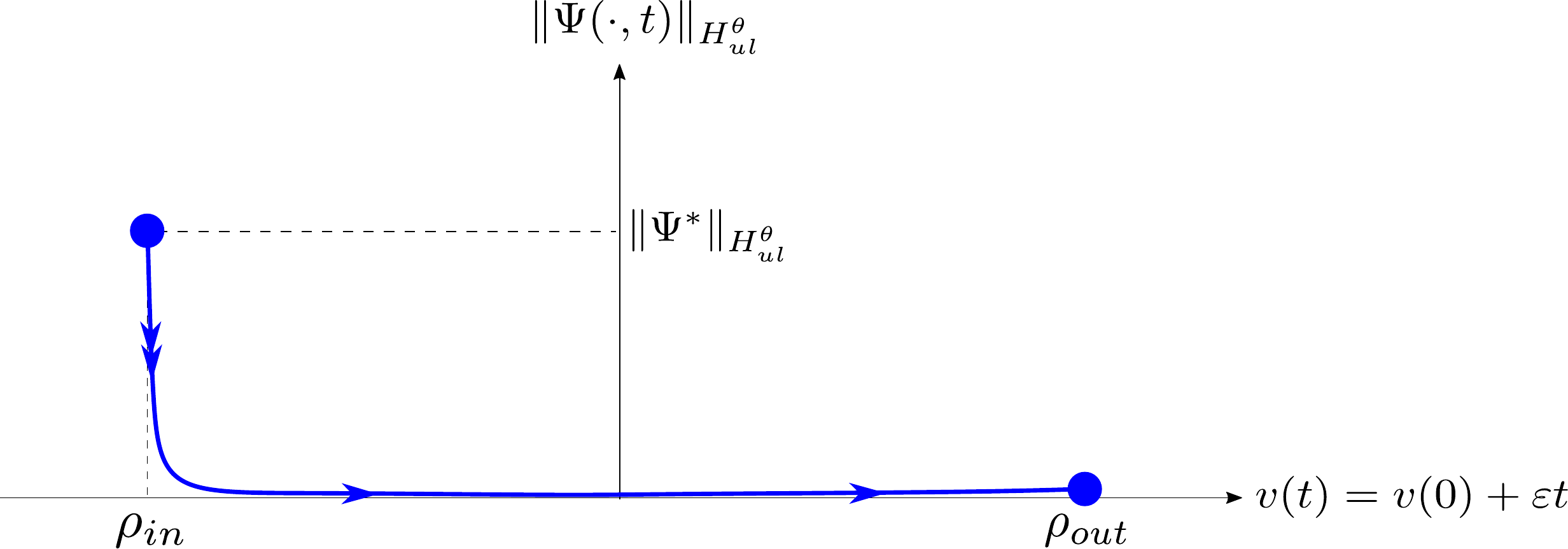}}
		\caption{$H_{ul}^\theta$ norm of $\Psi(\cdot,t)$ under forward evolution in time, starting from an initial condition in $\Delta^{in}$ with $v(0) = - \rho_{in} < 0$. Distinct behaviours are observed for general source functions $\mu(x) \neq 0$ in (a) vs.~the case $\mu(x) \equiv 0$ in (b). Figure (a) shows the situation described in Lemma \ref{lem:Ansatz_dynamics}, for which the solutions at $\Delta^{mid}$ are algebraically small but non-trivial if $\mu(x) \neq 0$. The delay phenomenon described by Lemma \ref{lem:Ansatz_dynamics_mu_0} in the case that $\mu(x) \equiv 0$ is sketched in Figure (b). According to \eqref{eq:delay_cond}, $\| \Psi(\cdot,T) \|_{H_{ul}^\theta}$ is exponentially small at $\Delta^{out}$ as long as the ratio $\rho_{out} / \rho_{in}$ is bounded below $1$.}
		\label{fig:approximation_dynamics}
	\end{figure}
	
	In contrast to the established asymptotic theory for the classical SH equation \eqref{eq:sh} with parameter values $v = O(\eps^{1/2}) = O(\delta^2)$, the leading order contribution in \eqref{eq:pi_mid_asymptotics} only arises at a higher order due to the effect of the source term $\eps \mu(x)$. In fact, it follows from the proof of Lemma \ref{lem:Ansatz_dynamics} in Section \ref{sec:Proof_of_thm_dynamics} and Lemma \ref{lem:Ansatz_dynamics_mu_0} below that the leading order approximation $\Psi_{GL} = r \psi_{GL}$, which is formally of order $O(\eps^{1/4}) = O(\delta)$ in this regime, is actually exponentially small with
	\[
	\| \Psi_{GL}(\cdot,T_{mid}) \|_{H_{ul}^\theta} \leq C \me^{- \kappa \rho_{in}^2 / 2 \eps} 
	\]
	for constants $C > 0$ and $\kappa \in (0,1)$. This is because the leading order GL-type approximation leads to a modulation equation of the form \eqref{eq:GL_eqn_global} (see also Proposition \ref{prop:GL_eqns_leading_order}), which features a \textit{delayed loss of stability} due in part to the presence of a left-right symmetry $A \leftrightarrow - A$. If $\mu(x) \neq 0$ with $\nu_1 \neq 0$, this symmetry is broken in the equations for $A_{\pm12}$, which explains why the leading order contribution to the asymptotics in \eqref{eq:pi_mid_asymptotics} is $\eps^{1/2} (A_{12,2}(\eps^{1/4} x, \eps^{1/2}T_{mid}) \me^{ix} + c.c.)$. If $\nu_1 = 0$ but $\nu_m \neq 0$ for some $m \neq \pm 1$, the symmetry is broken at a higher order and the leading order contribution is expected to be algebraic of size $O(\eps^{l/4})$ for a positive integer $l \geq 3$ which is increasing with $|m|$.
	
	\begin{remark}
		\label{rem:rolls}
		The modulation functions $A_{\pm12,2}$ determining the leading order asymptotics in \eqref{eq:pi_mid_asymptotics} are determined by the evolution equation for $A_{2,2}$ in system \eqref{eq:K2_eqns} (recall that we write $A_{j,l} = A_{\pm1j,l}$ in Propositions \ref{prop:K1_eqns}-\ref{prop:K3_eqns}). Since the formal leading order modulation functions $A_{\pm12,1}$ are exponentially small, $a_{2,1}$ and $-4i\partial_{x_2}^3 A_{1,1}$ are also exponentially small. One therefore expects that solutions for $A_{\pm12,2}$ are closely approximated by solutions of the linear inhomogeneous equations
		\[
		\partial_{t_2} A_{\pm12,2} = 4 \partial_{x_2}^2 A_{\pm12,2} + v(t_2) A_{\pm12,2} + \nu_{\pm1} , \qquad A_{\pm12,2}(x_2,0) = A_{\pm12,2}^\ast(x_2) .
		\]
		Based on the calculations in Section \ref{sec:Proof_of_thm_dynamics}, the leading order approximation for the initial condition is $A_{\pm12,2}(x_2) \sim \nu_1 \zeta^{-1/2} f_{2,1}(0)$, where $\zeta$ and $f_{2,1}(0)$ are small but fixed positive constants (see Lemma \ref{lem:K1_A2}). Applying the solution formula for the heat equation (see e.g.~\cite[Ch.~2.3]{Evans2010}) leads to the following prediction for the leading order correction
		\begin{equation}
			\label{eq:A12_guess}
			A_{\pm12,2}(x_2,t_2) \approx 
			\nu_1 \zeta^{-1/2} \left( f_{2,1}(0) + t_2 \right) \exp\left(- \frac{t_2}{\sqrt{\zeta}} + \frac{t_2^2}{2} \right) ,
		\end{equation}
		which is constant and bounded when evaluated at the time $T_2$ corresponding to the intersection of solutions with $\Delta^{mid}$. This suggests that all solutions described by Lemma \ref{lem:Ansatz_dynamics} are closely approximated by spatially periodic ``rolls" of size $O(\eps^{1/2})$ at $\Delta^{mid}$. We leave the rigorous justification of this claim for future work.
	\end{remark}
	
	Our second result concerns the map $\pi : \Delta^{in} \to \Delta^{out}$. It applies specifically to the symmetric case $\mu(x) \equiv 0$.
	
	\begin{lemma}
		\label{lem:Ansatz_dynamics_mu_0}
		\textup{(Approximation dynamics, $\mu(x) \equiv 0$).}
		Assume that $\mu(x) \equiv 0$ and that the initial condition $\Psi^\ast$ is such that $A_{11}^\ast \in H_{ul}^{\theta_A}$, where $\theta_A = 1 + 3 (n-4) + \theta$ and $\theta > 1/2$. Assume that
		\begin{equation}
			\label{eq:delay_cond}
			\frac{\rho_{out}}{\rho_{in}} \leq 1 - \omega ,
		\end{equation}
		where the constant $\omega = \omega(K) > 0$ \SJ{is fixed but} arbitrarily close to zero if $K > 0$ is sufficiently small, where $K$ is the constant which bounds size of initial conditions $\Psi^\ast$ in $\Delta^{in}$. Then there exists an $\eps_0 > 0$ such that for all $\eps \in (0,\eps_0)$, $\Psi(\cdot,t) \in H_{ul}^\theta$ for all $t \in [0,T]$. In particular, the map $\pi : \Delta^{in} \to \Delta^{out}$ is well-defined and given by
		\[
		\pi : 
		\begin{pmatrix}
			\Psi^\ast(x) \\
			- \rho_{in}
		\end{pmatrix}
		\mapsto 
		\begin{pmatrix}
			\Psi(x,T) \\
			\rho_{out}
		\end{pmatrix} ,
		\]
		where
		\begin{equation}
			\label{eq:Psi_bound}
			\| \Psi(\cdot, T) \|_{H_{ul}^\theta} \leq
			C \exp\left(- \frac{\kappa_-}{2\eps} \left(\rho_{in}^2 - \kappa_+ \rho_{out}^2 \right) \right) , 
		\end{equation}
		where $\kappa_- \in (0,1)$ and $\kappa_+ > 1$ are arbitrarily close to $1$ for sufficiently small $K$.
	\end{lemma}
	
	The situation is sketched in Figure \ref{fig:approximation_dynamics}(b). As for the proof of Lemma \ref{lem:Ansatz_dynamics}, the proof of Lemma \ref{lem:Ansatz_dynamics_mu_0} relies on direct estimates in charts $\mathcal K_l$, and is deferred to Section \ref{sec:Proof_of_thm_dynamics}. Lemma \ref{lem:Ansatz_dynamics} asserts that solutions with initial conditions $\Psi^\ast \in H_{ul}^{\theta}$ undergo a delayed loss of stability. Equation \eqref{eq:Psi_bound} provides a lower bound for the delay in terms of $\rho_{in}$ and $\rho_{out}$. The fact that the constants $\kappa_\pm$ can be chosen arbitrarily close to $1$ if $K>0$ is sufficiently small suggests that the delay is symmetric about $v = 0$, i.e.~that solutions which enter a sufficiently small (but $\eps$-independent) neighbourhood about the trivial solution $\Psi(x,t) = \Psi(t) \equiv 0$ at $v = - c < 0$ can be expected to leave a neighbourhood of $\Psi(x,t) = \Psi(t) \equiv 0$ at $v = c > 0$. We emphasise however that Lemma \ref{lem:Ansatz_dynamics_mu_0} only provides a lower bound for the delay time. In order to show rigorously that solutions depart from a neighbourhood of $\Psi(x,t) = \Psi(t) \equiv 0$ near a particular value for $v$, we would also require control on the lower bound for the norm $\| \Psi(\cdot, T) \|_{H_{ul}^\theta}$, similarly to the identification of delay times in \cite{Kaper2018} using the method of upper and lower solutions. This task is left for future work.
	
	\begin{remark}
		One reason to expect that a sharp transition occurs at the symmetric value $v = - \rho_{in}$ comes from considering the linearized problem, which is expected to serve as a good approximation when $\| \Psi(\cdot,t) \|_{H_{ul}^\theta}$ is small, as it is in Lemma \ref{lem:Ansatz_dynamics_mu_0}. If we consider only the leading order GL approximation, assume a periodic initial conditions $\Psi^\ast = \alpha \me^{ix} + c.c.$ with constant such that $|\alpha| \leq K$ and ignore the cubic nonlinear term in equation \eqref{eq:GL_eqn_global}, then the following explicit expression for the modulation functions $A_{\pm11,3}(x_3,T_3)$ in chart $\mathcal K_3$ at the exit section $\Delta^{out}$ can be identified using on the calculations in Section \ref{sec:Proof_of_thm_dynamics}:
		\[
		A_{1,3}(x_3,T_3) = A_{1,3}(T_3) = 
		\alpha \sqrt{\frac{\rho_{in}}{\rho_{out}}} \exp \left( - \frac{1}{2 \eps} (\rho_{in}^2 - \rho_{out}^2) + \frac{1}{2\zeta} (\rho_{in}^2 - \rho_{out}^2) \right) ,
		\]
		where $\zeta > 0$ is the same small but fixed constant as in Remark \ref{rem:rolls}. Explicit estimates for ``space-time buffer curves" corresponding to the hard transition to patterned states in the fast-slow complex GL equation have been derived formally based on comparisons to the linearized problem in \cite{Goh2022,Kaper2018}.
	\end{remark}
	
	\begin{remark}
		In contrast to the symmetric case $\mu(x) \equiv 0$, our results suggest that no delayed stability loss occurs if $\mu(x) \neq 0$. In particular, the conjectured approximation in \eqref{eq:A12_guess} grows exponentially in chart $\mathcal K_3$, where the time interval is $O(\eps^{-1})$. As opposed to a delayed bifurcation, these observations indicate an ``exchange of stability" similar to that described in the sequence of papers \cite{Butuzov2002b,Butuzov1999,Butuzov2000,Butuzov2001,Butuzov2002} and also in \cite{Engel2021}.
	\end{remark}
	
	\begin{remark}
		\SJJ{The results in Lemmas \ref{lem:Ansatz_dynamics} and \ref{lem:Ansatz_dynamics_mu_0} are stated for $\theta > 1/2$, although the spaces $H_{ul}^\theta$ are only defined in Section \ref{sec:Background} for integer values $\theta \in \mathbb N$. The requirement that $\theta > 1/2$ stems from the application of a Sobolev-type inequality that holds for a generalised definition of $H_{ul}^\theta$ which applies for all $\theta \geq 0$ and coincides with the simpler definition given in Section \ref{sec:Background} when $\theta \in \mathbb N$. We omit this more general definition for brevity, and refer to \cite[Sec.~2]{Schneider1994b} for details. See also Lemma \ref{lem:US_8.3.11} in Appendix \ref{app:technical_estimates} for the Sobolev inequality in question.}
	\end{remark}

	\subsubsection{Error estimates}
	
	Having described the dynamics of the approximation $\Psi$ in Lemmas \ref{lem:Ansatz_dynamics} and \ref{lem:Ansatz_dynamics_mu_0}, we turn the question of when and to what extent the approximation is valid. As we have seen, approximations like the improved GL approximation in \eqref{eq:GL_ansatz_n} and its dynamic generalisation via blow-up in \eqref{eq:blowup_Psi} are constructed by formally minimising a residual. However, formal smallness of the residual is only a necessary condition for accurate approximation by a set of GL-type modulation equations, and there are numerous counterexamples in the literature which show that modulation equations derived in a formally correct way can provide poor approximations due to errors which accumulate in time \cite{Schneider1995a,Schneider2005,Schneider2015}. A rigorous justification which parallels the classical result in Theorem \ref{thm:Error_static} is needed in order to prove statements about the dynamics of the dynamic SH equation \eqref{eq:sh_dynamic} based on the dynamics of the approximation $\Psi$.
	
	\
	
	The following result provides a rigorous quantification of the error $u - \Psi$ over a specified time interval, and should be compared with Theorem \ref{thm:Error_static}.
	
	\begin{thm}
		\label{thm:Error}
		\textup{(Validity of the approximation $\Psi$).}
		Let $\theta \geq 1$ and $A_{11} \in C([0,\widehat T], H_{ul}^{\theta_A})$ be a solution to the GL equation \eqref{eq:GL_eqn_global}, where $\theta_A = 1 + 3 (n-3) + \theta$, $\widehat T = T_{mid}$ if $\mu(x) \neq 0$, and $\widehat T = T$ if $\mu(x) \equiv 0$. Then there exists an $\eps_0 > 0$ such that for all $\eps \in (0,\eps_0)$, the following assertions hold for all SH solutions $u$ with initial conditions $u(x,0) = u^\ast(x)$ such that $\| u^\ast \|_{H_{ul}^\theta} \leq K$:
		\begin{enumerate}
			\item[(i)] The error satisfies $u(\cdot,t) - \Psi(\cdot,t) \in H_{ul}^\theta$ for all $t \in [0,T_{mid}]$, and 
			there exists a constant $C> 0$ such that
			\[
			\| u(\cdot,T_{mid}) - \Psi(\cdot,T_{mid}) \|_{H_{ul}^\theta} \leq C \eps^{(n-2)/4} .
			\]
			\item[(ii)] Let $\mu(x) \equiv 0$ and assume that 
			\begin{equation}
				\label{eq:error_cond}
				\frac{\rho_{in}}{\rho_{out}} \geq \sqrt {2 \kappa_+} ,
			\end{equation}
			where $\kappa_+ > 1$ is the same constant as in Lemma \ref{lem:Ansatz_dynamics_mu_0}, which can be chosen arbitrarily close to $1$ if $K > 0$ is sufficiently small. Then the error satisfies $u(\cdot,t) - \Psi(\cdot,t) \in H_{ul}^\theta$ for all $t \in [0,T]$, and 
			there exists a constant $C> 0$ such that
			\[
			\| u(\cdot,T) - \Psi(\cdot,T) \|_{H_{ul}^\theta}
			\leq C \exp\left( - \frac{\kappa_-}{2 \eps} \left( \rho_{in}^2 - 2 \kappa_+ \rho_{out}^2 \right) \right) ,
			\]
			where $\kappa_-$ is the same constant as in Lemma \ref{lem:Ansatz_dynamics_mu_0}, which can also be chosen arbitrarily close to $1$ if $K > 0$ is sufficiently small.
		\end{enumerate}
	\end{thm}
	
	Similarly to the proof of Lemmas \ref{lem:Ansatz_dynamics}-\ref{lem:Ansatz_dynamics_mu_0}, the proof of Theorem \ref{thm:Error}, which is deferred until Section \ref{sec:Proof_of_thm_error}, relies on direct estimates in charts $\mathcal K_l$ after blow-up. Similarly to the established method of proof in the static setting -- recall Theorem \ref{thm:Error_static}, for which we refer again to \cite[Ch.~10]{Schneider2017} and the references therein -- the idea is to use a-priori estimates on the norm of solutions to an evolution equation for the error $u - \Psi$. Since the equation for the error is posed in the blown-up space, however, we need to derive bounds in each coordinate chart $\mathcal K_l$. Moreover, we need to account for additional time-dependence induced by the forward evolution in $v$ in the dynamic SH equation \eqref{eq:sh_dynamic}.
	
	It is worthy to note that Theorem \ref{thm:Error} applies for a very general space of initial conditions $\{ u^\ast \in H_{ul}^{\theta} : \| u^\ast \|_{H_{ul}^\theta} \leq K \}$, where the constant $K > 0$ must be sufficiently small, but remains fixed and $O(1)$ with respect to $\eps \to 0$. In particular, we \textit{do not} require that $u^\ast = \Psi^\ast$. A similar result has been shown to hold for the static SH equation in \cite{Eckhaus1993,Schneider1995b}, except that in this case, initial conditions have to satisfy $\| u^\ast \|_{H_{ul}^\theta} \leq C \delta$; 
	recall the discussion following the statement of Theorem \ref{thm:Error_static} in Section \ref{sec:Background}. In the dynamic setting, the diameter of the space of initial conditions can be improved to $O(1)$ because we consider initial conditions with $v(0) = \rho_{in} < 0$ which undergo strong contraction for an initial time period during which $v < - c < 0$ for some small but fixed $c$. With regards to the error estimates themselves, observe that that the estimate in Assertion (i) agrees with the classical estimate in Theorem \ref{thm:Error_static} after recalling that $\eps = \delta^2$ in chart $\mathcal K_2$. The estimate in Assertion (ii) is valid only under the assumption that \eqref{eq:error_cond}, which is stronger than the delay criterion in \eqref{eq:delay_cond}, is satisfied. This is because in the proof of Theorem \ref{thm:Error} Assertion (ii) we are forced to adopt the stronger condition in \eqref{eq:error_cond} in order to control unwanted growth of the residual \eqref{eq:residual_def} in time.
	
	\begin{remark}
		\label{rem:improved_residual_estimates}
		Calculations in the proof of Assertion (ii) in Section \ref{sec:Proof_of_thm_error} suggest that the bound \eqref{eq:error_cond} is likely to be sub-optimal because the residual \eqref{eq:residual_def} can only be made algebraically (as opposed to exponentially) small as $\eps \to 0$. It is possible that the bound could be improved by considering the infinite series ansatz $\Psi_{\infty}$ obtained by setting $n=\infty$ in \eqref{eq:Psi_n_global} and, if necessary, considering exponentially small remainders `beyond all orders'. Exponentially small estimates have already been obtained for a variant of the (static) SH equation with a quintic nonlinearity in \cite{Dean2011} using beyond all orders asymptotics. Better estimates have also been obtained for the (static) Kuramoto-Sivashinsky equation in \cite{Schneider1996} via an alternate approach based on comparisons to solutions of a different approximating system which possesses an infinite dimensional center manifold. Thus the order of the approximation $n$, which does not appear in Assertion (ii), may still be important for the derivation of improved error estimates in the symmetric case $\mu(x) \equiv 0$. 
	\end{remark}

	\subsubsection{Swift-Hohenberg dynamics}
	
	By combining results on the approximation dynamics in Lemmas \ref{lem:Ansatz_dynamics}-\ref{lem:Ansatz_dynamics_mu_0} and the error estimates in Theorem \ref{thm:Error}, we obtain the following result, which characterises the dynamics of the dynamic SH equation \eqref{eq:sh_dynamic}.
	
	\begin{thm}
		\label{thm:Dynamics}
		\textup{(Dynamic Swift-Hohenberg dynamics).}
		Let $u$ be a solution to the dynamic SH equation \eqref{eq:sh_dynamic} with initial condition $(u^\ast(x),-\rho_{in})$ such that $\| u^\ast \|_{H_{ul}^\theta} \leq K$, and let $A_{11}$ be a solution to the GL equation \eqref{eq:GL_eqn_global} satisfying the assumptions of Theorem \ref{thm:Error}. There exists an $\eps_0 > 0$ such that for all $\eps \in (0,\eps_0)$ the following assertions are true:
		\begin{enumerate}
			\item[(i)] $u(\cdot,t) \in H_{ul}^\theta$ for all $t \in [0,T_{mid}]$, and
			\[
			u(x,T_{mid}) = \eps^{1/2} \left( A_{12,2}(\eps^{1/4} x, \eps^{1/2}T_{mid}) \me^{ix} + c.c. \right) + \eps^\alpha \tilde R(x,\eps) ,
			\]
			where $\alpha = \max \{3/4, (n-2)/4\}$ and $\tilde R(\cdot,\eps) \in H_{ul}^\theta$.
			\item[(ii)] Let $\mu(x) \equiv 0$ and assume that \eqref{eq:error_cond} is satisfied. Then $u(\cdot,t) \in H_{ul}^\theta$ for all $t \in [0,T]$, and there exists a constant $C > 0$ such that
			\[
			\| u(\cdot, T) \|_{H_{ul}^\theta} \leq
			C \exp\left(- \frac{\kappa_-}{2\eps} \left(\rho_{in}^2 - 2 \kappa_+ \rho_{out}^2 \right) \right) .
			\]
		\end{enumerate}
	\end{thm}
	
	\begin{proof}
		This follows directly from Lemmas \ref{lem:Ansatz_dynamics}-\ref{lem:Ansatz_dynamics_mu_0} and Theorem \ref{thm:Error}. Let $E := u - \Psi$ denote the error. We have
		\begin{equation}
			\label{eq:u_approx}
			u(x,t) = \Psi(x,t) + E(x,t) ,
		\end{equation}
		where $\| E(\cdot, \widehat T) \|_{H_{ul}^\theta}$ satisfies the bound in Theorem \ref{thm:Error} Assertion (i) if $\widehat T = T_{mid}$ and Assertion (ii) if $\widehat T = T$.
		
		Substituting $\widehat T = T_{mid}$ into \eqref{eq:u_approx} and applying Lemma \ref{lem:Ansatz_dynamics} leads to
		\[
		u(x,T_{mid}) = \eps^{1/2} \left( A_{12,2}(\eps^{1/4} x, \eps^{1/2}T_{mid}) \me^{ix} + c.c. \right) + \eps^{3/4} R(x,\eps) + E(x,T_{mid}) .
		\]
		Writing $\eps^\alpha \tilde R(x,\eps) := \eps^{3/4} R(x,\eps) + E(x,T_{mid})$, it follows by Theorem \ref{thm:Error} Assertion (i) that
		\[
		\eps^\alpha \| \tilde R(\cdot,\eps) \|_{H_{ul}^\theta} \leq  
		\eps^{3/4} \| R(\cdot,\eps) \|_{H_{ul}^\theta} + \| E(\cdot,T_{mid}) \|_{H_{ul}^\theta} \leq 
		C \max\{ \eps^{3/4}, \eps^{(n-2) / 4}\} ,
		\]
		for a constant $C > 0$. This proves Assertion (i).
		
		\
		
		Now let $\mu(x) \equiv 0$ and assume that \eqref{eq:error_cond} is satisfied. In this case Lemma \ref{lem:Ansatz_dynamics_mu_0} and Theorem \ref{thm:Error} Assertion (ii) apply, and we obtain
		\[
		\begin{split}
			\| u(\cdot,T) \|_{H_{ul}^\theta} &\leq 
			\| \Psi(\cdot,T) \|_{H_{ul}^\theta} + \| E(\cdot,T) \|_{H_{ul}^\theta} \\
			&\leq C \left( \exp\left(- \frac{\kappa_-}{2\eps} \left(\rho_{in}^2 - \kappa_+ \rho_{out}^2 \right) \right) + \exp\left(- \frac{\kappa_-}{2\eps} \left(\rho_{in}^2 - 2 \kappa_+ \rho_{out}^2 \right) \right) \right) \\
			&\leq C \exp\left(- \frac{\kappa_-}{2\eps} \left(\rho_{in}^2 - 2 \kappa_+ \rho_{out}^2 \right) \right) ,
		\end{split}
		\]
		for a constant $C > 0$ and for all $\eps \in (0,\eps_0)$ with $\eps_0 > 0$ sufficiently small. This proves Assertion (ii).
	\end{proof}
	
	Despite the fact that the approximation results in Lemmas \ref{lem:Ansatz_dynamics}-\ref{lem:Ansatz_dynamics_mu_0} only apply for initial conditions in $\mathcal M^{gl}$, Theorem \ref{thm:Dynamics} describes the forward evolution the much larger space of initial conditions $\{ u^\ast \in H_{ul}^{\theta} : \| u^\ast \|_{H_{ul}^\theta} \leq K \}$. This is possible because the error estimates in Theorem \ref{thm:Error} apply for this larger space of solutions. By Assertion (i), an approximation of order $n \geq 5$ is needed in the general case with $\mu(x) \neq 0$ in order to ensure that the remainder term is higher order with $\tilde R(x,\eps) = O(\eps^{3/4})$. Assertion (ii) shows that the delayed stability loss phenomenon identified for the approximating system in Lemma \ref{lem:Ansatz_dynamics_mu_0} is also present in the dynamic SH equation. In contrast to Lemma \ref{lem:Ansatz_dynamics_mu_0}, we have to assume the stronger and likely sub-optimal condition \eqref{eq:error_cond}. This is due to the unwanted growth in the error over time, recall the discussion following Theorem \ref{thm:Error} and Remark \ref{rem:improved_residual_estimates} above.
	
	\begin{remark}
		Time-dependent estimates can be obtained in the results of this section by rewriting the constants $\rho_{mid}$ and $\rho_{out}$ defining the exist sections $\Delta^{mid}$ and $\Delta^{out}$ in terms of the corresponding transition times $T_{mid}$ and $T$ respectively. Specifically, time-dependent estimates in Lemma \ref{lem:Ansatz_dynamics}, Theorem \ref{thm:Error} Assertion (i) and Theorem \ref{thm:Dynamics} Assertion (i) are obtained by substituting $\eps^{1/2} \rho_{mid} = - \rho_{in} + \eps T_{mid}$, where $T_{mid} = \eps^{-1} \rho_{in} + O(\eps^{-1/2})$, while time-dependent estimates in Lemma \ref{lem:Ansatz_dynamics_mu_0}, Theorem \ref{thm:Error} Assertion (ii) and Theorem \ref{thm:Dynamics} Assertion (ii) are obtained by substituting $\rho_{out} = - \rho_{in} + \eps T$, where $T = \eps^{-1} \rho_{in} + O(\eps^{-1})$. In particular, this yields lower bounds on the delay time $t_{delay}$ as measured from $v = 0$. From Lemma \ref{lem:Ansatz_dynamics_mu_0} it follows that the delay time for the approximation $\Psi$ satisfies $t_{delay} \geq \eps^{-1} \rho_{in} / \sqrt{\kappa_+}$, and from Theorem \ref{thm:Dynamics} Assertion (ii) it follows that the delay time for a solution to the SH equation $u$ has $t_{delay} \geq \eps^{-1} \rho_{in} / \sqrt{2 \kappa_+}$.
	\end{remark}

	\section{Formal derivation of the modulation equations via geometric blow-up}
	\label{sec:Amplitude_reduction_via_geometric_blow-up}
	
	In this section we derive the modulation equations presented in Section \ref{sub:results_modulation_eqns}. We start with an approximation $\Psi$ in the form proposed in \eqref{eq:Psi_global}, and derive the modulation equations for the functions $A_{mj}$ by imposing the formal requirement that $\textup{Res}(\Psi) = O(r^n)$. This procedure can be viewed as a direct extension of the established method for deriving modulation equations for the static SH equation in e.g.~\cite[Ch.~10]{Schneider2017}. The main differences here are the following:
	\begin{itemize}
		\item The `small parameter' $r(\bar t) \geq 0$ is now a variable which depends on time; 
		\item The geometry is more complicated (we work in a blown-up space);
		\item The time and space variables $\bar x$ and $\bar t$ depend non-trivially on the state-space via their defining equations in \eqref{eq:desingularization}.
	\end{itemize}
	Note that the modulation functions $A_{mj}$ are also constrained with respect to $v$ and $\eps$ via the blow-up map \eqref{eq:blowup_Psi}. 
	
	\begin{remark}
		The formal procedure for deriving modulation equations in both this work and classical modulation theory is similar to the established formal method for approximating center manifolds up to arbitrary but algebraic order; see e.g.~\cite{Haragus2010,Vanderbauwhede1992}. This is true despite the fact that the GL-manifold for system \eqref{eq:sh_dynamic} is not generally a center manifold; recall Remark \ref{rem:invariance}. This close relationship to center manifold theory has also been explored in detail in the context of the Kuramoto-Sivashinsky equation in \cite{Schneider1996}.
	\end{remark}
	
	We write
	\begin{equation}
		\label{eq:Psi_global}
		\Psi(x,t) = r(\bar t) \psi(x, \bar t) = \sum_{m \in I_N} c_m(\bar x, \bar t) \me^{i m x} , \qquad 
		c_m(\bar x, \bar t) = \sum_{j=1}^{\tilde \alpha(m)} r(\bar t)^{\alpha(m) + j} A_{mj}(\bar x, \bar t) ,
	\end{equation}
	and recall that $I_N = \{-N, \ldots , N\}$, $N := n-1 \geq 3$, $\alpha(m) := | |m| - 1|$ and $\tilde \alpha(m) := N - \alpha(m) - 2 \delta_{|m|,1}$. We also impose the reality condition $c_m = \overline{c_{-m}}$ for all $m \in I_N$ in order to ensure that $\Psi(x,t) \in \mathbb R$, and introduce local coordinate representations in charts $K_l$ via
	\begin{equation}
		\label{eq:psi_definition}
		\Psi_l(x,t) = r_l(t_l) \psi_l(x, t_l) = \sum_{m \in I_N} c_{m,l}(x_l, t_l) \me^{i m x} , \qquad 
		c_{m,l}(x_l, t_l) = \sum_{j=1}^{\tilde \alpha(m)} r_l(t_l)^{\alpha(m) + j} A_{mj,l}(x_l, t_l) ,
	\end{equation}
	for $l = 1,2,3$, where $x_l$ and $t_l$ are defined via $\partial_t = r_l(t_l)^2 \partial_{t_l}$ and $\partial_x = r_l(t_l) \partial_{x_l}$ (the local coordinate representation of \eqref{eq:desingularization}).
	The change of coordinates maps between charts $\mathcal K_l$ are denoted by $\kappa_{12} : \mathcal K_1 \to \mathcal K_2$ and $\kappa_{23} : \mathcal K_2 \to \mathcal K_3$, and given by
	\begin{equation}
		\label{eq:kappa_maps}
		\begin{aligned}
			\kappa_{12} :\ & \psi_1 = (-v_2)^{-1/2} \psi_2 , && r_1 = (-v_2)^{1/2} r_2 , && \eps_1 = v_2^{-2} , && v_2 < 0 , \\
			\kappa_{23} :\ & \psi_2 = \eps_3^{-1/4} \psi_3 , && v_2 = \eps_3^{-1/2} , && r_2 = \eps_3^{1/4} r_3 , && \eps_3 > 0 .
		\end{aligned}
	\end{equation}
	The change of coordinates formulae for $\psi_l$ imply the following formulae for the modulation functions $A_{mj,l}$.
	
	\begin{lemma}
		\label{lem:kappa_Amj}
		Modulation functions $A_{mj,l}$ in different charts $\mathcal K_l$ are related via 
		\begin{equation}
			\label{eq:kappa_Amj}
			A_{mj,1} = (-v_2)^{-(\alpha(m) + j) / 2} A_{mj,2} , \ v_2 < 0, \qquad
			A_{mj,2} = \eps_3^{- (\alpha(m) + j) / 4} A_{mj,3} , \ \eps_3 > 0.
		\end{equation}
		The formulae \eqref{eq:kappa_Amj} hold for each $m \in I_N$ and corresponding $j \in \{ 1, \ldots, \tilde \alpha(m)\}$.
	\end{lemma}
	
	\begin{proof}
		The definition \eqref{eq:Psi_global} imposes the requirement that $c_{m,l_1} = c_{m,l_2}$ where charts $\mathcal K_{l_1}$ and $\mathcal K_{l_2}$ overlap, for $(l_1, l_2) = (1,2)$ or $(l_1, l_2) = (2,3)$. This yields the requirement
		\[
		\sum_{j=1}^{\tilde \alpha(m)} r_{l_1}^{\alpha(m) + j} A_{mj,l_1} = 
		\sum_{j=1}^{\tilde \alpha(m)} r_{l_2}^{\alpha(m) + j} A_{mj,l_2} .
		\]
		If $(l_1, l_2) = (1,2)$ then using $r_1 = (-v_2)^{1/2} r_2$ (which is valid for $v_2 < 0$) we obtain
		\[
		\sum_{j=1}^{\tilde \alpha(m)} ((-v_2)^{1/2} r_2)^{\alpha(m) + j} A_{mj,1} = 
		\sum_{j=1}^{\tilde \alpha(m)} r_{2}^{\alpha(m) + j} A_{mj,2} ,
		\qquad v_2 < 0 ,
		\]
		which is only satisfied for all $r_2 \in [0,\vartheta]$ if $A_{mj,1} = (-v_2)^{-(\alpha(m) + j) / 2} A_{mj,2}$ for all $j \in \{1, \ldots, \tilde \alpha(m) \}$, as required. The case $(l_1, l_2) = (2,3)$ is similar.
	\end{proof}
	
	After applying the blow-up map \eqref{eq:blowup_Psi}, the residual \eqref{eq:residual_def} is given by
	\begin{equation}
		\label{eq:residual_def_blowup}
		\textrm{Res}(r \psi) = - \partial_t (r \psi) - (1 + \partial_x^2)^2 (r \psi) + r^3 \bar v \psi - r^3 \psi^3 + r^4 \bar \eps \mu(x) . 
	\end{equation}
	The aim is to formally minimise \eqref{eq:residual_def_blowup} in powers of $r$. First we expand the cubic nonlinearity via
	\begin{equation}
		\label{eq:cubic}
		\Psi^3 = (r \psi)^3 = \left( \sum_{m \in I_N} c_m \me^{imx} \right)^3 =
		\sum_{m = -3N}^{3N}
		b_m \me^{imx} ,
	\end{equation}
	where the coefficients $b_m$, which are defined via the right-most equality, depend the original coefficients $(c_m)_{m \in I_N}$ and therefore on $r$ via \eqref{eq:Psi_global}. 
	Note also that the $b_m$ are formally $O(r^3)$ or smaller since $\Psi^3 = O(r^3)$. Substituting \eqref{eq:cubic} and $\mu(x) = \sum_{m \in I_N} \nu_m \me^{imx}$ (recall the definition in \eqref{eq:mu}) into \eqref{eq:residual_def} leads to
	\begin{equation}
		\label{eq:residual_2}
		\textrm{Res} (r \psi) = \sum_{m \in I_N} \left( \mathcal L_m c_m -
		b_m + r^4 \bar \eps \nu_m \right) \me^{imx} - \sum_{m \in I_{3N} \setminus I_N} b_m \me^{imx} ,
	\end{equation}
	where $\mathcal L_m$ is the linear operator defined by
	\begin{equation}
		\label{eq:L_m}
		\mathcal L_m := - r^2 \partial_{\bar t} - (1-m^2)^2 - 4ir m(1-m^2) \partial_{\bar x} + \bar v r^2 - 2 r^2 (1-3m^2) \partial_{\bar x}^2 - 4ir^3 m \partial_{\bar x}^3 - r^4 \partial_{\bar x}^4 .
	\end{equation}
	Note that we expanded $-(1+\partial_x^2)^2 c_m \me^{imx}$ and applied the desingularization \eqref{eq:desingularization} in order to obtain the expression for $\mathcal L_m$, which is formulated in terms of the transformed $\bar t$ and $\bar x$. The second sum in \eqref{eq:residual_2} is $O(r^n)$. 
	
	Now we substitute the power series expressions for $c_m$ in \eqref{eq:Psi_global} and match in powers of $r$. The aim is to show that $(c_m)_{m \in I_N}$ can be chosen such that the first sum in \eqref{eq:residual_2} is $O(r^n)$. We first consider the linear part $\mathcal L_m c_m$. In order to distinguish terms at different orders in $r$ we define
	\[
	\mathcal L_m := \mathcal L_m^{(0)} + r \mathcal L_m^{(1)} + r^2 \mathcal L_m^{(2)} + r^3 \mathcal L_m^{(3)} + r^4 \mathcal L_m^{(4)} ,
	\]
	where
	\begin{equation}
		\label{eq:L_m^i}
		\begin{split}
			\mathcal L_m^{(0)} := - (1-m^2)^2 ,  \qquad &
			\mathcal L_m^{(1)} := - 4i m(1-m^2) \partial_{\bar x} ,  \qquad
			\mathcal L_m^{(2)} := - \partial_{\bar t} - 2 (1-3m^2) \partial_{\bar x}^2 + \bar v , \\
			& \ \ \ \mathcal L_m^{(3)} := - 4i m \partial_{\bar x}^3 ,  \qquad
			\mathcal L_m^{(4)} := - \partial_{\bar x}^4 ,
		\end{split}
	\end{equation}
	for each $m \in I_N$ (c.f.~the operators defined after equation \eqref{eq:Amj_algebraic_eqns} in Section \ref{sub:results_modulation_eqns}). Since $\mathcal L_m^{(2)} (r^{\alpha(m) + j} A_{mj}) \neq r^{\alpha(m) + j} \mathcal L_m^{(2)} A_{mj}$ due to the $\partial_{\bar t}$ derivative, we further define $\mathcal L_m^{(2)} (r^{\alpha(m) + j} A_{mj}) = r^{\alpha(m) + j} \tilde{\mathcal L_m^{(2)}} A_{mj}$, where
	\[
	\tilde{\mathcal L_m^{(2)}} :=  - \partial_{\bar t} - r^{-1} \partial_{\bar t} r - 2 (1-3m^2) \partial_{\bar x}^2 + \bar v ,
	\]
	in order to simplify the matching procedures below. We note that $r^{-1} \partial_{\bar t} r = O(1)$ as $r \to 0$, since
	\[
	\partial_{t_1} r_1 = \frac{1}{2} r_1 \eps_1 , \qquad \partial_{t_2} r_2 = 0 , \qquad \partial_{t_3} r_3 = - \frac{1}{2} r_3 \eps_3 ,
	\]
	in charts $\mathcal K_1$, $\mathcal K_2$ and $\mathcal K_3$ respectively. With the preceding definitions, the linear term $\mathcal L_m c_m$ can be rewritten as
	\begin{equation}
		\label{eq:Lm_operator}
		\begin{split}
			\mathcal L_m c_m 
			= & 
			\left( \mathcal L_m^{(0)} + r \mathcal L_m^{(1)} + r^2 \tilde{\mathcal L_m^{(2)}} + r^3 \mathcal L_m^{(3)} + r^4 \mathcal L_m^{(4)} \right) \sum_{j=1}^{\tilde \alpha(m)} r^{\alpha(m) + j} A_{mj} \\
			= & \sum_{j=1}^{\tilde \alpha(m)} r^{\alpha(m) + j} \left( \mathcal L_m^{(0)} + r \mathcal L_m^{(1)} + r^2 \tilde{\mathcal L_m^{(2)}} + r^3 \mathcal L_m^{(3)} + r^4 \mathcal L_m^{(4)} \right) A_{mj} ,
		\end{split}
	\end{equation}
	which is a power series in $r$.
	
	The role of the nonlinearity at each order in $r$ can be extracted by writing $b_m$
	as power series in $r$, i.e.~
	\begin{equation}
		\label{eq:bm_series}
		b_m = \sum_{j=1}^{\tilde \alpha(m)} a_{mj} r^{\alpha(m)+j} , \qquad m \in I_N .
	\end{equation}
	The $a_{mj}$ have an explicit formula which can be derived using Cauchy products if necessary, and do not depend on $r$. For our present purposes, it suffices to note that each $a_{mj}$ is a finite sum of cubic monomials of the form $l A_{m_1j_1} A_{m_2j_2} A_{m_3j_3}$, where $l \in \mathbb N$, and the $m_i \in I_N$ and $j_i \in \{0,\ldots,\tilde \alpha(m_i)\}$ satisfy the matching conditions
	\begin{itemize}
		\item $m = m_1 + m_2 + m_3$;
		\item $\alpha(m) + j + 2 = \alpha(m_1) + \alpha(m_2) + \alpha(m_3) + j_1 + j_2 + j_3$;
	\end{itemize}
	which follow directly from the definitions of the coefficients $(c_k)_{k \in I_N}$, $(b_k)_{k \in I_{3N}}$ and $(a_k)_{k \in I_{3N}}$ in \eqref{eq:Psi_global}, \eqref{eq:cubic} and \eqref{eq:bm_series} respectively. Specifically, the former condition comes from matching $\me^{imx}$ with $\me^{i(m_1 + m_2 + m_3) x}$, while the latter comes from matching $r^{\alpha(m_1) + \alpha(m_2) + \alpha(m_3) + j_1 + j_2 + j_3}$ with $r^{\alpha(m) + j + 2}$ (the `$+2$' appears because $b_m = O(r^3)$). In what follows we shall often appeal to the fact that $a_{mj}$ only depends on the critical modulation functions $A_{\pm1j}$ and the lower order modulation functions $A_{m'j'}$ for which $\alpha(m') + j' < \alpha(m) + j$. The following property in the case of critical modes with $|m| = 1$ will also be helpful.
	
	\begin{lemma}
		\label{lem:a_mj}
		If $j \geq 2$ then $a_{\pm1j}$ depends linearly on $A_{\pm1j}$, i.e.~the maps $A_{\pm1j} \mapsto a_{\pm1 j}$ and $A_{\mp1j} \mapsto a_{\pm1 j}$ are linear.
	\end{lemma}
	
	\begin{proof}
		Let $m \in \{-1,1\}$. Then $\alpha(m) = 0$ and therefore
		\[
		\alpha(m_1) + \alpha(m_2) + \alpha(m_3) + j_1 + j_2 + j_3 = j + 2.
		\]
		%
		Assume that $j_1 = j \geq 2$, i.e.~that $a_{\pm1j}$ depends on $A_{1j}$ or $A_{-1j}$. Then
		\[
		\alpha(m_1) + \alpha(m_2) + \alpha(m_3) + j_2 + j_3 = 2,
		\]
		implying that $\alpha(m_1) = \alpha(m_2) = \alpha(m_3) = 0$ and $j_2 = j_3 = 1$. Thus $a_{\pm1j}$ is a linear combination of cubic monomials $A_{m_1j} A_{m_21} A_{m_31}$ where $m_i \in \{-1,1\}$ for each $i = 1,2,3$. Hence $A_{\pm1j} \mapsto a_{\pm1 j}$ and $A_{\mp1j} \mapsto a_{\pm1 j}$ are linear maps.
	\end{proof}
	
	In particular, the proof of Lemma \ref{lem:a_mj} implies that if $j \geq 2$ the $a_{\pm1j}$ has the form $a_{\pm1j} =: a_{\pm1j}^L + a_{\pm1j}^N$, where
	\begin{equation}
		\label{eq:a_j_L}
		a_{\pm1j}^L(\tilde A_{1j}, \overline{\tilde A_{1j}}) = 
		\left( \alpha_1 A_{11}^2 + \alpha_2 |A_{11}|^2 + \alpha_3 \overline{A_{11}}^2 \right) A_{1j} + \left( \beta_1 A_{11}^2 + \beta_2 |A_{11}|^2 + \beta_3 \overline{A_{11}}^2 \right) \overline{A_{1j}} ,
	\end{equation}
	where $\alpha_1, \alpha_2, \alpha_3, \beta_1, \beta_2$ and $\beta_3$ are non-negative integers, and $a_{j}^N$ depends only modulation functions $A_{m'j'}$ at lower orders with $\alpha(m') + j' < j$.
	
	\
	
	Using \eqref{eq:Lm_operator} and \eqref{eq:bm_series}, the coefficients in the first sum in \eqref{eq:residual_2} can be rewritten as
	\begin{equation}
		\label{eq:matching_eqn}
		\begin{split}
			\mathcal L_m c_m - b_m &+ r^4 \bar \eps \nu_m = \\
			& \sum_{j=1}^{\tilde \alpha(m)} r^{\alpha(m) + j} \left( \big( \mathcal L_m^{(0)} + r \mathcal L_m^{(1)} + r^2 \tilde{\mathcal L_m^{(2)}} + r^3 \mathcal L_m^{(3)} + r^4 \mathcal L_m^{(4)} \big) A_{mj} - a_{mj} \right) + r^4 \bar \eps \nu_m ,
		\end{split}
	\end{equation}
	for each $m \in I_N$. In the following, modulation equations are derived by choosing $A_{mj}$ so that the coefficients of \eqref{eq:matching_eqn} vanish at each order in $r$, thereby (formally) minimising the residual \eqref{eq:residual_2}. The critical case $|m|=1$ is distinguished by the fact that $\mathcal L_{\pm 1}^{0} \equiv 0$ and $\mathcal L_{\pm 1}^{1} \equiv 0$. We consider this case first.

	\subsubsection*{Critical modulation equations for $m = \pm1$}
	
	For $|m| = 1$ we have $m \in \{-1,1\}$, $\alpha(m) = 0$ and $\tilde \alpha(m) = N-2$. Equation \eqref{eq:matching_eqn} becomes
	\begin{equation}
		\label{eq:zero_eqn_1}
		\mathcal L_m c_m - b_m + r^4 \bar \eps \nu_m = 
		\sum_{j=1}^{N-2} r^j \left( \big(r^2 \tilde{\mathcal L_m^{(2)}} + r^3 \mathcal L_m^{(3)} + r^4 \mathcal L_m^{(4)} \big) A_{mj} - a_{mj} \right) + r^4 \bar \eps \nu_m .
	\end{equation}
	Matching in powers of $r$ and requiring \eqref{eq:zero_eqn_1} to vanish at each order leads to the following $N-2$ equations:
	\begin{equation}
		\label{eq:A1j_matched}
		\begin{split}
			O(r^3) : \ & \tilde{\mathcal L_{\pm 1}^{(2)}} A_{\pm1 1} - 
			a_{\pm11} = 0, \\
			O(r^4) : \ & \tilde{\mathcal L_{\pm 1}^{(2)}} A_{\pm1 2} + \mathcal L_{\pm 1}^{(3)} A_{\pm 1 1} - 
			a_{\pm12} + \bar \eps \nu_{\pm1} = 0 , \\
			O(r^5) : \ & \tilde{\mathcal L_{\pm 1}^{(2)}} A_{\pm1 3} + \mathcal L_{\pm 1}^{(3)} A_{\pm 1 2} + \mathcal L_{\pm 1}^{(4)} A_{\pm 1 1} - 
			a_{\pm13} = 0 , \\
			& \qquad \qquad \qquad \qquad \quad \vdots \\
			O(r^{j}) : \ & \tilde{\mathcal L_{\pm 1}^{(2)}} A_{\pm1 j-2} + \mathcal L_{\pm 1}^{(3)} A_{\pm 1 j-3} + \mathcal L_{\pm 1}^{(4)} A_{\pm 1 j-4} -
			a_{\pm1j-2} = 0, \\
			& \qquad \qquad \qquad \qquad \quad \vdots \\
			O(r^{N}) : \ & \tilde{\mathcal L_{\pm 1}^{(2)}} A_{\pm1 N-2} + \mathcal L_{\pm 1}^{(3)} A_{\pm 1 N-3} + \mathcal L_{\pm 1}^{(4)} A_{\pm 1 N-4} - 
			a_{\pm1N-2} = 0 .
		\end{split}
	\end{equation}
	For our purposes it will suffice to write the first two equations explicitly. Direct calculations show that $a_{\pm11} = 3 A_{\pm11} |A_{\pm11}|^2$ and $a_{\pm12} = 3 A_{\pm11}^2 A_{\mp12} + 6 |A_{\pm11}|^2 A_{\pm12}$. Hence at $O(r^3)$ we obtain a GL equation
	\begin{equation}
		\label{eq:GL_eqn_1}
		\partial_{\bar t} A_{\pm11} = 4 \partial_{\bar x}^2 A_{\pm11} + \left( \bar v - r^{-1} \partial_{\bar t} r \right) A_{\pm11} - 3 A_{\pm11} |A_{\pm11}|^2 ,
	\end{equation}
	where $\bar v - r^{-1} \partial_{\bar t} r$ can be considered as a time-dependent coefficient (assuming that the equations for $\bar v(\bar t)$, $r(\bar t)$ and $\bar \eps(\bar t)$ can be solved explicitly in each chart $\mathcal K_l$, $l = 1, 2, 3$). Assuming the existence of solutions $A_{\pm11}(\bar x, \bar t)$, the equation at $O(r^4)$ is
	\begin{equation}
		\label{eq:GL_eqn_2}
		\partial_{\bar t} A_{\pm12} = \ 4 \partial_{\bar x}^2 A_{\pm12} + \left( \bar v - r^{-1} \partial_{\bar t} r \right) A_{\pm12} - 3 A_{\pm11}^2 A_{\mp12} - 6 |A_{\pm11}|^2 A_{\pm12} \mp 4i \partial_{\bar x}^3 A_{\pm11} + \bar \eps \nu_{\pm1} ,
	\end{equation}
	which takes the form of a linearized GL-type equation with a source term and coefficients depending on both time and space. 
	
	Assuming the existence of solutions $A_{m'j'}(\bar x,\bar t)$ at lower orders with $\alpha(m') + j_1 \leq \alpha(m) + j = j$ (recall that $\alpha(m) = 0$ if $m = \pm 1$), the remaining equations are given by
	\begin{equation}
		\label{eq:GL_eqn_j}
		\partial_{\bar t} A_{\pm1j} = \ 4 \partial_{\bar x}^2 A_{\pm1j} + \left( \bar v - r^{-1} \partial_{\bar t} r \right) A_{\pm1j} - a_{\pm1j} \mp 4i \partial_{\bar x}^3 A_{\pm1j-1} - \partial_{\bar x}^4 A_{\pm1j-2} ,
	\end{equation}
	for $j \in \{3, \ldots, N-2\}$, where we recall by Lemma \ref{lem:a_mj} that $a_{\pm1j}$ is linear in $A_{\pm1j}$ and otherwise dependent only on lower order solutions $A_{m'j'}(\bar x,\bar t)$. Since we assume the existence of solutions at lower orders, equation \eqref{eq:GL_eqn_j} can also be seen as a linearized GL-type equation with a source term and coefficients depending on both time and space.  Equations \eqref{eq:GL_eqn_1}, \eqref{eq:GL_eqn_2} and \eqref{eq:GL_eqn_j} are precisely those presented in system \eqref{eq:GL_eqns_global}, Section \ref{sub:results_modulation_eqns}.
	
	\begin{remark}
		Due to the recursive structure of the derivation of equations \eqref{eq:GL_eqn_1}, \eqref{eq:GL_eqn_2} and \eqref{eq:GL_eqn_j}, the existence of solutions the entire system of equations in \eqref{eq:GL_eqn_global} can be established on suitable time intervals by a recursive argument relying only on the existence of solutions to the lowest order GL-type equation \eqref{eq:GL_eqn_1}. This follows from the results of Section \ref{sec:Proof_of_thm_dynamics}. A similar result applies for the static GL reduction in \cite[Ch.~10]{Schneider2017}.
	\end{remark}
	
	\begin{remark}
		\label{rem:regularity}
		Due to the presence of derivatives in the higher order equations with $j \geq 2$, the required regularity increases in proportion to the order of the approximation $n$. As in the static theory, see e.g.~\cite[Rem.~10.2.3]{Schneider2017}, the most regularity is lost as a result of the $\partial_{\bar x}^3 A_{\pm1j-1}$ term. Specifically, we require $A_{\pm12}(\cdot, \bar t) \in H_{ul}^4$ to estimate $\partial_{\bar x}^4 A_{\pm12}(\cdot, \bar t)$, which imposes the requirement that $A_{\pm11}(\cdot, \bar t) \in H_{ul}^7$ because we do not take the smoothing properties of the semigroup into account. %
		For higher order approximations we have to assume that solutions to \eqref{eq:GL_eqn_1} satisfy $A_{\pm11}(\cdot, \bar t) \in H_{ul}^{\theta_A}$ on the relevant time interval, where $\theta_A := 3(n-3) + 1$. This suffices to compensate the loss of three derivatives at each step due to the $\partial_{\bar x}^3 A_{\pm1j-1}$ term. 
	\end{remark}

	\subsubsection*{Non-critical modulation equations for $m \neq \pm1$}
	
	Now assume that $|m| \neq 1$. In this case $\mathcal L_m^{(0)}$ and $\mathcal L_m^{(1)}$ are non-trivial. In particular, $\mathcal L_m^{(0)}$ is invertible with
	\[
	(\mathcal L_m^{(0)})^{-1} = - (1 - m^2)^{-2} \neq 0.
	\]
	We obtain the following equations after matching:
	\[
	\begin{split}
		O(r^{\alpha(m) + 1}) : \ & \mathcal L_m^{(0)} A_{m1} - a_{m1}
		+ \delta_{|m|,4} \bar \eps \nu_m = 0 , \\
		O(r^{\alpha(m) + 2}) : \ & \mathcal L_m^{(0)} A_{m2} + \mathcal L_m^{(1)} A_{m1} - a_{m2} + \delta_{|m|,3} \bar \eps \nu_m = 0 , \\
		O(r^{\alpha(m) + 3}) : \ & \mathcal L_m^{(0)} A_{m3} + \mathcal L_m^{(1)} A_{m2} + \tilde{\mathcal L_m^{(2)}} A_{m1} - 
		a_{m3} + (\delta_{|m|,0} + \delta_{|m|,2}) \bar \eps \nu_m = 0 , \\
		O(r^{\alpha(m) + 4}) : \ & \mathcal L_m^{(0)} A_{m4} + \mathcal L_m^{(1)} A_{m3} + \tilde{\mathcal L_m^{(2)}} A_{m2} + \mathcal L_m^{(3)} A_{m1} - 
		a_{m4} = 0 , \\
		O(r^{\alpha(m) + 5}) : \ & \mathcal L_m^{(0)} A_{m5} + \mathcal L_m^{(1)} A_{m4} + \tilde{\mathcal L_m^{(2)}} A_{m3} + \mathcal L_m^{(3)} A_{m2} + \mathcal L_m^{(4)} A_{m1} - 
		a_{m5} = 0 , \\
		& \qquad \qquad \qquad \qquad \quad \vdots \\
		O(r^{\alpha(m) + j}) : \ & \mathcal L_m^{(0)} A_{mj} + \mathcal L_m^{(1)} A_{mj-1} + \tilde{\mathcal L_m^{(2)}} A_{mj-2} + \mathcal L_m^{(3)} A_{mj-3} + \mathcal L_m^{(4)} A_{mj-4} - 
		a_{mj} = 0 , \\ 
		& \qquad \qquad \qquad \qquad \quad \vdots \\
		O(r^{N}) : \ & \mathcal L_m^{(0)} A_{m \tilde \alpha(m)} + \mathcal L_m^{(1)} A_{m \tilde \alpha(m) -1} + \tilde{\mathcal L_m^{(2)}} A_{m \tilde \alpha(m) -2} + \mathcal L_m^{(3)} A_{m \tilde \alpha(m) -3} \\
		& \qquad \qquad \qquad \qquad \qquad \qquad \qquad \qquad \qquad \quad + \mathcal L_m^{(4)} A_{m \tilde \alpha(m) -4} - a_{m \tilde \alpha(m)} = 0 .
	\end{split}
	\]
	Since $\mathcal L_m^{(0)}$ is invertible for each $m \neq \pm 1$, the above may written more concisely as
	\[
	\begin{split}
		A_{mj} = - (\mathcal L_m^{(0)})^{-1} \bigg( \mathcal L_m^{(1)} A_{mj-1} +  \tilde{\mathcal L_m^{(2)}} A_{mj-2} + \mathcal L_m^{(3)} A_{mj-3} + \mathcal L_m^{(4)} A_{mj-4} - a_{mj} + \delta_{\alpha(m)+j,4} \bar \eps \nu_m \bigg) ,
	\end{split}
	\]
	where we define $A_{mj-k} \equiv 0$ if $j-k \leq 0$, which is precisely equation \eqref{eq:Amj_algebraic_eqns} from Section \ref{sub:results_modulation_eqns}. Assuming the existence of solutions at lower orders $A_{m'j'}$ with $\alpha(m') + j' < \alpha(m) + j$ (including $A_{mj-1}$, $A_{mj-2}$, $A_{mj-3}$ and $A_{mj-4}$), the existence of a solution $A_{mj}$ with $|m| \neq 1$ is given by the \textit{algebraic equation} \eqref{eq:Amj_algebraic_eqns}, which is used to define the GL-manifold $\mathcal M^{gl}$. For clarity we provide the explicit form of the lowest order approximation when $n=4$. 
	We have
	\begin{equation}
		\label{eq:A01_etc}
		A_{01} = A_{02} = A_{\pm21} = A_{\pm22} \equiv 0 , \qquad A_{\pm31} = - \frac{A_{\pm11}^3}{64} ,
	\end{equation}
	such that the fourth order approximation is
	\[
	\Psi_4(x,t) =
	r(\bar t) \left(A_{11}(\bar x, \bar t) \me^{ix} + c.c.\right) - \frac{r(\bar t)^3}{64} \left( A_{11}(\bar x, \bar t)^3 \me^{3ix} + c.c.\right) .
	\]
	The approximation at order $n \geq 5$ is therefore formally 
	\[
	\Psi_n(x,t) =
	r(\bar t) \left(A_{11}(\bar x, \bar t) \me^{ix} + c.c.\right) + r(\bar t)^2 \left(A_{12}(\bar x, \bar t) \me^{ix} + c.c.\right) - \frac{r(\bar t)^3}{64} \left( A_{11}(\bar x, \bar t)^3 \me^{3ix} + c.c.\right) + O(r^4) .
	\]
	
	\begin{remark}
		\label{rem:center_manifold_similarity}
		The recursive structure of the approximation is such that the algebraic equations for non-critical modulation functions $A_{mj}$ with $m \neq \pm 1$ can be written as a graph over lower order critical modulation functions $A_{\pm1j'}$ with $j' < \alpha(m) + j$, similarly to the representation of hyperbolic variables as a graph over the center variables in a center manifold reduction. In particular, the nonlinear term $a_{mj}$ is linear in $A_{\pm1j}$ (recall Lemma \ref{lem:a_mj}) and otherwise dependent only on lower order critical modulation functions $A_{\pm1j'}$ with $j' < j$.
	\end{remark}
	
	For the remainder of this work we shall assume that the approximation $\Psi$ is given by \eqref{eq:Psi_global}, with $A_{\pm1 j}$ defined by equations \eqref{eq:GL_eqns_global} and $A_{mj}$ with $|m| \neq 1$ defined by equations \eqref{eq:Amj_algebraic_eqns}. For this choice of ansatz we have formally that
	\[
	\textrm{Res}(r\psi) = 
	\sum_{m \in I_N} \left( \mathcal L_m c_m -
	b_m + r^4 \bar \eps \nu_m \right) \me^{imx} - \sum_{m \in I_{3N} \setminus I_N} b_m \me^{imx} = O(r^n) ,
	\]
	since by construction the first term vanishes up to and including $O(r^N)$ and the second term is $O(r^n)$. The corresponding systems of modulation equations in local coordinate charts $\mathcal K_1$, $\mathcal K_2$ and $\mathcal K_3$ are described by Propositions \ref{prop:K1_eqns}, \ref{prop:K2_eqns} and \ref{prop:K3_eqns} respectively. We turn now to the dynamics of the modulation equations in charts.

	\section{Dynamics in charts and proofs for Lemmas \ref{lem:Ansatz_dynamics}-\ref{lem:Ansatz_dynamics_mu_0}}
	\label{sec:Proof_of_thm_dynamics}
	
	In this section we prove Lemmas \ref{lem:Ansatz_dynamics} and \ref{lem:Ansatz_dynamics_mu_0}, which are necessary for the proofs of both Theorems \ref{thm:Error} and \ref{thm:Dynamics}. In order to do so, we analyse the dynamics of the approximation $\Psi = r \psi$ in charts $K_l$, $l=1,2,3$. 
	
	Our aim is to describe the dynamics of $\Psi(x,t)$, in particular the maps $\pi^{mid} : \Delta^{in} \to \Delta^{mid}$ and $\pi : \Delta^{in} \to \Delta^{out}$ induced by forward evolution under system \eqref{eq:sh_dynamic} up to time $T_{mid} = \eps^{-1}(\rho_{in} + \eps^{1/2} \rho_{mid} )$ and $T = \eps^{-1} (\rho_{in} + \rho_{out})$ respectively; recall Section \ref{sub:results_error_estimates_and_dynamics}. We further restrict to initial conditions $\Psi(x,0) = \Psi^\ast(x)$ contained within the blown-down image of the GL manifold $\mathcal M^{gl}$, i.e.~initial conditions of the form
	\[
	\Psi^\ast(x) = 
	\sum_{m \in I_N} \sum_{j=1}^{\tilde \alpha(m)} |v(0)|^{(\alpha(m) + j) / 2} A_{mj}^\ast(|v(0)|^{1/2} x) \me^{imx} ,
	\]
	where $x \in \R$, subject to the requirement $\|\Psi^\ast\|_{H_{ul}^\theta} \leq \rho_{in}^{1/2} K$ for a small but fixed (with respect to $\eps$) constant $K > 0$. We analyse the maps $\pi^{mid}$ and $\pi$ via their pre-image after blow-up as a composition of three component maps $\pi_l$ (one for each chart $\mathcal K_l$).

	\subsection{Chart $\mathcal K_1$}
	\label{sub:K1_dynamics}
	
	We recall the equations in chart $\mathcal K_1$, i.e.
	\begin{equation}
		\label{eq:K1_eqns2}
		\begin{split}
			\partial_{t_1} A_{1,1} &= 4 \partial_{x_1}^2 A_{1,1} + \left( - 1 + \frac{\eps_1}{2} \right) A_{1,1} - 3 A_{1,1} |A_{1,1}|^2 , \\
			\partial_{t_1} r_1 &= - \frac{1}{2} r_1 \eps_1 , \\
			\partial_{t_1} \eps_1 &= 2 \eps_1^2 , \\
			\partial_{t_1} A_{2,1} &= 4 \partial_{x_1}^2 A_{2,1} + \left( -1 + \frac{\eps_1}{2} \right) A_{2,1} - a_{2,1} - 4i \partial_{x_1}^3 A_{1,1} + \eps_1 \nu_{1} , \\
			\partial_{t_1} A_{j,1} &=  4 \partial_{x_1}^2 A_{j,1} + \left( -1 + \frac{\eps_1}{2} \right) A_{j,1} - a_{j,1} - 4i \partial_{x_1}^3 A_{j-1,1} - \partial_{x_1}^4 A_{j-2,1} ,
		\end{split}
	\end{equation}
	where $j = 3, \ldots, N-2$. Our aim is to describe the map
	\begin{equation}
		\label{eq:pi_1_def}
		\pi_1 : \Sigma^{in}_1 \ni \left( \psi_1^\ast, \rho_{in}^{1/2}, \eps_1^\ast \right) \mapsto \left(\psi_1(\cdot,T_1), r_1^{out}(\eps_1^\ast), \zeta \right) \in \Sigma_1^{out} 
	\end{equation}
	induced by forward evolution of initial conditions in
	\[
	\Sigma_1^{in} = \left\{ ( \psi_1^\ast , \rho_{in}^{1/2}, \eps_1 ) : \| \psi_1^\ast \|_{H_{ul}^\theta} \leq K, \eps_1 \in [0,\zeta]  \right\} ,
	\]
	which \SJ{may be viewed as the blown-up preimage of the entry section $\Delta^{in}$ (or more precisely $\Delta_\eps^{in}$; recall the definition in \eqref{eq:Delta_in} and the discussion which follows it)}, up to the exit section
	\[
	\Sigma_1^{out} = \left\{ ( \psi_1 , r_1, \zeta ) : \| \psi_1 \|_{H_{ul}^\theta} \leq K, r_1 \in [0,\rho_{in}^{1/2}]  \right\} .
	\]
	Note that $\Sigma^1_{in}$ contains the relevant initial conditions since $\|\Psi^\ast\|_{H_{ul}^\theta} = \rho_{in}^{1/2} \|\psi_1^\ast\|_{H_{ul}^\theta} \leq \rho_{in}^{1/2} K$. The analysis here and in charts $\mathcal K_2$ and $\mathcal K_3$ proceeds in a number of steps. Since the structure of the argument remains the same in each chart, it is instructive to summarise the argument for $\mathcal K_1$ before continuing with the details. The steps are as follows:
	\begin{enumerate}
		\item[I] Solve the equations for $r_1(t_1)$, $\eps_1(t_1)$ and the transition time $T_1$ taken for solutions to reach $\Sigma_1^{out}$.
		\item[II] Substitute the solution for $\eps_1(t_1)$ into the equation for $A_{1,1}(x_1,t_1)$ in order to obtain a (decoupled) GL-type equation with time-dependent linear decay. Use this equation to estimate $A_{1,1}(x_1,t_1)$ for all $t_1 \in [0,T_1]$.
		\item[III] Given a solution for $\eps_1(t_1)$ and an estimate for $A_{1,1}(x_1,t_1)$, the equation for $A_{2,1}$ can be viewed as a (decoupled) linear, inhomogeneous GL-type equation with $(x_1,t_1)$-dependent decay and a source term. Estimate $A_{2,1}(x_1,t_1)$ for all $t_1 \in [0,T_1]$ using linear theory.
		\item[IV] Proceed iteratively with the equations for $A_{j,1}$, $j \geq 3$, which may be viewed as (decoupled) linear inhomogeneous GL-type equations depending on $\eps_1(t_1)$ and (previously solved for) $A_{m'j',1}(x_1,t_1)$ with $\alpha(m') + j' < j$. Estimate each $A_{j,1}(x_1,t_1)$ for all $t_1 \in [0,T_1]$ using linear theory.
		\item[V] Use the solutions and estimates obtained in Steps I-IV to estimate $\psi_1(x,t_1)$ for all $t_1 \in [0,T_1]$.
	\end{enumerate}
	
	\begin{remark}
		Throughout this section and the remainder of this manuscript we shall frequently denote different positive constants by $C$ in cases where the exact expression is unimportant or there is no need to distinguish it from other constants. Constants denoted by $C$ are understood to be independent of small parameters such as $\eps_1^\ast$.
	\end{remark}

	\subsubsection*{Step I: Solve for $r_1(t_1)$, $\eps_1(t_1)$ and $T_1$}
	
	Recall that $\eps_1^\ast := \eps_1(0)$. We have the following result.
	
	\begin{lemma}
		\label{lem:K1_r1eps1}
		The transition time is given by
		\[
		T_1 = \frac{1}{2 \eps_1^\ast} \left(1 - \frac{\eps_1^\ast}{\zeta} \right) ,
		\]
		and for all $t_1 \in [0,T_1]$ we have
		\[
		r_1(t_1) = \rho_{in}^{1/2} (1 - 2 \eps_1^\ast t_1)^{1/4} , \qquad
		\eps_1(t_1) = \frac{\eps_1^\ast}{1 - 2 \eps_1^\ast t_1} .
		\]
		In particular, $r_1^{out}(\eps_1^\ast) = \rho_{in}^{1/2} (\eps_1^\ast / \zeta )^{1/4}$.
	\end{lemma}
	
	\begin{proof}
		The expressions for $r_1(t_1)$ and $\eps_1(t_1)$ follow by direct integration of the corresponding equations in \eqref{eq:K1_eqns2}, and $T_1$ is obtained from the requirement that $\eps_1(T_1) = \zeta$. The expression for $r_1^{out}(\eps_1^\ast)$ follows from $r_1^{out}(\eps_1) = r_1(T_1)$.
	\end{proof}

	\subsubsection*{Step II: Estimate $A_{1,1}(x_1,t_1)$}
	
	We now show that solutions $A_{1,1}(x_1,t_1)$ are exponentially attracted to the trivial homogeneous solution in $H_{ul}^\theta$ for all $t_1 \in (0,T_1]$.
	
	\begin{remark}
		In the following we assume that $A_{1,1}^\ast \in H_{ul}^{\theta_A}$ and specify the degree of regularity $\theta_A$ as we go. By Remark \ref{rem:regularity}, we expect that a higher degree regularity is necessary in each successive step of the approximation.
	\end{remark}
	
	\begin{lemma}
		\label{lem:K1_A1}
		Fix $\theta_A > 1/2$ and assume that $A_{1,1}^\ast \in H_{ul}^{\theta_A}$. Then $A_{1,1}(\cdot,t_1) \in H_{ul}^{\theta_A}$ for all $t_1 \in [0,T_1]$, and
		\[
		\| A_{1,1}(\cdot,t_1) \|_{H_{ul}^{\theta_A}} \leq 
		C \frac{\me^{- \tilde c t_1}} {\sqrt[4]{1 - 2 \eps_1^\ast t_1}}  \|A_{1,1}^\ast\|_{H_{ul}^{\theta_A}} , 
		\qquad t_1 \in [0,T_1] ,
		\]
		where $\tilde c \in (0,1)$ can be chosen arbitrarily close to $1$ by choosing $K>0$ sufficiently small. In particular, for each fixed $\kappa \in (0,\tilde c)$, if $\zeta > 0$ is sufficiently small then
		\begin{equation}
			\label{eq:A11_T1_bound}
			\| A_{1,1}(\cdot, T_1) \|_{H_{ul}^{\theta_A}} \leq
			C \me^{- \kappa  / 2 \eps_1^\ast} \|A_{1,1}^\ast\|_{H_{ul}^{\theta_A}} .
		\end{equation}
	\end{lemma}
	
	\begin{proof}	
		Substituting the expression for $\eps_1(t_1)$ in Lemma \ref{lem:K1_r1eps1} into the equation for $A_{1,1}$ in \eqref{eq:K1_eqns2} yields a GL-type equation with non-constant linear damping:
		\begin{equation}
			\label{eq:A1_decoupled_K1}
			\partial_{t_1} A_{1,1} = 4 \partial_{x_1}^2 A_{1,1} + \left( -1 + \frac{\eps_1(t_1)}{2} \right) A_{1,1} - 3 A_{1,1} |A_{1,1}|^2 , \qquad
			A_{1,1}(x_1,0) = A_{1,1}^\ast(x_1) .
		\end{equation}
		In order to control the time-dependent coefficient $\eps_1(t_1) / 2$ we introduce an integrating factor $\me^{I(t_1)}$, where 
		\[
		I(t_1) = \int_0^{t_1} \left( - \eta + \frac{\eps_1^*}{2(1-2\eps_1^*s_1)} \right) ds_1 = 
		- \eta t_1 - \frac{1}{4} \ln (1-2\eps_1^*t_1), \qquad 
		\me^{I(t_1)} = \frac{\me^{- \eta t_1}}{\sqrt[4]{1 - 2 \eps_1^\ast t_1}} ,
		\]
		where $\eta > 0$ is arbitrarily small but fixed, and define the transformation $A_{1,1}(x_1,t_1) =: \me^{I(t_1)} \tilde A_{1,1}(x_1,t_1)$. This leads to an equation with constant linear decay, namely
		\begin{align}
			\partial_{t_1} \tilde A_{1,1} = 4 \partial_{x_1}^2 \tilde A_{1,1} - (1 - \eta) \tilde{A}_{1,1} - 3 \me^{2 I(t_1)} \tilde A_{1,1} | \tilde A_{1,1} |^2 , \qquad
			\tilde A_{1,1}(x_1,0) = A_{1,1}^\ast(x_1) .
		\end{align}
		Variation of constants yields
		\begin{equation}
			\label{Eq:Tilde_A_1}
			\tilde A_{1,1}(\cdot,t_1) = \me^{t_1\Lambda_1} A_{1,1}^\ast  - 3 \int_0^{t_1} \me^{(t_1 - s_1) \Lambda_{1,1}}  \me^{2 I(s_1)} \tilde A_{1,1}(\cdot,s_1) | \tilde A_{1,1}(\cdot,s_1) |^2 ds_1 ,
		\end{equation}
		where $(\me^{t_1 \Lambda_1})_{t_1 \geq 0}$ is the semigroup generated by $\Lambda_{1,1} := 4 \partial_{x_1}^2-(1-\eta)$. In order to estimate the norm of $\tilde A_{1,1}(\cdot,t_1)$ on $t_1 \in [0,T_1]$, we need estimates for the semigroup and the cubic nonlinearity $\tilde A_{1,1}(\cdot,s_1) | \tilde A_{1,1}(\cdot,s_1) |^2$.
		
		\
		
		The semigroup can be bounded using Fourier multiplier techniques. Specifically, we use \cite[Lemma 8.3.7]{Schneider2017} (see also \cite{Schneider1994}); see Lemma \ref{lem:US_8.3.7} in Appendix \ref{app:technical_estimates} for a reformulation in the present notation. Since for each $c\in(0,1)$ there is a constant $C>0$ such that
		\[
		\| \xi_1 \mapsto  \exp(-(4 \xi_1^2 + 1) t_1) \|_{C^2_b} \leq C\me^{-(c - \eta) t_1} ,
		\]
		Lemma \ref{lem:US_8.3.7} implies that for each $c\in(0,1)$, there is a constant $C>0$ depending on ${\theta_A}$ such that
		\[
		\|\me^{t_1 \Lambda_{1,1}}\|_{H^{\theta_A}_{ul} \to H_{ul}^{\theta_A}} \leq  C\me^{- (c - \eta) t_1}.
		\]
		
		If we use the fact that $\me^{2 I(t_1)} = \me^{- 2 \eta t_1}(1 - 2 \eps_1^\ast t_1)^{-1/2} \leq C$ for all $t_1 \in [0,T_1]$, it remains to control the nonlinear term $\tilde A_{1,1}(\cdot,s_1) | \tilde A_{1,1}(\cdot,s_1) |^2$. For $\tilde K > K > 0$, define
		\begin{equation}
			\label{eq:tilde_T1}
			\tilde T_1 := \min \left\{ T_1, \sup \{t_1 > 0 : \|\tilde A_{1,1}(\cdot,t_1)\|_{H_{ul}^{\theta_A}} \leq \tilde K \} \right\}
		\end{equation}
		such that $A_{1,1}(\cdot,t_1) \in H_{ul}^{\theta_A}$ for all $t_1 \in [0,\tilde T_1]$. We use the fact that $H_{ul}^{\theta_A}$ is an algebra for ${\theta_A} > 1/2$. Using \cite[Lemma 8.3.11]{Schneider2017} (see Lemma \ref{lem:US_8.3.11} in Appendix \ref{app:technical_estimates} for a formulation in the present notation) we obtain
		\[
		\| \tilde A_{1,1}(\cdot,s_1) | \tilde A_{1,1}(\cdot,s_1) |^2 \|_{H^{\theta_A}_{ul}} \leq 
		C \| \tilde A_{1,1}(\cdot,s_1) \|_{H^{\theta_A}_{ul}}^3 \leq 
		C \tilde K^2 \|\tilde A_{1,1}(\cdot,s_1) \|_{H^{\theta_A}_{ul}}, 
		\]
		for all $s_1 \in [0,\tilde T_1]$.
		
		\
		
		Using the bounds obtained above and equation \eqref{Eq:Tilde_A_1} we obtain
		\[
		\| \tilde{A}_{1,1}(\cdot,t_1) \|_{H^{\theta_A}_{ul}} \leq C \me^{- (c - \eta) t_1} \|A_{1,1}^\ast \|_{H^{\theta_A}_{ul}} + C \tilde K^2 \int_0^{t_1} \me^{-(c - \eta) (t_1 - s_1)} \| \tilde{A}_{1,1}(\cdot,s_1) \|_{H^{\theta_A}_{ul}} ds_1 ,
		\]
		for $t_1 \in [0,\tilde T_1]$. Applying the Gr\"onwall inequality in \cite[Lemma 2.8]{Hummel2022} (see also Lemma \ref{lem:HK_2.8} in Appendix \ref{app:technical_estimates}) yields
		\begin{equation}
			\label{eq:A11_tilde_bound}
			\| \tilde{A}_{1,1}(\cdot,t_1) \|_{H^{\theta_A}_{ul}} \leq C \me^{- (\tilde c - \eta) t_1} \|A_{1,1}^\ast \|_{H^{\theta_A}_{ul}} , \qquad t_1 \in [0, \tilde T_1] ,
		\end{equation}
		where $\tilde c := c - C \tilde K^2$ can be chosen arbitrarily close to $c$ if $K$ is chosen sufficiently small ($\tilde K > K$ can be chosen arbitrarily close to $K$). Choosing $K$ and $\eta$ sufficiently small so that $\tilde c - \eta > 0$ ensures that $\tilde T_1 = T_1$, i.e.~that \eqref{eq:A11_tilde_bound} holds for all $t_1 \in [0,T_1]$. Thus we obtain the desired bounds via
		\[
		\| A_{1,1}(\cdot,t_1) \|_{H^{\theta_A}_{ul}} = \me^{I(t_1)} \| \tilde{A}_{1,1}(\cdot,t_1) \|_{H^{\theta_A}_{ul}} 
		\leq C \frac{\me^{- \tilde c t_1}}{\sqrt[4]{1-2\eps_1^*t_1}}  \|A_{1,1}^\ast \|_{H^{\theta_A}_{ul}} ,
		\]
		for all $t_1 \in [0,T_1]$, and
		\begin{equation}
			\label{eq:A11_bound}
			\| A_{1,1}(\cdot,T_1) \|_{H^{\theta_A}_{ul}}  \leq 
			C \me^{- \kappa / 2\eps_1^\ast} \|A_{1,1}^\ast \|_{H_{ul}^{\theta_A}} ,
		\end{equation}
		for any fixed $\kappa \in (0, \tilde c)$, as long as $\zeta > 0$ is fixed and sufficiently small.
	\end{proof}
	
	\begin{remark}
		\label{rem:T_def}
		In many of the proofs below we make use of definitions similar to that for $\tilde T_1$ in equation \eqref{eq:tilde_T1}. Specifically, we shall denote by $\tilde T_{j,l}$ the time
		\[
		\tilde T_{j,l} := \min \left\{ T_l, \sup \{t_l > 0 : \|\tilde A_{j,l}(\cdot,t_l)\|_{H_{ul}^{\theta}} \leq \tilde K \} \right\} 
		\]
		in chart $\mathcal K_l$, where the constant $\tilde K > K > 0$ is fixed but arbitrarily close to $K$. In cases where $\tilde K$ must be larger (i.e.~not close to $K$), we use $\tilde M$ instead of $\tilde K$ to denote an arbitrarily large but fixed positive constant. In both cases, $\tilde A_{j,l}(\cdot,t_l) \in H_{ul}^\theta$ for all $t_l \in [0,\tilde T_l]$, by definition.
	\end{remark}

	\subsubsection*{Step III: Estimate $A_{2,1}(x_1,t_1)$}
	
	We prove the following result.
	
	\begin{lemma}
		\label{lem:K1_A2}
		Let $\theta_A = 3 + \theta$ where $\theta > 1/2$. Then $A_{2,1}(\cdot,t_1) \in H_{ul}^\theta$ for all $t_1 \in [0,T_1]$. Moreover, we have that
		\begin{equation}
			\label{eq:A21_asymptotics}
			A_{2,1}(x_1,T_1) = \nu_1 f_{2,1}(\eps_1^\ast) + A_{2,1,h}(x_1,T_1) ,
		\end{equation}
		where
		\[
		f_{2,1}(\eps_1^\ast) = \me^{1/2\zeta} \sqrt[4]{\frac{\zeta}{8}} \Gamma\left( \frac{1}{4}, \frac{1}{2 \zeta} \right) + O({\eps_1^\ast}^{3/4} \me^{-1/2\eps_1^\ast}) 
		\approx \zeta + O({\eps_1^\ast}^{3/4} \me^{-1/2\eps_1^\ast})
		\]
		as $\eps_1^\ast \to 0$, and
		\begin{equation}
			\label{eq:A21h_bound}
			\| A_{2,1,h}(\cdot, T_1) \|_{H_{ul}^\theta} \leq C \frac{\me^{- \tilde c / 2 \eps_1^\ast}}{{\eps_1^\ast}^{3/4}} \max_{k = 1, 2} \| A_{k,1}^\ast  \|_{H_{ul}^{\theta_k}} ,
		\end{equation}
		where $\tilde c \in (0,1)$ is the same constant as in Lemma \ref{lem:K1_A1} and 
		\begin{equation}
			\label{eq:theta_k}
			\theta_k :=
			\begin{cases}
				\theta_A, & k=1 , \\
				\theta, & k=2.
			\end{cases}
		\end{equation}
	\end{lemma}
	
	\begin{proof}
		Given solutions for $\eps_1(t_1)$ and $A_{1,1}(x_1,t_1)$ as described in Lemmas \ref{lem:K1_r1eps1} and \ref{lem:K1_A1} respectively, the equation for $A_{2,1}$ decouples from system \eqref{eq:K1_eqns2}. We obtain a linear inhomogeneous GL-type equation with a non-constant coefficient and a time-varying source, namely
		\[
		\begin{split}
			\partial_{t_1} A_{2,1} = 4 \partial_{x_1}^2 A_{2,1} &+ \left( -1 + \frac{\eps_1(t_1)}{2} \right) A_{2,1} \\
			&-  3 A_{1,1}(x_1,t_1)^2 \overline{A_{2,1}} - 6 |A_{1,1}(x_1,t_1)|^2 A_{2,1} - 4i \partial_{x_1}^3 A_{1,1}(x_1,t_1) + \nu_1 \eps_1(t_1) ,
		\end{split}
		\]
		with $A_{2,1}(x_1,0) = A_{2,1}^\ast(x_1)$ for $t_1 \in [0,T_1]$ and $x_1 \in \R$. Similarly to the analysis in Step II, we define an integrating factor 
		\[
		I_1(t_1) = \int_0^{t_1} \frac{\eps_1^*}{2(1-2\eps_1^*s_1)} ds_1 = 
		- \frac{1}{4} \ln (1-2\eps_1^*t_1), \qquad 
		\me^{I(t_1)} = \frac{1}{\sqrt[4]{1 - 2 \eps_1^\ast t_1}} ,
		\]
		and apply a transformation $A_{2,1}(x_1,t_1) =: \me^{I_1(t_1)} \tilde A_{2,1}(x_1,t_1)$ to obtain
		\[
		\begin{split}
			\partial_{t_1} \tilde A_{2,1} = 4 \partial_{x_1}^2 \tilde A_{2,1} - \tilde A_{2,1} -  3 A_{1,1}(x_1,t_1)^2 \overline{\tilde A_{2,1}} &- 6 |A_{1,1}(x_1,t_1)|^2 \tilde A_{2,1} \\
			&- 4i \me^{-I_1(t_1)} \partial_{x_1}^3 A_{1,1}(x_1,t_1) + \nu_1 \me^{-I_1(t_1)} \eps_1(t_1) ,
		\end{split}
		\]
		with $\tilde A_{2,1}(x_1,0) = A_{2,1}^\ast(x_1)$. After applying the variation of constants formula, the solution may be split into two parts by setting $\tilde A_{2,1} := \tilde A_{2,1,h} + \tilde A_{2,1,s}$, where
		\begin{equation}
			\begin{split}
				\label{eq:A21_voc}
				\tilde A_{2,1,h}(\cdot,t_1) &= 
				\me^{t_1 \Lambda_1} A_{2,1}^\ast  - \\
				3 \int_0^{t_1} & \me^{(t_1 - s_1) \Lambda_1} \left( A_{1,1}(\cdot,s_1)^2 \overline{\tilde A_{2,1}}(\cdot, s_1) + 2 |A_{1,1}(\cdot,s_1)|^2 \tilde A_{2,1}(\cdot, s_1) - 4i \me^{-I_1(s_1)} \partial_{x_1}^3 A_{1,1}(\cdot,s_1) \right) ds_1 , \\
				\tilde A_{2,1,s}(\cdot,t_1) &= 
				\nu_1 \int_0^{t_1} \me^{(t_1 - s_1) \Lambda_1} \me^{- I_1(s_1)} \eps_1(s_1) ds_1 ,
			\end{split}
		\end{equation}
		and $\Lambda_1 := 4 \partial_{x_1}^2 - 1$. In the following we derive a bound in $H_{ul}^\theta$ by defining $A_{2,1,h} := \me^{I_1} \tilde A_{2,1,h}$, $A_{2,1,s} := \me^{I_1} \tilde A_{2,1,s}$, and considering each part separately.
		
		We first consider $\tilde A_{2,1,h}$.
		A bound for the semigroup can be derived using Lemma \ref{lem:US_8.3.7}, as in Step II. In this case we find that for all $c \in (0,1)$ there is a constant $C > 0$ depending on $\theta$ such that $\| \me^{t_1 \Lambda_1} \|_{H_{ul}^\theta \to H_{ul}^\theta} \leq C \me^{- c t_1}$. It follows that 
		\[
		\begin{split}
			\| \tilde A_{2,1,h}(\cdot,t_1) \|_{H_{ul}^\theta} \leq 
			C \me^{- c t_1} \| A_{2,1}^\ast  \|_{H_{ul}^\theta} &+ 
			C \tilde K \int_0^{t_1} \me^{- c (t_1 - s_1)} \| A_{1,1}(\cdot, s_1) \|_{H_{ul}^\theta}^2 ds_1 \\
			&+ C \int_0^{t_1} \me^{- c (t_1 - s_1)} \me^{-I_1(s_1)} \| A_{1,1}(\cdot, s_1) \|_{H_{ul}^{\theta_A}} ds_1 ,
		\end{split}
		\]
		for all $t_1 \in [0,\tilde T_{2,1}]$, where we used Lemma \ref{lem:US_8.3.11} and the fact that $\| \tilde A_{1,2}(\cdot, s_1) \|_{H_{ul}^\theta} \leq \tilde K$ in order to derive the first integral, and the fact that
		\[
		A_{1,1}(\cdot,t_1) \in H^{\theta_A}_{ul} \ \implies \ 
		\| \partial_{x_1}^3 A_{1,1}(\cdot,t_1) \|_{H_{ul}^\theta} \leq
		C \| A_{1,1}(\cdot,t_1) \|_{H_{ul}^{\theta+3}} = C \| A_{1,1}(\cdot,t_1) \|_{H_{ul}^{\theta_A}} ,
		\]
		in order to derive the second (this step explains the requirement $\theta_A = 3 + \theta$). Both integrals can be bounded using Lemma \ref{lem:K1_A1}. We have
		\begin{equation}
			\label{eq:A21h_bound_2}
			\begin{split}
				\int_0^{t_1} \me^{- c (t_1 - s_1)} \| A_{1,1}(\cdot, s_1) \|_{H_{ul}^\theta}^2 ds_1 &=
				\me^{- c t_1} \int_0^{t_1} \me^{c s_1} \| A_{1,1}(\cdot, s_1) \|_{H_{ul}^\theta}^2 ds_1 \\
				&\leq C \me^{-c t_1} \|A_{1,1}^\ast \|_{H_{ul}^\theta}^2 \int_0^{t_1} \frac{\me^{- (2 \tilde c - c) s_1}} {\sqrt{1 - 2 \eps_1^\ast s_1}} ds_1 ,
			\end{split}
		\end{equation}
		and
		\begin{equation}
			\label{eq:A21h_bound_3}
			\begin{split}
				\int_0^{t_1} \me^{- c (t_1 - s_1)} \me^{-I_1(s_1)} \| A_{1,1}(\cdot, s_1) \|_{H_{ul}^{\theta_A}} ds_1 &=
				\me^{- c t_1} \int_0^{t_1} \me^{c s_1} \sqrt{1 - 2 \eps_1^\ast s_1}^{1/4} \| A_{1,1}(\cdot, s_1) \|_{H_{ul}^{\theta_A}} ds_1 \\
				&\leq C \me^{-c t_1} \|A_{1,1}^\ast \|_{H_{ul}^{\theta_A}} \int_0^{t_1} \me^{(c - \tilde c) s_1} ds_1 \\
				&\leq C \me^{-\tilde c t_1} \|A_{1,1}^\ast \|_{H_{ul}^{\theta_A}} ,
			\end{split}
		\end{equation}
		for $t_1 \in [0, \tilde T_{2,1}]$. Note that $c - \tilde c \in (0,1)$ and $2 \tilde c - c = 2 - 2 C \tilde K^2 \in (0,1)$ as long as $\tilde K>K>0$ is chosen sufficiently small. Using Lemma \ref{lem:Gamma_integrals} and the gamma function asymptotics in Appendix \ref{app:technical_estimates} (equation \eqref{eq:Gamma_asymptotics}) to evaluate the integral in \eqref{eq:A21h_bound_2}, we obtain 
		\begin{equation}
			\label{eq:A21h_tilde_bound}
			\begin{split}
				\| \tilde A_{2,1,h}(\cdot,t_1) \|_{H_{ul}^\theta} &\leq 
				C \me^{- c t_1} \| A_{2,1}^\ast  \|_{H_{ul}^\theta} + C \me^{-\tilde c t_1} \|A_{1,1}^\ast \|_{H_{ul}^{\theta_A}} \\
				&+ 
				C i \me^{-c t_1} \me^{-(2 \tilde c - c) / 2 \eps_1^\ast} \frac{ \|A_{1,1}^\ast \|_{H_{ul}^{\theta}}^2}{\sqrt{\eps_1^\ast}} \Gamma \left( \frac{1}{2}, - (2 \tilde c - c) \left(\frac{1}{2 \eps_1^\ast} - s_1 \right) \right) \bigg|_0^{t_1} , \\
				&\leq C \me^{-\tilde c t_1} \max_{k = 1, 2} \| A_{k,1}^\ast  \|_{H_{ul}^{\theta_k}} ,
			\end{split}
		\end{equation}
		for all $t_1 \in [0, \tilde T_{2,1}]$, where $\theta_k$ is given by \eqref{eq:theta_k}. We therefore have that 
		\[
		\| \tilde A_{2,1,h}(\cdot,T_1) \|_{H_{ul}^\theta} \leq 
		C \frac{\me^{- \tilde c / 2 \eps_1^\ast}}{\sqrt{\eps_1^\ast}} \max_{k = 1, 2} \| A_{k,1}^\ast  \|_{H_{ul}^{\theta_k}} .
		\]
		Since $c$ and $\tilde K$ are chosen so that $\tilde c = c - C\tilde K^2 \in (0,1)$ we can set $\tilde T_{2,1} = T_1$. It follows that
		\[
		\| A_{2,1,h}(\cdot,T_1) \|_{H_{ul}^\theta} =
		\me^{I_1(T_1)} \| \tilde A_{2,1,h}(\cdot,T_1) \|_{H_{ul}^\theta} \leq 
		C \frac{\me^{- \tilde c / 2 \eps_1^\ast}}{{\eps_1^\ast}^{3/4}} \max_{k = 1, 2} \| A_{k,1}^\ast  \|_{H_{ul}^{\theta_k}} ,
		\]
		thereby proving the desired bound \eqref{eq:A21h_bound}.
		
		\
		
		The equation for $\tilde A_{2,1,s}(\cdot,t_1)$ can be solved explicitly, since $\tilde A_{2,1,s}$ solves a heat equation with linear damping:
		\begin{equation}
			\label{eq:A2_s_K1}
			\partial_{t_1} \tilde A_{2,1,s} = 4 \partial_{x_1}^2 \tilde A_{2,1,s} - \tilde A_{2,1,s} + \me^{- I_1(t_1)} \eps_1(t_1) \nu_1 , \qquad  
			\tilde A_{2,1,s}(x_1,0) \equiv 0 .
		\end{equation}
		Using the explicit solution formula for the heat equation, see e.g.~\cite[Ch.~2.3]{Evans2010}, we obtain
		\begin{equation}
			\label{eq:A21f_integral}
			\begin{split}
				\tilde A_{2,1,s}(x_1,t_1) &= \frac{\eps_1^\ast \nu_1}{4 \sqrt \pi} \me^{-t_1} \int_0^{t_1} \frac{\me^{s_1}}{(t_1 - s_1)^{1/2} (1 - 2 \eps_1^\ast s_1)^{3/4}}
				\int_{-\infty}^\infty \me^{-(x_1 - \xi_1)^2/ 16 (t_1 - s_1)} d\xi_1 ds_1 \\
				&= \eps_1^\ast \nu_1 \me^{-t_1} \int_0^{t_1} \frac{\me^{s_1}}{(1 - 2 \eps_1^\ast s_1)^{3/4}} ds_1 \\
				& = \sqrt[4]{\eps_1^\ast} \nu_1 \me^{-t_1} \frac{\me^{1/2\eps_1^\ast}}{2^{3/4}} \Gamma\left( \frac{1}{4} , \frac{1}{2\eps_1^\ast} - s_1 \right) \bigg|_0^{t_1} ,
			\end{split}
		\end{equation}
		for all $x_1 \in \R$ and $t_1 \in [0,T_1]$, where we applied Lemma \ref{lem:Gamma_integrals} again in order to obtain the final expression. Hence
		\[
		\begin{split}
			A_{2,1}(x_1, t_1) &= \me^{I_1(t_1)} \left( \tilde A_{2,1,s}(x_1,t_1) + \tilde A_{2,1,h}(x_1,t_1) \right) \\
			&= \nu_1 \sqrt[4]{\eps_1^\ast} \frac{\me^{-t_1 + 1/2\eps_1^\ast}}{2^{3/4} \sqrt[4]{1 - 2 \eps_1^\ast t_1}} \Gamma\left( \frac{1}{4} , \frac{1}{2\eps_1^\ast} - s_1 \right) \bigg|_0^{t_1} + \frac{1}{\sqrt[4]{1 - 2 \eps_1^\ast t_1}} \tilde A_{2,1,h}(x_1,t_1) ,
		\end{split}
		\]
		for all $x_1 \in \R$ and $t_1 \in [0,T_1]$. Substituting the expression for $t_1 = T_1$ in Lemma \ref{lem:K1_r1eps1} and applying the known asymptotic formula for the incomplete gamma function in Appendix \ref{app:technical_estimates} (equation \eqref{eq:Gamma_asymptotics}) 
		yields the expression in \eqref{eq:A21_asymptotics}.
	\end{proof}

	\subsubsection*{Step IV: Estimate $A_{j,1}(x_1,t_1)$, $j = 3, \ldots, N-2$}
	
	We now consider modulation functions $A_{j,1}$ with $j \geq 3$. We recall from Section \ref{sec:Amplitude_reduction_via_geometric_blow-up}, in particular Lemma \ref{lem:a_mj}, that the linear $A_{j,1}$-dependent part of the $a_{j,1}$ term in the equation for $A_{j,1}$ in \eqref{eq:K1_eqns2} can be extracted by setting $a_{j,1} =: a_{j,1}^L + a_{j,1}^N$, where
	\begin{equation}
		\label{eq:a_j1_L}
		a_{j,1}^L(\tilde A_{j,1}, \overline{\tilde A_{j,1}}) = 
		\left( \alpha_1 A_{1,1}^2 + \alpha_2 |A_{1,1}|^2 + \alpha_3 \overline{A_{1,1}}^2 \right) \tilde A_{j,1} + \left( \beta_1 A_{1,1}^2 + \beta_2 |A_{1,1}|^2 + \beta_3 \overline{A_{1,1}}^2 \right) \overline{\tilde A_{j,1}} ,
	\end{equation}
	where $\alpha_1, \alpha_2, \alpha_3, \beta_1, \beta_2$ and $\beta_3$ are non-negative integers, and $a_{j,1}^N$ depends only modulation functions $A_{m'j',1}$ at lower orders with $\alpha(m') + j' < j$.
	
	\
	
	We prove the following result.
	
	\begin{lemma}
		\label{lem:K1_Aj}
		Let $\theta_A = \theta_A(j) = 3(j-1) + \theta$ where $j \in \{ 3, \ldots, N-2 \}$ and $\theta > 1/2$, and assume that $a_{j,1}^N(\cdot,t_1) \in H_{ul}^\theta$ for all $t_1 \in [0,T_1]$. Then $A_{j,1}(\cdot,t_1) \in H_{ul}^\theta$ for all $t_1 \in [0,T_1]$. If additionally $\mu(x) \equiv 0$ and
		\begin{equation}
			\label{eq:a_jN_cond}
			\| a_{j,1}^N(\cdot,t_1) \|_{H_{ul}^\theta} \leq 
			C \max_{k = 1,\ldots,j-1} \| A_{k,1}(\cdot,t_1) \|_{H_{ul}^{\theta_{k,j}}}^3 ,
		\end{equation}
		then for $K > 0$ sufficiently small it holds that
		\begin{equation}
			\label{eq:Aj1_nu10_bound}
			\| A_{j,1}(\cdot,T_1) \|_{H_{ul}^\theta} \leq 
			C \me^{- \kappa / 2\eps_1^\ast} \max_{k = 1, \ldots, j} \| A_{k,1}^\ast  \|_{H_{ul}^{\theta_{k,j}}} 
		\end{equation}
		for each fixed $\kappa \in (0,\tilde c)$, where
		\begin{equation}
			\label{eq:theta_k_2}
			\theta_{k,j} := 3(j-k) + \theta =
			\begin{cases}
				\theta_A(j), & k=1 , \\
				3 (j-k) + \theta, & k=2, \ldots, j-1, \\
				\theta , & k = j ,
			\end{cases}
		\end{equation}
		extends the definition of $\theta_k$ in \eqref{eq:theta_k}.
	\end{lemma}
	
	\begin{proof}
		Since the structure of the equation for $A_{j,1}$ in \eqref{eq:K1_eqns2} is the same for each $j \geq 3$, we can show 
		that $A_{j,1}(\cdot,t_1) \in H_{ul}^\theta$ for all $t_1 \in [0,T_1]$ by an inductive argument over $k = 3, \ldots, j-1 \leq N - 3$. We omit an explicit treatment of the base case $k=3$, which is simpler and can be treated with minor adaptations of the arguments applied in the induction step below. The induction hypothesis is that $A_{k,1}(\cdot,t_1) \in H_{ul}^{\theta_{k,j}}$ for each $k = 3, \ldots, j-1$, for all $t_1 \in [0,T_1]$. 
		
		\
		
		The equation for $A_{j,1}$ is given by
		\[
		\partial_{t_1} A_{j,1} =  4 \partial_{x_1}^2 A_{j,1} + \left( -1 + \frac{\eps_1(t_1)}{2} \right) A_{j,1} - a_{j,1} - 4i \partial_{x_1}^3 A_{j-1,1} - \partial_{x_1}^4 A_{j-2,1} ,
		\quad A_{j,1}(x_1,0) = A_{j,1}^\ast(x_1) ,
		\]
		where $x_1 \in \R$, $t_1 \in [0,T_1]$ and $a_{j,1}$ is given by $a_{j,1} = a_{j,1}^L + a_{j,1}^N$ with $a_{j,1}^L$ as in \eqref{eq:a_j1_L}. This equation also decouples from system \eqref{eq:K1_eqns2}, and many of the estimates in $H_{ul}^\theta$ can be obtained via arguments similar to those used in Steps II-III. We therefore present fewer details. Both $- 4i \partial_{x_1}^3 A_{j-1,1}$ and $\partial_{x_1}^4 A_{j-2,1}$ may be considered as source terms, since we given solutions $A_{j-1,1}(x_1,t_1)$ and $A_{j-2,1}(x_1,t_1)$ for all $x_1 \in \mathbb R$ and $t_1 \in [0,T_1]$ either via either Lemma \ref{lem:K1_A1}, Lemma \ref{lem:K1_A2}, or the induction hypothesis (depending on $j$).
		
		Using the same integrating factor as in Step III leads to
		\[
		\partial_{t_1} \tilde A_{j,1} = 4 \partial_{x_1}^2 \tilde A_{j,1} - \tilde A_{j,1} - a_{j,1}^L(\tilde A_{j,1}, \overline{\tilde A_{j,1}}) - \me^{-I_1(t_1)} \left( a^N_{j,1} + 4 i \partial_{x_1}^3 A_{j-1,1} + \partial_{x_1}^4 A_{j-2,1} \right) ,
		\]
		together with $\tilde A_{j,1}(x_1,0) = A_{j,1}^\ast(x_1)$, after the transformation $A_{j,1}(x_1,t_1) =: \me^{I_1(t_1)} \tilde A_{j,1}(x_1,t_1)$. As in Step III, we apply the variation of constants formula and split the solution into two parts $\tilde A_{j,1} := \tilde A_{j,1,h} + \tilde A_{j,1,s}$ given by
		\begin{equation}
			\label{eq:Aj1_voc}
			\begin{split}
				\tilde A_{j,1,h}(\cdot,t_1) &= 
				\me^{t_1 \Lambda_1} A_{j,1}^\ast  - \int_0^{t_1} \me^{(t_1 - s_1) \Lambda_1} a_{j,1}^L(\tilde A_{j,1}(\cdot,s_1), \overline{\tilde A_{j,1}}(\cdot,s_1)) ds_1 , \\
				\tilde A_{j,1,s}(\cdot,t_1) &= 
				- \int_0^{t_1} \me^{(t_1 - s_1) \Lambda_1} \me^{-I_1(s_1)} \left( a^N_{j,1}(\cdot,s_1) + 4 i \partial_{x_1}^3 A_{j-1,1}(\cdot,s_1) + \partial_{x_1}^4 A_{j-2,1}(\cdot,s_1) \right) ds_1 ,
			\end{split}
		\end{equation}
		for $t_1 \in [0,T_1]$. Note that the $\partial_{x_1}^3 A_{j-1,1}$ contribution has been included in the source terms, in contrast to Step III. Arguments analogous to those which led to the bound for $\|A_{2,1,h}(\cdot,T_1) \|_{H_{ul}^\theta}$ in Lemma \ref{lem:K1_A2} lead to a similar bound for $A_{j,1,h}(\cdot,T_1) =: \me^{I_1(T_1)} \tilde A_{j,1,h}(\cdot,T_1)$. Specifically, 
		we obtain 
		\begin{equation}
			\label{eq:Aj1h_bound_t}
			\begin{split}
				\| \tilde A_{j,1,h}(\cdot,t_1) \|_{H_{ul}^\theta} &\leq 
				C \me^{- c t_1} \| A_{j,1}^\ast \|_{H_{ul}^\theta} +
				C i \me^{-c t_1} \me^{-(2 \tilde c - c) / 2 \eps_1^\ast} \frac{\|A_{1,1}^\ast\|_{H_{ul}^{\theta}}^2}{\sqrt{\eps_1^\ast}} \Gamma \left( \frac{1}{2}, -(2 \tilde c - c) \left(\frac{1}{2 \eps_1^\ast} - s_1 \right) \right) \bigg|_0^{t_1} \\
				&\leq C \me^{-c t_1} \max_{k = 1,\ldots,j} \| A_{k,1}^\ast  \|_{H_{ul}^\theta} ,
			\end{split}
		\end{equation}
		for all $t_1 \in [0, \tilde T_{j,1}]$, such that
		\[
		\| A_{j,1,h}(\cdot,T_1) \| = 
		\me^{I_1(T_1)} \| \tilde A_{j,1,h}(\cdot,T_1) \| \leq 
		C \frac{\me^{- c / 2 \eps_1^\ast}}{{\eps_1^\ast}^{3/4}} \max_{k = 1,\ldots,j} \| A_{k,1}^\ast  \|_{H_{ul}^\theta} .
		\]
		
		We now consider the second term $\tilde A_{j,1,s}$. Using the same bound for the semigroup as in Step III we obtain the estimate
		\begin{equation}
			\label{eq:Aj1s_voc_bound}
			\begin{split}
				\| \tilde A_{j,1,s}(\cdot,t_1) \|_{H_{ul}^\theta} & \leq \\
				C \me^{- c t_1} & \int_0^{t_1} \me^{c s_1} \sqrt[4]{1 - 2\eps_1^\ast s_1} \left\| a^N_{j,1}(\cdot,s_1) + 4 i \partial_{x_1}^3 A_{j-1,1}(\cdot,s_1) + \partial_{x_1}^4 A_{j-2,1}(\cdot,s_1) \right\|_{H_{ul}^\theta} ds_1 ,
			\end{split}
		\end{equation}
		for all $t_1 \in [0, \tilde T_{j,1}]$. The $a^N_{j,1}$ term is controlled by the assumption that $a^N_{j,1}(\cdot,s_1) \in H_{ul}^\theta$ for all $s_1 \in [0,T_1]$, and the derivative terms $\partial_{x_1}^3 A_{j-1,1}(\cdot,s_1)$ and $\partial_{x_1}^4 A_{j-2,1}(\cdot,s_1)$ are easily controlled because we assume sufficient regularity. Specifically,
		\[
		\begin{split}
			\| \partial_{x_1}^3 A_{j-1,1}(\cdot,s_1) \|_{H_{ul}^\theta} &\leq 
			C \| A_{j-1,1}(\cdot,s_1) \|_{H_{ul}^{\theta+3}} =
			C \| A_{j-1,1}(\cdot,s_1) \|_{H_{ul}^{\theta_{j-1,j}}} , \\
			\| \partial_{x_1}^4 A_{j-2,1}(\cdot,s_1) \|_{H_{ul}^\theta} &\leq 
			C \| A_{j-2,1}(\cdot,s_1) \|_{H_{ul}^{\theta+4}} \leq
			C \| A_{j-2,1}(\cdot,s_1) \|_{H_{ul}^{\theta_{j-2,j}}} ,
		\end{split}
		\]
		for all $s_1 \in [0, \tilde T_{j,1}]$. Since $A_{j-1,1}(\cdot,t_1) \in H_{ul}^{\theta_{j-1,j}}$ and $A_{j-2,1}(\cdot,t_1) \in H_{ul}^{\theta_{j-2,j}}$ for all $t_1 \in [0,T_1]$ by the induction hypothesis, it follows that the $H_{ul}^\theta$ norm in the integrand of \eqref{eq:Aj1s_voc_bound} is bounded above by a constant for all $s_1 \in [0,T_1]$. It follows that
		\[
			\| \tilde A_{j,1,s}(\cdot,t_1) \|_{H_{ul}^\theta} \leq 
			C \me^{- c t_1} \int_0^{t_1} \me^{c s_1} \sqrt[4]{1 - 2\eps_1^\ast s_1} ds_1
			= C \me^{- c t_1} \me^{c / 2\eps_1^\ast} \sqrt[4]{\eps_1^\ast} \Gamma \left( \frac{5}{4} , c \left( \frac{1}{2\eps_1^\ast} - s_1 \right) \right) \bigg|_0^{t_1} ,
		\]
		for all $t_1 \in [0,\tilde T_{j,1}]$, where we used Lemma \ref{lem:Gamma_integrals} again. After multiplying through by the integrating factor 
		we obtain
		\begin{equation}
			\label{eq:Aj1_nu_bound}
			\| A_{j,1}(\cdot,t_1) \|_{H_{ul}^\theta} \leq 
			C \frac{\me^{-c t_1}}{\sqrt[4]{1 - 2\eps_1^\ast t_1}} \max_{k = 1,\ldots,j} \| A_{k,1}^\ast  \|_{H_{ul}^\theta}
			+ C \frac{\me^{-c (t_1 - 1/2\eps_1^\ast)}}{\sqrt[4]{1 - 2\eps_1^\ast t_1}}  \sqrt[4]{\eps_1^\ast} \Gamma \left( \frac{5}{4} , c \left( \frac{1}{2 \eps_1^\ast} - s_1 \right) \right) \bigg|_0^{t_1} ,
		\end{equation}
		for $t_1 \in [0, \tilde T_{j,1}]$. Since \eqref{eq:Aj1_nu_bound} is bounded for all $t_1 \in [0,T_1]$ (this follows directly from the asymptotics of the incomplete gamma function and the monotonicity properties of the integral described in Appendix \ref{app:technical_estimates}), it follows that $\tilde T_{j,1} = T_1$, such that $A_{j,1}(\cdot,t_1) \in H_{ul}^\theta$ for all $t_1 \in [0,T_1]$.
		
		\
		
		We now set $\mu(x) \equiv 0$ and assume that the condition \eqref{eq:a_jN_cond} holds. In order to verify \eqref{eq:Aj1_nu10_bound} we need to improve the bound on the right-most term in \eqref{eq:Aj1_nu_bound}, which is currently $O(1)$ at $t_1 = T_1$ as $\eps_1^\ast \to 0$. This can be achieved by obtaining better bounds for $\| a^N_{j,1}(\cdot,s_1) \|_{H_{ul}^\theta}$, $\| \partial_{x_3}^4 A_{j-1,1}(\cdot,s_1) \|_{H_{ul}^\theta}$ and $\| \partial_{x_1}^4 A_{j-2,1}(\cdot,s_1) \|_{H_{ul}^\theta}$. It suffices to show that
		\begin{equation}
			\label{eq:Aki_bounds}
			\| A_{k,1} (\cdot, t_1) \|_{H_{ul}^{\theta_{k,j}}} \leq 
			C \frac{\me^{- \tilde c t_1}}{\sqrt[4]{1 - 2 \eps_1^\ast t_1}} \max_{i = 1, \ldots, k} \| A_{i,1}^\ast  \|_{H_{ul}^{\theta_{i,j}}} ,
		\end{equation}
		for all $t_1 \in [0,T_1]$ and $k = 1, \ldots , j$, where $\tilde c = c - C \tilde K^2 \in (0,1)$ as before. This is 
		because \eqref{eq:Aki_bounds} implies
		\[
		\| A_{k,1} (\cdot, T_1) \|_{H_{ul}^{\theta_{k,j}}} \leq 
		C \me^{- \kappa / 2\eps_1^\ast} \max_{i = 1, \ldots, k} \| A_{i,1}^\ast  \|_{H_{ul}^{\theta_{i,j}}} 
		\]
		for any fixed $\kappa \in (0,\tilde c)$, for all $k = 1, \ldots , j$, assuming that $\zeta > 0$ is fixed sufficiently small. Thus \eqref{eq:Aki_bounds} implies \eqref{eq:Aj1_nu10_bound}.
		
		For $k=1$, \eqref{eq:Aki_bounds} follows immediately from Lemma \ref{lem:K1_A1}. For $k=2$ we have that $A_{2,1} = A_{2,1,h}$ when $\mu(x) \equiv 0$. Therefore, equation \eqref{eq:A21h_tilde_bound} in the proof of Lemma \ref{lem:K1_A2} implies that
		\begin{equation}
			\label{eq:A21h_bound_4}
			\| A_{2,1}(\cdot,t_1) \|_{H_{ul}^{\theta_{2,j}}} \leq 
			C \frac{\me^{- \tilde c t_1}}{\sqrt[4]{1 - 2 \eps_1^\ast t_1}} \max_{i = 1, 2} \| A_{i,1}^\ast \|_{H_{ul}^{\theta_{i,j}}} ,
		\end{equation}
		for all $t_1 \in [0,T_1]$, as required. 
		
		We show \eqref{eq:Aki_bounds} for $k=3,\ldots,j$ by induction. For the base case $k=3$, equation \eqref{eq:a_jN_cond} and the bounds in \eqref{eq:Aki_bounds} with $k=1,2$ imply that
		\[
		\| a^N_{3,1}(\cdot,s_1) \|_{H_{ul}^{\theta_{3,j}}} \leq 
		C \frac{\me^{- 3 \tilde c s_1}}{(1 - 2 \eps_1^\ast s_1)^{3/4}} \max_{i = 1,2} \| A_{i,1}^\ast \|_{H_{ul}^{\theta_{i,j}}}^3 \leq 
		C \frac{\me^{- \tilde c s_1}}{(1 - 2 \eps_1^\ast s_1)^{1/4}} \max_{i = 1,2} \| A_{i,1}^\ast \|_{H_{ul}^{\theta_{i,j}}} ,
		\]
		for $s_1 \in [0,T_1]$ (note that $\me^{- 3 c s_1} / (1 - 2 \eps_1^\ast s_1)^{3/4} \leq \me^{- c s_1} / \sqrt[4]{1 - 2 \eps_1^\ast s_1}$ on $s_1 \in [0,T_1]$). Since we assume sufficient regularity, bounds for $\partial_{x_1}^3 A_{1,2}(\cdot,s_1)$ and $\partial_{x_1}^4 A_{1,1}(\cdot,s_1)$ are derived directly via the corresponding bounds for $A_{1,2}(\cdot,s_1)$ and $A_{1,1}(\cdot,s_1)$ respectively. Specifically, we have
		\[
		\begin{split}
			\| \partial_{x_1}^3 A_{1,2}(\cdot,s_1) \|_{H_{ul}^{\theta_{3,j}}} &\leq 
			C \| A_{1,2}(\cdot,s_1) \|_{H_{ul}^{\theta_{2,j}}} \leq 
			C \frac{\me^{- \tilde c s_1}}{\sqrt[4]{1 - 2 \eps_1^\ast s_1}} \max_{i = 1, 2} \| A_{i,1}^\ast  \|_{H_{ul}^{\theta_{2,j}}} , \\
			\| \partial_{x_1}^4 A_{1,1}(\cdot,s_1) \|_{H_{ul}^{\theta_{3,j}}} &\leq 
			C \| A_{1,1}(\cdot,s_1) \|_{H_{ul}^{\theta_{3,j} + 4}} \leq 
			C \frac{\me^{- \tilde c s_1}}{\sqrt[4]{1 - 2 \eps_1^\ast s_1}} \| A_{1,1}^\ast \|_{H_{ul}^{\theta_A(j)}} ,
		\end{split}
		\]
		for $s_1 \in [0,T_1]$. We therefore obtain
		\[
			\big\| a^N_{3,1}(\cdot,s_1) + 4i \partial_{x_1}^3 A_{1,2}(\cdot,s_1) + \partial_{x_1}^4 A_{1,1}(\cdot,s_1) \big\|_{H_{ul}^{\theta_{3,j}}} \leq  
			C \frac{\me^{- \tilde c s_1}}{\sqrt[4]{1 - 2 \eps_1^\ast s_1}} 
			\max_{i = 1, 2} \| A_{i,1}^\ast \|_{H_{ul}^{\theta_{i,j}}} ,
		\]
		for $s_1 \in [0,T_1]$. Combining this with \eqref{eq:Aj1s_voc_bound} leads to
		\[
		\| \tilde A_{3,1,s}(\cdot,t_1) \|_{H_{ul}^{\theta_{3,j}}} \leq 
		C \me^{- c t_1} \max_{i = 1, 2} \| A_{i,1}^\ast \|_{H_{ul}^{\theta_{i,j}}} \int_0^{t_1} \me^{(c-\tilde c) s_1} ds_1 \\
		\leq C \me^{- \tilde c t_1} \max_{i = 1, 2} \| A_{i,1}^\ast  \|_{H_{ul}^{\theta_{i,j}}} ,
		\]
		for all $t_1 \in [0,T_1]$. After multiplying through by the integrating factor we therefore obtain the estimate
		\[
		\| A_{3,1}(\cdot,t_1) \|_{H_{ul}^{\theta_{3,j}}} \leq 
		C \frac{1}{\sqrt[4]{1 - 2\eps_1^\ast t_1}} \| \tilde A_{3,1,h}(\cdot,t_1) \|_{H_{ul}^{\theta_{3,j}}} + C \frac{\me^{- \tilde c t_1}}{\sqrt[4]{1 - 2\eps_1^\ast t_1}} \max_{i = 1, 2} \| A_{i,1}^\ast  \|_{H_{ul}^{\theta_{i,j}}} ,
		\]
		for $t_1 \in [0,T_1]$. The desired bound \eqref{eq:Aki_bounds} follows after using the bound for $\| \tilde A_{3,1,h}(\cdot,t_1) \|_{H_{ul}^{\theta_{3,j}}}$ in \eqref{eq:Aj1h_bound_t}, thereby proving the desired result for the base case $k = 3$.
		
		It remains to carry out the induction step, for which we assume that \eqref{eq:Aki_bounds} holds for $k=3,\ldots, j-1 \leq N-3$. We present fewer details since the structure of the argument is the same. Using \eqref{eq:a_jN_cond} it follows that 
		\[
		\| a^N_{j,1}(\cdot,s_1) \|_{H_{ul}^{\theta}} \leq 
		C \frac{\me^{- 3 \tilde c s_1}}{(1 - 2 \eps_1^\ast s_1)^{3/4}}  \max_{i = 1,\ldots,j-1} \| A_{i,1}^\ast \|_{H_{ul}^{\theta_{i,j}}}^3 \leq
		C \frac{\me^{- \tilde c s_1}}{\sqrt[4]{1 - 2 \eps_1^\ast s_1}}  \max_{i = 1,\ldots,j-1} \| A_{i,1}^\ast \|_{H_{ul}^{\theta_{i,j}}} ,
		\]
		for $s_1 \in [0,T_1]$. Due to our regularity assumptions and induction hypothesis we have
		\[
		\begin{split}
			\| \partial_{x_1}^3 A_{j-1,1}(\cdot,s_1) \|_{H_{ul}^{\theta}} &\leq 
			C \| A_{j-1,1}(\cdot,s_1) \|_{H_{ul}^{\theta_{j-1,j}}} 
			\leq C \frac{\me^{- \tilde c s_1}}{\sqrt[4]{1 - 2 \eps_1^\ast s_1}} \max_{i = 1, \ldots, j-1} \| A_{i,1}^\ast \|_{H_{ul}^{\theta_{i,j}}} , \\
			\| \partial_{x_1}^4 A_{j-2,1}(\cdot,s_1) \|_{H_{ul}^{\theta}} &\leq 
			C \| A_{j-2,1}(\cdot,s_1) \|_{H_{ul}^{\theta_{j-2,j}}} 
			\leq C \frac{\me^{- \tilde c s_1}}{\sqrt[4]{1 - 2 \eps_1^\ast s_1}} \max_{i = 1, \ldots, j-2} \| A_{i,1}^\ast \|_{H_{ul}^{\theta_{i,j}}} ,
		\end{split}
		\]
		for $s_1 \in [0,T_1]$. The bound for $\| A_{j,1}(\cdot,t_1) \|_{H_{ul}^{\theta}}$ on $t_1 \in [0,T_1]$ then follows analogously to the base case $k=3$. By induction, it follows that \eqref{eq:Aki_bounds} holds for all $k = 3, \ldots, j \leq N-2$. 
	\end{proof}
	
	\begin{remark}
		\label{rem:a_N}
		In order to verify the assumptions on $a_{j,1}^N$ in Lemma \ref{lem:K1_Aj} we need to control the norm of non-critical modulation functions $A_{mj,1}$ with $|m| \neq 1$. To see that the conditions on $a_{j,1}^N$ are `natural', recall from Remark \ref{rem:center_manifold_similarity} that $a_{j,1}^N$ can be written in terms of lower order critical modulation functions $A_{\pm1j'}$, where $j' < j$. It will follow from Lemma \ref{lem:K1_gl_coefficients} and its proof in Section \ref{subsub:psi_1} below that the relevant conditions are always satisfied.
	\end{remark}

	\subsubsection{Step V: Estimate $\psi_1$}
	\label{subsub:psi_1}
	
	Lemmas \ref{lem:K1_r1eps1}, \ref{lem:K1_A1}, \ref{lem:K1_A2} and \ref{lem:K1_Aj} provide sufficient information to characterise the dynamics of the modulation equations \eqref{eq:K1_eqns2}. In order to properly characterise the dynamics of the `full approximation' $\psi_1$ and the map $\pi_1$ defined in \eqref{eq:pi_1_def}, however, we must also control the non-critical modulation functions $A_{mj,1}$ with $|m| \neq 1$. Moreover, we need to verify the assumptions on $a_{j,1}^N$ 
	(recall Remark \ref{rem:a_N}). In the following we reinstate the full notation $A_{\pm1j,1}$ instead of $A_{j,1}$ in order to distinguish $A_{mj,1}$ with $|m|=1$ and $|m|\neq1$. \SJ{We also recall that $\textbf A_{mj,1}$ denotes the set of lower order critical modulation functions $A_{\pm1j',1}$ with $j' \in \{1,\ldots,\alpha(m) + j - 1\}$.}
	
	\begin{lemma}
		\label{lem:K1_gl_coefficients}
		Let $\theta_A = 3(N-3) + \theta$ where $\theta > 1/2$. We have that $A_{\pm1j,1}(\cdot,t_1) \in H_{ul}^{\theta_{j,N-2}}$ for all $j = 1, \ldots, N-2$ and $t_1 \in [0,T_1]$, and
		\begin{equation}
			\label{eq:gl_coefficients_smooth}
			A_{mj,1}(\cdot, t_1) = g_{mj,1}^{gl} \left( \textbf{\textup{A}}_{mj,1}(\cdot,t_1) , \eps_1(t_1) \right) \in H_{ul}^{\theta_{\alpha(m)+j,N}} ,
		\end{equation}
		for all $t_1 \in [0,T_1]$, $m \in I_N \setminus \{-1,1\}$ and $j = 1, \ldots, \tilde \alpha(m)$. If additionally $\mu(x) \equiv 0$ then for each $m \in I_N \setminus \{-1,1\}$ we have
		\begin{equation}
			\label{eq:gl_coefficients_bounds}
				\big\| g_{mj,1}^{gl} \left( \textbf{\textup{A}}_{mj,1}(\cdot,T_1) , \eps_1(T_1) \right) \big\|_{H_{ul}^{\theta_{\alpha(m)+j,N}}}
				\leq C \me^{- \kappa / 2 \eps_1^\ast}
				\max_{k = 1, \ldots, \alpha(m)+j} \| A_{1k,1}^\ast  \|_{H_{ul}^{\theta_{k,N-2}}} ,
		\end{equation}
		for all $j = 1, \ldots, \tilde \alpha(m)$.
	\end{lemma}
	
	\begin{proof}
		Similarly to the established modulation theory for the static SH equation in e.g.~\cite[Chapter 10]{Schneider2017}, this follows from the recursive construction of $\Psi$ in Section \ref{sec:Amplitude_reduction_via_geometric_blow-up} and the existence of solutions $A_{11,1} \in H_{ul}^{\theta_A}$ due to Lemma \ref{lem:K1_A1}. The idea is to proceed iteratively with terms of order $O(r_1^{\alpha(m) + j})$. For the first three orders, we find the following:
		\begin{itemize}
			\item $\alpha(m) + j = 1$: In this case $(m,j) \in \{ (\pm 1, 1) \}$, and $A_{\pm11,1}(\cdot,t_1) \in H_{ul}^{\theta_A}$ by Lemma \ref{lem:K1_A1}.
			\item $\alpha(m) + j = 2$: In this case $(m,j) \in \{ (\pm1,2), (0,1), (\pm 2,1) \}$. We have $A_{\pm21,1}(\cdot,t_1) \in H_{ul}^{\theta_{2,N-2}}$ by Lemma \ref{lem:K1_A2}, and $A_{01,1} \equiv A_{\pm21,1} \equiv 0$ after restriction to $\mathcal M_1^{gl}$; recall \eqref{eq:A01_etc}.
			\item $\alpha(m) + j = 3$: In this case $(m,j) \in \{ (\pm1,3), (0,2), (\pm2,2), (\pm3,1) \}$. Restriction to $\mathcal M_1^{gl}$ yields $A_{02,1} = 0$, $A_{\pm22,1} = 0$ and $A_{\pm31,1} = - A_{\pm11,1}^3 / 64$ (see again \eqref{eq:A01_etc}). Note that $A_{\pm31,1}(\cdot,t_1) \in H_{ul}^{\theta_A}$ since $A_{\pm11,1}(\cdot,t_1) \in H_{ul}^{\theta_A}$ and $H_{ul}^{\theta_A}$ is an algebra. This implies that $a_{3,1}^N(\cdot,t_1) \in H_{ul}^{\theta_A}$ for all $t_1 \in [0,T_1]$. Therefore $A_{\pm13,1}(\cdot,t_1) \in H_{ul}^{\theta_{3,N-2}}$ by Lemma \ref{lem:K1_Aj}.
		\end{itemize}
		
		The recursive structure for higher orders is as follows: At the level $\alpha(m) + j = l \leq N$ one has a finite set of equations for $A_{mj,1}$ with $(m,j) \in \mathcal I^l := \{(m,j) \in I_N \setminus \{-1,1\} \times \{1, \ldots, \tilde \alpha(m)\} : \alpha(m) + j = l \}$ determined via restriction to $\mathcal M_1^{gl}$. 
		The graph equation defining $\mathcal M_1^{gl}$ in \eqref{eq:GL_manifold_K1} only depends on $A_{m'j',1}$ with $\alpha(m') + j' < l$, which in turn depend only on critical modulation functions $A_{\pm1l'}$ of lower order with $l' < l$; recall Remark \ref{rem:center_manifold_similarity}. 
		In particular, $a_{j,1}^N$ is a polynomial in modulation functions $A_{\pm1 j',1}$ with $j' < j$. This allows one to show that $a_{j,1}^N(\cdot,t_1) \in H_{ul}^{\theta_{j,N-2}}$ for all $t_1 \in [0,T_1]$, since Lemmas \ref{lem:K1_A1} and \ref{lem:K1_A2} can be used to initiate an inductive proof which relies on the fact that $A_{\pm1j',1}(\cdot,t_1) \in H_{ul}^{j',N-2}$ for all lower orders $j' = 1, \ldots, j-1$. The details are standard and omitted for brevity. After showing that $a_{j,1}^N(\cdot,t_1) \in H_{ul}^{\theta_{j,N-2}}$ for all $t_1 \in [0,T_1]$, Lemma \ref{lem:K1_Aj} can be applied to establish that $A_{\pm j,1}(\cdot,t_1) \in H_{ul}^{\theta_{j,N-2}}$ for all $t_1 \in [0,T_1]$.
		
		\
		
		Assume now that $\mu(x) \equiv 0$, in which case $\nu_m = 0$ for all $m \in I_N$ and
		\begin{equation}
			\label{eq:gl_coefficients_mu_zero}
			g_{mj,1}^{gl} \left(\textbf{A}_{mj,1}, \eps_1 \right) = 
			- (\mathcal L_{m,1}^{(0)})^{-1} \bigg( \mathcal L_{m,1}^{(1)} A_{mj-1,1} + \tilde{\mathcal L_{m,1}^{(2)}} A_{mj-2,1} 
			+ \mathcal L_{m,1}^{(3)} A_{mj-3,1} + \mathcal L_{m,1}^{(4)} A_{mj-4,1} - 
			a_{mj,1} \bigg) .
		\end{equation}
		In order to prove the bound in \eqref{eq:gl_coefficients_bounds}, we note first that the matching condition $\alpha(m_1) + \alpha(m_2) + \alpha(m_3) + j_1 + j_2 + j_3 \leq \alpha(m) + j + 2$ implies that
		\begin{equation}
			\label{eq:a_bound}
			\begin{split}
				\| a_{mj,1}(\cdot,t_1) \|_{H_{ul}^{\theta_{\alpha(m)+j,N}}} \leq 
				C \sup_{\{(m',j') : \alpha(m')+j' \leq \alpha(m) + j - 2 \}} \| A_{m'j',1}(\cdot,t_1) \|_{H_{ul}^{\theta_{\alpha(m')+j',N}}} ,
			\end{split}
		\end{equation}
		for all $t_1 \in [0,T_1]$ and $m \neq 1$. Since assume sufficient regularity to control the derivatives in the linear operators $\mathcal L_{m,1}^{(l)}$ for $l = 1,2,3,4$, we obtain
		\[
		\begin{split}
			\| g_{mj,1}^{gl} & \left(\textbf{A}_{mj,1}(\cdot,t_1), \eps_1 \right)  \|_{H_{ul}^{\theta_{\alpha(m) + j,N}}} \leq 
			C \bigg( \| A_{mj-1,1}(\cdot,t_1) \|_{H_{ul}^{\theta_{\alpha(m) + j-1,N} + 1}} \\
			&+ \| A_{mj-2,1}(\cdot,t_1) \|_{H_{ul}^{\theta_{\alpha(m) + j - 2,N} + 2}} + 
			\| A_{mj-3,1}(\cdot,t_1) \|_{H_{ul}^{\theta_{\alpha(m) + j - 3,N} + 3}} \\
			&+ \| A_{mj-4,1}(\cdot,t_1) \|_{H_{ul}^{\theta_{\alpha(m) + j - 4,N} + 4}} +
			\sup_{\{(m',j') : \alpha(m')+j' \leq \alpha(m) + j - 2 \}} \| A_{m'j',1}(\cdot,t_1) \|_{H_{ul}^{\theta_{\alpha(m')+j',N}}} \bigg) ,
		\end{split}
		\]
		for all $t_1 \in [0,T_1]$. If we assume as an induction hypothesis that
		\[
		\| A_{mj,1}(\cdot,t_1) \|_{H_{ul}^{\theta_{\alpha(m)+j,N}}} = \| g_{mj,1}^{gl} \left(\textbf{A}_{mj,1}(\cdot,t_1), \eps_1 \right) \|_{H_{ul}^{\theta_{\alpha(m)+j,N}}} \leq 
		C \sup_{k = 1, \ldots, j-2} \| A_{\pm1 k,1}(\cdot,t_1) \|_{H_{ul}^{\theta_{k,N-2}}} 
		\]
		for all $t_1 \in [0,T_1]$, 
		(the base case follows by the arguments above in dot points), then we obtain
		\begin{equation}
			\label{eq:Amj1_bounds}
			\| g_{mj,1}^{gl} \left(\textbf{A}_{mj,1}(\cdot,t_1), \eps_1 \right) \|_{H_{ul}^{\theta_{\alpha(m) + j,N}}} \leq
			C \sup_{k = 1, \ldots, j-2} \| A_{\pm1 k,1}(\cdot,t_1) \|_{H_{ul}^{\theta_{k,N-2}}} ,
		\end{equation}
		for all $t_1 \in [0,T_1]$. The bound in \eqref{eq:gl_coefficients_mu_zero} follows after applying the bounds for $A_{\pm1,j}$ when $\mu(x) \equiv 0$ from Lemmas \ref{lem:K1_A1}, \ref{lem:K1_A2} and \eqref{lem:K1_Aj} when $\mu(x) \equiv 0$ (the condition \eqref{eq:a_jN_cond} is satisfied due to \eqref{eq:a_bound}).
	\end{proof}
	
	\begin{remark}
		The proof of Lemma \ref{lem:K1_gl_coefficients} also shows that the assumptions on $a_{j,1}^N$ for the applicability of Lemma \ref{lem:K1_Aj} are satisfied.
	\end{remark}
	
	Combining the results of Lemmas \ref{lem:K1_r1eps1}, \ref{lem:K1_A1}, \ref{lem:K1_A2}, \ref{lem:K1_Aj} and \ref{lem:K1_gl_coefficients} we obtain the following result, which summarises the dynamics in chart $\mathcal K_1$.
	
	\begin{proposition}
		\label{prop:K1_summary}
		Let $\theta_A = 1 + 3(N-3) + \theta$ where $\theta > 1/2$. Then for $\zeta, K > 0$ sufficiently small but fixed we have that $\psi_1(\cdot, t_1) \in H_{ul}^\theta$ for all $t_1 \in [0,T_1]$. In particular, the map $\pi_1 : \Sigma_1^{in} \to \Sigma_1^{out}$ is well-defined and given by
		\begin{equation}
			\label{eq:pi_1}
			\pi_1 : ( \psi_1^\ast, \rho_{in}^{1/2}, \eps_1^\ast ) \mapsto \left( \psi_1(\cdot,T_1), \rho_{in}^{1/2}\left(\frac{\eps_1^\ast}{\zeta}\right)^{1/4}, \zeta \right) ,
		\end{equation}
		where
		\begin{equation}
			\label{eq:psi_1_out}
			\psi_1(x,T_1) = 
			{\eps_1^\ast}^{1/4} \frac{\rho_{in}^{1/2}}{\zeta^{1/4}} \left( \nu_1 f_{2,1}(\eps_1^\ast) \me^{ix} + c.c. \right) + {\eps_1^\ast}^{1/2} h_{rem,1} (x) ,
		\end{equation}
		where $f_{2,1}(\eps_1^\ast)$ is the function described in Lemma \ref{lem:K1_A2} and $h_{rem,1} \in H_{ul}^{\theta}$. If additionally $\mu(x) \equiv 0$ then we have
		\[
		\|\psi_1(\cdot,T_1)\|_{H_{ul}^\theta} \leq 
		C \me^{- \kappa / 2 \eps_1^\ast} \max_{j = 1, \ldots, N-2} \| A_{j,1}^\ast \|_{H_{ul}^{\theta+1}} ,
		\]
		where $\kappa \in (0,\tilde c)$ can be fixed arbitrarily close to $\tilde c$ if $K, \zeta > 0$ are sufficiently small.
	\end{proposition}
	
	\begin{proof}
		The form of the map in \eqref{eq:pi_1} follows immediately from Lemma \ref{lem:K1_r1eps1}. The fact that $\psi_1(\cdot,t_1) \in H_{ul}^\theta$ for all $t_1 \in [0,T_1]$ follows from its definition in \eqref{eq:psi_definition} and the fact that $A_{mj,1}(\cdot,t_1) \in H_{ul}^{\theta+1}$ for all $m \in I_N$ with corresponding $j \in \{ 1, \ldots , \tilde \alpha(m)\}$ and for all $t_1 \in [0,T_1]$; recall 
		Lemma \ref{lem:K1_gl_coefficients} and note that we assume one additional degree of regularity. The extra regularity is needed since we use Lemma \ref{lem:norm_rescaling} in order to take the $H_{ul}^\theta$ with respect to $x$ instead of $x_1$.
		
		In order to obtain the asymptotic expression in \eqref{eq:psi_1_out} we evaluate the series expansion \eqref{eq:psi_definition} at $t_1 = T_1$. This yields
		\[
		\begin{split}
			r_1(T_1) \psi_1(x,T_1) &= \sum_{m \in I_N} \sum_{j=1}^{\tilde \alpha(m)} 
			r_1(T_1)^{\alpha(m) + j} A_{mj,1}(x_1,T_1) \me^{imx} \\
			&= r_1(T_1) \psi_{GL,1}(x,T_1) + r_1(T_1)^2 \left( A_{12,1}(x_1,T_1) \me^{ix} + c.c. \right) + r_1(T_1)^3 \tilde h_{rem,1}(x,T_1) ,
		\end{split}
		\]
		where $\psi_{GL,1}(x,T_1) = A_{11,1} (x_1,T_1) \me^{ix} + c.c.$ and we used the fact that $A_{01,1}(\cdot,T_1) = A_{21,1}(\cdot,T_1) \equiv 0$; recall the proof of Lemma \ref{lem:K1_gl_coefficients}. The remainder function satisfies $\tilde h_{rem,1}(\cdot,T_1) \in H_{ul}^{\theta}$, since it is a linear combination of terms of the form $r_1(T_1)^{\alpha(m) + j} A_{mj,1}(\cdot,T_1) \me^{imx}$ with $r_1(T_1) \leq \rho_{in}^{1/2}$ by Lemma \ref{lem:K1_r1eps1} and $A_{mj,1}(\cdot,T_1) \in H_{ul}^{\theta+1}$ by Lemma 
		\ref{lem:K1_gl_coefficients}. The formal leading order approximation $\psi_{GL,1}(x,T_1)$ is in fact higher order, since by Lemmas \ref{lem:K1_A1} and \ref{lem:norm_rescaling} we have
		\[
		\| \psi_{GL,1}(x,T_1) \|_{H_{ul}^\theta} \leq C \| A_{11,1} (\cdot,T_1) \|_{H_{ul}^{\theta+1}} \leq 
		C \me^{-\kappa / 2\eps_1^\ast} \|A_{1,1}^\ast\|_{H_{ul}^{\theta+1}} .
		\]
		The expression \eqref{eq:psi_1_out} follows after substituting the asymptotic expression for $A_{2,1} = A_{12,1}$ in Lemma \ref{lem:K1_A2}, which yields
		\begin{equation}
			\label{eq:psi_1_asymptotics}
			\begin{split}
				\psi_1(x,T_1) = r_1(T_1) \left( \nu_1 f_{2,1}(\eps_1^\ast) \me^{ix} + c.c. \right) + r_1(T_1)^2 h_{rem,1}(x) ,
			\end{split}
		\end{equation}
		where we defined a new remainder term
		\[
		h_{rem,1}(x) := 
		r(T_1)^{-1} (A_{2,1,h}(x_1,T_1) \me^{ix} + c.c.) + r_1(T_1)^{-2} \psi_{GL,1}(x,T_1) + \tilde h_{rem,1}(x,T_1) .
		\]
		A direct application of the triangle inequality and Lemma \ref{lem:norm_rescaling} shows that $h_{rem,1} \in H_{ul}^{\theta}$ (note that the norm of $A_{2,1,h}$ is exponentially small due to Lemma \ref{lem:K1_A2}). Substituting the solution $r_1(T_1) = \rho_{in}^{1/2} \sqrt[4]{\eps_1^\ast/\zeta}$ into \eqref{eq:psi_1_asymptotics} yields \eqref{eq:psi_1_out}.
		
		\
		
		Now assume that $\mu(x) \equiv 0$. It follows from Lemmas \ref{lem:K1_A1}, \ref{lem:K1_A2}, \ref{lem:K1_Aj} and \ref{lem:K1_gl_coefficients} that for each $m \in I_N$ and corresponding $j = 1,\ldots,N-2$ we have
		\[
		\| A_{mj,1}(\cdot,T_1) \|_{H_{ul}^{\theta+1}} \leq 
		C \me^{- \kappa / 2\eps_1^\ast} \max_{j = 1,\ldots,N-2} \| A_{j,1}^\ast \|_{H_{ul}^{\theta+1}} .
		\]
		Applying this bound to the series expansion for $\psi_1(\cdot,T_1)$ together with a triangle inequality and Lemma \ref{lem:norm_rescaling} yields the desired result.
	\end{proof}
	
	\begin{remark}
		Recall from Lemma \ref{lem:K1_A2} that $f_{21}(\eps_1^\ast) \approx \zeta$ for small $\zeta$. Thus a close numerical approximation is $\psi_1(x,T_1) \approx 2 \nu_1 \rho_{in}^{1/2} \zeta^{3/4} {\eps_1^\ast}^{1/4} \cos x$.
	\end{remark}

	\subsection{Chart $\mathcal K_2$}
	\label{sub:K2_dynamics}
	
	We recall the equations in chart $\mathcal K_2$:
	\begin{equation}
		\label{eq:K2_eqns2}
		\begin{split}
			\partial_{t_2} A_{1,2} &= 4 \partial_{x_2}^2 A_{1,2} + v_2 A_{1,2} - 3 A_{1,2} | A_{1,2} |^2 , \\
			\partial_{t_2} v_2 &= 1 , \\
			\partial_{t_2} A_{2,2} &= 4 \partial_{x_2}^2 A_{2,2} + v_2 A_{2,2} - a_{2, 2} - 4i \partial_{x_2}^3 A_{1,2} + \nu_{1} , \\
			\partial_{t_2} A_{j,2} &=  4 \partial_{x_2}^2 A_{j,2} + v_2 A_{j,2} - a_{j,2} - 4i \partial_{x_2}^3 A_{j-1,2} - \partial_{x_2}^4 A_{j-2,2} ,
		\end{split}
	\end{equation}
	for $j = 3, \ldots, N-2$. Our aim in chart $\mathcal K_2$ is to describe the map
	\[
	\pi_2 : \Sigma^{in}_2 \ni \left( \psi_2^\ast, - \zeta^{-1/2}, r_2 \right) \mapsto \left(\psi_2(\cdot,T_2), \zeta^{-1/2}, r_2 \right) \in \Sigma_2^{out} ,
	\]
	induced by forward evolution of initial conditions in
	\[
	\Sigma_2^{in} = \left\{ ( \psi_2^\ast , - \zeta^{-1/2}, r_2 ) : \| \psi_2^\ast \|_{H_{ul}^\theta} \leq \zeta^{1/4} K, r_2 \in [0,\vartheta]  \right\} ,
	\]
	which corresponds to the exit section $\Sigma_1^{out}$ in chart $\mathcal K_1$ coordinates ($\vartheta > 0$ can be chosen such that $\Sigma_2^{in} = \kappa_{12}(\Sigma^{out}_1)$), up to the exit section
	\[
	\Sigma_2^{out} = \left\{ ( \psi_2 , \rho_{mid}, r_2 ) : \| \psi_2 \|_{H_{ul}^\theta} \leq \zeta^{1/4} K, r_2 \in [0,\vartheta]  \right\} ,
	\]
	which corresponds to the section $\Delta^{mid}$ defined in \eqref{eq:Delta_mid} after blow-down \SJ{after fixing $\eps$ and using the natural identification with points in $\Delta^{mid}_\eps = \Delta^{mid} \times [0,\eps_0]$}. The analysis proceeds via steps similar to those of the $\mathcal K_1$ analysis.

	\subsubsection*{Step I: Solve for $v_2(t_2)$ and $T_2$}
	
	We have the following result.
	
	\begin{lemma}
		\label{lem:K2_T2}
		The transition time is $T_2 = \zeta^{-1/2} + \rho_{mid}$, and on the interval $t_2 \in [0,T_2]$ we have $v_2(t_2) = - \zeta^{-1/2} + t_2$.
	\end{lemma}
	
	\begin{proof}
		This follows by direct integration and the boundary constraint $v_2(T_2) = \rho_{mid}$.
	\end{proof}

	\subsubsection*{Step II: Estimate $A_{1,2}(x_2,t_2)$}
	
	We prove the following result.
	
	\begin{lemma}
		\label{lem:K2_A1}
		Fix $\theta_A > 1/2$ and assume that $A_{1,2}^\ast \in H_{ul}^{\theta_A}$. Then $A_{1,2}(\cdot,t_2) \in H_{ul}^{\theta_A}$ on $t_2 \in [0,T_2]$. In particular,
		\begin{equation}
			\label{eq:A12_bound}
			\| A_{1,2}(\cdot,t_2) \|_{H_{ul}^{\theta_A}} \leq 
			C \|A_{1,2}^\ast\|_{H_{ul}^{\theta_A}} ,
		\end{equation}
		for all $t_2 \in [0,T_2]$.
	\end{lemma}
	
	\begin{proof}
		Substituting the expressions for $v_2(t_2)$ into the equation for $A_{1,2}$ yields a GL-type equations with time-dependent linear damping:
		\begin{equation}
			\label{eq:A1_decoupled_K2}
			\partial_{t_2} A_{1,2} = 4 \partial_{x_2}^2 A_{1,2} + v_2(t_2) A_{1,2} - 3 A_{1,2} |A_{1,2}|^2 ,
			\qquad A_{1,2}(x_2,0) = A_{1,2}^\ast(x_2) ,
		\end{equation}
		where $v_2(t_2) = -\zeta^{-1/2} + t_2$ by Lemma \ref{lem:K2_T2}. Equation \eqref{eq:A1_decoupled_K2} decouples from system \eqref{eq:K2_eqns2}, so we may consider it independently. The effect of the non-constant coefficient $v_2(t_2)$ can be mitigated via the introduction of an integrating factor $\me^{I_{2}(t_2)}$, where
		\begin{equation}
			\label{eq:integrating_factor_A11_K2}
			I_{2}(t_2) = \int_0^{t_2} v_2(s_2) ds_2 = - \zeta^{-1/2} t_2 + \frac{t_2^2}{2} .
		\end{equation}
		Defining $A_{1,2}(x_2,t_2) =: \me^{I_{2}(t_2)} \tilde A_{1,2}(x_2,t_2)$ leads to the equation
		\[
		\partial_{t_2} \tilde A_{1,2} = 4 \partial_{x_2}^2 \tilde A_{1,2} - 3 \me^{2 I_{2}(t_2)} \tilde A_{1,2} |A_{1,2}|^2 ,
		\qquad \tilde A_{1,2}(x_2,0) = A_{1,2}^\ast(x_2) ,
		\]
		for which the variation of constants formula yields
		\begin{equation}
			\label{eq:A12_voc}
			\tilde A_{1,2}(\cdot,t_2) = \me^{t_2\Lambda_2} A_{1,2}^\ast - 3 \int_0^{t_2} \me^{(t_2 - s_2) \Lambda_2}  \me^{2 I_{2}(s_2)} \tilde A_{1,2}(\cdot,s_2) |\tilde A_{1,2}(\cdot,s_2)|^2 ds_2 ,
		\end{equation}
		where $(\me^{t_2 \Lambda_2})_{t_2 \geq 0}$ is the (heat) semigroup generated by $\Lambda_2 := 4 \partial_{x_2}^2$. 
		Since the transition time $T_2$ is finite, there exists a constant $C > 0$ such that $\me^{2I_{2}(t_2)} \leq C$ for all $t_2 \in [0,T_2]$. Moreover, it follows from Lemma \ref{lem:US_8.3.7} that there is a constant $C > 0$ such that
		\[
		\| \me^{t_2 \Lambda_2} \|_{H_{ul}^{\theta_A} \to H_{ul}^{\theta_A}} \leq C ,
		\]
		for all $t_2 \in [0,T_2]$. Taking the $H^{\theta_A}_{ul}$ norm in \eqref{eq:A12_voc}, we obtain
		\[
		\| \tilde A_{1,2}(\cdot,t_2) \|_{H_{ul}^{\theta_A}} \leq 
		C \| A_{1,2}^\ast \|_{H_{ul}^{\theta_A}} + 3 C \int_0^{t_2} \| \tilde A_{1,2}(\cdot,s_2) |\tilde A_{1,2}(\cdot,s_2)|^2 \|_{H_{ul}^{\theta_A}} ds_2 .
		\]
		By Lemma \ref{lem:US_8.3.11} we have $\| \tilde A_{1,2}(\cdot,s_2) |\tilde A_{1,2}(\cdot,s_2)|^2 \|_{H_{ul}^{\theta_A}} \leq C \| \tilde A_{1,2}(\cdot,s_2) \|_{H_{ul}^{\theta_A}}^3 \leq \tilde C \tilde K^2 \| \tilde A_{1,2}(\cdot,s_2) \|_{H_{ul}^{\theta_A}}$ for all $s_2 \in [0,\tilde T_2]$. Therefore
		\[
		\| \tilde A_{1,2}(\cdot,t_2) \|_{H_{ul}^{\theta_A}} \leq 
		C \| A_{1,2}^\ast \|_{H_{ul}^{\theta_A}} + \tilde C \tilde K^2 \int_0^{t_2} \| \tilde A_{1,2}(\cdot,s_2) \|_{H_{ul}^{\theta_A}} ds_2 ,
		\]
		for $t_2 \in [0,\tilde T_2]$. Applying Gr\"onwall's inequality yields the estimate
		\begin{equation}
			\label{eq:A21_bound_t}
			\| \tilde A_{1,2}(\cdot,t_2) \|_{H_{ul}^{\theta_A}} \leq 
			C \| A_{1,2}^\ast \|_{{H_{ul}^{\theta_A}}} \me^{\tilde C \tilde K^2 t_2} \leq 
			C \| A_{1,2}^\ast \|_{{H_{ul}^{\theta_A}}} , 
		\end{equation}
		for $t_2 \in [0,\tilde T_2]$ ($\me^{\tilde C \tilde K^2 t_2}$ is bounded since $\tilde T_2$ is finite). Since \eqref{eq:A21_bound_t} is finite for all $t_2 \leq T_2$, we have that $A_{1,2}(\cdot,t_2) \in H_{ul}^{\theta_A}$ for all $t_2 \in [0,T_2]$. The bound in \eqref{eq:A12_bound} follows from \eqref{eq:A21_bound_t} and the fact that $\|A_{1,2}(\cdot,t_2)\|_{H_{ul}^{\theta_A}} \leq C \| \tilde A_{1,2}(\cdot,t_2)\|_{H_{ul}^{\theta_A}}$.
	\end{proof}

	\subsubsection*{Step III: Estimate $A_{2,2}(x_2,t_2)$}
	
	We prove the following result.
	
	\begin{lemma}
		\label{lem:K2_A2}
		Let $\theta_A = 3 + \theta$ where $\theta > 1/2$. Then $A_{2,2}(\cdot,t_2) \in H_{ul}^\theta$ for all $t_2 \in [0,T_2]$. Moreover, we have that
		\begin{equation}
			\label{eq:A22_asymptotics}
			A_{2,2}(x_2,T_2) = \nu_1 f_{2,2} + A_{2,2,h}(x_2,T_2) ,
		\end{equation}
		where $f_{2,2}$ is a constant
		\[
		f_{2,2} := 
		\sqrt{\frac{\pi}{2}} \me^{\rho_{mid}^2/2} \left( \textnormal{erf}\left(\frac{\rho_{mid}}{\sqrt{2}} \right) + \textnormal{erf}\left(\frac{1}{\sqrt{2 \zeta}} \right) \right) > 0,
		\]
		and
		\begin{equation}
			\label{eq:A22h_bound}
			\| A_{2,2,h}(\cdot,T_2) \|_{H_{ul}^\theta} \leq 
			C \max_{k = 1, 2} \| A_{k,2}^\ast \|_{H_{ul}^{\theta_k}} ,
		\end{equation}
		where $\theta_k$ is given by \eqref{eq:theta_k}.
	\end{lemma}
	
	\begin{proof}	
		Given solutions for $v_2(t_2)$ and $A_{1,2}(x_2,t_2)$ as described in Lemmas \ref{lem:K2_T2} and \ref{lem:K2_A1} respectively, the equation for $A_{2,2}$ decouples. We obtain
		\[
		\partial_{t_2} A_{2,2} = 4 \partial_{x_2}^2 A_{2,2} + v_2(t_2) A_{2,2} - 3 A_{1,2}(x_2,t_2)^2 \overline{A_{2,2}} - 6 |A_{1,2}(x_2,t_2)|^2 A_{2,2} - 4i \partial_{x_2}^3 A_{1,2}(x_2,t_2) + \nu_1 , 
		\]
		together with $A_{2,2}(x_2,0) = A_{2,2}^\ast(x_2)$. After applying the transformation $A_{2,2}(x_2,t_2) =: \me^{I_2(t_2)} \tilde A_{2,2}(x_2,t_2)$, where $I_2(t_2)$ is the same integrating factor \eqref{eq:integrating_factor_A11_K2} as in Step II, we obtain
		\[
		\partial_{t_2} \tilde A_{2,2} = 4 \partial_{x_2}^2 \tilde A_{2,2} - 3 A_{1,2}(x_2,t_2)^2 \overline{\tilde A_{2,2}} - 6 |A_{1,2}(x_2,t_2)|^2 \tilde A_{2,2} - 4i \me^{-I_2(t_2)} \partial_{x_2}^3 A_{1,2}(x_2,t_2) + \me^{-I_2(t_2)} \nu_1 , 
		\]
		together with $A_{2,2}(x_2,0) = A_{2,2}^\ast(x_2)$. Again we apply the variation of constants formula and split the solution into two parts via $\tilde A_{2,2} := \tilde A_{2,2,h} + \tilde A_{2,2,s}$ defined by
		\begin{equation}
			\label{eq:A22_voc}
			\begin{split}
				\tilde A_{2,2,h}(\cdot,t_2) &= \me^{t_2 \Lambda_2} A_{2,2}^\ast  \ - \\ \int_{0}^{t_2} & \me^{(t_2 - s_2) \Lambda_2} \left( 3 A_{1,2}(\cdot,s_2)^2 \overline{\tilde A_{2,2}}(\cdot,s_2) + 6 |A_{1,2}(\cdot,s_2)|^2 \tilde A_{2,2}(\cdot,s_2) + \me^{-I_2(s_2)} 4i \partial_{x_2}^3 A_{1,2}(\cdot,s_2) \right) ds_2 , \\
				\tilde A_{2,2,s}(\cdot,t_2) &= \nu_1 \int_{0}^{t_2} \me^{(t_2 - s_2) \Lambda_2} \me^{- I_{2}(s_2)} ds_2 ,
			\end{split}
		\end{equation}
		where as in Step II, $(\me^{t_2 \Lambda_2})_{t_2 \geq 0}$ denotes the heat semigroup generated by $\Lambda_2 = 4 \partial_{x_2}^2$. In order to estimate the norm of $A_{2,2,h} := \me^{I_{2}} \tilde A_{2,2,h}$, first recall from Step II above that $\|\me^{t_2\Lambda_2}\|_{H_{ul}^\theta \to H_{ul}^\theta} \leq C$ for all $t_2 \in [0,T_2]$. Then for all $t_2 \in [0,\tilde T_{2,2}]$, where $\tilde T_{2,2} := \min\{ T_2, \sup \{t_2 > 0 : \|\tilde A_{2,2}(\cdot,t_2)\|_{H_{ul}^{{\theta}}} \leq \tilde M \} \}$ for an arbitrarily large but fixed constant $\tilde M > 0$ (c.f.~Remark \ref{rem:T_def}), we obtain
		\begin{equation}
			\label{eq:A22_bound_t}
			\| \tilde A_{2,2,h}(\cdot,t_2) \|_{H_{ul}^\theta} \leq 
			C \| A_{2,2}^\ast  \|_{H_{ul}^\theta} + 
			C \tilde M^2 T_2 \| A_{1,2}^\ast \|_{H_{ul}^{\theta_A}} 
			\leq C \max_{k=1,2} \| A_{k,2}^\ast \|_{H_{ul}^{\theta_k}} ,
		\end{equation}
		where we used Lemma \ref{lem:US_8.3.11}, the fact that $\|A_{1,2}(\cdot,t_2)\|_{H_{ul}^\theta} \leq C \tilde K \| A_{1,2}^\ast  \|_{H_{ul}^\theta}$ for all $t_2 \in [0, T_2]$ due to Lemma \ref{lem:K2_A1}, and $\me^{-I_2(s_2)} \| \partial_{x_2}^3 A_{1,2}(\cdot,s_2) \|_{H_{ul}^\theta} \leq C \| \partial_{x_2}^3 A_{1,2}(\cdot,s_2) \|_{H_{ul}^\theta} \leq C \| A_{1,2}(\cdot,s_2) \|_{H_{ul}^{\theta_A}}$ in order to derive the the right-hand side. 
		Since \eqref{eq:A22_bound_t} is bounded for all $t_2 \in [0,T_2]$ and the integrating factor is bounded, $A_{2,2,h}(\cdot,t_2) \in H_{ul}^\theta$ and satisfies \eqref{eq:A22h_bound} for all $t_2 \in [0,T_2]$.
		
		We can calculate $\tilde A_{2,2,s}$ explicitly by viewing it as the solution to a heat equation with a constant source term $\nu_1$ and initial data $\tilde A_{2,2,s}(x_2,0) \equiv 0$. A direct calculation using the explicit solution formula gives the spatially homogeneous expression
		\[
		\tilde A_{2,2,s}(x_2,t_2) = \tilde A_{2,2,s}(t_2) = \nu_1 \sqrt{\frac{\pi}{2}} \me^{1/2\zeta} \left( \textnormal{erf}\left(\frac{\sqrt{\zeta} t_2 - 1}{\sqrt{2 \zeta}} \right) + \textnormal{erf}\left(\frac{1}{\sqrt{2 \zeta}} \right) \right) ,
		\]
		for all $t_2 \in [0,T_2]$. In particular,
		\[
		\tilde A_{2,2,s}(T_2) = \nu_1 \sqrt{\frac{\pi}{2}} \me^{1/2\zeta} \left( \textnormal{erf}\left(\frac{\rho_{mid}}{\sqrt{2}} \right) + \textnormal{erf}\left(\frac{1}{\sqrt{2 \zeta}} \right) \right) .
		\]
		The expression in \eqref{eq:A22_asymptotics} follows after multiplying through by the integrating factor.
	\end{proof}

	\subsubsection*{Step IV: Estimate $A_{j,2}(x_2,t_2)$, $j = 3, \ldots, N-2$}
	
	In order to consider modulation functions $A_{j,2}$ with $j \geq 3$ we use Lemma \ref{lem:a_mj} again in order to separate $a_{j,2} =: a_{j,2}^L + a_{j,2}^N$ into two parts. Similarly to the analysis in chart $\mathcal K_1$ Step IV, we write
	\[
	a_{j,2}^L(\tilde A_{j,2}, \overline{\tilde A_{j,2}}) = 
	\left( \alpha_1 A_{1,2}^2 + \alpha_2 |A_{1,2}|^2 + \alpha_3 \overline{A_{1,2}}^2 \right) \tilde A_{j,2} + \left( \beta_1 A_{1,2}^2 + \beta_2 |A_{1,2}|^2 + \beta_3 \overline{A_{1,2}}^2 \right) \overline{\tilde A_{j,2}} ,
	\]
	and note that $a_{j,2}^N$ depends only modulation functions $A_{m'j',2}$ at lower orders with $\alpha(m') + j' < j$.
	
	\
	
	We prove the following result.
	
	\begin{lemma}
		\label{lem:K2_Aj}
		Let $\theta_A = \theta_A(j) = 3(j-1) + \theta$ where $j \in \{ 3, \ldots, N-2 \}$ and $\theta > 1/2$, and assume that $a_{j,2}^N(\cdot,t_2) \in H_{ul}^\theta$ for all $t_2 \in [0,T_2]$. Then $A_{j,2}(\cdot,t_2) \in H_{ul}^\theta$ for all $t_2 \in [0,T_2]$. If additionally $\mu(x) \equiv 0$ and
		\begin{equation}
			\label{eq:a_jN_cond_K2}
			\| a_{j,2}^N(\cdot,t_2) \|_{H_{ul}^\theta} \leq 
			C \max_{k = 1,\ldots,j-1} \| A_{k,2}(\cdot,t_2) \|_{H_{ul}^{\theta_{k,j}}}^3 ,
		\end{equation}
		then for $K > 0$ sufficiently small it holds that
		\begin{equation}
			\label{eq:Aj2_nu10_bound}
			\| A_{j,2}(\cdot,T_2) \|_{H_{ul}^\theta} \leq 
			C \max_{k = 1, \ldots, j} \| A_{k,2}^\ast  \|_{H_{ul}^{\theta_{k,j}}} ,
		\end{equation}
		where $\theta_{k,j}$ is given by \eqref{eq:theta_k_2}.
		%
		%
	\end{lemma}

	\begin{proof}
		We proceed similarly to the proof of Lemma \ref{lem:K1_Aj} in Step IV of the chart $\mathcal K_1$ analysis, and therefore present fewer details. As an induction hypothesis, we assume that $A_{k,2}(\cdot,t_2) \in H_{ul}^{\theta_{k,j}}$ for all $k=3,\ldots,j-1 \leq N-3$. We need to show that $A_{j,2}(\cdot,t_2) \in H_{ul}^{\theta}$. We omit the explicit treatment of the base case again since it is similar to the induction step. 
		
		\
		
		The equation for $A_{j,2}$ is given by
		\[
		\partial_{t_2} A_{j,2} =  4 \partial_{x_2}^2 A_{j,2} + v_2(t_2) A_{j,2} - a_{j,2} - 4i \partial_{x_2}^3 A_{j-1,2} + \partial_{x_2}^4 A_{j-2,2} ,
		\qquad
		A_{j,2}(x_2,0) = A_{j,2}^\ast(x_2) ,
		\]
		and we obtain a new equation
		\[
		\partial_{t_2} \tilde A_{j,2} = 
		4 \partial_{x_2}^2 \tilde A_{j,2} - a_{j,2}^L\left(\tilde A_{j,2}, \overline{\tilde A_{j,2}} \right) - \me^{-I_{2}(t_2)} \left( a_{j,2}^N + 4i \partial_{x_2}^3 A_{j-1,2} + \partial_{x_2}^4 A_{j-2,2} \right) ,
		\]
		with $\tilde A_{j,2}(x_2,0) = A_{j,2}^\ast(x_2)$, after the transformation $A_{j,2}(x_2,t_2) =: \me^{I_{2}(t_2)} \tilde A_{j,2}(x_2,t_2)$. We now apply the variation of constants formula and split the solution into two parts $\tilde A_{j,2} = \tilde A_{j,2,h} + \tilde A_{j,2,s}$, where
		\begin{equation}
			\label{eq:Aj2_voc}
			\begin{split}
				\tilde A_{j,2,h}(\cdot,t_2) &= 
				\me^{t_2 \Lambda_2} A_{j,2}^\ast  - \int_0^{t_2} \me^{(t_2 - s_2) \Lambda_2} a_{j,2}^L\left(\tilde A_{j,2}(\cdot,s_2), \overline{\tilde A_{j,2}}(\cdot,s_2) \right) ds_2 , \\
				\tilde A_{j,2,s}(\cdot,t_2) &= 
				- \int_0^{t_2} \me^{(t_2 - s_2) \Lambda_2} \me^{-I_{2}(s_2)} \left( a^N_{j,2} (\cdot,s_2) + 4i \partial_{x_2}^3 A_{j-1,2}(\cdot,s_2) + \partial_{x_2}^4 A_{j-2,2}(\cdot,s_2) \right) ds_2 ,
			\end{split}
		\end{equation}
		for $t_2 \in [0,T_2]$.
		
		Using the same bound for the semigroup as in Steps II-III above, as well as the bound for $A_{1,2}$ in Lemma \ref{lem:K2_A1}, we obtain
		\begin{equation}
			\label{eq:A2h_bound}
			\| \tilde A_{j,2,h}(\cdot,t_2) \|_{H_{ul}^\theta} \leq 
			C \| A_{j,1}^\ast \|_{H_{ul}^\theta} + C \tilde T_{j,2} \| A_{1,2}^\ast \|_{H_{ul}^\theta}^2 
			\leq C \max_{k = 1,\ldots, j} \| A_{j,1}^\ast \|_{H_{ul}^\theta} ,
		\end{equation}
		for $t_2 \in [0, \tilde T_{j,2}]$, where $\tilde T_{j,2} := \min\{T_2, \sup\{ t_2 > 0 : \| \tilde A_{j,2}(\cdot,s_2) \|_{H_{ul}^\theta} \leq \tilde M \} \}$.
		
		For the source term $\tilde A_{j,2,s}$ we have
		\begin{equation}
			\label{eq:A2s_bound}
			\begin{split}
				\| \tilde A_{j,2,s}(\cdot,t_2) \|_{H_{ul}^\theta} &\leq
				C \int_0^{t_2} \big\| a^N_{j,2}(\cdot,s_2) + 4i \partial_{x_2}^3 A_{j-1,2}(\cdot,s_2) + \partial_{x_2}^4 A_{j-2,2}(\cdot,s_2) \big\|_{H_{ul}^\theta} ds_2 \\
				&\leq C \tilde T_{j,2} \left( \big\| a^N_{3,2}(\cdot,t_2) \big\|_{H_{ul}^\theta} + \| \partial_{x_2}^3 A_{j-1,2}(\cdot,t_2) \big\|_{H_{ul}^\theta} + \| \partial_{x_2}^4 A_{j-2,2}(\cdot,t_2) \big\|_{H_{ul}^\theta} \right) ,
			\end{split}
		\end{equation}
		for $t_2 \in [0,\tilde T_{j,2}]$. Since $a^N_{3,2}(\cdot,t_2) \in H_{ul}^\theta$ by assumption and
		\[
		\begin{split}
			\| \partial_{x_2}^3 A_{j-1,2}(\cdot,t_2) \big\|_{H_{ul}^\theta} &\leq
			C \| A_{j-1,2}(\cdot,t_2) \big\|_{H_{ul}^{\theta_{j-1,j}}} \leq C , \\
			\| \partial_{x_2}^4 A_{j-2,2}(\cdot,t_2) \big\|_{H_{ul}^\theta} &\leq 
			C \| A_{j-2,2}(\cdot,t_2) \big\|_{H_{ul}^{\theta_{j-2,j}}} \leq C ,
		\end{split}
		\]
		by our regularity assumptions and the induction hypothesis, $\tilde A_{j,2,s}(\cdot,t_2) \in H_{ul}^\theta$ for all $t_2 \in [0,\tilde T_{j,2}]$.
		
		Since the integrating factor $e^{I_2(t_2)}$ is also bounded for all $t_2 \in [0,\tilde T_{2,j}]$, it follows that $A_{j,2}(\cdot,t_2) = \me^{I_2(t_2)} (\tilde A_{j,2,h}(\cdot,t_2) + \tilde A_{j,2,s}(\cdot,t_2)) \in H_{ul}^\theta$ for all $t_2 \in [0,\tilde T_{j,2}]$. We may set $\tilde T_{j,2} = T_2$, because the bounds above are finite for all $t_2 \in [0,T_2]$. 
		
		\
		
		It remains to prove the bound in \eqref{eq:Aj2_nu10_bound}. Under the additional assumptions that \eqref{eq:a_jN_cond_K2} is satisfied and $\mu(x) \equiv 0$ we can obtain a better bound on the norm of $\tilde A_{j,2,s}$. This is achieved by using \eqref{eq:a_jN_cond_K2} and
		\[
		\begin{split}
			\| \partial_{x_2}^3 A_{j-1,2}(\cdot,t_2) \big\|_{H_{ul}^\theta} &\leq
			C \| A_{j-1,2}(\cdot,t_2) \big\|_{H_{ul}^{\theta_{j-1,j}}} \leq 
			C \max_{k=1,\ldots,j-1} \| A_{k,2}^\ast \|_{H_{ul}^{\theta_{k,j}}} , \\
			\| \partial_{x_2}^4 A_{j-2,2}(\cdot,t_2) \big\|_{H_{ul}^\theta} &\leq 
			C \| A_{j-2,2}(\cdot,t_2) \big\|_{H_{ul}^{\theta_{j-2,j}}} \leq 
			C \max_{k=1,\ldots,j-2} \| A_{k,2}^\ast \|_{H_{ul}^{\theta_{k,j}}} ,
		\end{split}
		\]
		where the right-most bounds can be derived using a recursive argument base on the improved bounds from Lemmas \ref{lem:K2_A2} and \ref{lem:K2_Aj} in case $\mu(x) \equiv 0$. Substituting these bounds into the right-hand side of \eqref{eq:A2s_bound}, combining the resulting bound with \eqref{eq:A2h_bound} and multiplying through by the (bounded) integrating factor yields the result in \eqref{eq:Aj2_nu10_bound}.
	\end{proof}

	\subsubsection{Step V: Estimate $\psi_2$}
	
	As for the $\mathcal K_1$ analysis in Section \ref{sub:K1_dynamics}, we need to control the non-critical modulation functions $A_{mj,2}$ with $|m| \neq 1$ in order to describe the dynamics of $\psi_2(\cdot,T_2)$ and the map $\pi_2 : \Sigma^{in}_2 \to \Sigma^{out}_2$. This is also necessary in order to verify the assumptions on $a_{j,2}^N$ in Lemma \ref{lem:K2_Aj}. As before, we reinstate the full notation $A_{\pm1j,2}$ instead of $A_{j,2}$ in order to distinguish modulation functions with $|m|=1$ and $|m|\neq1$, \SJ{and recall that $\textbf A_{mj,2}$ denotes the set of lower order critical modulation functions $A_{\pm1j',2}$ with $j' \in \{1,\ldots,\alpha(m) + j - 1\}$.}
	
	\begin{lemma}
		\label{lem:K2_gl_coefficients}
		Let $\theta_A = 1 + 3(N-3) + \theta$ where $\theta > 1/2$. Then the conditions on $a_{j,2}^N$ for the applicability of Lemma \ref{lem:K2_Aj} are satisfied. We have that $A_{\pm1j,2}(\cdot,t_2) \in H_{ul}^{\theta_{j,N-2}}$ for all $j = 1, \ldots, N-2$ and $t_2 \in [0,T_2]$, and
		\begin{equation}
			\label{eq:gl_coefficients_smooth_K2}
			A_{mj,2}(\cdot, t_1) = g_{mj,2}^{gl} \left( \textbf{\textup{A}}_{mj,2}(\cdot,t_2) , v_2(t_2) \right) \in H_{ul}^{\theta_{\alpha(m)+j,N}} ,
		\end{equation}
		for all $t_2 \in [0,T_2]$, $m \in I_N \setminus \{-1,1\}$ and $j = 1, \ldots, \tilde \alpha(m)$. If additionally $\mu(x) \equiv 0$ then for each $m \in I_N \setminus \{-1,1\}$ we have
		\begin{equation}
			\label{eq:gl_coefficients_bounds_K2}
			\big\| g_{mj,2}^{gl} \left( \textbf{\textup{A}}_{mj,2}(\cdot,T_2) , v_2(T_2) \right) \big\|_{H_{ul}^{\theta_{j,N-2}}}
			\leq C \max_{k = 1, \ldots, j-2} \| A_{1k,2}^\ast \|_{H_{ul}^{\theta_{k,N-2}}} ,
		\end{equation}
		for all $j = 1, \ldots, \tilde \alpha(m)$.
	\end{lemma}
	
	\begin{proof}
		The proof is analogous to the proof of Lemma \ref{lem:K1_gl_coefficients}, so we omit the details. The reason that we obtain constant bounds in the case $\mu(x) \equiv 0$ (as opposed the the exponential bounds in Lemma \ref{lem:K1_gl_coefficients}) is that we have only constant bounds when $\nu_1 = 0$ in Lemmas \ref{lem:K2_A1}, \ref{lem:K2_A2} and \ref{lem:K2_Aj}.
	\end{proof}
	
	Combining the results of Lemmas \ref{lem:K2_T2}, \ref{lem:K2_A1}, \ref{lem:K2_A2}, \ref{lem:K2_Aj} and \ref{lem:K2_gl_coefficients} we obtain the following result, which summarises the dynamics in chart $\mathcal K_2$.
	
	\begin{proposition}
		\label{prop:K2_summary}
		Let $\theta_A = 1 + 3(N-3) + \theta$ where $\theta > 1/2$. Then for $\vartheta, K > 0$ sufficiently small but fixed we have $\psi_2(\cdot,t_2) \in H_{ul}^\theta$ for all $t_2 \in [0,T_2]$. In particular, the map $\pi_2 : \Sigma_2^{in} \to \Sigma_2^{out}$ is well-defined and given by
		\begin{equation}
			\label{eq:pi_2}
			\pi_2 : ( \psi_2^\ast, -\zeta^{-1/2}, r_2 ) \mapsto \left( \psi_2(\cdot,T_2), \rho_{max}, r_2 \right) ,
		\end{equation}
		with
		\begin{equation}
			\label{eq:psi_2}
			\psi_2(x,t_2) = \psi_{GL,2}(x,t_2) + r_2 \left( A_{2,2,h}(x,t_2) \me^{ix} + c.c.\right) 
			+ r_2^2 h_{rem,2}(x,t_2) ,
		\end{equation}
		where $\psi_{GL,2}(x,t_2) = A_{1,2} (x_2,T_2) \me^{ix} + c.c.$ satisfies $\| \psi_{GL,2}(\cdot,t_2) \|_{H_{ul}^{\theta_A-1}} \leq C \|A_{1,2}^\ast \|_{H_{ul}^{\theta_A}}$, the function $A_{2,2}$ is described in Lemma \ref{lem:K2_A2}, and the remainder term $h_{rem,2}(\cdot,t_2) \in H_{ul}^{\theta}$ for all $t_2 \in [0,T_2]$.
		If additionally $\mu(x) \equiv 0$ then we have
		\begin{equation}
			\label{eq:psi_2_bound}
			\| \psi_2(\cdot,t_2) \|_{H_{ul}^\theta} \leq 
			C \| A_{1,2}^\ast \|_{H_{ul}^{\theta+1}} + r_2 C \max_{k=1,2} \| A_{k,2}^\ast  \|_{H_{ul}^{\theta+1}} + r_2^2 C \max_{k = 1, \ldots, N-2} \| A_{k,2}^\ast \|_{H_{ul}^{\theta+1}} ,
		\end{equation}
		for all $t_2 \in [0,T_2]$.
	\end{proposition}
	
	\begin{proof}
		The form of the map in \eqref{eq:pi_2} follows immediately from Lemma \ref{lem:K2_T2}. The fact that $\psi_2(\cdot, t_2) \in H_{ul}^\theta$ for all $t_2 \in [0,T_2]$ follows from its definition in \eqref{eq:psi_definition} and the fact that $A_{mj,2}(\cdot,t_2) \in H_{ul}^{\theta+1}$ for all $m \in I_N$, $j = 1, \ldots , \tilde \alpha(m)$ and $t_2 \in [0,T_2]$ due to Lemmas \ref{lem:K2_A1}, \ref{lem:K2_A2}, \ref{lem:K2_Aj} and \ref{lem:K2_gl_coefficients}. Similarly to the proof of Proposition \ref{prop:K1_summary}, we need one extra degree of regularity wherever we use Lemma \ref{lem:norm_rescaling} to take the $H_{ul}^\theta$ norm with respect to $x$ instead of $x_2$.
		
		The asymptotic expression in \eqref{eq:psi_2} is obtained directly by evaluating the series expansion \eqref{eq:psi_definition} and evaluating it at $t_2 = T_2$. The details are similar to the derivation of the asymptotic formula \eqref{eq:psi_1_out} in the proof of Proposition \ref{prop:K1_summary}, and so omitted here for brevity.
		
		Finally, the bound in \eqref{eq:psi_2_bound} follows after applying a triangle inequality to the series definitions of $\psi_2$ and applying the assertions pertaining to the case $\mu(x) \equiv 0$ in Lemmas \ref{lem:K2_A1}, \ref{lem:K2_A2}, \ref{lem:K2_Aj} and \ref{lem:K2_gl_coefficients}.
	\end{proof}

	\subsection{Proof of Lemma \ref{lem:Ansatz_dynamics}}
	\label{sub:proof_of_approximation_lemma_1}
	
	Lemma \ref{lem:Ansatz_dynamics} can be proved with the information obtained in charts $\mathcal K_1$ and $\mathcal K_2$. The idea is to analyse $\pi^{mid} : \Delta^{in} \to \Delta^{mid}$ via its pre-image $\pi_2 \circ \kappa_{12} \circ \pi_1 : \Sigma_1^{in} \to \Sigma_2^{out}$ in the blown-up space, where $\kappa_{12}$ is the change of coordinates map defined in \eqref{eq:kappa_maps} and $\pi_1$, $\pi_2$ are the maps described in Propositions \ref{prop:K1_summary}, \ref{prop:K2_summary} respectively. The fact that $\Psi(\cdot,t) = r(\bar t) \psi(\cdot, \bar t) \in H_{ul}^\theta$ for all $t \in [0,T_{mid}]$ follows directly from the fact that $\psi_l(\cdot, t_l) \in H_{ul}^\theta$ for all $t_l \in [0,T_l]$ in charts $\mathcal K_l$ and $\mathcal K_2$; recall Propositions \ref{prop:K1_summary} and \ref{prop:K2_summary}. It therefore remains to show that the map $\pi^{mid}$ is given by \eqref{eq:pi_mid_asymptotics}.
	
	\
	
	We consider an initial condition $(\Psi^\ast , -\rho_{in}) \in \Delta^{in}$, which corresponds to an initial condition of the form $(\psi_1^\ast, \sqrt{\rho_{in}} , \eps_1^\ast) \in \Sigma_1^{in}$ defined by
	\begin{equation}
		\label{eq:original_coords_ic}
		\psi_1^\ast  = \frac{\Psi^\ast }{\sqrt{\rho_{in}}} , \qquad
		r_1^\ast = \sqrt{\rho_{in}} , \qquad \eps_1^\ast = \frac{\eps}{\rho_{in}^2} .
	\end{equation}
	Since $\Sigma_2^{in} = \kappa_{12}(\Sigma_1^{out})$, we obtain the following bounds for the initial modulation functions in chart $\mathcal K_2$ using the change of coordinates formula $A_{mj,2}^\ast  = \zeta^{-(\alpha(m) + j)/4} A_{mj,1}^\ast(\cdot, T_1)$, recall \eqref{eq:kappa_Amj}, together with Lemmas \ref{lem:K1_A1}, \ref{lem:K1_A2}, \ref{lem:K1_Aj} and \ref{lem:K1_gl_coefficients}:
	\begin{equation}
		\label{eq:K2_ic_bounds}
		\begin{split}
			&\| A_{1,2}^\ast \|_{H_{ul}^{\theta_A}} = \zeta^{-1/4} \| A_{1,1}(\cdot, T_1) \|_{H_{ul}^{\theta_A}} \leq C \me^{-\kappa / 2 \eps_1^\ast} \| A_{1,1}^\ast \|_{H_{ul}^{\theta_A}} , \\
			&\| A_{2,2}^\ast \|_{H_{ul}^{\theta_{2,N-2}}} = \zeta^{-1/2} \| A_{2,1}(\cdot, T_1) \|_{H_{ul}^{\theta_{2,N-2}}} \leq C \left( |\nu_1| f_{2,1}(\eps_1^\ast) + \me^{-\kappa / 2 \eps_1^\ast} \max_{k=1,2} \| A_{k,1}^\ast  \|_{H_{ul}^{\theta_{k,N-2}}} \right) , \\
			&\| A_{j,2}^\ast \|_{H_{ul}^{\theta_{j,N-2}}} = \zeta^{-j/4} \| A_{j,1}(\cdot, T_1) \|_{H_{ul}^{\theta_{j,N-2}}} \leq C , \\
			&\| A_{mj,2}^\ast \|_{H_{ul}^{\theta_{\alpha(m)+j,N}}} = \zeta^{-(\alpha(m) + j)/4} \| A_{mj,1}(\cdot, T_1) \|_{H_{ul}^{\theta_{\alpha(m)+j,N}}} \leq C ,
		\end{split}
	\end{equation}
	where $j \geq  3$ in the second-to-last line and $|m| \neq 1$ in the last. This provides sufficient information on initial conditions
	\[
	\left(\psi_2^\ast, -\zeta^{-1/2}, r_2 \right) = \kappa_{12} \left( \psi_1(\cdot,T_1), \sqrt{\rho_{in}} \left(\frac{\eps_1^\ast}{\zeta}\right)^{1/4} , \eps_1^\ast \right) = 
	\left( \zeta^{-1/4} \psi_1(\cdot,T_1) , -\zeta^{-1/2} , \sqrt{\rho_{in}} {\eps_1^\ast}^{1/4} \right) 
	\]
	in $\Sigma_2^{in}$, now in chart $\mathcal K_2$ coordinates. The evolution from $\Sigma_2^{in}$ up to $\Sigma_2^{out}$ is described by Proposition \ref{prop:K2_summary}. Using Proposition \ref{prop:K2_summary} and the bound for $\| A_{1,2}^\ast \|_{H_{ul}^{\theta_A}}$ in \eqref{eq:K2_ic_bounds} we have that
	\[
	\| \psi_{GL,2}(\cdot,T_2) \|_{{H_{ul}^{\theta_A-1}}} \leq 
	C \|A_{1,2}^\ast \|_{H_{ul}^{\theta_A}} \leq 
	C \me^{-\kappa / 2 \eps_1^\ast} \| A_{1,1}^\ast \|_{H_{ul}^{\theta_A}} .
	\]
	Thus $\psi_{GL,2}(\cdot,T_2)$ can be pushed into the higher order terms in \eqref{eq:psi_2}, such that we obtain
	\begin{equation}
		\label{eq:psi_2_out}
		\psi_2(x,t_2) = r_2 \left( A_{2,2}(x_2,t_2) \me^{ix} + c.c. \right) + r_2^2 \tilde h_{rem,2}(x,t_2) ,
	\end{equation}
	where $\tilde h_{rem,2}(\cdot,t_2) \in H_{ul}^\theta$ for all $t_2 \in [0,T_2]$ and 
	\[
	\left\| A_{2,2}(r_2 \cdot , T_2) \me^{i \cdot} + c.c. \right\|_{H_{ul}^{\theta_{2,N-2}}} \leq 
	C \left( |\nu_1|  + \me^{-\kappa / 2 \eps_1^\ast} 
	\right) ,
	\]
	due to \eqref{eq:K2_ic_bounds} and Lemma \ref{lem:K2_A2}.
	
	In order to express \eqref{eq:psi_2_out} in terms of the original function $\Psi$ we simply apply the blow-down transformations
	\[
	\Psi(x,T_{mid}) = r_2 \psi_2(x, T_2) = 
	\eps^{1/4} \psi_2 \left( x, \sqrt{\rho_{in}} {\eps_1^\ast}^{1/4} \right) 
	\]
	and $\eps_1^\ast = \eps / \rho_{in}^2$, thereby obtaining
	\[
	\Psi(x,T_{mid}) = \eps^{1/2} \left( A_{2,2}(\eps^{1/4} x, T_2) \me^{ix} + c.c. \right) + \eps^{3/4} R(x,\eps) ,
	\]
	where
	\[
	\left\| A_{2,2}(\eps^{1/4} \cdot, T_2) \me^{i\cdot} + c.c. \right\|_{H_{ul}^\theta} \leq 
	C \left( |\nu_1| + \me^{-\kappa \rho_{in}^2 / 2 \eps} \right)
	\]
	and $R(\cdot,\eps) := \tilde h_{rem,2}(\cdot,T_2) \in H_{ul}^\theta$ as required. This concludes the proof of Lemma \ref{lem:Ansatz_dynamics}.
	\qed

	\subsection{Chart $\mathcal K_3$}
	\label{sub:K3_dynamics}
	
	In order to prove Lemma \ref{lem:Ansatz_dynamics_mu_0} we also need to understand the dynamics in chart $\mathcal K_3$. We recall the equations in chart $\mathcal K_3$:
	\begin{equation}
		\label{eq:K3_eqns2}
		\begin{split}
			\partial_{t_3} A_{1,3} &= 4 \partial_{x_3}^2 A_{1,3} + \left(1 - \frac{\eps_3}{2} \right) A_{1,3} - 3 A_{1,3} |A_{1,3}|^2 , \\
			\partial_{t_3} r_3 &= \frac{1}{2} r_3 \eps_3 , \\
			\partial_{t_3} \eps_3 &= - 2 \eps_3^2 , \\
			\partial_{t_3} A_{2,3} &= 4 \partial_{x_3}^2 A_{2,3} + \left( 1 - \frac{\eps_3}{2} \right) A_{2,3} - a_{2,3} - 4i \partial_{x_3}^3 A_{1,3} 
			, \\
			\partial_{t_3} A_{j,3} &=  4 \partial_{x_3}^2 A_{j,3} + \left( 1 - \frac{\eps_3}{2} \right) A_{j,3} - a_{j,3} - 4i \partial_{x_3}^3 A_{j-1,3} - \partial_{x_3}^4 A_{j-2,3} ,
		\end{split}
	\end{equation}
	for $j = 3, \ldots, N-2$. Note that we have set $\nu_1 = 0$ in the equation for $A_{2,3}$ since our goal is to prove Lemma \ref{lem:Ansatz_dynamics_mu_0}, which has $\mu(x) \equiv 0$. Our aim in chart $\mathcal K_3$ is to describe the map
	\[
	\pi_3 : \Sigma^{in}_3 \ni \left( \psi_3^\ast , r_3^\ast, \zeta \right) \mapsto \left(\psi_3(\cdot,T_3), \rho_{out}^{1/2}, \eps_3^{out}(r_3^\ast) \right) \in \Sigma_3^{out} ,
	\]
	induced by forward evolution of initial conditions in
	\[
	\Sigma_3^{in} = \left\{ ( \psi_3^\ast  , r_3, \zeta ) : \| \psi_3^\ast  \|_{H_{ul}^\theta} \leq K, r_3 \in [0, \zeta^{-1/4} \vartheta]  \right\} ,
	\]
	which corresponds to the exit section $\Sigma^{out}_2$ in chart $\mathcal K_2$ coordinates via $\Sigma_3^{in} = \kappa_{23}(\Sigma_2^{out})$ if we set $\rho_{mid} = \zeta^{-1/2}$ (which we shall assume is the case from this point on), up to the exit section
	\[
	\Sigma_3^{out} = \left\{ ( \psi_3^\ast  , \rho_{out}^{1/2}, \eps_3 ) : \| \psi_3^\ast  \|_{H_{ul}^\theta} \leq K, \eps_3 \in [0,\zeta]  \right\} ,
	\]
	which corresponds to the original exit section $\Delta^{out}$ defined in \eqref{eq:Delta_out} \SJ{after blowing down, fixing $\eps$, and using the natural identification with points in $\Delta^{out}_\eps = \Delta^{out} \times [0,\eps_0]$}. We proceed again in steps.

	\subsubsection*{Step I: Solve for $r_3(t_3)$, $\eps_3(t_3)$ and $T_3$}
	
	We have the following result.
	
	\begin{lemma}
		\label{lem:K3_r3eps3}
		Let $r_3^\ast := r_3(0)$. The transition time $T_3$ is given by
		\[
		T_3 = \frac{1}{2 \zeta {r_3^\ast}^4} \left(\rho_{out}^2 - {r_3^\ast}^4 \right) ,
		\]
		and for $t_3 \in [0,T_3]$ we have
		\[
		r_3(t_3) = r_3^\ast (1 + 2 \zeta t_3)^{1/4} , \qquad
		\eps_3(t_3) = \frac{\zeta}{1 + 2 \zeta t_3} .
		\]
		In particular, $\eps_3^{out}(r_3^\ast) = \zeta {r_3^\ast}^4 / \rho_{out}^2$.
	\end{lemma}
	
	\begin{proof}
		The expressions for $r_3(t_3)$ and $\eps_3(t_3)$ follow by direct integration, and $T_3$ is obtained from the requirement that $r_3(T_3) = \rho_{out}^{1/2}$. The expression for $\eps_3^{out}(r_3^\ast)$ follows after evaluating $\eps_3(T_3)$.
	\end{proof}

	\subsubsection*{Step II: Estimate $A_{1,3}(x_3,t_3)$}
	
	We prove the following result.
	
	\begin{lemma}
		\label{lem:K3_A1}
		Fix $\theta_A > 1/2$ and assume that $A_{1,3}^\ast \in H_{ul}^{\theta_A}$. If $K, \zeta > 0$ are sufficiently small then for each $\tilde \sigma_0 \geq \tilde \sigma > 1$ there is exists a constant $C > 0$ such that if initial conditions satisfy
		\begin{equation}
			\label{eq:A13_ic_bound}
			\|A_{1,3}^\ast \|_{H_{ul}^{\theta_A}}  \leq C \me^{- \tilde \sigma_0 \rho_{out}^2 / 2 \zeta {r_3^\ast}^4} ,
		\end{equation}
		then $A_{1,3}(\cdot,t_3) \in H_{ul}^{\theta_A}$ for all $t_3 \in [0,T_3]$, with
		\[
		\|A_{1,3}(\cdot,T_3)\|_{H_{ul}^{\theta_A}}  \leq C \me^{\tilde \sigma \rho_{out}^2 / 2 \zeta {r_3^\ast}^4} \|A_{1,3}^\ast \|_{H_{ul}^{\theta_A}} \leq C \me^{-(\tilde \sigma_0 - \tilde \sigma) \rho_{out}^2 / 2 \zeta {r_3^\ast}^4} .
		\]
	\end{lemma}
	
	\begin{proof}
		The proof proceeds via arguments similar to those used to prove Lemma \ref{lem:K1_A1}. Substituting the expressions for $\eps_3(t_3)$ into the equation for $A_{1,3}$ yields
		\begin{equation}
			\label{eq:A1_decoupled_K3}
			\partial_{t_3} A_{1,3} = 4 \partial_{x_3}^2 A_{1,3} + \left( 1 - \frac{\eps_3(t_3)}{2} \right) A_{1,3} - 3 A_{1,3} |A_{1,3}|^2 , \qquad
			A_{1,3}(x_3,0) = A_{1,3}^\ast(x_3) .
		\end{equation}
		Defining an integrating factor $\me^{I_3(t_3)}$ via
		\begin{equation}
			\label{eq:integrating_factor_A11_K3}
			I_3(t_3) = -\int_0^{t_3} \frac{\eps_3(s_3)}{2} ds_3 = - \frac{1}{4} \ln (1 + 2 \zeta t_3) , \qquad
			\me^{I_3(t_3)} = \frac{1}{\sqrt[4]{1 + 2 \zeta t_3}} ,
		\end{equation}
		and applying the transformation $A_{1,3}(x_3,t_3) =: \me^{I_3(t_3)} \tilde A_{1,3}(x_3,t_3)$ leads to
		\[
		\partial_{t_3} \tilde A_{1,3} = 4 \partial_{x_3}^2 \tilde A_{1,3} + \tilde A_{1,3} - 3 \me^{2 I_3(t_3)} \tilde A_{1,3} |\tilde A_{1,3}|^2, \qquad 
		\tilde A_{1,3}(x_3,0) = A_{1,3}^\ast(x_3) .
		\]
		Using the variation of constants formula, this equation can be recast in the integral form
		\begin{equation}
			\label{eq:A13_voc}
			\tilde A_{1,3}(\cdot,t_3) = \me^{t_3\Lambda_3} A_{1,3}^\ast  - 3 \int_0^{t_3} \me^{(t_3 - s_3) \Lambda_3}  \me^{2 I_3(s_3)} \tilde A_{1,3}(\cdot,s_3) |\tilde A_{1,3}(\cdot,s_3)|^2 ds_3 ,
		\end{equation}
		where $(\me^{t_3 \Lambda_3})_{t_3 \geq 0}$ is the semigroup generated by $\Lambda_3 := 4 \partial_{x_3}^2 + 1$. In order to bound $\|\tilde A_{1,3}(\cdot,t_3)\|_{H_{ul}^{\theta_A}}$ we need to bound $(\me^{t_3 \Lambda_3})_{t_3 \geq 0}$ and the cubic nonlinearity $\tilde A_{1,3}(\cdot,s_3) |\tilde A_{1,3}(\cdot,s_3)|^2$ in $H_{ul}^{\theta_A}$. We shall also use the fact that $\me^{2I_3(s_3)} \leq 1$ for all $s_3 \geq 0$.
		
		The semigroup can be bounded using Lemma \ref{lem:US_8.3.7}. We find that for each $\sigma > 1$ there exists a $C>0$ depending on ${\theta_A}$ such that
		\[
		\| \me^{t_3 \Lambda_3} \|_{H_{ul}^{\theta_A} \to H_{ul}^{\theta_A}} \leq C \me^{\sigma t_3} 
		\]
		for all $t_3 \geq 0$.
		
		Using the semigroup bound and Lemma \ref{lem:US_8.3.7} we obtain
		\[
		\| \tilde{A}_{1,3}(\cdot,t_3) \|_{H^{\theta_A}_{ul}} \leq C \me^{\sigma t_3} \|A_{1,3}^\ast \|_{H^{\theta_A}_{ul}} + C \tilde K^2 \int_0^{t_3} \me^{\sigma (t_3 - s_3)} \| \tilde{A}_{1,3}(\cdot,s_3) \|_{H^{\theta_A}_{ul}}\,ds_3 ,
		\]
		for $t_3 \in [0,\tilde T_3]$, and applying the Gr\"onwall inequality in Lemma \ref{lem:HK_2.8} leads to
		\[
		\| \tilde{A}_{1,3}(\cdot,t_3) \|_{H^{\theta_A}_{ul}} \leq C \me^{\tilde \sigma t_3} \|A_{1,3}^\ast \|_{H^{\theta_A}_{ul}} , 
		\]
		for $t_3 \in [0, \tilde T_3]$, where $\tilde \sigma := \sigma + 2 C \tilde K^2 > 1$. Therefore
		\begin{equation}
			\label{eq:A13_t3_bound}
			\| A_{1,3}(\cdot,t_3) \|_{H^{\theta_A}_{ul}} \leq C \frac{\me^{\tilde \sigma t_3}}{\sqrt[4]{1 + 2 \zeta t_3}} \|A_{1,3}^\ast \|_{H^{\theta_A}_{ul}} , 
		\end{equation}
		for $t_3 \in [0, \tilde T_3]$. Since the right-hand side is increasing on $t_3 \in [0,T_3]$, it follows that $A_{1,3}(\cdot,t_3) \in H_{ul}^{\theta_A}$ for all $t_3 \in [0,T_3]$ if the initial condition $A_{1,3}^\ast$ is sufficiently small in $H_{ul}^{\theta_A}$. A sufficient condition is that
		\begin{equation}
			\label{eq:A13_K_bound}
			\| A_{1,3}(\cdot,T_3) \|_{H^{\theta_A}_{ul}} \leq 
			C \frac{\me^{\tilde \sigma T_3}}{\sqrt[4]{1 + 2 \zeta T_3}} \|A_{1,3}^\ast \|_{H^{\theta_A}_{ul}} < K.
		\end{equation}
		This condition is satisfied due to \eqref{eq:A13_ic_bound}, which can be seen by substituting the expression for $T_3$ in Lemma \ref{lem:K3_r3eps3}. In this case we can set $\tilde T_3 = T_3$ in \eqref{eq:A13_t3_bound}, thereby proving the desired result.
	\end{proof}
	
	\begin{remark}
		Since the source term $\eps \mu(x)$ does not influence the lowest order approximation, Lemma \ref{lem:K3_A1} applies whether or not we have $\mu(x) \equiv 0$.
	\end{remark}

	\subsubsection*{Step III: Estimate $A_{2,3}(x_3,t_3)$}
	
	We prove the following result.
	
	\begin{lemma}
		\label{lem:K3_A2}
		Let $\theta_A = \theta + 3$ where $\theta > 1/2$ and assume that
		\begin{equation}
			\label{eq:A23_ic_bound}
			\|A_{k,3}^\ast \|_{H_{ul}^{\theta_k}}  \leq C \me^{- \tilde \sigma_0 \rho_{out}^2 / 2 \zeta {r_3^\ast}^4} ,
		\end{equation}
		for $k=1,2$, where $\tilde \sigma_0 > \tilde \sigma > \sigma > 1$ is the same constant as in Lemma \ref{lem:K3_A1} and $\theta_k$ is defined in \eqref{eq:theta_k}. Then $A_{2,3}(\cdot,t_3) \in H_{ul}^\theta$ for all $t_3 \in [0,T_3]$ and
		\begin{equation}
			\label{eq:A23_bound}
			\| A_{2,3}(\cdot,T_3) \|_{H_{ul}^\theta}
			\leq C {r_3^\ast}^{-1} \me^{ \tilde \sigma \rho_{out}^2 / 2 \zeta {r_3^\ast}^4} \max_{k = 1,2} \|A_{k,3}^\ast \|_{H_{ul}^{\theta_k}} \leq
			C 
			\me^{ -(\tilde \sigma_0 - \tilde \sigma) \rho_{out}^2 / 2 \zeta {r_3^\ast}^4} .
		\end{equation}
	\end{lemma}
	
	\begin{proof}
		The proof proceeds via arguments similar to those in the proof of Lemma \ref{lem:K1_A2}. 
		We have
		\[
		\partial_{t_3} A_{2,3} = 4 \partial_{x_3}^2 A_{2,3} + \left( 1 - \frac{\eps_3(t_3)}{2} \right) A_{2,3} -  3 A_{1,3}(x_3,t_3)^2 \overline{A_{2,3}} - 6 |A_{1,3}(x_3,t_3)|^2 A_{2,3} - 4i \partial_{x_3}^3 A_{1,3}(x_3,t_3) ,
		\]
		with $A_{2,3}(x_3,0) = A_{2,3}^\ast(x_3)$ for $t_3 \in [0,T_3]$ and $x_3 \in \R$. Following the transformation $A_{2,3}(x_3,t_3) =: \me^{I_3(t_3)} \tilde A_{2,3}(x_3,t_3)$, where $\me^{I_3(t_3)}$ is the same integrating factor used in Step II (recall \eqref{eq:integrating_factor_A11_K3}), we obtain
		\[
		\partial_{t_3} \tilde A_{2,3} = 4 \partial_{x_3}^2 \tilde A_{2,3} + \tilde A_{2,3} -  3 A_{1,3}(x_3,t_3)^2 \overline{\tilde A_{2,3}} - 6 |A_{1,3}(x_3,t_3)|^2 \tilde A_{2,3} - 4i \me^{-I_3(t_3)} \partial_{x_3}^3 A_{1,3}(x_3,t_3) ,
		\]
		with $A_{2,3}(x_3,0) = A_{2,3}^\ast(x_3)$. After applying the variation of constants formula we obtain
		\begin{equation}
			\begin{split}
				\label{eq:A23_voc}
				\tilde A_{2,3} & (\cdot,t_3) = 
				\me^{t_3 \Lambda_3} A_{2,3}^\ast \ - \\
				\int_0^{t_3} & \me^{(t_3 - s_3) \Lambda_3} \left(3 A_{1,3}(\cdot,s_3)^2 \overline{\tilde A_{2,3}}(\cdot,s_3) + 6 |A_{1,3}(\cdot,s_3)|^2 \tilde A_{2,3}(\cdot,s_3) + 4i \me^{-I_3(s_3)} \partial_{x_3}^3 A_{1,3}(\cdot,s_3) \right) ds_3 ,
			\end{split}
		\end{equation}
		where $\Lambda_3 = 4 \partial_{x_3}^2 + 1$ as in Step II. We have the same bound for the semigroup as in Step II, i.e.~for each $\sigma > 1$ there is a constant $C > 0$ depending on $\theta$ such that $\| \me^{t_3 \Lambda_3} \|_{H_{ul}^\theta \to H_{ul}^\theta} \leq C \me^{\sigma t_3}$. Given the bounds for $A_{1,3}$ in Lemma \ref{lem:K3_A1} and its proof (in particular \eqref{eq:A13_t3_bound}), it follows via arguments analogous to those applied in the proof of Lemma \ref{lem:K1_A2} that for all $t_3 \in [0,\tilde T_{2,3}]$ we have 
		\begin{equation}
			\label{eq:A23_tilde_bound}
			\begin{split}
				\| \tilde A_{2,3}(\cdot,t_3) \|_{H_{ul}^\theta} &\leq 
				C \me^{\sigma t_3} \| A_{2,3}^\ast  \|_{H_{ul}^\theta} + C \me^{\tilde \sigma t_3} \| A_{1,3}^\ast \|_{H_{ul}^{\theta_A}} \\
				&+ 
				C i \me^{\sigma t_3} \|A_{1,3}^\ast \|_{H_{ul}^{\theta}}^2 \Gamma \left( \frac{1}{2}, - (-\sigma + 2 \tilde \sigma) \left(\frac{1}{2 \zeta} + s_3 \right) \right) \bigg|_{t_3}^0 ,
			\end{split}
		\end{equation}
		where $- \sigma + 2 \tilde \sigma = \sigma + 4 C \tilde K^2 > 1$ and $-\sigma + \tilde \sigma = 2C \tilde K^2 \in (0,1)$. A direct calculation using the upper gamma function asymptotics in Appendix \ref{app:technical_estimates} (equation \eqref{eq:Gamma_asymptotics}) shows that
		\begin{equation}
			\label{eq:A23_t3_bound}
			\| \tilde A_{2,3}(\cdot,t_3) \|_{H_{ul}^\theta} \leq 
			C \me^{\tilde \sigma t_3} \max_{k=1,2} \| A_{k,3}^\ast \|_{H_{ul}^{\theta_k}} + C \me^{2 \tilde \sigma t_3} \| A_{1,3}^\ast \|_{H_{ul}^\theta}^2 
			\leq C \me^{\tilde \sigma t_3} \max_{k=1,2} \| A_{k,3}^\ast \|_{H_{ul}^{\theta_k}} ,
		\end{equation}
		for $t_3 \in [0, \tilde T_{2,3}]$, where we used the bound on $\|A_{1,3}^\ast\|_{H_{ul}^\theta}$ from \eqref{eq:A23_ic_bound} in the second inequality. Using the bounds in \eqref{eq:A23_ic_bound} again leads to
		\begin{equation}
			\label{eq:A23_tilde_bound_T3}
			\| \tilde A_{2,3}(\cdot,T_3) \|_{H_{ul}^\theta} \leq 
			C \me^{-(\tilde \sigma_0 - \tilde \sigma) \rho_{out}^2 / 2 \zeta {r_3^\ast}^4} .
		\end{equation}
		Since \eqref{eq:A23_tilde_bound} is increasing on $t_3 \in [0,T_3]$ and \eqref{eq:A23_tilde_bound_T3} is bounded because $\tilde \sigma_0 > \tilde \sigma$, we can ensure that $\tilde T_{2,3} = T_3$ by choosing $\vartheta > 0$ sufficiently small. The bound in \eqref{eq:A23_bound} follows after multiplying through by the integrating factor.
	\end{proof}

	\subsubsection*{Step IV: Estimate $A_{j,3}(x_3,t_3)$, $j = 3, \ldots, N-2$}
	
	Similarly to the Step IV analysis in charts $\mathcal K_1$ and $\mathcal K_2$ we write $a_{j,3} =: a_{j,3}^L + a_{j,3}^N$, where
	\[
	a_{j,3}^L(\tilde A_{j,3}, \overline{\tilde A_{j,3}}) = 
	\left( \alpha_1 A_{1,3}^2 + \alpha_2 |A_{1,3}|^2 + \alpha_3 \overline{A_{1,3}}^2 \right) \tilde A_{j,3} + \left( \beta_1 A_{1,3}^2 + \beta_2 |A_{1,3}|^2 + \beta_3 \overline{A_{1,3}}^2 \right) \overline{\tilde A_{j,3}} ,
	\]
	and $a_{j,3}^N$ depends only modulation functions $A_{m'j',3}$ at lower orders with $\alpha(m') + j' < j$.
	
	\
	
	We prove the following result.
	
	\begin{lemma}
		\label{lem:K3_Aj}
		Let $\theta_A = \theta_A(j) = 3(j-1) + \theta$ where $j \in \{ 3, \ldots, N-2 \}$ and $\theta > 1/2$, assume
		\begin{equation}
			\label{eq:a_jN_cond_K3}
			\| a_{j,3}^N(\cdot,t_3) \|_{H_{ul}^\theta} \leq 
			C \max_{k = 1,\ldots,j-1} \| A_{k,3}(\cdot,t_3) \|_{H_{ul}^{\theta_{j,k}}}^3 ,
		\end{equation}
		and that
		\begin{equation}
			\label{eq:Aj3_ic_bound}
			\|A_{k,3}^\ast \|_{H_{ul}^{\theta_{k,j}}} \leq C \me^{- \tilde \sigma_0 \rho_{out}^2 / 2 \zeta {r_3^\ast}^4} ,
		\end{equation}
		for all $k = 1, \ldots, j$, where $\tilde \sigma_0 > \tilde \sigma > \sigma > 1$ is the same constant as in Lemmas \ref{lem:K3_A1}-\ref{lem:K3_A2} and $\theta_{k,j}$ is given by \eqref{eq:theta_k_2}. Then $A_{j,3}(\cdot,t_3) \in H_{ul}^\theta$ for all $t_3 \in [0,T_3]$ and
		\begin{equation}
			\label{eq:Aj3_bound}
			\| A_{j,3}(\cdot,T_3) \|_{H_{ul}^\theta}
			\leq C \me^{ -(\tilde \sigma_0 - \tilde \sigma) \rho_{out}^2 / 2 \zeta {r_3^\ast}^4} .
		\end{equation}
	\end{lemma}
	
	\begin{proof}	
		We proceed similarly to the proof of Lemmas \ref{lem:K1_Aj}, \ref{lem:K2_Aj}, and show that $A_{j,3}(\cdot,t_3) \in H_{ul}^\theta$ for all $j = 1, \ldots, N-2$ using an inductive argument over $k = 3, \ldots, j-1 \leq N-3$. 
		As an induction hypothesis we assume the existence of solutions $A_{k,3}(\cdot,t_3) \in H_{ul}^{\theta_{k,j}}$ for each $k = 3, \ldots, j-1 \leq N - 3$, for all $t_3 \in [0,T_3]$. As in the proofs of Lemmas \ref{lem:K1_Aj} and \ref{lem:K2_Aj}, we skip the base case for brevity since it is similar to the inductive step.
		
		\
		
		The equation for $A_{j,3}$ is given by
		\[
		\partial_{t_3} A_{j,3} =  4 \partial_{x_3}^2 A_{j,3} + \left( 1 - \frac{\eps_3(t_3)}{2} \right) A_{j,1} - a_{j,3} - 4i \partial_{x_3}^3 A_{j-1,3} - \partial_{x_3}^4 A_{j-2,3} ,
		\qquad A_{j,3}(x_3,0) = A_{j,3}^\ast(x_3) ,
		\]
		for all $x_3 \in \R$ and $t_3 \in [0,T_3]$. 
		Many of the estimates below are obtained using arguments similar to those used in Steps II-III, so we present fewer details. Note that $\partial_{x_3}^3 A_{j-1,3}$ and $\partial_{x_3}^4 A_{j-2,3}$ can be considered as source terms, since we are given solutions $A_{j-1,3}(\cdot,t_3) \in H_{ul}^{\theta_{j-1,j}}$ and $A_{j-2,3}(\cdot,t_3) \in H_{ul}^{\theta_{j-2,j}}$ for all $t_3 \in [0,T_3]$ via Lemma \ref{lem:K3_A1}, Lemma \ref{lem:K3_A2} or induction hypothesis depending on $j$.
		
		\
		
		Setting $A_{j,3}(x_3,t_3) =: \me^{I_3(t_3)} \tilde A_{j,3}(x_3,t_3)$ where $I_3(t_3)$ is given by \eqref{eq:integrating_factor_A11_K3} leads to
		\[
		\partial_{t_3} \tilde A_{j,3} = 4 \partial_{x_3}^2 \tilde A_{j,3} - \tilde A_{j,3} - a_{j,3}^L(\tilde A_{j,3}, \overline{\tilde A_{j,3}}) - \me^{-I_3(t_3)} \left( a^N_{j,3} + 4 i \partial_{x_3}^3 A_{j-1,3} + \partial_{x_3}^4 A_{j-2,3} \right) ,
		\]
		together with $\tilde A_{j,3}(x_3,0) = A_{j,3}^\ast(x_3)$. We apply the variation of constants formula and split the solution into two parts $\tilde A_{j,3} =: \tilde A_{j,3,h} + \tilde A_{j,3,s}$, where
		\begin{equation}
			\label{eq:Aj3_voc}
			\begin{split}
				\tilde A_{j,3,h}(\cdot,t_3) &= 
				\me^{t_3 \Lambda_3} A_{j,3}^\ast  - \int_0^{t_3} \me^{(t_3 - s_3) \Lambda_3} a_{j,3}^L(\tilde A_{j,3}(\cdot,s_3), \overline{\tilde A_{j,3}}(\cdot,s_3)) ds_3 , \\
				\tilde A_{j,3,s}(\cdot,t_3) &= 
				- \int_0^{t_3} \me^{(t_3 - s_3) \Lambda_3} \me^{-I_3(s_3)} \left( a^N_{j,3} (\cdot,s_3) + 4 i \partial_{x_3}^3 A_{j-1,3}(\cdot,s_3) + \partial_{x_3}^4 A_{j-2,3}(\cdot,s_3) \right) ds_3 ,
			\end{split}
		\end{equation}
		for $t_3 \in [0,T_3]$. Arguments analogous to those which led to the bound for $\|A_{2,3}(\cdot,t_3) \|_{H_{ul}^\theta}$ in Lemma \ref{lem:K3_A2} show that $A_{j,3,h}(\cdot,t_3) = \me^{I_3(t_3)} \tilde A_{j,3,h}(\cdot,t_3)$ has a similar bound. Specifically, we have 
		\begin{equation}
			\label{eq:Aj3h_bound_t}
			\begin{split}
				\| \tilde A_{j,3,h}(\cdot,t_3) \|_{H_{ul}^\theta} &\leq 
				C \me^{\sigma t_3} \| A_{j,3}^\ast  \|_{H_{ul}^\theta} + 
				C i \me^{\sigma t_3} \|A_{1,3}^\ast \|_{H_{ul}^{\theta}}^2 \Gamma \left( \frac{1}{2}, - (- \sigma + 2 \tilde \sigma) \left(\frac{1}{2 \zeta} + s_3 \right) \right) \bigg|_{t_3}^0 , \\
				& \leq C \me^{\sigma t_3} \| A_{j,3}^\ast  \|_{H_{ul}^\theta} + C \me^{2 \tilde \sigma t_3} \| A_{1,3}^\ast \|_{H_{ul}^\theta}^2 \\
				& \leq C \me^{\sigma t_3} \| A_{j,3}^\ast \|_{H_{ul}^\theta} + C \me^{\tilde \sigma t_3} \| A_{1,3}^\ast \|_{H_{ul}^\theta} \\
				& \leq C \me^{\tilde \sigma t_3} \max_{k=1,\ldots,j} \| A_{k,3}^\ast \|_{H_{ul}^{\theta}} ,
			\end{split}
		\end{equation}
		for all $t_3 \in [0, \tilde T_{j,3}]$, where we used the gamma function asymptotics \eqref{eq:Gamma_asymptotics} in Appendix \ref{app:technical_estimates} in the second inequality, and the bound for $A_{1,3}^\ast$ in \eqref{eq:Aj3_ic_bound} in the second to last inequality. Using the expression for $T_3$ in Lemma \ref{lem:K3_r3eps3} and the bounds \eqref{eq:Aj3_ic_bound}, we obtain
		\begin{equation}
			\label{eq:Aj3h_bound}
			\| A_{j,3,h}(\cdot,T_3) \|_{H_{ul}^\theta}
			\leq 
			C \me^{ -(\tilde \sigma_0 - \tilde \sigma) \rho_{out}^2 / 2 \zeta {r_3^\ast}^4} .
		\end{equation}
		
		\
		
		Now consider the source term $\tilde A_{j,3,s}$. Using the same bound for the semigroup as in Steps II-III we obtain the estimate
		\begin{equation}
			\label{eq:Aj3s_voc_bound}
			\begin{split}
				\| \tilde A_{j,3,s}(\cdot,t_3) \|_{H_{ul}^\theta} & \leq \\
				C \me^{\sigma t_3} & \int_0^{t_3} \me^{- \sigma s_3} \sqrt[4]{1 + 2\zeta s_3} \left\| a^N_{j,3} (\cdot,s_3)+ 4 i \partial_{x_3}^3 A_{j-1,3}(\cdot,s_3) + \partial_{x_3}^4 A_{j-2,3}(\cdot,s_3) \right\|_{H_{ul}^\theta} ds_3 ,
			\end{split}
		\end{equation}
		for all $t_3 \in [0, \tilde T_{j,3}]$. From this point on the proof resembles the part of the proof of Lemma \ref{lem:K1_Aj} devoted to the proving the bound \eqref{eq:Aj1_nu10_bound} when $\mu(x) \equiv 0$. We show that
		\begin{equation}
			\label{eq:Aj3s_bound}
			\| \tilde A_{j,3,s}(\cdot,t_3) \|_{H_{ul}^\theta} \leq 
			C \me^{ -(\tilde \sigma_0 - \tilde \sigma) \rho_{out}^2 / 2 \zeta {r_3^\ast}^4} ,
		\end{equation}
		by an inductive argument over $k = 3, \ldots , j$. The idea is to bound each $\| \tilde A_{k,3,s}(\cdot,t_3) \|_{H_{ul}^\theta}$ using bounds established for lower orders. 
		
		In order to prove the base case for $k=3$, we use the bounds
		\[
		\| A_{1,3}(\cdot,t_3) \|_{H_{ul}^{\theta_A(j)}} \leq
		C \frac{\me^{\tilde \sigma t_3}}{\sqrt[4]{1 + 2 \zeta t_3}} \|A_{1,3}^\ast \|_{H_{ul}^{\theta_A(j)}} , \qquad
		\| A_{2,3}(\cdot,t_3) \|_{H_{ul}^{\theta_{2,j}}} \leq \\
		C \frac{\me^{\tilde \sigma t_3}}{\sqrt[4]{1 + 2 \zeta t_3}} \max_{i = 1, 2} \| A_{k,3}^\ast \|_{H_{ul}^{\theta_{k,j}}} ,
		\]
		which apply for all $t_3 \in [0,T_3]$ due to the proofs of Lemmas \ref{lem:K3_A1} and \ref{lem:K3_A2} (c.f.~equations \eqref{eq:A13_t3_bound} and \eqref{eq:A23_t3_bound} respectively). 
		
		In order to obtain a bound on the integrand in \eqref{eq:Aj3s_voc_bound} we use \eqref{eq:a_jN_cond_K3}, which implies
		\[
		\begin{split}
			\| a^N_{3,3}(\cdot,s_3) \|_{H_{ul}^{\theta_{3,j}}} &\leq 
			C \max_{i=1,2} \| A_{i,3}(\cdot,t_3) \|_{H_{ul}^{\theta_{i,j}}}^3 \\
			&\leq C \frac{\me^{3 \tilde \sigma t_3}}{(1 + 2 \zeta t_3)^{3/4}} \max_{i=1,2} \| A_{i,3}^\ast \|_{H_{ul}^{\theta_{i,j}}}^3 \\ 
			&\leq C \frac{\me^{\tilde \sigma t_3}}{\sqrt[4]{1 + 2 \zeta t_3}} \max_{i = 1, 2} \| A_{i,3}^\ast \|_{H_{ul}^{\theta_{i,j}}} ,
		\end{split}
		\]
		where we used Lemma \ref{lem:K3_A1} in the second inequality and \eqref{eq:Aj3_ic_bound} together with $(1 + 2 \zeta t_3)^{-3/4} \leq (1 + 2 \zeta t_3)^{-1/4}$ in the last inequality. The derivative terms can be bounded using the bounds for $A_{1,3}$ and $A_{2,3}$ above. Specifically,
		\[
		\begin{split}
			\| \partial_{x_3}^3 A_{2,3}(\cdot,s_3) \|_{H_{ul}^{\theta_{3,j}}} & \leq 
			C \| A_{2,3}(\cdot,s_3) \|_{H_{ul}^{\theta_{2,j}}} \leq 
			C \frac{\me^{\tilde \sigma t_3}}{\sqrt[4]{1 + 2 \zeta t_3}} \max_{i = 1, 2} \| A_{i,3}^\ast \|_{H_{ul}^{\theta_{i,j}}} , \\
			\| \partial_{x_3}^4 A_{1,3}(\cdot,s_3) \|_{H_{ul}^{\theta_{3,j}}} & \leq 
			C \| A_{1,3}(\cdot,s_3) \|_{H_{ul}^{\theta_A(j)}} \leq 
			C \frac{\me^{\tilde \sigma t_3}}{\sqrt[4]{1 + 2 \zeta t_3}} \|A_{1,3}^\ast \|_{H_{ul}^{\theta_A(j)}} ,
		\end{split}
		\]
		for $s_3 \in [0,T_3]$. It follows that
		\[
			\left\| a^N_{3,3}(\cdot,s_3) + 4 i \partial_{x_3}^3 A_{2,3}(\cdot,s_3) + \partial_{x_3}^4 A_{1,3}(\cdot,s_3) \right\|_{H_{ul}^{\theta_{3,j}}} \leq 
			C \frac{\me^{\tilde \sigma t_3}}{\sqrt[4]{1 + 2 \zeta t_3}} \max_{i = 1, 2} \| A_{i,3}^\ast \|_{H_{ul}^{\theta_{i,j}}} ,
		\]
		for all $s_3 \in [0,\tilde T_{3,3}]$. Using this to bound \eqref{eq:Aj3s_voc_bound} with $j=3$, we obtain
		\[
		\| \tilde A_{3,3,s}(\cdot,t_3) \|_{H_{ul}^{\theta_{3,j}}} \leq 
		C \me^{\sigma t_3} \max_{i = 1, 2} \| A_{i,3}^\ast \|_{H_{ul}^{\theta_{i,j}}} \int_0^{t_3} \me^{(\tilde \sigma - \sigma) s_3} ds_3 \leq
		C \me^{\tilde \sigma t_3} \max_{i = 1, 2} \| A_{i,3}^\ast \|_{H_{ul}^{\theta_{i,j}}} ,
		\]
		for $t_3 \in [0, \tilde T_{3,3}]$. 
		
		Multiplying through by the integrating factor and using the bounds on the initial conditions in \eqref{eq:Aj3_ic_bound} we obtain
		\begin{equation}
			\label{eq:A33s_bound}
			\| A_{3,3,s}(\cdot,T_3) \|_{H_{ul}^{\theta_{3,j}}} \leq 
			C \me^{ -(\tilde \sigma_0 - \tilde \sigma) \rho_{out}^2 / 2 \zeta {r_3^\ast}^4} ,
		\end{equation}
		which is valid for each $\tilde \sigma_0 > \tilde \sigma$ if $\vartheta > 0$ is sufficiently small. The desired bound for $\| A_{3,3}(\cdot,t_3) \|_{H_{ul}^{\theta_{3,j}}}$ is obtained using \eqref{eq:Aj3h_bound}, \eqref{eq:A33s_bound} and the triangle inequality. Note that $\tilde T_{3,3} = T_3$ as required since the right-hand sides in \eqref{eq:Aj3h_bound} and \eqref{eq:A33s_bound} are bounded below $\tilde K$.
		
		\
		
		We turn now to the induction step, and assume that \eqref{eq:Aj3s_bound} holds for $k=3,\ldots, j-1$. Similarly to the proof of Lemma \ref{lem:K1_Aj}, we present fewer details since the structure of the argument is the same as the base case. Using \eqref{eq:a_jN_cond_K3}, the induction hypothesis and the bounds in \eqref{eq:Aj3_ic_bound} we obtain 
		\[
		\| a^N_{j,3}(\cdot,s_3) \|_{H_{ul}^{\theta}} \leq 
		C \frac{\me^{3 \tilde \sigma s_1}}{(1 + 2 \zeta s_3)^{3/4}} \max_{i = 1,\ldots,j-1} \| A_{i,3}^\ast \|_{H_{ul}^{\theta_{i,j}}}^3 \leq
		C \frac{\me^{\tilde \sigma s_3}}{\sqrt[4]{1 + 2 \zeta s_3}} \max_{i = 1,\ldots,j-1} \| A_{i,3}^\ast \|_{H_{ul}^{\theta_{i,j}}} ,
		\]
		for $s_3 \in [0,T_3]$. The derivative terms are controlled because we assume sufficient regularity, so that the induction hypothesis implies
		\[
		\begin{split}
			\| \partial_{x_3}^3 A_{j-1,3}(\cdot,s_3) \|_{H_{ul}^{\theta}} &\leq 
			C \| A_{j-1,3}(\cdot,s_3) \|_{H_{ul}^{\theta_{j-1,j}}} 
			\leq C \frac{\me^{\tilde \sigma s_3}}{\sqrt[4]{1 + 2 \zeta s_3}} \max_{i = 1, \ldots, j-1} \| A_{i,3}^\ast \|_{H_{ul}^{\theta_{i,j}}} , \\
			\| \partial_{x_3}^4 A_{j-2,3}(\cdot,s_3) \|_{H_{ul}^{\theta}} &\leq 
			C \| A_{j-2,3}(\cdot,s_3) \|_{H_{ul}^{\theta_{j-2,j}}} 
			\leq C \frac{\me^{\tilde \sigma s_3}}{\sqrt[4]{1 + 2 \zeta s_3}} \max_{i = 1, \ldots, j-2} \| A_{i,3}^\ast \|_{H_{ul}^{\theta_{i,j}}} ,
		\end{split}
		\]
		for $s_3 \in [0,T_3]$. From here the bound for $\| A_{j,3}(\cdot,t_3) \|_{H_{ul}^{\theta}}$ on $t_3 \in [0,T_3]$ follows analogously to the base case $k=3$. 
	\end{proof}

	\subsubsection{Step V: Estimate $\psi_3$}
	
	We need the following result before we can describe the dynamics of $\psi_3$. As in the chart $\mathcal K_2$ and $\mathcal K_3$ analyses, we reinstate the full notation $A_{\pm1j,3}$ instead of $A_{j,3}$ in order to distinguish modulation functions with $|m|=1$ and $|m|\neq1$, \SJ{and recall that $\textbf A_{mj,3}$ denotes the set of lower order critical modulation functions $A_{\pm1j',3}$ with $j' \in \{1,\ldots,\alpha(m) + j - 1\}$.}
	
	\begin{lemma}
		\label{lem:K3_gl_coefficients}
		Let $\theta_A = 3(N-3) + \theta$ where $\theta > 1/2$ and assume that initial conditions $A_{\pm1j,3}^\ast$ satisfy \eqref{eq:Aj3_ic_bound}. Then the conditions on $a_{j,3}^N$ for the applicability of Lemma \ref{lem:K3_Aj} are satisfied, and $A_{\pm1j,3}(\cdot,t_3) \in H_{ul}^{\theta_{j,N-2}}$ for all $j = 1, \ldots, N-2$ and $t_3 \in [0,T_3]$. Moreover,
		\begin{equation}
			\label{eq:gl_coefficients_smooth_K3}
			A_{mj,3}(\cdot, t_3) = g_{mj,3}^{gl} \left( \textbf{\textup{A}}_{mj,3}(\cdot,t_3) , \eps_3(t_3) \right) \in H_{ul}^{\theta_{\alpha(m)+j,N}} ,
		\end{equation}
		and
		\begin{equation}
			\label{eq:gl_coefficients_bounds_K3}
			\big\| g_{mj,3}^{gl} \left( \textbf{\textup{A}}_{mj,3}(\cdot,T_3) , \eps_3(T_3) \right) \big\|_{H_{ul}^{\theta_{\alpha(m)+j,N}}}
			\leq C \me^{ -(\tilde \sigma_0 - \tilde \sigma) \rho_{out}^2 / 2 \zeta {r_3^\ast}^4} ,
		\end{equation}
		for all $t_2 \in [0,T_2]$, $m \in I_N \setminus \{-1,1\}$ and $j = 1, \ldots, \tilde \alpha(m)$.
	\end{lemma}
	
	\begin{proof}
		The proof proceeds similarly to the (second part of the) proof of Lemma \ref{lem:K1_gl_coefficients}, so we omit the details for brevity. It is necessary to assume the initial condition bounds \eqref{eq:Aj3_ic_bound} in order to apply Lemmas \ref{lem:K3_A1}, \ref{lem:K3_A2} and \ref{lem:K3_Aj}. 
	\end{proof}
	
	Combining the results of Lemmas \ref{lem:K3_r3eps3}, \ref{lem:K3_A1}, \ref{lem:K3_A2}, \ref{lem:K3_Aj} and \ref{lem:K3_gl_coefficients} we obtain the following result, which summarises the dynamics in chart $\mathcal K_3$.
	
	\begin{proposition}
		\label{prop:K3_summary}
		Let $\theta_A = 1 + 3(N-3) + \theta$ where $\theta > 1/2$ and assume that initial conditions are small according to \eqref{eq:Aj3_ic_bound}. Then for $\zeta, K > 0$ sufficiently small but fixed we have $\psi_3(\cdot,t_3) \in H_{ul}^\theta$ for all $t_3 \in [0,T_3]$. In particular, the map $\pi_3 : \Sigma_3^{in} \to \Sigma_3^{out}$ is well-defined and given by 
		\begin{equation}
			\label{eq:pi_3}
			\pi_3 : \left( \psi_3^\ast, r_3^\ast, \zeta \right) \mapsto 
			\left( \psi_3(\cdot,T_3) , \rho_{out}^{1/2}, \zeta \frac{{r_3^\ast}^4}{\rho_{out}^2} \right) ,
		\end{equation}
		where $\psi_3(x,T_3)$ satisfies
		\begin{equation}
			\label{eq:psi_3_bound}
			\| \psi_3(\cdot,T_3) \|_{H_{ul}^\theta} \leq 
			C \me^{ -(\tilde \sigma_0 - \tilde \sigma) \rho_{out}^2 / 2 \zeta {r_3^\ast}^4} ,
		\end{equation}
		with $\tilde \sigma_0 > \tilde \sigma > \sigma > 1$.
	\end{proposition}
	
	\begin{proof}
		The conditions of Proposition \ref{prop:K3_summary} imply that Lemmas \ref{lem:K3_r3eps3}, \ref{lem:K3_A1}, \ref{lem:K3_A2}, \ref{lem:K3_Aj} and \ref{lem:K3_gl_coefficients} apply. It follows that all modulation functions $A_{mj,3}(\cdot,t_3) \in H_{ul}^{\theta+1}$ for all $t_3 \in [0,T_3]$, implying that $\psi_3(\cdot,t_3) \in H_{ul}^\theta$ for all $t_3 \in [0,T_3]$ (recall the definition of $\psi_3$ in \eqref{eq:psi_definition}).
		
		The form of the map in \eqref{eq:pi_3} follows from Lemma \ref{lem:K3_r3eps3}. To verify the bound in \eqref{eq:psi_3_bound}, notice that Lemmas \ref{lem:K3_A1}, \ref{lem:K3_A2}, \ref{lem:K3_Aj} and \ref{lem:K3_gl_coefficients} imply that for each $m \in I_N$ we have
		\[
		\| A_{mj,3}(\cdot, T_3) \|_{H_{ul}^{\theta+1}} \leq 
		C \me^{ -(\tilde \sigma_0 - \tilde \sigma) \rho_{out}^2 / 2 \zeta {r_3^\ast}^4} ,
		\]
		for all $j = 1, \ldots, \tilde \alpha(m)$. Thus
		\[
		\begin{split}
			\| \psi_3(\cdot,T_3) \|_{H_{ul}^\theta} &\leq 
			\frac{1}{r_3(T_3)} \sum_{m \in I_N} \sum_{j=1}^{\tilde \alpha(m)} r_3(T_3)^{\alpha(m) + j} \| A_{mj,3}(r_3(T_s) \cdot,T_3) \me^{imx} \|_{H_{ul}^{\theta}} \\
			&\leq C \max_{m \in I_N, j \in \{1, \ldots, \tilde \alpha(m)\}} \| A_{mj,3}(\cdot, T_3) \|_{H_{ul}^{\theta+1}} \\
			&\leq C \me^{ -(\tilde \sigma_0 - \tilde \sigma) \rho_{out}^2 / 2 \zeta {r_3^\ast}^4} ,
		\end{split}
		\]
		as required. The bound \eqref{eq:psi_3_bound} implies that $\pi_3$ is well-defined because $\| \psi_3(\cdot,T_3) \|_{H_{ul}^\theta} \leq K$ for all $r_3^\ast \in [0,\zeta^{-1/4}\vartheta]$ with $\vartheta > 0$ sufficiently small. 
	\end{proof}

	\subsection{Proof of Lemma \ref{lem:Ansatz_dynamics_mu_0}}
	\label{sub:proof_of_approximation_lemma_2}
	
	We are now in a position to prove Lemma \ref{lem:Ansatz_dynamics_mu_0} by combining results obtained for the case $\mu(x) \equiv 0$ in each chart $\mathcal K_l$, $l = 1,2,3$. The fact that $\Psi(\cdot,t) = r(\bar t) \psi(\cdot, \bar t) \in H_{ul}^\theta$ for all $t \in [0,T]$ follows directly from the fact that $\psi_l(\cdot, t_l) \in H_{ul}^\theta$ for all $t_l \in [0,T_l]$ in each chart $\mathcal K_l$; recall Propositions \ref{prop:K1_summary}, \ref{prop:K2_summary} and \ref{prop:K3_summary}. It therefore remains to prove the bound \eqref{eq:Psi_bound}.
	
	\
	
	Similarly to the proof of Lemma \ref{lem:Ansatz_dynamics} in Section \ref{sub:proof_of_approximation_lemma_1}, we consider an initial condition $(\Psi^\ast , -\rho_{in}) \in \Delta^{in}$, which can be rewritten as initial condition $(\psi_1^\ast, \sqrt{\rho_{in}} , \eps_1^\ast) \in \Sigma_1^{in}$ in chart $\mathcal K_1$ using \eqref{eq:K2_ic_bounds}. Since $\mu(x) \equiv 0$, Proposition \ref{prop:K1_summary} implies that solutions are exponentially small in $H_{ul}^\theta$ at $\Sigma_2^{in} = \kappa_{12}(\Sigma_1^{out})$. This time, we obtain the following bounds for the initial modulation functions $A_{mj,2}^\ast  = \zeta^{-(\alpha(m) + j)/4} A_{mj,1}^\ast(\cdot, T_1)$ in chart $\mathcal K_2$ using the assertions of Lemmas \ref{lem:K1_A1}, \ref{lem:K1_A2}, \ref{lem:K1_Aj} and \ref{lem:K1_gl_coefficients} relevant to the case $\mu(x) \equiv 0$:
	\begin{equation}
		\label{eq:K2_ic_bounds_2}
		\begin{split}
			& \| A_{1,2}^\ast \|_{H_{ul}^{\theta_A}} = \zeta^{-1/4} \| A_{1,1}(\cdot, T_1) \|_{H_{ul}^{\theta_A}} \leq C \me^{-\kappa / 2\eps_1^\ast} \| A_{1,1}^\ast  \|_{H_{ul}^{\theta_A}} , \\
			& \| A_{2,2}^\ast \|_{H_{ul}^{\theta_{2,N-2}}} = \zeta^{-1/2} \| A_{2,1}(\cdot, T_1) \|_{H_{ul}^{\theta_{2,N-2}}} \leq C \me^{- \kappa / 2 \eps_1^\ast} \max_{k=1,2} \| A_{k,1}^\ast \|_{H_{ul}^{\theta_{k,N-2}}} , \\
			& \| A_{j,2}^\ast \|_{H_{ul}^{\theta_{j,N-2}}} = \zeta^{-j/4} \| A_{j,1}(\cdot, T_1) \|_{H_{ul}^{\theta_{j,N-2}}} \leq C \me^{-\kappa / 2 \eps_1^\ast} \max_{k=1,\ldots,j} \| A_{k,1}^\ast \|_{H_{ul}^{\theta_{k,N-2}}} , \\ 
			& \| A_{mj,2}^\ast \|_{H_{ul}^{\theta_{\alpha(m)+j,N}}} = \zeta^{-(\alpha(m) + j)/4} \| A_{mj,1}(\cdot, T_1) \|_{H_{ul}^{\theta_{\alpha(m)+j,N}}} \leq C \me^{- \kappa / 2 \eps_1^\ast} \max_{k=1,\ldots,j-2} \| A_{k,1}^\ast \|_{H_{ul}^{\theta_{k,N-2}}} , 
		\end{split}
	\end{equation}
	where $j \geq  3$ in the second-to-last line and $|m| \neq 1$ in the last. 
	The evolution from $\Sigma_2^{in}$ up to $\Sigma_2^{out}$ is understood in chart $\mathcal K_2$ and summarised in Proposition \ref{prop:K2_summary}. Bounds for the initial modulation functions $A_{mj,3}^\ast  = \zeta^{(\alpha(m) + j)/4} A_{mj,2}^\ast(\cdot, T_2)$ in chart $\mathcal K_3$ are obtained via Lemmas \ref{lem:K2_A1}, \ref{lem:K2_A2}, \ref{lem:K2_Aj} and \ref{lem:K2_gl_coefficients}. These bounds can be formulated in terms of initial data in $\mathcal K_1$ using \eqref{eq:K2_ic_bounds_2}. We obtain
	\begin{equation}
		\label{eq:K3_ic_bounds}
		\begin{split}
			&\| A_{1,3}^\ast \|_{H_{ul}^{\theta_A}} = \zeta^{1/4} \| A_{1,2}(\cdot, T_2) \|_{H_{ul}^{\theta_A}} \leq C \me^{- \kappa / 2\eps_1^\ast} \| A_{1,1}^\ast \|_{H_{ul}^{\theta_A}} , \\
			&\| A_{2,3}^\ast \|_{H_{ul}^{\theta_{2,N-2}}} = \zeta^{1/2} \| A_{2,2}(\cdot, T_2) \|_{H_{ul}^{\theta_{2,N-2}}} \leq C \me^{-\kappa / 2 \eps_1^\ast} \max_{k=1,2} \| A_{k,1}^\ast \|_{H_{ul}^{\theta_{k,N-2}}} , \\
			&\| A_{j,3}^\ast \|_{H_{ul}^{\theta_{j,N-2}}} = \zeta^{j/4} \| A_{j,2}(\cdot, T_2) \|_{H_{ul}^{\theta_{j,N-2}}} \leq C \me^{-\kappa / 2\eps_1^\ast} \max_{k=1,\ldots,j} \| A_{k,1}^\ast \|_{H_{ul}^{\theta_{k,N-2}}} , \\ 
			&\| A_{mj,3}^\ast \|_{H_{ul}^{\theta_{\alpha(m)+j,N}}} = \zeta^{(\alpha(m) + j)/4} \| A_{mj,2}(\cdot, T_2) \|_{H_{ul}^{\theta_{\alpha(m)+j,N}}} \leq C \me^{- \kappa / 2\eps_1^\ast} \max_{k=1,\ldots,j-2} \| A_{k,1}^\ast \|_{H_{ul}^{\theta_{k,N-2}}} , 
		\end{split}
	\end{equation}
	where again, $j \geq  3$ in the second-to-last line and $m \neq 1$ in the last. The idea from here is to apply Proposition \ref{prop:K3_summary}. In order to do so, we need to check the bounds \eqref{eq:Aj3_ic_bound}.
	
	Since the argument is the same for each $A_{j,1}^\ast$, we present the details for $j=1$ only. A sufficient condition for \eqref{eq:Aj3_ic_bound} to be satisfied is
	\begin{equation}
		\label{eq:A13_ic_bound_2}
		\| A_{1,3}^\ast  \|_{H_{ul}^\theta} \leq
		C \me^{-\kappa / 2\eps_1^\ast} \| A_{1,1}^\ast \|_{H_{ul}^\theta} \leq
		C \me^{- \tilde \sigma_0 \rho_{out}^2 / 2 \zeta {r_3^\ast}^4} ,
	\end{equation}
	which is satisfied for all $\eps_1^\ast$ and $r_3^\ast$ sufficiently small if
	\begin{equation}
		\label{eq:exponent_inequality}
		\frac{\kappa}{2 \eps_1^\ast} \geq \frac{\tilde \sigma_0 \rho_{out}^2}{2 \zeta {r_3^\ast}^4} .
	\end{equation}
	It follows from the form of the maps $\pi_1$ and $\pi_2$ in Propositions \ref{prop:K1_summary} and \ref{prop:K2_summary} respectively that
	\begin{equation}
		\label{eq:r_3_ast}
		r_3^\ast = \zeta^{-1/4} r_2(T_2) = \zeta^{-1/4} r_2^\ast = r_1(T_1) = \rho_{in}^{1/2} \sqrt[4]{\frac{\eps_1^\ast}{\zeta}} .
	\end{equation}
	We also recall that for sufficiently small (but fixed) $K > 0$, the constants $\kappa \in (0,1)$ and $\tilde \sigma_0 > 1$ can be chosen arbitrarily close to $1$. It follows that \eqref{eq:exponent_inequality} (and therefore \eqref{eq:A13_ic_bound_2}) is satisfied if
	\[
	\frac{\rho_{out}}{\rho_{in}} \leq \sqrt{\frac{\kappa}{\tilde \sigma_0}} =: 1 - \omega ,
	\]
	where $\omega \in (0,1)$ can be chosen arbitrarily close to $1$ if $K > 0$ is fixed but sufficiently small, which is precisely the assumption in the statement of Lemma \ref{lem:Ansatz_dynamics_mu_0}.
	
	The preceding arguments show that Proposition \ref{prop:K3_summary} applies. Setting $\tilde \sigma_0 = \kappa (\rho_{in} /\rho_{out})^2$ in order to maximise the contraction, we obtain the following via \eqref{eq:psi_3_bound}, \eqref{eq:r_3_ast} and the blow-down transformation $\eps_1^\ast = \eps / \rho_{in}^2$:
	\[
	\| \psi_3(\cdot,T_3) \|_{H_{ul}^\theta} \leq C \exp\left( -\frac{\kappa}{2\eps} \left( \rho_{in}^2 - \frac{\tilde \sigma}{\kappa} \rho_{out}^2 \right) \right) ,
	\]
	where $\tilde \sigma / \kappa > 1$ can be chosen arbitrarily close to $1$ if $K > 0$ is sufficiently small. Setting $\kappa_- := \kappa$ and $\kappa_+ := \tilde \sigma / \kappa$ and applying the blow-down transformation
	\[
	\psi_3(\cdot, T_3) = \frac{\Psi(\cdot, T)}{\sqrt{\rho_{out}}}
	\]
	yields the desired result.
	\qed

	\section{Proof of Theorem \ref{thm:Error}}
	\label{sec:Proof_of_thm_error}

	In this section we prove Theorem \ref{thm:Error}. 
	In order to do so, we need to bound the norm of the error $\| u(\cdot,t) - \Psi(\cdot,t) \|_{H_{ul}^\theta}$ at particular times $t = T_{mid}$ and $t = T$. Similarly to the established approaches for the static SH problem \eqref{eq:sh} in e.g.~\cite[Ch.~10]{Schneider2017}, we define a weighted error function $R$ via
	\begin{equation}
		\label{eq:error}
		r(\bar t)^\beta R(x, \bar t) := u(x,t) - r(\bar t) \psi(x, \bar t) ,
	\end{equation}
	where the exponent $\beta > 0$ is expected to depend upon the order $n$ of the approximation $\Psi$, but left unspecified for now. The first step is to obtain an evolution equation for $R$ in terms of $\psi$ and $\textrm{Res}(r\psi)$. This equation is expected to differ in the dynamic setting since the variable $r$, which plays the role of the small parameter, depends on time (recall the discussion following the statement of Theorem \ref{thm:Error}).
	
	\
	
	Differentiating $u$ with respect to $t$ and applying the definition \eqref{eq:error} leads to
	\[
	\begin{split}
		\partial_t u &= \partial_t (r^\beta R) + \partial_t (r \psi) \\
		&= \partial_t (r^\beta R) - (1 + \partial_x^2)^2 r \psi + r^3 \bar v \psi - r^2 \psi^3 + r^4 \eps \mu(x) - \textrm{Res} (r \psi) \\
		& = - (1 + \partial_x^2)^2 (r^\beta R + r \psi) + r^2 \bar v (r^\beta R + r \psi) - (r^\beta R + r \psi)^3 + r^4 \bar \eps \mu(x) .
	\end{split}
	\]
	Expanding $\partial_t (r^\beta R) = r^\beta \partial_t R + \beta r^{\beta - 1} R \partial_t r$, changing to the desingularized time $\bar t$ and rearranging a little, we obtain the following equation for the (weighted) error function $R$:
	\begin{equation}
		\label{eq:error_eqn}
		\partial_{\bar t} R = - r^{-2} (1 + \partial_x^2)^2 R + \left( \bar v - \beta r^{-1} \partial_{\bar t} r - 3 \psi^2 \right) R - 3 r^{\beta - 1} R^2 \psi - r^{2 \beta - 2} R^3 + r^{- \beta - 2} \textrm{Res} (r \psi) .
	\end{equation}
	Some of the time-dependence in the linear part can be mitigated via the use of an integrating factor. Defining a new function $\tilde R$ via $R(x,\bar t) =: \me^{I(\bar t)} \tilde R(x,\bar t)$, where
	\[
	I(\bar t) = - \beta \int_0^{\bar t} \frac{\partial_{\bar s} r (\bar s)}{r(\bar s)} d\bar s ,
	\]
	leads to
	\[
	\partial_{\bar t} \tilde R = - r^{-2} (1 + \partial_x^2)^2 \tilde R + \bar v \tilde R - 3 \psi^2 \tilde R - 3 r^{\beta - 1} \me^{I} \tilde R^2 \psi - r^{2 \beta - 2} \me^{2I} \tilde R^3 + r^{- \beta - 2} \me^{-I} \textrm{Res} (r \psi) .
	\]
	Using the variation of constants formula we obtain
	\begin{equation}
		\label{eq:R}
		\tilde R(\cdot, \bar t) = \me^{\bar t \Lambda} R^\ast  + \int_{0}^{\bar t} \me^{(\bar t - \bar s) \Lambda} \left( F(\cdot, \bar s) + r(\bar s)^{- \beta - 2} \me^{-I(\bar s)} \textrm{Res} (r(\bar s) \psi(\cdot, \bar s)) \right) d \bar s ,
	\end{equation}
	where $(\me^{\bar t \Lambda})_{\bar t \geq 0}$ is the evolution family generated by $\Lambda := - r(\bar t)^{-2} (1 + \partial_x^2)^2 + \bar v(\bar t)$, the initial condition satisfies $R^\ast := \tilde R(\cdot,0) = \tilde R(\cdot,0)$, and
	\begin{equation}
		\label{eq:nonlinear_terms}
		F(\cdot, \bar s) := - 3 \psi(\cdot, \bar s)^2 \tilde R(\cdot, \bar s) - 3 r(\bar s)^{\beta - 1} \me^{I(\bar s)} \tilde R(\cdot, \bar s)^2 \psi(\cdot, \bar s) - r(\bar s)^{2 \beta - 2} \me^{2 I(\bar s)} \tilde R(\cdot, \bar s)^3 .
	\end{equation}
	
	In order to estimate the norm of \eqref{eq:R} (and thus also the norm of $R$) in $H_{ul}^\theta$, we need to estimate the following:
	\begin{itemize}
		\item The evolution family $(\me^{\bar t \Lambda})_{\bar t \geq 0}$;
		\item The residual $\textrm{Res} (r(\bar s) \psi(\cdot, \bar s))$;
		\item The `nonlinear' terms in $F(\cdot, \bar s)$.
	\end{itemize}
	We derive estimates for each of these terms below. Since the equation for the error in \eqref{eq:error_eqn} is posed in the blown-up space, we will often work in coordinate charts $\mathcal K_l$.

	\subsection{Evolution family estimates}
	\label{sub:evolution_family_estimates}
	
	We prove the following result.
	
	\begin{lemma}
		\label{lem:error_evolution_family}
		The following estimates for the evolution family are obtained in coordinate charts $\mathcal K_l$, and hold for every $\theta \geq 0$:
		\begin{enumerate}
			\item[(i)] Chart $\mathcal K_1$: For each $c_1 \in (0,1)$ there exists a constant $C > 0$ depending on $\theta$ such that
			\[
			\| \me^{t_1 \Lambda_1} \|_{H_{ul}^\theta \to H_{ul}^\theta} \leq C \me^{-c_1 t_1}
			\]
			for all $t_1 \in [0,T_1]$, where $\Lambda_1 = - r_1(t_1)^{-2} (1 + \partial_x^2)^2 - 1$.
			\item[(ii)] Chart $\mathcal K_2$: There exists a constant $C > 0$ such that
			\[
			\| \me^{t_2 \Lambda_2} \|_{H_{ul}^\theta \to H_{ul}^\theta} \leq C
			\]
			for all $t_2 \in [0,T_2]$, where $\Lambda_2 = - r_2^{-2} (1 + \partial_x^2)^2 + v_2(t_2)$.
			\item[(iii)] Chart $\mathcal K_3$: For each $c_3 > 1$ there exists a constant $C > 0$ depending on $\theta$ such that
			\[
			\| \me^{t_3 \Lambda_3} \|_{H_{ul}^\theta \to H_{ul}^\theta} \leq C \me^{c_3 t_3}
			\]
			for all $t_3 \in [0,T_3]$, where $\Lambda_3 = - r_3(t_3)^{-2} (1 + \partial_x^2)^2 + 1$.
		\end{enumerate}
	\end{lemma}
	
	\begin{proof}
		We begin in global coordinates. The evolution family $(\me^{\bar t \Lambda})_{\bar t \geq 0}$ is induced by the linear equation
		\[
		\partial_{\bar t} \tilde{w}(x, \bar t) = -\frac{(1 + \partial_{x}^2)^2}{r_1(\bar t)^2}\tilde{w}(x, \bar t) + \bar v(\bar s) \tilde{w}(x, \bar t) , \qquad \tilde{w}(x,0) = w^\ast(x), \qquad (\bar t \geq 0, \ x \in \R ) .
		\]
		Fourier transforming in space leads to a family of linear non-autonomous ODEs
		\[
		\partial_{\bar t} \hat{\tilde{w}}(\xi, \bar t) = -\frac{(1 - \xi^2)^2}{r(\bar t)^2} \hat{\tilde{w}}(\xi, \bar t) + \bar v(\bar s) \hat{\tilde{w}}(\xi, \bar t) , \qquad
		\hat{\tilde{w}}(\xi,0) = \hat{w}^\ast(\xi), \qquad 
		(\bar t \geq 0, \ \xi \in \R ) ,
		\]
		with solutions
		\[
		\hat{\tilde{w}}(\xi, \bar t) = \exp\left( \int_{0}^{\bar t} \left( - \frac{(1 - \xi^2)^2}{r(\bar s)^2} + \bar v(\bar s) \right) d\bar s \right) \hat{w}^\ast(\xi) .
		\]
		Taking the inverse Fourier transform leads to the solution formula
		\[
		\tilde{w}(\cdot, \bar t) = \mathscr{F}^{-1} \exp \left( \int_{0}^{\bar t} \left(- \frac{(1-\xi^2)^2}{r(\bar s)^2}  + \bar v(\bar s) \right) d\bar s \right) \mathscr{F} w^\ast  =: P( \bar t) w^\ast .
		\]
		Fourier multiplier techniques allow us to study the mapping properties of $\tilde{P}(\bar t) \colon H^\theta_{ul} \to H^\theta_{ul}, \ w^\ast \mapsto P(\bar t)w^\ast$ via its symbol
		\[
		\tilde{p}(\xi, \bar t) := \exp \left( \int_{0}^{\bar t} \left(- \frac{(1-\xi^2)^2}{r(\bar s)^2}  + \bar v(\bar s) \right) d\bar s \right).
		\]
		In particular, Lemma \ref{lem:US_8.3.7} shows that it suffices to derive estimates for the $C^2_b$-norm of $\tilde{p}(\cdot,\bar t)$. These can be estimated directly using the above and
		\[
		\begin{split}
			\partial_\xi \tilde p(\xi, \bar t) &= 4 \xi (1 - \xi^2) \left( \int_0^{\bar t} \frac{1}{r(\bar s)^2} d\bar s \right) \tilde p(\xi, \bar t) , \\
			\partial_\xi^2 \tilde p(\xi, \bar t) &= \left( 4 - 12 \xi^2 + 16 \xi^2 (1 - \xi^2)^2 \right) \left( \int_0^{\bar t} \frac{1}{r(\bar s)^2} d\bar s \right)^2 \tilde p(\xi, \bar t) .
		\end{split}
		\]
		In order to estimate $\tilde p(\xi,\bar t)$ and the integral terms, we work in charts $\mathcal K_l$.
		
		\
		
		In $\mathcal K_1$ we have that $r(\bar t) = r_1(t_1) = \rho_{in}^{1/2} (1 - 2\eps_1^\ast t_1)^{1/4}$ and $\bar v(\bar t) = -1$ for all $t_1 \in [0,T_1]$; recall Lemma \ref{lem:K1_r1eps1}. There are algebraic growth rates in time due to the fact that
		\[
		\int_0^{t_1} \frac{1}{r_1(s_1)^2} ds_1 = \rho_{in}^{-1} {\eps_1^\ast}^{-1} (1 - \sqrt{1 - 2 \eps_1^\ast t_1}) ,
		\]
		but these are compensated by the exponential decay in the expression for $\tilde p(\xi,\bar t)$, due to the fact that $\bar v = -1$ in chart $\mathcal K_1$. In particular, letting $\tilde p_1$ denote the representation of $\tilde p$ in $\mathcal K_1$, we find that for each $c_1 \in (0,1)$ there is a constant $C > 0$ depending on $\theta$ such that
		\[
		\| \tilde p_1(\cdot, t_1) \|_{C_b^2} \leq C \me^{- c_1 t_1} ,
		\]
		for all $t_1 \in [0,T_1]$. Lemma \ref{lem:US_8.3.7} implies the corresponding result in $H_{ul}^\theta$, thereby proving Assertion (i).
		
		\
		
		In $\mathcal K_2$ we have $r(\bar t) = r_2$ and $\bar v(\bar t) = v_2(t_2) = - \zeta^{-1/2} + t_2$ for all $t_2 \in [0,T_2]$; recall Lemma \ref{lem:K2_T2}. A direct calculation shows that
		\[
		\| \tilde p_2(\cdot, t_2) \|_{C_b^2} \leq C \frac{t_2^2}{r_2^4} \me^{- (1 - \xi^2)^2 t_2 / r_2^2} 
		\]
		for all $t_2 \in [0,T_2]$. Thus there is exponential decay for all wavenumbers $\xi$ bounded away from the critical modes $\xi \pm 1$. It turns out that the unwanted algebraic growth close to $\xi \pm 1$ can be mitigated using \textit{mode filters}. Due to the similarity with pre-existing arguments in e.g.~\cite{Mielke1995,Schneider1994,Schneider1994b,Schneider2017}, we omit the details, noting simply that close for $\xi = \pm 1 \pm r_2 \bar \xi$ close to $\xi = \pm 1$ one has $-(1 - \xi^2)^2 = - 4 (r_2 \bar \xi)^2 + O((r_2 \bar \xi)^4)$ so that
		\[
		\exp \left( - \int_0^{t_2} \left( 4 \bar \xi^2 + O((r_2 \bar \xi)^2 \right) ds_2 \right) = \exp \left( - 4 \bar \xi^2 t_2 + O((r_2 \bar \xi)^2 t_2) \right) ,
		\]
		which is bounded for all $t_2 \in [0,T_2]$. This leads to a $C^2_b$ norm which is growing at a rate proportional to $t_2^2$, which is finite on $[0,T_2]$. Applying Lemma \ref{lem:US_8.3.7} yields the $H_{ul}^\theta$ estimate in Assertion (ii).
		
		\
		
		In $\mathcal K_3$ we have $r(\bar t) = r_3(t_3) = r_3^\ast (1 + 2 \zeta t_3)^{1/4}$ and $\bar v(\bar t) = 1$ for all $t_3 \in [0,T_3]$; recall Lemma \ref{lem:K3_r3eps3}. In this case we also identify algebraic growth rates in time, however these are dominated by the exponential growth due to the fact that $\bar v = 1$ in chart $\mathcal K_3$. In particular, for each $c_3 > 1$ there is a $c_3$ such that
		\[
		\| \tilde p_3(\xi,t_3) \|_{C_b^2} \leq C \me^{c_3 t_3} 
		\]
		for all $t_3 > 0$. The bound in Assertion (iii) follows after applying Lemma \ref{lem:US_8.3.7}.
	\end{proof}

	\subsection{Residual estimates}
	\label{sub:residual_estimates}
	
	In the following we require extra regularity in order to control higher order contributions to the residual because of the reasons given in Remark \ref{rem:regularity}. In order to keep track of this extra regularity we use
	\[
	\theta_{j,k} := 1 + 3 (j - k + 1) + \theta = 
	\begin{cases}
		\theta_A(j) = 1 + 3j + \theta, & k=1 , \\
		3 (j - k + 1) + \theta, & k=2, \ldots, j-1, \\
		4 + \theta , & k = j ,
	\end{cases}
	\]
	instead of the earlier definition in \eqref{eq:theta_k_2}. We obtain the following result.
	
	\begin{lemma}
		\label{lem:error_residual}
		Let $\theta_A = 1 + 3(n-3) + \theta$ and $\theta \geq 1$. Then
		\begin{equation}
			\label{eq:residual_error}
			\| \textup{Res} (r_l(t_l) \psi_l(\cdot,t_l)) \|_{H_{ul}^\theta} \leq C r_l(t_l)^n \left( \delta_{n,4} \max_{m \in I_N}|\nu_m| + \sup_{m \in I_N , j = 1,\ldots,\tilde \alpha(m)} \| A_{mj,l}(\cdot,t_l) \|_{H_{ul}^{\alpha(m)+j,N}} \right) ,
		\end{equation}
		for $t_l \in [0, T_l]$ and $l = 1,2$, where 
		\[
		\delta_{n,4} := 
		\begin{cases}
			1, & n = 4 , \\
			0, & n \geq 4 .
		\end{cases}
		\]
		If additionally $\mu(x) \equiv 0$ and the conditions of Proposition \ref{prop:K3_summary} are satisfied, then 
		\begin{equation}
			\label{eq:residual_error_mu_0}
			\| \textup{Res} (r_l(t_l) \psi_l(\cdot,t_l)) \|_{H_{ul}^\theta} \leq C r_l(t_l)^n \sup_{k = 1,\ldots,N-2} \| A_{1k,l}(\cdot,t_l) \|_{H_{ul}^{\theta_{k,N-2}}} ,
		\end{equation}
		for all $t_1 \in [0,T_l]$ and $l=1,2,3$.
	\end{lemma}
	
	\begin{proof}
		We begin in global coordinates. It follows from the formal construction for $\Psi = r \psi$ in Section \ref{sec:Amplitude_reduction_via_geometric_blow-up} that
		\begin{equation}
			\label{eq:residual_3}
			\begin{split}
				\textrm{Res} (r \psi) &= 
				\delta_{n,4} r^4 \bar \eps \sum_{m \in I_N} \nu_m \me^{imx} \\
				&+ \sum_{m \in I_N} \sum_{j=1}^{\tilde \alpha(m)} r^{\alpha(m) + j} \big( r \mathcal L_m^{(1)} + r^2 \tilde{\mathcal L_m^{(2)}} + r^3 \mathcal L_m^{(3)} + r^4 \mathcal L_m^{(4)} \big) A_{mj} \me^{imx} - \sum_{m \in I_{3N} \setminus I_N} b_m \me^{imx} .
			\end{split}
		\end{equation}
		The first sum is easily controlled. The second sum has terms which vanish up to and including $O(r^N)$, however we need to control the higher order terms at $O(r^{N+1}) = O(r^n)$, $O(r^{n+1})$, $O(r^{n+2})$ and $O(r^{n+3})$. Since $\tilde \alpha(m) \leq N-1$ for all $m \in I_N$, terms involving $r \mathcal L_m^{(1)}$ must be $O(r^n)$ with $\alpha(m) + j = N$. It follows that
		\[
		\sum_{m \in I_N} \sum_{j=1}^{\tilde \alpha(m)} r^{\alpha(m) + j} r \mathcal L_m^{(1)} A_{mj} \me^{imx} = 
		r^n \mathcal L_m^{(1)} \left( A_{0N-1} + A_{2N-1} \me^{2ix} + A_{3N-2} \me^{3ix} + \cdots + A_{N1} \me^{Nix} \right) + c.c. ,
		\]
		the norm of which is bounded by $C r^n$ since the non-critical modulation functions $A_{mj}$ with $|m| \neq 1$ inside the term in parentheses are bounded according to Lemmas \ref{lem:K1_gl_coefficients}, \ref{lem:K2_gl_coefficients} and \ref{lem:K3_gl_coefficients}. More precisely, we have that
		\[
		\bigg\| x \mapsto \sum_{m \in I_N} \sum_{j=1}^{\tilde \alpha(m)} r^{\alpha(m) + j} r \mathcal L_m^{(1)} A_{mj} \me^{imx} \bigg\|_{H_{ul}^\theta} \leq 
		C r^n \sup_{ \{(m,j): \alpha(m) + j = N\} } \| A_{mj}(\cdot, \bar t) \|_{H_{ul}^{\theta + 2}} ,
		\]
		where we require two extra degrees of regularity (one is needed to apply Lemma \ref{lem:norm_rescaling}, the other is needed to control the $\partial_{\bar x}$ derivative in $\mathcal L_m^{(1)}$). Similar arguments relying on the bounds obtained for the modulation functions $A_{mj}$ in Section \ref{sec:Proof_of_thm_dynamics} show that
		\begin{equation}
			\label{eq:residual_linear_part_bounds}
			\begin{split}
				\bigg\| x \mapsto \sum_{m \in I_N} \sum_{j=1}^{\tilde \alpha(m)} r^{\alpha(m) + j} r^2 \tilde{\mathcal L}_m^{(2)} A_{mj} \me^{imx} \bigg\|_{H_{ul}^\theta} &\leq 
				C r^n \sup_{ \{(m,j): \alpha(m) + j = N-1\} } \| A_{mj}(\cdot, \bar t) \|_{H_{ul}^{\theta + 3}} \\
				&+ C r^{n+1} \sup_{ \{(m,j): \alpha(m) + j = N\} } \| A_{mj}(\cdot, \bar t) \|_{H_{ul}^{\theta + 3}} , \\
				\bigg\| x \mapsto \sum_{m \in I_N} \sum_{j=1}^{\tilde \alpha(m)} r^{\alpha(m) + j} r^3 \mathcal L_m^{(3)} A_{mj} \me^{imx} \bigg\|_{H_{ul}^\theta} &\leq 
				C r^n \sup_{ \{(m,j): \alpha(m) + j = N-2\} } \| A_{mj}(\cdot, \bar t) \|_{H_{ul}^{\theta + 4}} \\
				&+ C r^{n+1} \sup_{ \{(m,j): \alpha(m) + j = N-1\} } \| A_{mj}(\cdot, \bar t) \|_{H_{ul}^{\theta + 4}} \\
				&+ C r^{n+2} \sup_{ \{(m,j): \alpha(m) + j = N\} } \| A_{mj}(\cdot, \bar t) \|_{H_{ul}^{\theta + 4}} , \\
				\bigg\| x \mapsto \sum_{m \in I_N} \sum_{j=1}^{\tilde \alpha(m)} r^{\alpha(m) + j} r^4 \mathcal L_m^{(4)} A_{mj} \me^{imx} \bigg\|_{H_{ul}^\theta} &\leq 
				C r^n \sup_{ \{(m,j): \alpha(m) + j = N-3\} } \| A_{mj}(\cdot, \bar t) \|_{H_{ul}^{\theta + 5}} \\
				&+ C r^{n+1} \sup_{ \{(m,j): \alpha(m) + j = N-2\} } \| A_{mj}(\cdot, \bar t) \|_{H_{ul}^{\theta + 5}} \\
				&+ C r^{n+2} \sup_{ \{(m,j): \alpha(m) + j = N-1\} } \| A_{mj}(\cdot, \bar t) \|_{H_{ul}^{\theta + 5}} \\
				&+ C r^{n+3} \sup_{ \{(m,j): \alpha(m) + j = N\} } \| A_{mj}(\cdot, \bar t) \|_{H_{ul}^{\theta + 5}} ,
			\end{split}
		\end{equation}
		from which it follows that
		\[
		\begin{split}
			\bigg\| x \mapsto \sum_{m \in I_N} \sum_{j=1}^{\tilde \alpha(m)} r^{\alpha(m) + j} \big( r \mathcal L_m^{(1)} + r^2 \tilde{\mathcal L_m^{(2)}} + r^3 \mathcal L_m^{(3)} &+ r^4 \mathcal L_m^{(4)} \big) A_{mj} \me^{imx} \bigg\|_{H_{ul}^\theta} \leq \\
			&C r^n \sup_{\{(m,j) : N-3 \leq \alpha(m) + j \leq N\}} \| A_{mj}(\cdot,\bar t) \|_{H_{ul}^{\theta+5}} .
		\end{split}
		\]

		\
		
		It remains to consider the third sum in \eqref{eq:residual_3}. Recall that for each $m \in I_{3N} \setminus I_N$, $b_m$ is a linear combination of terms of the form
		\[
		r^{\alpha(m_1) + \alpha(m_2) + \alpha(m_3) + j_1 + j_2 + j_3} A_{m_1j_1} A_{m_2j_2} A_{m_3j_3} ,
		\]
		with $m_i \in I_N$ and $j_i \in \{1 , \ldots, \tilde \alpha(m_i)\}$ for each $i=1,2,3$. Moreover,
		\[
		\alpha(m_1) + \alpha(m_2) + \alpha(m_3) + j_1 + j_2 + j_3 = \alpha(m) + j \geq |m| - 1 + 1 \geq n ,
		\]
		since $m = m_1 + m_2 + m_3 \in I_{3N} \setminus I_N$ and $j \geq 1$. For $l = 1,2$ and additionally $l=3$ if $\mu(x) \equiv 0$ and the conditions of Proposition \ref{prop:K3_summary} are satisfied, the results in Section \ref{sec:Proof_of_thm_dynamics} show that all modulation functions satisfy $A_{mj,l}(\cdot,t_l) \in H_{ul}^{\theta+1}$, for all $t_l \in [0, T_l]$; we refer again to Lemmas \ref{lem:K1_gl_coefficients}, \ref{lem:K2_gl_coefficients} and \ref{lem:K3_gl_coefficients}. Using Lemmas \ref{lem:norm_rescaling} and \ref{lem:US_8.3.11} we therefore obtain
		\[
		\begin{split}
			\bigg\| x \mapsto \sum_{m \in I_{3N} \setminus I_N} b_m(r_l(t_l) x,t_l) \me^{imx} \bigg\|_{H_{ul}^\theta} & \leq
			\sum_{m \in I_{3N} \setminus I_N} \| x \mapsto b_m(r_l(t_l) x,t_l) \me^{imx} \|_{H_{ul}^\theta} \\
			&\leq C r_l(t_l)^n \sup_{m \in I_N,j \in \{1,\ldots,\tilde \alpha(m)\}} \| x \mapsto A_{mj,l}(r_l(t_l) x,t_l) \me^{imx} \|_{H_{ul}^\theta} \\
			&\leq C r_l(t_l)^n \sup_{m \in I_N,j \in \{1,\ldots,\tilde \alpha(m)\}} \| A_{mj,l}(\cdot,t_l) \|_{H_{ul}^{\theta+1}} ,
		\end{split}
		\]
		for all $t_l \in [0,T_l]$, $l = 1,2$ and also for $l=3$ if $\mu(x) \equiv 0$ and the conditions of Proposition \ref{prop:K3_summary} are satisfied. 
		
		\
		
		Taking a triangle inequality in \eqref{eq:residual_3} and using the bounds for the second and third sums derived above yields the bound in \eqref{eq:residual_error}. If $\mu(x) \equiv 0$, then $\nu_m = 0$ for all $m \in I_N$ in \eqref{eq:residual_error}. It also follows from the proof of Lemmas \ref{lem:K1_Aj} and \ref{lem:K1_gl_coefficients} (see equations \eqref{eq:Aki_bounds} and \eqref{eq:Amj1_bounds} in particular) that
		\[
		\|A_{mj,l}(\cdot,t_l)\|_{H_{ul}^{\theta_{\alpha(m)+j,N}}} \leq 
		C \sup_{k = 1,\ldots,j-2} \| A_{\pm1k,l}(\cdot,t_1) \|_{H_{ul}^{\theta_{k,N-2}}} ,
		\]
		for all $t_l \in [0,T_l]$ and $l=1,2,3$. Combining this with \eqref{eq:residual_error} yields the bound in \eqref{eq:residual_error_mu_0}.
	\end{proof}
	
	\begin{remark}
		In Lemma \ref{lem:error_residual} we require $\theta_A = 1 + 3(n-3) + \theta$ as opposed to lesser requirement that $\theta_A = 1 + 3(N-3) + \theta = 1 + 3(n-4) + \theta$, which is all that was needed in order to describe the approximation dynamics in Lemmas \ref{lem:Ansatz_dynamics} and \ref{lem:Ansatz_dynamics_mu_0}. Greater regularity is needed to bound the residual because we need to control the norms in \eqref{eq:residual_linear_part_bounds}; see also Remark \ref{rem:regularity}.
	\end{remark}

	\subsection{Estimates for the nonlinear terms}
	\label{sub:estimates_for_the_nonlinear_terms}
	
	Let $F_l(x,t_l)$ denote the local representation of $F(x,\bar t)$ in chart $\mathcal K_l$, and define
	\[
	\tilde T_l := \min\left\{ T_l , \sup \left\{ t_l > 0 : \| \tilde R_l(\cdot,t_l) \|_{H_{ul}^\theta} \leq \tilde M \right\} \right\}
	\]
	where $\tilde M > M \geq \|R^\ast\|_{H_{ul}^\theta}$, for each $l=1,2,3$. We prove the following result.
	
	\begin{lemma}
		\label{lem:error_nonlinear}
		Let $\theta_A = 1 + 3(N-3) + \theta$ where $\theta > 1/2$. 
		We obtain the following estimates in charts $\mathcal K_l$:
		\begin{enumerate}
			\item[(i)] In chart $\mathcal K_1$ we have
			\[
			\begin{split}
				\| F_1(\cdot, s_1) \|_{H_{ul}^\theta} \leq 
				C \bigg( \| \psi_1(\cdot, s_1) \|_{H_{ul}^\theta}^2 + \frac{\rho_{in}^{(\beta - 1)/2}}{(1 - 2 \eps_1^\ast s_1)^{1 / 4}} \| \tilde R_1(\cdot, s_1) \|_{H_{ul}^\theta} & \| \psi_1(\cdot, s_1) \|_{H_{ul}^\theta} \\ 
				+ \frac{\rho_{in}^{\beta - 1}}{(1 - 2 \eps_1^\ast s_1)^{1 / 2}} & \| \tilde R_1(\cdot, s_1) \|_{H_{ul}^\theta}^2 \bigg) \| \tilde R_1(\cdot, s_1) \|_{H_{ul}^\theta} ,
			\end{split}
			\]
			for all $t_1 \in [0, \tilde T_1]$.
			\item[(ii)] In chart $\mathcal K_2$ we have
			\[
			\begin{split}
				\| F_2(\cdot, s_2) \|_{H_{ul}^\theta} \leq 
				C \bigg( \| \psi_2(\cdot, s_2) \|_{H_{ul}^\theta}^2 + r_2^{\beta - 1} \| \tilde R_2(\cdot, s_2) \|_{H_{ul}^\theta} & \| \psi_2(\cdot, s_2) \|_{H_{ul}^\theta} \\ 
				+ r_2^{2\beta - 2} & \| \tilde R_2(\cdot, s_2) \|_{H_{ul}^\theta}^2 \bigg) \| \tilde R_2(\cdot, s_2) \|_{H_{ul}^\theta} ,
			\end{split}
			\]
			for all $t_2 \in [0, \tilde T_2]$.
			\item[(iii)] If $\mu(x) \equiv 0$ and the conditions of Proposition \ref{prop:K3_summary} are satisfied then in chart $\mathcal K_3$ we have
			\[
			\begin{split}
				\| F_3(\cdot, s_3) \|_{H_{ul}^\theta} \leq 
				C \bigg( \| \psi_3(\cdot, s_3) \|_{H_{ul}^\theta}^2 + \frac{{r_3^\ast}^{\beta - 1}}{(1 + 2 \zeta s_3)^{1 / 4}} \| \tilde R_3(\cdot, s_3) \|_{H_{ul}^\theta} & \| \psi_3(\cdot, s_3) \|_{H_{ul}^\theta} \\ 
				+ \frac{{r_3^\ast}^{2\beta - 2}}{(1 + 2 \zeta s_3)^{1 / 2}} & \| \tilde R_3(\cdot, s_3) \|_{H_{ul}^\theta}^2 \bigg) \| \tilde R_3(\cdot, s_3) \|_{H_{ul}^\theta} ,
			\end{split}
			\]
			for all $t_3 \in [0, \tilde T_3]$.
		\end{enumerate}
	\end{lemma}	
	
	\begin{proof}
		By Propositions \ref{prop:K1_summary}, \ref{prop:K2_summary} and \ref{prop:K3_summary} we have that $\psi_l(\cdot,t_l) \in H_{ul}^\theta$ for all $t_l \in [0,\tilde T_l]$ (although only if $\mu(x) \equiv 0$ and the conditions of Proposition \ref{prop:K3_summary} are satisfied in case $l=3$). Since $H_{ul}^\theta$ is closed under multiplication and $\tilde R_l(\cdot,t_l) \in H_{ul}^\theta$ for all $t_l \in [0,\tilde T_l]$ by the definition of $\tilde T_l$, it follows that for any fixed integers $n_1, n_2 \geq 1$ there is a constant $C>0$ such that
		\[
		\| \psi_l(\cdot,s_l)^{n_1} \tilde R_l(\cdot, s_l)^{n_2} \|_{H_{ul}^\theta} \leq 
		C \| \psi_l(\cdot,s_l) \|_{H_{ul}^\theta}^{n_1} \| \tilde R_l(\cdot,s_l) \|_{H_{ul}^\theta}^{n_2} , \qquad t_l \in [0,\tilde T_l] .
		\]
		Thus we obtain the following after taking the norm of \eqref{eq:nonlinear_terms}:
		\[
		\begin{split}
			\| F_l(\cdot, s_l) \|_{H_{ul}^\theta} \leq 
			C \bigg( \| \psi_l(\cdot, s_l) \|_{H_{ul}^\theta}^2 + r_l(s_l)^{\beta - 1} \me^{I_l(s_l)} \| \tilde R_l(\cdot, s_l) \|_{H_{ul}^\theta} & \| \psi_l(\cdot, s_l) \|_{H_{ul}^\theta} \\ 
			+ r_l(s_l)^{2 \beta - 2} \me^{2 I_l(s_l)} & \| \tilde R_l(\cdot, s_l) \|_{H_{ul}^\theta}^2 \bigg) \| \tilde R_l(\cdot, s_l) \|_{H_{ul}^\theta} ,
		\end{split}
		\]
		where $I_l(s_l)$ denotes the local representation of $I(\bar s)$ in chart $\mathcal K_l$. The expressions in Assertions (i), (ii) and (iii) follow from the explicit solution formulae for $r_l(t_l)$ (see Lemmas \ref{lem:K1_r1eps1}, \ref{lem:K3_r3eps3} and recall that $r_2(t_2) = r_2$) together with
		\[
		\me^{I_1(s_1)} = \frac{1}{(1 - 2 \eps_1^\ast s_1)^{\beta / 4}} , \qquad 
		\me^{I_2(s_2)} = 1 , \qquad 
		\me^{I_3(s_3)} = \frac{1}{(1 + 2 \zeta s_3)^{\beta / 4}} .
		\]
	\end{proof}

	\subsection{Proof of Theorem \ref{thm:Error}}
	\label{sub:proof_of_theorem_error}
	
	In this section we use the estimates derived in Sections \ref{sub:evolution_family_estimates}, \ref{sub:residual_estimates} and \ref{sub:estimates_for_the_nonlinear_terms} above to prove Theorem \ref{thm:Error}. We begin by deriving estimates for the (weighted) error function $R_l$ in local coordinate charts $\mathcal K_l$.

	\subsubsection{Error estimates in charts}
	
	Using \eqref{eq:R} and the results obtained above we can estimate the error in each chart $\mathcal K_l$. 
	
	\begin{proposition}
		\label{prop:R_bounds_Kl}
		Let $\theta_A = 1 + 3(n-3) + \theta$ where $\theta \geq 1$, and $\beta = n - 2$. Then for $l=1,2$ we have that $R_l(\cdot, t_l) \in H_{ul}^\theta$ for all $t_l \in [0, T_l]$, with
		\[
		\| R_1(\cdot,T_1) \|_{H_{ul}^\theta} \leq
		C \left( 1 + \| R_1^\ast \|_{H_{ul}^\theta} \me^{- \tilde c_1 / 2\eps_1^\ast} \right) , \qquad
		\| R_2(\cdot,T_2) \|_{H_{ul}^\theta}
		\leq C \left( 1 + \| R_2^\ast \|_{H_{ul}^\theta} \right) ,
		\]
		where $\tilde c_1 = c_1 - 2 C \tilde C(\tilde K,\tilde M) \in (0,1)$ can be chosen arbitrarily close to $1$ by fixing $K, M > 0$ sufficiently small. If additionally $\mu(x) \equiv 0$ then the following assertions are true:
		\begin{enumerate}
			\item[(i)] In chart $\mathcal K_1$ we have
			\[
			\| R_1(\cdot,T_1) \|_{H_{ul}^\theta} \leq 
			C \left( \| R_1^\ast  \|_{H_{ul}^\theta} + \max_{k = 1, \ldots, N-2} \| A_{k,1}^\ast \|_{H_{ul}^{\theta_{k,N-2}}} \right) \me^{-\kappa / 2\eps_1^\ast} .
			\]
			\item[(ii)] In chart $\mathcal K_2$ we have
			\[
			\| R_2(\cdot,T_2) \|_{H_{ul}^\theta} \leq 
			C \left( \| R_2^\ast \|_{H_{ul}^\theta} + \max_{k = 1, \ldots, N-2} \| A_{k,2}^\ast \|_{H_{ul}^{\theta_{k,N-2}}} \right) .
			\]
			\item[(iii)] If the conditions of Proposition \ref{prop:K3_summary} are satisfied, in chart $\mathcal K_3$ we have that $R_3(\cdot,t_3) \in H_{ul}^\theta$ for all $t_3 \in [0,T_3]$ and
			\[
			\| R_3(\cdot,T_3) \|_{H_{ul}^\theta} \leq 
			C {r_3^\ast}^{n-2} \left( \| R_3^\ast \|_{H_{ul}^\theta} + \me^{-(\tilde \sigma_0 - \tilde \sigma) \rho_{out}^2 / 2 \zeta {r_3^\ast}^4} \right) \me^{\tilde c_3 \rho_{out}^2 / 2 \zeta {r_3^\ast}^4} ,
			\]
			where $\tilde c_3 := c_3 + 2 C \tilde C(\tilde K,\tilde M) > 1$ can be set arbitrarily close to $1$ by fixing sufficiently small $K, M > 0$.
		\end{enumerate}
	\end{proposition}
	
	\begin{proof}
		In chart $\mathcal K_1$ we have
		\[
		\begin{split}
			\tilde R_1(\cdot,t_1) = \me^{t_1 \Lambda_1} R_1^\ast  &+ \int_0^{t_1} \me^{(t_1 - s_1) \Lambda_1} F_1(\cdot,s_1) ds_1 \\
			&+ \int_0^{t_1} \me^{(t_1 - s_1) \Lambda_1} r_1(s_1)^{-\beta-2} \me^{-I_1(s_1)} \textup{Res}(r_1(s_1) \psi_1(\cdot, s_1)) ds_1 ,
		\end{split}
		\]
		for $t_1 \in [0, \tilde T_1]$. Using the bound for the evolution family in Lemma \ref{lem:error_evolution_family}, we obtain
		\begin{equation}
			\label{eq:R1_tilde_voc}
			\begin{split}
				\| \tilde R_1(\cdot,t_1) \|_{H_{ul}^\theta} \leq C \me^{-c_1 t_1} \| R_1^\ast  \|_{H_{ul}^\theta} &+ C \int_0^{t_1} \me^{-c_1(t_1 - s_1)} \| F_1(\cdot,s_1) \|_{H_{ul}^\theta} ds_1 \\
				+ \ C \int_0^{t_1} \me^{- c_1 (t_1 - s_1)} & r_1(s_1)^{-\beta-2} \me^{-I_1(s_1)} \| \textup{Res}(r_1(s_1) \psi_1(\cdot, s_1)) \|_{H_{ul}^\theta} ds_1 ,
			\end{split}
		\end{equation}
		for all $t_1 \in [0,\tilde T_1]$. The right-most integral can be bounded using Lemma \ref{lem:error_residual} and the fact that $r_1(s_1)^{-\beta-2} \me^{-I_1(s_1)} = 1 / (\rho_{in}^{\beta / 2} r_1(s_1)^2)$. We obtain
		\begin{equation}
			\label{eq:R_bound_K1}
			\begin{split}
				\int_0^{t_1} \me^{- c_1 (t_1 - s_1)} r_1(s_1)^{-\beta-2} \me^{-I_1(s_1)} & \| \textup{Res}(r_1(s_1) \psi_1(\cdot, s_1)) \|_{H_{ul}^\theta} ds_1 \\ \leq 
				\frac{C}{\rho_{in}^{\beta / 2}} \int_0^{t_1} \me^{- c_1 (t_1 - s_1)} & r_1(s_1)^{n - 2} \left( \delta_{n,4} \max_{m \in I_N}|\nu_m| + \sup_{m,j} \| A_{mj,1}(\cdot,s_1) \|_{H_{ul}^{\theta_{\alpha(m)+j,N}}} \right) ds_1 ,
			\end{split}
		\end{equation}
		for all $t_1 \in [0,T_1]$, where the supremum is taken with respect to modulation functions with $m \in I_N$ and $j \in \{1,\ldots,\tilde\alpha(m)\}$. By Lemma \ref{lem:K1_gl_coefficients}, all such modulation functions in chart $\mathcal K_1$ satisfy $A_{mj,1}(\cdot,s_1) \in H_{ul}^{\theta + 1}$ for all $s_1 \in [0,T_1]$. It follows that
		\[
		\begin{split}
			\frac{C}{\rho_{in}^{\beta / 2}} \int_0^{t_1} \me^{- c_1 (t_1 - s_1)} r_1(s_1)^{n - 2} 
			\left( \delta_{n,4} \max_{m \in I_N}|\nu_m| + \sup_{m,j} \| A_{mj,1}(\cdot,s_1) \|_{H_{ul}^{\theta_{\alpha(m)+j,N}}} \right) ds_1 &\leq \\
			\frac{C}{\rho_{in}^{\beta / 2}} \int_0^{t_1} \me^{- c_1 (t_1 - s_1)} & r_1(s_1)^{n - 2} ds_1 .
		\end{split}
		\]
		The right-most integral can be evaluated directly using the expression for $r_1(s_1)$ in Lemma \ref{lem:K1_r1eps1} together with the integral identity and gamma function properties in Lemma \ref{lem:Gamma_integrals}. We obtain
		\[
		\begin{split}
			\me^{-c_1 t_1} \int_0^{t_1} \me^{c_1 s_1} r_1(s_1)^{n - 2} ds_1 &= 
			C \rho_{in}^{n-2} \me^{- c_1 (t_1 - 1 / 2 \eps_1^\ast)} {\eps_1^\ast}^{(n-2) / 4} \Gamma \left( \frac{n+2}{4}, c_1 \left( \frac{1}{2 \eps_1^\ast} - s_1 \right) \right) \bigg|_0^{t_1} \\
			&\leq C \rho_{in}^{n-2} {\eps_1^\ast}^{(n-2) / 4} ,
		\end{split}
		\]
		for all $t_1 \in [0,T_1]$. Combining this with \eqref{eq:R1_tilde_voc} leads to
		\[
		\| \tilde R_1(\cdot,t_1) \|_{H_{ul}^\theta} \leq C \left( \me^{-c_1 t_1} \| R_1^\ast  \|_{H_{ul}^\theta} + \rho_{in}^{n - \beta - 2} {\eps_1^\ast}^{(n-2) / 4} \right) + C \int_0^{t_1} \me^{-c_1(t_1 - s_1)} \| F_1(\cdot,s_1) \|_{H_{ul}^\theta} ds_1 ,
		\]
		for $t_1 \in [0,\tilde T_1]$ (at this point the estimates in Lemma \ref{lem:error_nonlinear} are only shown to hold for $t_1 \in [0,\tilde T_1]$). 
		For our present purposes it suffices to note that by Lemma \ref{lem:error_nonlinear} Assertion (i) we have
		\[
		\| F_1(\cdot,s_1) \|_{H_{ul}^\theta} \leq \tilde C(\tilde M,\tilde K) \| \tilde R_1(\cdot,s_1) \|_{H_{ul}^\theta}
		\]
		for all $s_1 \in [0,\tilde T_1]$, where the constant $\tilde C(\tilde K,\tilde M) > 0$ can be chosen arbitrarily small by choosing $K, M > 0$ sufficiently small (we only require that $\tilde K > K$ and $\tilde M > M$). Thus
		\[
		\| \tilde R_1(\cdot,t_1) \|_{H_{ul}^\theta} \leq C \left( \me^{-c_1 t_1} \| R_1^\ast  \|_{H_{ul}^\theta} + \rho_{in}^{n - \beta - 2} {\eps_1^\ast}^{(n-2) / 4} \right) + C \tilde C(\tilde K,\tilde M) \int_0^{t_1} \me^{-c_1(t_1 - s_1)} \| \tilde R_1(\cdot,s_1) \|_{H_{ul}^\theta} ds_1 ,
		\]
		for all $t_1 \in [0, \tilde T_1]$. Applying the Gr\"onwall inequality in Lemma \ref{lem:HK_2.8} yields
		\[
		\begin{split}
			\| \tilde R_1(\cdot,t_1) \|_{H_{ul}^\theta} &\leq 
			C \left( \rho_{in}^{n - \beta - 2} {\eps_1^\ast}^{(n-2)/4} + \| R_1^\ast  \|_{H_{ul}^\theta} \right) \me^{- (c_1 - 2 C \tilde C(\tilde K,\tilde M)) t_1} \\
			&+  {\eps_1^\ast}^{(n-2)/4} \rho_{in}^{n - \beta - 2} \int_0^{t_1} \me^{- (c_1 - 2 C \tilde C(\tilde K,\tilde M)) (t_1 - s_1)} ds_1 \\
			&= C \left( \rho_{in}^{n - \beta - 2} {\eps_1^\ast}^{(n-2)/4} + \| R_1^\ast  \|_{H_{ul}^\theta} \right) \me^{- (c_1 - 2 \tilde C) t_1} \\
			&+ C \rho_{in}^{n - \beta - 2} {\eps_1^\ast}^{(n-2)/4} \left( \frac{1 - \me^{- (c_1 - 2 \tilde C) t_1}}{c_1 - 2 C \tilde C(\tilde K, \tilde M)} \right) ,
		\end{split}
		\]
		for all $t_1 \in [0,\tilde T_1]$. Choosing $K, M > 0$ sufficiently small allows us to choose $\tilde C(\tilde K,\tilde M)$ such that  $c_1 - 2 C \tilde C(\tilde K, \tilde M) \in (0,1)$. In this case the right-hand side is finite for all $t_1 \in [0,T_1]$, such that $\tilde T_1 = T_1$. Multiplying through by the integrating factor yields
		\begin{equation}
			\label{eq:R1_bound}
			\begin{split}
				\| R_1(\cdot,t_1) \|_{H_{ul}^\theta} &= 
				\me^{I_1(t_1)} \| \tilde R_1(\cdot,t_1) \|_{H_{ul}^\theta} \\
				&\leq \frac{C}{(1 - 2 \eps_1^\ast t_1)^{\beta/4}} \left( \| R_1^\ast \|_{H_{ul}^\theta} \me^{- \tilde c_1 t_1} + \rho_{in}^{n-\beta-2} {\eps_1^\ast}^{(n-2)/4} \right) ,
			\end{split}
		\end{equation}
		where $\tilde c_1 = c_1 - 2 C \tilde C(\tilde K,\tilde M) \in (0,1)$. Since $(1 - 2 \eps_1^\ast t_1)^{-\beta/4} \in [1,(\zeta / \eps_1^\ast)^{\beta/4}]$ for $t_1 \in [0,T_1]$, \eqref{eq:R1_bound} is bounded above by a constant if we set $\beta = n-2$, as required. Subsequent evaluation at $t_1 = T_1$ yields the desired bound.
		
		\
		
		In order get a better bound when $\mu(x) \equiv 0$, we start again with \eqref{eq:R_bound_K1} except with $\nu_m = 0$ for all $m \in I_N$. We use the improved bound \eqref{lem:error_residual} in Lemma \ref{lem:error_residual}, which applies when $\mu(x) \equiv 0$.
		Substituting this for right-most bound into \eqref{eq:R_bound_K1} leads to the upper bound
		\[
		\begin{split}
			C \me^{- \tilde c t_1} \max_{k = 1, \ldots, N-2} \| A_{k,1}^\ast \|_{H_{ul}^{\theta_{k,N-2}}} \int_0^{t_1} \me^{(c_1 - \tilde c) s_1} (1 - 2 \eps_1^\ast s_1)^{(n-3)/4} ds_1 &\leq \\ 
			C \me^{- \tilde c t_1} & \max_{k = 1, \ldots, N-2} \| A_{k,1}^\ast \|_{H_{ul}^{\theta_{k,N-2}}} ,
		\end{split}
		\]
		which we obtain by using $(1 - 2 \eps_1^\ast s_1)^{(n-3)/4} \leq 1$ and setting $c_1 = \tilde c$. Thus we have shown that
		\[
		\begin{split}
			\| \tilde R_1(\cdot,t_1) \|_{H_{ul}^\theta} \leq C \me^{-c_1 t_1} & \left( \| R_1^\ast  \|_{H_{ul}^\theta} + \max_{k = 1, \ldots, N-2} \| A_{k,1}^\ast \|_{H_{ul}^{\theta_{k,N-2}}} \right) \\
			& \qquad \qquad + C \tilde C(\tilde K, \tilde M) \int_0^{t_1} \me^{-c_1(t_1 - s_1)} \| \tilde R_1(\cdot,s_1) \|_{H_{ul}^\theta} ds_1 .
		\end{split}
		\]
		Applying the Gr\"onwall inequality in Lemma \ref{lem:HK_2.8} gives
		\[
		\| \tilde R_1(\cdot,t_1) \|_{H_{ul}^\theta} \leq 
		C \left( \| R_1^\ast \|_{H_{ul}^\theta} + \max_{k = 1, \ldots, N-2} \| A_{k,1}^\ast \|_{H_{ul}^{\theta_{k,N-2}}} \right) \me^{-(c_1 - 2 C \tilde C(\tilde K, \tilde M)) t_1} .
		\]
		Multiplying through by the integrating factor and evaluating at $t_1 = T_1$ we obtain
		\[
		\| R_1(\cdot,T_1) \|_{H_{ul}^\theta} \leq 
		C \left( \| R_1^\ast  \|_{H_{ul}^\theta} + \max_{k = 1, \ldots, N-2} \| A_{k,1}^\ast \|_{H_{ul}^{\theta_{k,N-2}}} \right) \me^{-\kappa / 2\eps_1^\ast} ,
		\]
		as required.
		
		\
		
		In chart $\mathcal K_2$ we have
		\begin{equation}
			\label{eq:R2_voc}
				R_2(\cdot,t_2) = \me^{t_2 \Lambda_2} R_2^\ast  + \int_0^{t_2} \me^{(t_2 - s_2) \Lambda_2} F_2(\cdot,s_2) ds_2 \\
				+ \int_0^{t_2} \me^{(t_2 - s_2) \Lambda_2} r_2^{-\beta-2} \textup{Res}(r_2 \psi_2(\cdot, s_2)) ds_2 ,
		\end{equation}
		for $t_2 \in [0, \tilde T_2]$ (recall that $\me^{I_2(t_2)} = 1$ so that we can work directly with $R_2 = \tilde R_2$). Using the bound for the evolution family in Lemma \ref{lem:error_evolution_family} and the bound for the residual in Lemma \ref{lem:error_residual}, we obtain
		\[
		\begin{split}
			\| R_2(\cdot,t_2) \|_{H_{ul}^\theta} \leq 
			&C \| R_2^\ast  \|_{H_{ul}^\theta} + C \int_0^{t_2} \| F_2(\cdot,s_2) \|_{H_{ul}^\theta} ds_2 \\
			&+ C r_2^{n-\beta-2} \int_0^{t_2} \left( \delta_{n,4} \max_{m \in I_N}|\nu_m| + \sup_{m,j} \| A_{mj,2}(\cdot,s_2) \|_{H_{ul}^{\theta_{\alpha(m)+j,N}}} \right) ds_2 .
		\end{split}
		\]
		If we set $\beta = n-2$ and use Lemma \ref{lem:error_nonlinear} together with the fact that each $A_{mj,2}(\cdot,s_2) \in H_{ul}^{\theta_{\alpha(m)+j,N}}$ for all $s_2 \in [0,T_2]$ (this follows from Lemma \ref{lem:K2_Aj}), we obtain
		\[
		\| R_2(\cdot,t_2) \|_{H_{ul}^\theta} \leq 
		C \left( 1 + \| R_2^\ast \|_{H_{ul}^\theta} \right) + C  \tilde C(\tilde M,\tilde K) \left(1 + r_2^{n-3} + r_2^{2(n-3)} \right) \int_0^{t_2} \| R_2(\cdot,s_2) \|_{H_{ul}^\theta} ds_2 ,
		\]
		for all $t_2 \in [0,\tilde T_2]$, where the constant $\tilde C(\tilde M,\tilde K) > 0$ can again be chosen as small as necessary by choosing sufficiently small $K, M > 0$. Applying Gr\"onwall's inequality yields
		\[
		\| R_2(\cdot,t_2) \|_{H_{ul}^\theta}
		\leq C \left( 1 + \| R_2^\ast \|_{H_{ul}^\theta} \right) \exp \left( C  \tilde C(\tilde M,\tilde K) \left(1 + r_2^{n-3} + r_2^{2(n-3)} \right) t_2 \right) .
		\]
		It follows that $R_2(\cdot,t_2) \in H_{ul}^\theta$ for all $t_2 \in [0,T_2]$ (recall that $T_2$ is finite).
		
		\
		
		Similarly to the $\mathcal K_1$ analysis above, the bound on the residual integral can be improved if $\mu(x) \equiv 0$ because we can use the bound \eqref{eq:residual_error_mu_0} in Lemma \ref{lem:error_residual}. 
		Using this result 
		to improve the bound on the right-most integral in \eqref{eq:R2_voc} leads to
		\[
		\begin{split}
			\| R_2(\cdot,t_2) \|_{H_{ul}^\theta} \leq 
			C & \left( \| R_2^\ast \|_{H_{ul}^\theta} + \max_{k = 1, \ldots, N-2} \| A_{k,2}^\ast \|_{H_{ul}^{\theta_{k,N-2}}} \right) + \\
			&\qquad \qquad C \tilde C(\tilde M,\tilde K) \left(1 + r_2^{n-3} + r_2^{2(n-3)} \right) \int_0^{t_2} \| R_2(\cdot,s_2) \|_{H_{ul}^\theta} ds_2 ,
		\end{split}
		\]
		which holds for all $t_2 \in [0,\tilde T_2]$. The bound in Assertion (ii) follows after an application of Gr\"onwall's inequality. 
		
		\
		
		It remains to prove Assertion (iii). We work in chart $\mathcal K_3$, set $\mu(x) \equiv 0$ and assume that the hypotheses of Proposition \ref{prop:K3_summary} are satisfied. We have
		\[
		\begin{split}
			\tilde R_3(\cdot,t_3) = \me^{t_3 \Lambda_3} R_3^\ast  &+ \int_0^{t_3} \me^{(t_3 - s_3) \Lambda_3} F_3(\cdot,s_3) ds_3 \\
			&+ \int_0^{t_3} \me^{(t_3 - s_3) \Lambda_3} r_3(s_3)^{-\beta-2} \me^{-I_3(s_3)} \textup{Res}(r_3(s_3) \psi_3(\cdot, s_3)) ds_3 ,
		\end{split}
		\]
		for $t_3 \in [0, \tilde T_3]$. Using the bound for the evolution family in Lemma \ref{lem:error_evolution_family}, we obtain
		\[
		\begin{split}
			\| \tilde R_3(\cdot,t_3) \|_{H_{ul}^\theta} \leq C \me^{c_3 t_3} \| R_3^\ast  \|_{H_{ul}^\theta} &+ C \int_0^{t_3} \me^{c_3(t_3 - s_3)} \| F_3(\cdot,s_3) \|_{H_{ul}^\theta} ds_3 \\
			+ C \int_0^{t_3} \me^{c_3 (t_3 - s_3)} & r_3(s_3)^{-\beta-2} \me^{-I_3(s_3)} \| \textup{Res}(r_3(s_3) \psi_3(\cdot, s_3)) \|_{H_{ul}^\theta} ds_3 ,
		\end{split}
		\]
		for all $t_3 \in [0,\tilde T_3]$. The right-most integral can be bounded using Lemma \ref{lem:error_residual} and the fact that $r_3(s_1)^{-\beta-2} \me^{-I_3(s_1)} = 1 / ({r_3^\ast}^{\beta} r_3(s_3)^2)$. We obtain
		\[
		\begin{split}
			\int_0^{t_3} \me^{c_3 (t_3 - s_3)} r_3(s_3)^{-\beta-2} \me^{-I_3(s_3)} & \| \textup{Res}(r_3(s_3) \psi_3(\cdot, s_3)) \|_{H_{ul}^\theta} ds_3 \\ \leq 
			& \frac{C}{{r_3^\ast}^\beta} \int_0^{t_3} \me^{c_3 (t_3 - s_3)} r_3(s_3)^{n - 2} \sup_{k = 1,\ldots,N-2} \| A_{k,3}(\cdot,s_3) \|_{H_{ul}^{\theta_{k,N-2}}} ds_3 ,
		\end{split}
		\]
		for all $t_1 \in [0,T_3]$. Under the assumptions of Proposition \ref{prop:K3_summary} we have that
		\[
		\sup_{k = 1,\ldots,N-2} \| A_{k,3}(\cdot,s_3) \|_{H_{ul}^{\theta+1}} \leq
		C \me^{-(\tilde \sigma_0 - \tilde \sigma) \rho_{out}^2 / 2 \zeta {r_3^\ast}^4} ,
		\]
		for all $t_3 \in [0,T_3]$. Using this and Lemma \ref{lem:Gamma_integrals}, we obtain
		\[
		\begin{split}
			\frac{C}{{r_3^\ast}^\beta} & \int_0^{t_3} \me^{c_3 (t_3 - s_3)} r_3(s_3)^{n - 2} \sup_{k = 1,\ldots,N-2} \| A_{k,3}(\cdot,s_3) \|_{H_{ul}^{\theta_{k,N-2}}} ds_3 \\
			&\leq C {r_3^\ast}^{n-\beta-2} \me^{-(\tilde \sigma_0 - \tilde \sigma) \rho_{out}^2 / 2 \zeta {r_3^\ast}^4} \me^{c_3 t_3} \int_0^{t_3} \me^{- c_3 s_3} (1 + 2 \zeta s_3)^{(n-2)/4} ds_3 \\
			&\leq C {r_3^\ast}^{n-\beta-2} \me^{-(\tilde \sigma_0 - \tilde \sigma) \rho_{out}^2 / 2 \zeta {r_3^\ast}^4} \me^{c_3 t_3} \Gamma\left( \frac{n+2}{4}, c_3 \left( \frac{1}{2 \zeta} + s_3 \right) \right) \bigg|_0^{t_3} \\
			&\leq C \me^{-(\tilde \sigma_0 - \tilde \sigma) \rho_{out}^2 / 2 \zeta {r_3^\ast}^4} \me^{c_3 t_3} ,
		\end{split}
		\]
		where we set $\beta = n-2$ and used the gamma function asymptotics \eqref{eq:Gamma_asymptotics} in Appendix \ref{app:technical_estimates} to derive the last inequality. We shall also use
		\[
		\| F_3(\cdot,s_3) \|_{H_{ul}^\theta} \leq 
		\tilde C(\tilde K, \tilde M) \| \tilde R_3(\cdot,s_3) \|_{H_{ul}^\theta} , \qquad t_3 \in [0,\tilde T_3] ,
		\] 
		which follows from Lemma \ref{lem:error_nonlinear}. Applying the Gr\"onwall inequality in Lemma \ref{lem:HK_2.8} to
		\[
		\| \tilde R_3(\cdot,t_3) \|_{H_{ul}^\theta} \leq 
		C \me^{c_3 t_3} \left( \| R_3^\ast \|_{H_{ul}^\theta} + \me^{-(\tilde \sigma_0 - \tilde \sigma) \rho_{out}^2 / 2 \zeta {r_3^\ast}^4} \right) + C \tilde C(\tilde K, \tilde M) \int_0^{t_3} \me^{c_3(t_3 - s_3)} \| \tilde R_3(\cdot,s_3) \|_{H_{ul}^\theta} ds_3 ,
		\]
		yields the estimate
		\[
		\| \tilde R_3(\cdot,t_3) \|_{H_{ul}^\theta} \leq 
		C \left( \| R_3^\ast \|_{H_{ul}^\theta} + \me^{-(\tilde \sigma_0 - \tilde \sigma) \rho_{out}^2 / 2 \zeta {r_3^\ast}^4} \right) \me^{(c_3 + 2 C \tilde C(\tilde K,\tilde M))t_3} ,
		\]
		for all $t_3 \in [0,\tilde T_3]$. Multiplying through by the integrating factor gives
		\[
		\| R_3(\cdot,t_3) \|_{H_{ul}^\theta} \leq 
		\frac{C}{(1 + 2 \zeta t_3)^{(n-2)/4}} \left( \| R_3^\ast \|_{H_{ul}^\theta} + \me^{-(\tilde \sigma_0 - \tilde \sigma) \rho_{out}^2 / 2 \zeta {r_3^\ast}^4} \right) \me^{(c_3 + 2 C \tilde C(\tilde K,\tilde M))t_3} .
		\]
		Evaluating this at $t_3 = T_3$ yields the estimate in Assertion (iii).
	\end{proof}
	
	\subsubsection{Proof of Theorem \ref{thm:Error}}
	
	We are now in a position to prove Theorem \ref{thm:Error}. The fact that $u(\cdot,t) - \Psi(\cdot,t) \in [0, \widehat T]$, where $\widehat T = T_{mid}$ for general $\mu(x)$ and $\widehat T = T$ if $\mu(x) \equiv 0$, follows from the definition of the weighted error function $R$ in \eqref{eq:R} and the fact that $R_l(\cdot,t_l) \in H_{ul}^\theta$ for all $t_l \in [0,T_l]$ by Proposition \ref{prop:R_bounds_Kl}, for each $l = 1,2,3$ (assuming the additional conditions including $\mu(x) \equiv 0$ are satisfied in case $l = 3$).
	
	\
	
	In order to prove the bounds in Assertions (i) and (ii), we need to relate estimates for the (weighted) error functions $R_l$ in different charts $\mathcal K_l$. This can be done with the following change of coordinate transformations, which are derived from the form of the maps $\kappa_{ij}$ in \eqref{eq:kappa_maps}:
	\[
	R_1 = (-v_2)^{-\beta/2} R_2 , \ v_2 < 0 , \qquad 
	R_2 = \eps_3^{-\beta/4} R_3 , \ \eps_3 > 0 .
	\]
	In particular, we have
	\[
	R_1(\cdot,T_1) = \zeta^{-\beta/4} R_2^\ast ,  \qquad 
	R_2(\cdot,T_2) = \zeta^{-\beta/4} R_3^\ast .
	\]
	
	\
	
	We start with the general case where $\mu(x)$ is given by \eqref{eq:mu}, i.e.~we first prove Assertion (i). We need to consider the size of the (weighted) error function $R$ as it evolves forward in time from an initial condition in $\Sigma_1^{in}$ in chart $\mathcal K_1$ up to the exit section $\Sigma^{out}_2$ in chart $\mathcal K_2$. At $\Sigma_2^{in} = \kappa_{12}(\Sigma_1^{out})$ we have
	\[
	\| R_2^\ast \|_{H_{ul}^\theta} = \zeta^{\beta / 4} \| R_1(\cdot,T_1) \|_{H_{ul}^\theta} 
	\leq C \left( 1 + \| R_1^\ast \|_{H_{ul}^\theta} \me^{- \tilde c_1 / 2\eps_1^\ast} \right) ,
	\]
	and therefore
	\[
	\| R_2(\cdot,T_2) \|_{H_{ul}^\theta} \leq C \left( 1 + \| R_1^\ast \|_{H_{ul}^\theta} \me^{- \tilde c_1 / 2\eps_1^\ast} \right) ,
	\]
	by Proposition \ref{prop:R_bounds_Kl}. This implies that
	\[
	r_2^\beta \| R_2(\cdot,T_2) \|_{H_{ul}^\theta} = 
	\| u(x,T_{mid}) - \Psi(x,T_{mid}) \|_{H_{ul}^\theta} \leq 
	C r_2^\beta \left( 1 + \| R_1^\ast \|_{H_{ul}^\theta} \me^{- \tilde c_1 / 2\eps_1^\ast} \right) .
	\]
	Applying the blow-down transformations $\eps_1^\ast = \eps / \rho_{in}^2$, $r_2 = \eps^{1/4}$ and setting $\beta = n-2$ yields
	\[
	\|u(\cdot,T_{mid}) - \Psi(\cdot,T_{mid}) \|_{H_{ul}^\theta} \leq C \eps^{(n-2)/4} \left( 1 + \| R_1^\ast \|_{H_{ul}^\theta} \me^{- \tilde c_1 / 2\eps_1^\ast} \right) \leq C \eps^{(n-2)/4} ,
	\]
	thereby proving Assertion (i) in Theorem \ref{thm:Error}.
	
	\
	
	It remains to prove Assertion (ii), for which we set $\mu(x) \equiv 0$. In this case, the error is exponentially small in chart $\mathcal K_1$, and Proposition \ref{prop:R_bounds_Kl} implies the following at bound at $\Sigma_2^{in} = \kappa_{12}(\Sigma_1^{out})$:
	\[
	\| R_2^\ast \|_{H_{ul}^\theta} = \zeta^{\beta / 4} \| R_1(\cdot,T_1) \|_{H_{ul}^\theta} 
	\leq C \left( \| R_1^\ast \|_{H_{ul}^\theta} + \max_{k = 1,\ldots,N-2} \|A_{k,1}^\ast\|_{H_{ul}^{\theta_{k,N-2}}} \right) \me^{-\kappa / 2\eps_1^\ast} \leq C \me^{-\kappa / 2\eps_1^\ast} .
	\]
	Applying Proposition \ref{prop:R_bounds_Kl} again, this time at $\Sigma_3^{in} = \kappa_{23}(\Sigma_2^{out})$, we obtain
	\[
	\| R_3^\ast \|_{H_{ul}^\theta} = \zeta^{\beta / 4} \| R_2(\cdot,T_2) \|_{H_{ul}^\theta} 
	\leq C \left( \| R_2^\ast \|_{H_{ul}^\theta} + \max_{k = 1,\ldots,N-2} \|A_{k,2}^\ast\|_{H_{ul}^{\theta_{k,N-2}}} \right) .
	\]
	Using the bound for $\| R_2^\ast \|_{H_{ul}^\theta}$ above and the bounds for $\|A_{k,2}^\ast\|_{H_{ul}^\theta}$ in \eqref{eq:K2_ic_bounds_2}, which apply when $\mu(x) \equiv 0$, we have that
	\[
	\| R_3^\ast \|_{H_{ul}^\theta} \leq C \me^{-\kappa / 2\eps_1^\ast} .
	\]
	Hence, Assertion (iii) of Proposition \ref{prop:R_bounds_Kl} implies that
	\[
	\begin{split}
		\| R_3(\cdot,t_3) \|_{H_{ul}^\theta} &\leq 
		C {r_3^\ast}^{n-2} \left( \| R_3^\ast \|_{H_{ul}^\theta} + \me^{-(\tilde \sigma_0 - \tilde \sigma) \rho_{out}^2 / 2 \zeta {r_3^\ast}^4} \right) \me^{(c_3 + 2 C \tilde C(\tilde K,\tilde M))t_3} \\
		&\leq C {r_3^\ast}^{n-2} \left( \me^{-\kappa / 2\eps_1^\ast} + \me^{-(\tilde \sigma_0 - \tilde \sigma) \rho_{out}^2 / 2 \zeta {r_3^\ast}^4} \right) \me^{(c_3 + 2 C \tilde C(\tilde K,\tilde M))t_3} .
	\end{split}
	\]
	We recall from the proof of Lemma \ref{lem:Ansatz_dynamics_mu_0} in Section \ref{sub:proof_of_approximation_lemma_1} that the optimal choice for $\tilde \sigma_0 > \tilde \sigma > 1$ is $\tilde \sigma_0 \rho_{out}^2 / 2 \zeta {r_3^\ast}^4 = \kappa / 2\eps_1^\ast$. Applying this choice and $\eps_1^\ast = \eps / \rho_{in}^2$ leads to
	\[
	\| R_3(\cdot,T_3) \|_{H_{ul}^\theta} \leq
	C \exp\left( - \frac{\kappa}{2 \eps} \left( \rho_{in}^2 - \frac{\tilde \sigma}{\kappa} \rho_{out}^2 - \frac{\tilde c_3}{\kappa} \rho_{out}^2 \right) \right) ,
	\]
	where $\tilde c_3 := c_3 + 2 C \tilde C(\tilde K,\tilde M) > c_3 > 1$ can be chosen arbitrarily close to $c_3$ (and thus to $1$) if $K, M > 0$ are chosen sufficiently small. Since the same is true of $\tilde \sigma$, we can choose $\tilde c_3 = \tilde \sigma$, leading to
	\[
	\| R_3(\cdot,t_3) \|_{H_{ul}^\theta} \leq
	C \exp\left( - \frac{\kappa}{2 \eps} \left( \rho_{in}^2 - 2 \frac{\tilde \sigma}{\kappa} \rho_{out}^2 \right) \right) ,
	\]
	which is exponentially small as $\eps \to 0$ as long as
	\[
	\frac{\rho_{in}}{\rho_{out}} \geq \sqrt{\frac{2 \tilde \sigma}{\kappa}} > \sqrt 2,
	\]
	where the strict inequality on the right can be arbitrarily close to equality if $K, M>0$ are sufficiently small. This is precisely the assumption in Theorem \ref{thm:Error}. It follows that
	\[
	\| u(\cdot,T) - \Psi(\cdot,T) \|_{H_{ul}^\theta} = r_3(T_3)^\beta \| R_3(\cdot,T_3) 	\|_{H_{ul}^\theta} 
	\leq C \exp\left( - \frac{\kappa}{2 \eps} \left( \rho_{in}^2 - 2 \frac{\tilde \sigma}{\kappa} \rho_{out}^2 \right) \right) .
	\]
	Writing $\kappa_- := \kappa$ and $\kappa_+ := \tilde \sigma / \kappa$ as in the proof of Lemma \ref{lem:Ansatz_dynamics_mu_0} yields the bound in Assertion (ii).
	\qed

	\section{Summary}
	\label{sec:Summary_and_outlook}
	
	Turing instabilities have been identified as a key local mechanism for the emergence of patterned steady states in a wide variety of different contexts \cite{Cross1993,Hoyle2006,Murray1989,Rabinovich2000}. The dynamic counterpart, i.e.~the slow passage problem, is also expected to arise in a wide variety of applications, based on the fact that realistic models often feature system parameters which evolve slowly in time. Nevertheless, the slow passage through a Turing bifurcation has, by comparison, received far less attention in the literature. As notable exceptions we refer again to \cite{Chen2015} for a detailed numerical study of a model with applications in vegetation patterns, and to \cite{Avitabile2020} for the first rigorous results on the local mechanisms for delay phenomena for systems with a slow passage through an $O(2)$-symmetric Turing bifurcation using the center manifold theory of \cite{Haragus2010,Vanderbauwhede1992}, which applied for systems possessing a spectral gap. The inapplicability of reduction methods based on center manifold theory or Lyapunov-Schmidt reduction on large or unbounded spatial domains presents a significant analytical challenge, which arises naturally in the context of pattern forming systems for which the spatial scale characterising the inhomogeneity of the pattern is small in comparison to the size of the domain itself. Modulation theory has developed in order to address this problem in classical bifurcation theory, but has as yet, to the best of our knowledge, not been extended to the fast-slow setting which is natural for the study of slow passage problems.
	
	\
	
	This work constitutes a first step towards the development of modulation theory for fast-slow and singularly perturbed systems in infinite dimensions. Although we developed our methods in the context of the dynamic SH equation \eqref{eq:sh_dynamic}, the methods themselves are intended to be general in the sense that they can be applied to a wide variety of applications. In Section \ref{sec:Amplitude_reduction_via_geometric_blow-up}, we showed that the formal modulation ansatz \eqref{eq:GL_ansatz_n} which is known from classical modulation theory (and can be derived for different applications according to their `clustered mode distribution'; recall Figure \ref{fig:mode_distribution}) admits of a fast-slow generalisation in the form of the blow-up transformation \eqref{eq:blowup_Psi}. This blow-up transformation is easy to derive if the multiple-scales ansatz from classical modulation theory is already known. The idea is to replace the small parameter $\delta$ in the classical ansatz with a time-dependent variable $r(\bar t)$, and to replace the simple rescalings $\bar x = \delta^2 x$ and $\bar t = \delta^2 t$ with the time-dependent desingularizations in \eqref{eq:desingularization}.
	
	Applying this method to the dynamic SH equation \eqref{eq:sh_dynamic}, we derived a set of modulation equations which take the form of non-autonomous GL equations posed in the blown-up space. The modulation equations were presented in both local and global coordinates in Section \ref{sub:results_modulation_eqns}. Lemmas \ref{lem:Ansatz_dynamics} and \ref{lem:Ansatz_dynamics_mu_0} characterise the dynamics of the modulation equations, which are still infinite dimensional, but dynamically simpler in the sense that they are lower order equations. Lemma \ref{lem:Ansatz_dynamics} provides rigorous asymptotics for the approximation $\Psi(x,t)$ in the case that the source term is given by \eqref{eq:mu}. In contrast to the static theory, we found that the leading order dynamics in this regime are actually governed by the formal second order term at $O(\eps^{1/2})$ (recall equation \eqref{eq:pi_mid_asymptotics}), since the formal leading order term is exponentially small due to a delay effect. 
	Lemma \ref{lem:Ansatz_dynamics_mu_0} describes a delayed loss of stability in the symmetric approximation with $\mu(x) \equiv 0$. The requirement \eqref{eq:delay_cond} and the bound in \eqref{eq:Psi_bound} together imply that initial conditions with $v(0) = - \rho_{in}$ remain exponentially small up to $v(T) \approx \rho_{in}$. In order to prove that a hard onset of patterned states occurs near $v = \rho_{in}$ we would also need to provide a lower bound on the norm of $\Psi(\cdot,T)$. Such an approach has been successfully applied in the context of scalar reaction-diffusion problems in \cite{Kaper2018,Nefedov2003}, but the analysis for equation \eqref{eq:sh_dynamic} is left for future work.
	
	\
	
	As in classical modulation theory, rigorous bounds for the error associated to the approximation are necessary in order to relate the dynamics of the approximation $\Psi$ and solutions to the SH problem \eqref{eq:sh_dynamic}. We provided detailed error estimates in $H_{ul}^\theta$ spaces in Theorem \ref{thm:Error}. Different bounds were obtained in the general case $\mu(x) \neq 0$ and the symmetric case $\mu(x) \equiv 0$. In the first case we were able to control the error up to a distance of $O(\eps^{1/2})$ from the static bifurcation point at $v=0$, i.e.~up to the exit section $\Delta^{mid}$ in chart $\mathcal K_2$. In the symmetric case $\mu(x) \equiv 0$ we were able to show that the error remains exponentially small up to an $O(1)$ distance from $v=0$, but estimate obtained in Assertion (ii) is likely to be sub-optimal due to the unwanted growth of the residual. It is conjectured that a better bound can be obtained if the algebraic control over the residual is improved to exponential control; see again Remark \ref{rem:improved_residual_estimates}. Finally, we note once more that the error estimates of Theorem \ref{thm:Error} apply for a rather large set of initial conditions $u(x,0) = u^\ast(x)$. We require only that $\|u^\ast\|_{H_{ul}^\theta} \leq K$ for some constant $K>0$ which is small but $O(1)$ as $\eps \to 0$. This is an improvement on known results for the static SH equation in e.g.~\cite{Eckhaus1993,Schneider1995b}, which only apply for initial conditions that are $O(\delta)$ as $\delta \to 0$; see again the classical result in Theorem \ref{thm:Error_static} and the discussion which follows it. This is because of additional contraction undergone by solutions in the fast-slow setting during the time spend in the asymptotically stable regime with $v(t) < 0$.
	
	Finally, we presented rigorous results on the dynamics of SH solutions in Theorem \ref{thm:Dynamics}. Assertion (i) shows that a large space of SH solutions (all those with initial conditions $(u(x,0),v(0)) = (u^\ast(x),-\rho_{in})$ such that $\|u^\ast\|_{H_{ul}^\theta} \leq K$) have the same asymptotic form at $\Delta^{mid}$, i.e.~at a distance of $O(\eps^{1/2})$ past the static bifurcation point. Assertion (ii) characterises the delayed stability loss in the symmetric case $\mu(x) \equiv 0$. A lower bound on the delay time is given, but limited by the (likely sub-optimal) bound on the error estimates in Theorem \ref{thm:Error}. Nevertheless, this result constitutes, to the best of our knowledge, the first rigorous and \textit{global} result on the existence of delayed stability loss in systems featuring a slow passage through a Turing bifurcation.

	\section*{Acknowledgements}
	
	SJ and CK acknowledge funding from the SFB/TRR 109 Discretization and Geometry in Dynamics grant. FH and CK acknowledge support of the EU within the TiPES project funded the European Unions Horizon 2020 research and innovation programme under grant agreement No.~820970. CK has also been supported by a Lichtenberg Professorship of the VolkswagenStiftung.

	\section*{Data availability}
	
	Data sharing not applicable to this article as no datasets were generated or analysed during the current study.

	\bibliographystyle{siam}
	\bibliography{swift_hohenberg}

	\appendix

	\section{Technical estimates}
	\label{app:technical_estimates}
	
	In the following we provide a number of known results (or slight adaptations thereof) that are useful for the proofs in Sections \ref{sec:Proof_of_thm_dynamics} and \ref{sec:Proof_of_thm_error}.

	\subsubsection*{Estimate methods in $H_{ul}^\theta$}
	
	Functions approximated by the blow-up approximations $\Psi_{GL}$ and $\Psi_n$ defined in  Section \ref{sub:results_modulation_eqns} depend on two distinct spatial scales $x$ and $\bar x$, where the latter is defined via \eqref{eq:desingularization}. In order to to estimate the norm of functions like $\Psi_{GL}(x,t) = r(t) (A(\bar x,\bar t) \me^{ix} + c.c.)$ in $H_{ul}^\theta$ spaces, we use the following scaling property to relate the norms of $\| x \mapsto u(x) \|_{H_{ul}^\theta}$ and $\| x \mapsto u(rx) \|_{H_{ul}^\theta}$ for a non-negative scaling factor $r \geq 0$. 
	
	\begin{lemma}
		\label{lem:norm_rescaling}
		\textup{c.f.~\cite[Ch.~10]{Schneider2017}} 
		Let $u \in H_{ul}^{\theta + 1}$ and $r \geq 0$. Then 
		\[
		\| x \mapsto u(r x) \|_{H_{ul}^\theta} \leq 
		C \| u \|_{H_{ul}^{\theta + 1}} .
		\]
	\end{lemma}
	
	\begin{proof}
		It suffices to show that $\| x \mapsto u(r x) \|_{L_{ul}^2} \leq 
		C \| u \|_{H_{ul}^1}$, which can be done by a direct calculation. Applying the definition of $\|\cdot\|_{L_{ul}^2}$ in \eqref{eq:L2_norm} we obtain
		\[
		\begin{split}
			\| u(r \cdot) \|_{H_{ul}^\theta} &= \sup_{y \in \R} \left( \int_y^{y+1} | u(rx) |^2 dx \right)^{1/2} \\
			&\leq \sup_{y \in \R} \left\{ \left( \int_{y-1/2}^{y+1/2} dx \right)^{1/2} \sup_{x \in [y-1/2,y+1/2]} | u(rx) | \right\} \leq \| u \|_{C_b^0} \leq C \| u \|_{H_{ul}^1} ,
		\end{split}
		\]
		where we used a Sobolev embedding in the final inequality, c.f.~\cite[Lemma 8.3.11]{Schneider2017} and Lemma \ref{lem:US_8.3.11} below.
	\end{proof}
	
	The next result comes from \textit{multiplier theory}, and is useful in order to estimate the norm of semigroup and evolution family operators in uniformly local Sobolev spaces $H_{ul}^\theta$. Multiplier theory takes advantage of the Fourier transform properties of Sobolev spaces in order to provide methods for estimating norms of translation invariant operators in $H_{ul}^\theta$, which is typically a very difficult task in physical space. An operator $M : H_{ul}^q \to H_{ul}^\theta$ 
	is a called a \textit{multiplier} if the corresponding Fourier operator $\widehat M : \R \to \mathbb C$ is a multiplication operator, i.e.~if $M = \mathscr F^{-1} (\widehat M \mathscr F)$. It is known that the the norms for multipliers on Sobolev spaces $H^\theta$ can be bounded in terms of the $C_b^0(\R, \mathbb C)$ norm (see e.g.~\cite[Lemma 8.3.6]{Schneider2017}). The following result extends this to multipliers on uniformly local Sobolev spaces $H_{ul}^\theta$. Although it was originally shown in \cite{Schneider1994}, we present here it in the formulation of \cite{Mielke1995,Schneider1994b,Schneider2017}, along with a specialisation to the particular case of interest for the proofs in Sections \ref{sec:Proof_of_thm_dynamics} and \ref{sec:Proof_of_thm_error}.
	
	\begin{lemma}
		\label{lem:US_8.3.7}
		\textup{\cite[Lemma 8.3.7]{Schneider2017}}
		For $\theta, q \geq 0$ and $w_{\theta-q}(\xi) = (1 + \xi^2)^{(\theta - q)/2} \widehat M(\xi) \in C^2_b(\R, \mathbb C)$, the map $M_{ul} : H_{ul}^q \to H_{ul}^\theta$, $u \mapsto \mathscr F^{-1} (\widehat M \mathscr F u)$ is well-defined and satisfies
		\[
		\| M u \|_{H_{ul}^\theta} \leq C \| w_{\theta-q} \|_{C_b^2(\R,\mathbb C)} \| u \|_{H_{ul}^q} ,
		\]
		where the constant $C$ depends on $q$ and $\theta$ but not on $\widehat M$. For the case $q = \theta$ occurring frequently in this work, $w_{\theta-q}(\xi) = w_0(\xi) = \widehat M(\xi)$ such that
		\[
		\| M_{ul} u \|_{H_{ul}^\theta} \leq C \| w_0 \|_{C_b^2(\R,\mathbb C)} \| u \|_{H_{ul}^\theta} ,
		\]
		i.e.~$\| M_{ul} u \|_{H_{ul}^\theta \to H_{ul}^\theta} \leq C \| w_0 \|_{C_b^2(\R,\mathbb C)}$.
	\end{lemma}
	
	Finally, we have the following result, which extends a well-known result on the multiplicative closure of Sobolev spaces $H^\theta$ to the uniformly local Sobolev spaces $H_{ul}^\theta$.
	
	\begin{lemma}
		\label{lem:US_8.3.11}
		\textup{\cite[Lemma 8.3.11]{Schneider2017}}
		Fix $\theta > 1/2$. Then $H_{ul}^\theta$ is an algebra with
		\[
		\| uv \|_{H_{ul}^\theta} \leq 9 \| u \|_{H_{ul}^\theta} \| v \|_{H_{ul}^\theta} 
		\]
		for all $u, v \in H_{ul}^\theta$, and can be continuously embedded in $C_{b,unif}^0$. In particular, Sobolev's embedding theorem gives
		\[
		\| u \|_{C_b^\theta} \leq \| u \|_{H_{ul}^{\theta+1}} .
		\]
	\end{lemma}

	\subsubsection*{Integral identities and inequalities}
	
	The following Gr\"onwall-type inequality from \cite{Hummel2022} is frequently used in Sections \ref{sec:Proof_of_thm_dynamics} and \ref{sec:Proof_of_thm_error} in order to control the norms of integral equations obtained after applying the variation of constants formula.
	
	\begin{lemma}
		\label{lem:HK_2.8}
		\textup{\cite[Lemma 2.8]{Hummel2022}} Let $\alpha \in \R$, $\beta, N, T > 0$ and $\gamma \in (0,1]$. Assume that $a, b : [0,T] \to [0,\infty)$ are continuous, the derivative $b'$ is locally integrable in $L^2$ and that the function $[t \mapsto \me^{- \alpha t / \beta} b(t)]$ is non-decreasing. If
		\[
		a(t) \leq b(t) + N \int_0^t \frac{\me^{\alpha (t - s) / \beta}}{\beta^\gamma (t - s)^{1 - \gamma}} a(s) ds ,
		\]
		for all $t \in [0,T]$, then
		\[
		a(t) \leq 2 b(0) \me^{\tilde \alpha t / \beta} + 2 \int_0^t \left( b'(s) - \frac{\alpha}{\beta} b(s) \right) \me^{\tilde \alpha (t - s) / \beta} ds ,
		\]
		where $\tilde \alpha = \alpha + 2 N^{1/\gamma} (2 / \gamma)^{(1 - \gamma) / \gamma}$, for all $t \in [0,T]$. In the special case that $\beta = \gamma = 1$ and $b(t) = C \me^{\alpha t}$ for some constant $C > 0$, which arises frequently in this work, we have
		\[
		a(t) \leq 2 C \me^{(\alpha + 2N) t} ,
		\]
		for all $t \in [0,T]$.
	\end{lemma}
	
	Finally, we provide an integral identity and a number of basic properties of the \textit{upper incomplete gamma function}
	\begin{equation}
		\label{eq:Gamma_def}
		\Gamma(a,z) := \int_z^\infty s^{a-1} \me^{-s} ds ,
	\end{equation}
	which arises frequently in the proofs in Sections \ref{sec:Proof_of_thm_dynamics} and \ref{sec:Proof_of_thm_error}. We refer to \cite[Ch.~11.2]{Temme1996} for more background and properties on incomplete gamma functions. The asymptotics
	\begin{equation}
		\label{eq:Gamma_asymptotics}
		\Gamma(a,z) \sim \me^{-z} z^{a-1} \left( 1 + O(z^{-1}) \right) , \qquad |z| \to \infty ,
	\end{equation}
	are particularly useful for our purposes. The integral identity presented in the following result is also helpful for many calculations.
	
	\begin{lemma}
		\label{lem:Gamma_integrals}
		Let $\alpha, \beta, \gamma \in \R$ and $t \in [0, |\beta^{-1}|)$. Then
		\begin{equation}
			\label{eq:Gamma_identity}
			Q(t) = \int_0^{t} \frac{\me^{- \alpha s}}{(1 + \beta s)^\gamma} ds = 
			\frac{\me^{ \alpha / \beta}}{\beta} \left( \frac{\beta}{\alpha} \right)^{1-\gamma} \left[ \Gamma \left(1 - \gamma, \frac{\alpha}{\beta}\right) - \Gamma \left(1 - \gamma, \frac{\alpha}{\beta} (1 + \beta t) \right) \right] ,
		\end{equation}
		which is increasing on $t \in [0,|\beta^{-1}|)$.
	\end{lemma}
	
	\begin{proof}
		This can be shown by a direct calculation using the substitution $\eta = \alpha (1 + \beta s) / \beta$ and the definition of $\Gamma(\cdot,\cdot)$ in \eqref{eq:Gamma_def}, or verified by directly differentiating the right-hand side in \eqref{eq:Gamma_identity} and appealing to the fundamental theorem of calculus. The  increasing property follows from the fact that
		\[
		Q'(t) = \frac{\me^{- \alpha t}}{(1 + \beta t)^\gamma} > 0 
		\]
		for all $t \in [0,|\beta^{-1}|)$.
	\end{proof}

\end{document}